\documentclass[a4paper,12pt]{amsart}
\pdfoutput=1
\usepackage{mathmode}
\usepackage{tikz,tikz-cd}
\usepackage[hmargin=1.29in,vmargin=1.37in]{geometry}
\usepackage{graphicx,array,subfigure}
\usepackage{stackengine,scalerel}
\usepackage{amsthm}
\usepackage{enumitem}
\usepackage{mathtools}
 \usepackage{multirow}
\usepackage[T2A]{fontenc}
\usepackage[utf8]{inputenc}
\usepackage{caption}
\usepackage{textcomp}

\NewDocumentCommand{\circledtext}{m}{%
  \raisebox{0.23ex}{%
    \begin{tikzpicture}[baseline=0pt]
      \node[circle,draw,inner sep=0pt,minimum size=0.85em,font=\tiny,line width=0.3pt,anchor=base] (char) {#1};
    \end{tikzpicture}%
  }%
}
\usepackage{pgffor}
\usepackage{float}
\floatstyle{plain} %
\newfloat{diagram}{tbp}{lod} %
\floatname{diagram}{Diagram} %

\edef\strudelccode{\the\catcode`\@}
\catcode`\@=11

\ifx\@gobble\undefined
    \long\def\@gobble#1{}
\fi

\def\m@strip{\expandafter\@gobble\string}

\ifx\@ifnextchar\undefined
    \def\@@ifnextchar{%
        \ifx\@reg@D\@reg@A %
            \expandafter\@reg@B%
        \else%
            \expandafter\@reg@C%
        \fi%
    }

    \long\def\@ifnextchar#1#2#3{%
        \let\@reg@A=#1\relax\def\@reg@B{#2}\def\@reg@C{#3}%
        \futurelet\@reg@D\@@ifnextchar%
    }
\fi

\ifx\@firstoftwo\undefined
    \long\def\@firstoftwo#1#2{#1}
    \long\def\@secondoftwo#1#2{#2}
\fi

\long\def\pdfmsym@afterfi#1#2\fi{\fi#1}

\def\pdfmsym@repeated#1#2{%
    \ifnum#1>0 %
        \pdfmsym@afterfi{#2\expandafter\pdfmsym@repeated\expandafter{\the\numexpr #1-1\relax}{#2}}%
    \fi%
}

\begingroup\lccode`\?=`\p \lccode`\!=`\t %
\lowercase{\endgroup
\def\@ignorept#1.#2?!{\ifnum#2=0 #1\else \ifnum#1=0 \expandafter\@remzero\fi#1.#2\fi}}
\def\@remzero#10{#1}
\def\@nopt#1{\expandafter\@ignorept\the\dimexpr #1\relax}

\def\pdf@literal#1{}

\unless\ifx\saveboxresource\undefined
    \let\pdfxform=\saveboxresource
    \let\pdfrefxform=\useboxresource
    \let\pdflastxform=\lastsavedboxresourceindex
\fi

\def\@normtrans/{0.996264 0 0 0.996264 0 0 cm}
\def\pdfmsymsettransforms{%
    \edef\@pdfmsym@ytrans/{0.996264 0 0 \@nopt{0.996264pt * \@font@scale} 0 0 cm}%
    \edef\@pdfmsym@trans/{\@nopt{0.996264pt * \@font@scale} 0 0 \@nopt{0.996264pt * \@font@scale} 0 0 cm}%
}

\def\pdfmsymsetscalefactor#1{%
    \edef\@font@scale{#1 / 10}%
    \pdfmsymsettransforms%
    \math@sym@defs%
}

\def\@@linehead@type#1#2#3#4[#5][#6]{\hbox to \dimexpr #4pt * #5 * \@font@scale\relax{%
    \pdf@literal{%
        q \@pdfmsym@trans/
        0 #6 0 0 #6 0 0 cm
        0 1 0 0 1 0 #2 cm
        1 j 1 J #1 w
        #3 Q%
    }\hss}%
}
\def\@linehead@type#1#2#3#4{%
    \@ifnextchar[ {\@@linehead@type{#3}{#4}{#1}{#2}}%
                  {\@@linehead@type{#3}{#4}{#1}{#2}[1][1]}%
}

\def\@emptylinehead{\@linehead@type{}{0}}

\def\@rarrow {\@linehead@type{0 0 m 2 0 l 1 .2 0 1.2 0 1.5 c 2 0 m 1 -.2 0 -1.2 0 -1.5 c S}{2}}
\def\@larrow {\@linehead@type{2 0 m 0 0 l 1 .2 2 1.2 2 1.5 c 0 0 m 1 -.2 2 -1.2 2 -1.5 c S}{2}}
\def\@rharp  {\@linehead@type{0 0 m 2 0 l 1 .2 0 1.2 0 1.5 c S}{2}}
\def\@lharp  {\@linehead@type{2 0 m 0 0 l 1 .2 2 1.2 2 1.5 c S}{2}}
\def\@rdharp {\@linehead@type{0 0 m 2 0 l 1 -.2 0 -1.2 0 -1.5 c S}{2}}
\def\@ldharp {\@linehead@type{2 0 m 0 0 l 1 -.2 2 -1.2 2 -1.5 c S}{2}}
\def\@linecap{\@linehead@type{0 0 m 1 0 l S}{1}}
\def\@mapcap {\@linehead@type{0 1.5 m 0 -1.5 l 0 0 m 1 0 l S}{1}}
\def\@mapsfromcap {\@linehead@type{1 1.5 m 1 -1.5 l 1 0 m 0 0 l S}{1}}
\def\@rsarrow{\@linehead@type{0 0 m 2 0 l 0 1 l 2 0 m 0 -1 l S}{2}}
\def\@lsarrow{\@linehead@type{2 0 m 0 0 l 2 1 l 0 0 m 2 -1 l S}{2}}
\def\@backuphook{\@linehead@type{1 2 m .5 2 0 1.5 0 1 c 0 .5 .5 0 1 0 c S}{1}}
\def\@frontuphook{\@linehead@type{0 2 m .5 2 1 1.5 1 1 c 1 .5 .5 0 0 0 c S}{1}}
\def\@backdownhook{\@linehead@type{1 0 m .5 0 0 -.5 0 -1 c 0 -1.5 .5 -2 1 -2 c S}{1}}
\def\@frontdownhook{\@linehead@type{0 0 m .5 0 1 -.5 1 -1 c 1 -1.5 .5 -2 0 -2 c S}{1}}
\def\@doublerarrow{\@linehead@type{0 0 m 2 0 l 1 .2 0 1.2 0 1.5 c 2 0 m 1 -.2 0 -1.2 0 -1.5 c 2 0 m 4 0 l 3 .2 2 1.2 2 1.5 c 4 0 m 3 -.2 2 -1.2 2 -1.5 c S}{4}}
\def\@doublelarrow{\@linehead@type{4 0 m 2 0 l 3 .2 4 1.2 4 1.5 c 2 0 m 3 -.2 4 -1.2 4 -1.5 c 2 0 m 0 0 l 1 .2 2 1.2 2 1.5 c 0 0 m 1 -.2 2 -1.2 2 -1.5 c S}{4}}
\def\@circlecap{\@linehead@type{0 0 m 0 .5 .5 1 1 1 c 1.5 1 2 .5 2 0 c 2 -.5 1.5 -1 1 -1 c .5 -1 0 -.5 0 0 c S}{2}}

\def\@Rarrow    {\@linehead@type{0 -1 m 1 -1 l 0 1 m 1 1 l 3 0 m 1.5 .3 0 2 0 2.5 c 3 0 m 1.5 -.3 0 -2 0 -2.5 c S}{3}}
\def\@Larrow    {\@linehead@type{3 -1 m 2 -1 l 3 1 m 2 1 l 0 0 m 1.5 .3 3 2 3 2.5 c 0 0 m 1.5 -.3 3 -2 3 -2.5 c S}{3}}
\def\@Linecap   {\@linehead@type{0 -1 m 1 -1 l 0 1 m 1 1 l S}{1}}
\def\@Rightcirclecap{\@linehead@type{0 1 m .5 1 1 .5 1 0 c 1 -.5 .5 -1 0 -1 c S}{1}}
\def\@Leftcirclecap{\@linehead@type{1 1 m .5 1 0 .5 0 0 c 0 -.5 .5 -1 1 -1 c S}{1}}
\def\@Rightsquarecap{\@linehead@type{0 1 m 1 1 l 1 -1 l 0 -1 l S}{1}}
\def\@Leftsquarecap{\@linehead@type{1 1 m 0 1 l 0 -1 l 1 -1 l S}{1}}
\def\@Rightribboncap{\@linehead@type{0 1 m 2 1 l 0 0 l 2 -1 l 0 -1 l S}{2}}
\def\@Leftribboncap{\@linehead@type{1.5 1 m 0 1 l 1.5 0 l 0 -1 l 1.5 -1 l S}{1.5}}
\def\@BigLinecap{\@linehead@type{0 -1.5 m 1 -1.5 l 0 1.5 m 1 1.5 l S}{1}}
\def\@Rightarrows{\@linehead@type{0 -1.5 m 2 -1.5 l 1 -1.3 0 -.5 0 0 c 2 -1.5 m 1 -1.7 0 -2.5 0 -3 c 0 1.5 m 2 1.5 l 1 1.7 0 2.5 0 3 c 2 1.5 m 1 1.3 0 .5 0 0 c S}{2}}
\def\@Leftarrows {\@linehead@type{2 -1.5 m 0 -1.5 l 1 -1.3 2 -.5 2 0 c 0 -1.5 m 1 -1.7 2 -2.5 2 -3 c 2 1.5 m 0 1.5 l 1 1.7 2 2.5 2 3 c 0 1.5 m 1 1.3 2 .5 2 0 c S}{2}}
\def\@Rightunderarrow{\@linehead@type{0 -1.5 m 2 -1.5 l 1 -1.3 0 -.5 0 0 c 2 -1.5 m 1 -1.7 0 -2.5 0 -3 c 0 1.5 m 2 1.5 l S}{2}}
\def\@Leftunderarrow{\@linehead@type{2 -1.5 m 0 -1.5 l 1 -1.3 2 -.5 2 0 c 0 -1.5 m 1 -1.7 2 -2.5 2 -3 c 2 1.5 m 0 1.5 l S}{2}}
\def\@Rightoverarrow{\@linehead@type{0 1.5 m 2 1.5 l 1 1.7 0 2.5 0 3 c 2 1.5 m 1 1.3 0 .5 0 0 c 0 -1.5 m 2 -1.5 l S}{2}}
\def\@Leftoverarrow{\@linehead@type{2 1.5 m 0 1.5 l 1 1.7 2 2.5 2 3 c 0 1.5 m 1 1.3 2 .5 2 0 c 2 -1.5 m 0 -1.5 l S}{2}}
\def\@Circlescap{\@linehead@type{0 -1 m 0 -.5 .5 0 1 0 c 1.5 0 2 -.5 2 -1 c 2 -1.5 1.5 -2 1 -2 c .5 -2 0 -1.5 0 -1 c
    0 1 m 0 1.5 .5 2 1 2 c 1.5 2 2 1.5 2 1 c 2 .5 1.5 0 1 0 c .5 0 0 .5 0 1 c S}{2}}

\def\@TripleLinecap{\@linehead@type{0 -2 m 1 -2 l 0 0 m 1 0 l 0 2 m 1 2 l S}{1}}
\def\@TripleRarrow {\@linehead@type{0 -2 m .7 -2 l 0 0 m 3.5 0 l 0 2 m .7 2 l 3.5 0 m 2 .5 .5 2 .2 3.2 c 3.5 0 m 2 -.5 .5 -2 .2 -3.2 c S}{3.5}}
\def\@TripleLarrow {\@linehead@type{3.5 -2 m 2.8 -2 l 3.5 0 m 0 0 l 3.5 2 m 2.8 2 l 0 0 m 1.5 .5 3 2 3.3 3.2 c 0 0 m 1.5 -.5 3 -2 3.3 -3.2 c S}{3.5}}

\def\vecc@w{.4} \def\vecc@hw{.2}  \def\vecc@skew{2.5}
\def\vecc@displaystyle@s{1}       \def\vecc@displaystyle@sf{1}
\def\vecc@textstyle@s{1}          \def\vecc@textstyle@sf{1}
\def\vecc@scriptstyle@s{.8}       \def\vecc@scriptstyle@sf{8 / 10}
\def\vecc@scriptscriptstyle@s{.6} \def\vecc@scriptscriptstyle@sf{6 / 10}
\def\vecc@skip{.4}

\def\@@@vecc@type#1#2#3#4#5#6{%
    \vbox{\offinterlineskip%
        \ialign{##\cr
            $\m@th\mkern\vecc@skew mu%
            #2{\vecc@w}{\vecc@hw}[#4][#5]%
            \cleaders\hrule height \dimexpr \vecc@w pt * #4 * \@font@scale\relax\hfill%
            #3{\vecc@w}{\vecc@hw}[#4][#5]%
            \mkern\vecc@skew mu$\cr\noalign{\kern\dimexpr\vecc@skip ex * #4\relax}
            $\m@th#6#1$\cr%
       }%
    }%
}

\def\@@vecc@type#1#2{\@@@vecc@type#2{\csname vecc@\m@strip#1@sf\endcsname}{\csname vecc@\m@strip#1@s\endcsname}{#1}}
\def\@vecc@type#1#2#3{\mathpalette\@@vecc@type{{#1}{#2}{#3}}}

\def\@@@undervecc@type#1#2#3#4#5#6{%
    \vtop{\offinterlineskip%
        \ialign{##\cr
            $\m@th#6#1$\cr\noalign{\kern\dimexpr\vecc@skip ex * #4\relax}%
            $\m@th\mkern\vecc@skew mu%
            #2{\vecc@w}{\vecc@hw}[#4][#5]%
            \cleaders\hrule height \dimexpr \vecc@w pt * #4 * \@font@scale\relax\hfill%
            #3{\vecc@w}{\vecc@hw}[#4][#5]%
            \mkern\vecc@skew mu$\cr%
       }%
    }%
}

\def\@@undervecc@type#1#2{\@@@undervecc@type#2{\csname vecc@\m@strip#1@sf\endcsname}{\csname vecc@\m@strip#1@s\endcsname}{#1}}
\def\@undervecc@type#1#2#3{\mathpalette\@@undervecc@type{{#1}{#2}{#3}}}

\def\@vecc@def#1#2#3{%
    \expandafter\def\csname #1\endcsname##1{\@vecc@type{##1}{#2}{#3}}%
    \expandafter\def\csname short#1\endcsname##1{{\def\vecc@skew{0}\csname #1\endcsname{##1}}}%
}

\def\@undervecc@def#1#2#3{%
    \expandafter\def\csname #1\endcsname##1{\@undervecc@type{##1}{#2}{#3}}%
    \expandafter\def\csname short#1\endcsname##1{{\def\vecc@skew{0}\csname #1\endcsname{##1}}}%
}

\def\@overunder@def#1#2#3{\@vecc@def{over#1}{#2}{#3}\@undervecc@def{under#1}{#2}{#3}}

\@vecc@def{vecc}\@linecap\@rarrow
\@vecc@def{lvecc}\@larrow\@linecap
\@vecc@def{straightvecc}\@linecap\@rsarrow
\@vecc@def{straightlvecc}\@lsarrow\@linecap
\@undervecc@def{undervecc}\@linecap\@rarrow
\@undervecc@def{underlvecc}\@larrow\@linecap
\@undervecc@def{understraightvecc}\@linecap\@rsarrow
\@undervecc@def{understraightlvecc}\@lsarrow\@linecap

\@overunder@def{rightharp}\@linecap\@rharp
\@overunder@def{leftharp}\@lharp\@linecap
\@overunder@def{leftrightvecc}\@larrow\@rarrow
\@overunder@def{leftrightharp}\@lharp\@rdharp
\@overunder@def{rightleftharp}\@ldharp\@rharp

\unless\ifx\pdfxform\undefined 
    \def\@@constvec#1#2#3{{%
        \setbox1=\hbox{$\m@th#1#3$}%
        \setbox0=\hbox{$\m@th#1\smash{#3}\vrule height1.3ex width0pt depth\dimexpr\dp1+1pt\relax$}%
        \pdfxform0%
        #2{\smash{\pdfrefxform\pdflastxform}\vphantom{x}}%
    }}
    \def\@constvec#1#2{\@@constvec#1#2}
    \def\constvec#1#2{\mathpalette\@constvec{{#1}{#2}}}
\fi

\def\arrow@skip{\mkern1mu}
\def\xarrow@buffer{1}
\def\@@@xarrow@type#1#2#3#4#5#6#7{\mathrel{\arrow@skip%
    \vcenter{\hbox{$\m@th#7%
        #1{\vecc@w}{\vecc@hw}[#5][#6]%
        \vrule width\dimexpr \xarrow@buffer pt * \@font@scale\relax height \dimexpr \vecc@w pt * #5 * \@font@scale\relax depth0pt%
        \smash{\mathord{\mathop{\kern\z@\leaders\hrule height \dimexpr \vecc@w pt * #5 * \@font@scale\relax\hfill}\limits^{#3}_{#4}}}%
        \vrule width\dimexpr \xarrow@buffer pt * \@font@scale\relax height \dimexpr \vecc@w pt * #5 * \@font@scale\relax depth0pt%
        #2{\vecc@w}{\vecc@hw}[#5][#6]%
    $}}%
    \vphantom{\mathop{-}^{#3}_{#4}}
    \arrow@skip%
}}

\def\@@xarrow@type#1#2{\@@@xarrow@type#2{\csname vecc@\m@strip#1@sf\endcsname}{\csname vecc@\m@strip#1@s\endcsname}{#1}}
\def\@xarrow@type#1#2#3#4{\mathpalette\@@xarrow@type{{#1}{#2}{#3}{#4}}}

\def\@@@arrow@type#1#2#3#4#5#6{\mathrel{\arrow@skip%
    \vcenter{\hbox{$\m@th#6%
        #1{\vecc@w}{\vecc@hw}[#4][#5]%
        \vrule width \dimexpr #3pt * #4 * \@font@scale\relax height \dimexpr \vecc@w pt * #4 * \@font@scale\relax depth 0pt%
        #2{\vecc@w}{\vecc@hw}[#4][#5]%
    $}}\arrow@skip%
}}

\def\@@arrow@type#1#2{\@@@arrow@type#2{\csname vecc@\m@strip#1@sf\endcsname}{\csname vecc@\m@strip#1@s\endcsname}{#1}}

\def\@arrow@type#1#2#3{\mathpalette\@@arrow@type{{#1}{#2}{#3}}}

\def\arrow@len{6}
\def\long@arrow@len{12}
\def\@arrow@def#1#2#3{%
    \expandafter\def\csname #1\endcsname{\@arrow@type{#2}{#3}{\arrow@len}}%
    \expandafter\def\csname long#1\endcsname{\@arrow@type{#2}{#3}{\long@arrow@len}}%
    \expandafter\def\csname @x#1\endcsname##1[##2]{\@xarrow@type{#2}{#3}{##1}{##2}}%
    \expandafter\def\csname x#1\endcsname##1{%
        \@ifnextchar[ {\csname @x#1\endcsname{##1}}%
                      {\csname @x#1\endcsname{##1}[]}%
    }%
}

\@arrow@def{varrightarrow}\@linecap\@rarrow
\@arrow@def{varleftarrow}\@larrow\@linecap
\@arrow@def{varrightharp}\@linecap\@rharp
\@arrow@def{varleftharp}\@lharp\@linecap
\@arrow@def{varleftrightarrow}\@larrow\@rarrow
\@arrow@def{varleftrightharp}\@lharp\@rdharp
\@arrow@def{varrightleftharp}\@ldharp\@rharp
\@arrow@def{varmapsto}\@mapcap\@rarrow
\@arrow@def{varmapsfrom}\@larrow\@mapsfromcap
\@arrow@def{varuphookrightarrow}\@backuphook\@rarrow
\@arrow@def{varuphookleftarrow}\@larrow\@frontuphook
\@arrow@def{vardownhookrightarrow}\@backdownhook\@rarrow
\@arrow@def{vardownhookleftarrow}\@larrow\@frontdownhook
\@arrow@def{vardoublerightarrow}\@linecap\@doublerarrow
\@arrow@def{vardoubleleftarrow}\@doublelarrow\@linecap
\@arrow@def{varcirclerightarrow}\@circlecap\@rarrow
\@arrow@def{varcircleleftarrow}\@larrow\@circlecap

\def\@subscriptconvdef#1#2{\expandafter\let\csname \m@strip#1@subconv\endcsname=#2}
\def\@subscriptconv#1{\csname \m@strip#1@subconv\endcsname}
\@subscriptconvdef\displaystyle\scriptstyle
\@subscriptconvdef\textstyle\scriptstyle
\@subscriptconvdef\scriptstyle\scriptscriptstyle
\@subscriptconvdef\scriptscriptstyle\scriptscriptstyle

\def\@@Arrow@rule#1#2#3{%
    \lower\dimexpr #2 * #3 * \@font@scale\relax
    \vbox{%
        \hrule width #1 height \dimexpr \vecc@w pt * #3 * \@font@scale\relax depth 0pt%
        \kern\dimexpr #2 * #3 * \@font@scale * 2 - \vecc@w pt * #3 * \@font@scale\relax%
        \hrule width #1 height \dimexpr \vecc@w pt * #3 * \@font@scale\relax depth 0pt%
    }%
}

\def\@Arrow@rule#1#2#3{%
    \@@Arrow@rule{\dimexpr #1pt * #3 * \@font@scale\relax}{#2}{#3}%
}

\def\@@@xArrow@type#1#2#3#4#5#6#7#8{\mathrel{\arrow@skip%
    {\setbox0=\hbox{$\m@th#8-$}\raise.5\dimexpr\ht0-\dp0-(\vecc@w pt * #6 * \@font@scale)\relax\hbox{$\m@th#8%
        \setbox0=\hbox{$\m@th\@subscriptconv#8#5$}%
        #1{\vecc@w}{\vecc@hw}[#6][#7]%
        \@Arrow@rule{\xarrow@buffer}{#3 pt}{#6}%
        \lower\dimexpr (#3 pt + \vecc@skip ex) * #6 * \@font@scale+\ht0\relax
        \vbox{\offinterlineskip%
            \ialign{\hfil##\hfil\cr
                $\m@th\@subscriptconv#8#4$\cr\noalign{\kern\dimexpr\vecc@skip ex * #6 * \@font@scale\relax}%
                \leaders\hrule height \dimexpr \vecc@w pt * #6 * \@font@scale\relax\hfill\cr%
                \noalign{\kern\dimexpr #3 pt * #6 * \@font@scale * 2 - \vecc@w pt * #6 * \@font@scale\relax}%
                \leaders\hrule height \dimexpr \vecc@w pt * #6 * \@font@scale\relax\hfill\cr\noalign{\kern\dimexpr\vecc@skip ex * #6 * \@font@scale\relax}%
                $\m@th\@subscriptconv#8#5$\cr%
            }%
        }%
        \@Arrow@rule{\xarrow@buffer}{#3 pt}{#6}%
        #2{\vecc@w}{\vecc@hw}[#6][#7]%
    $}}%
    \arrow@skip%
}}

\def\@@xArrow@type#1#2{\@@@xArrow@type#2{\csname vecc@\m@strip#1@sf\endcsname}{\csname vecc@\m@strip#1@s\endcsname}{#1}}
\def\@xArrow@type#1#2#3#4#5{\mathpalette\@@xArrow@type{{#1}{#2}{#3}{#4}{#5}}}

\def\@@@Arrow@type#1#2#3#4#5#6#7{\mathrel{\arrow@skip%
    \vcenter{\hbox{$\m@th#7%
        #1{\vecc@w}{\vecc@hw}[#5][#6]%
        \@Arrow@rule{#3}{#4pt}{#5}%
        #2{\vecc@w}{\vecc@hw}[#5][#6]%
    $}}\arrow@skip%
}}

\def\@@Arrow@type#1#2{\@@@Arrow@type#2{\csname vecc@\m@strip#1@sf\endcsname}{\csname vecc@\m@strip#1@s\endcsname}{#1}}
\def\@Arrow@type#1#2#3#4{\mathpalette\@@Arrow@type{{#1}{#2}{#3}{#4}}}

\def\Arrow@len{5}
\def\long@Arrow@len{9}
\def\@Arrow@def#1#2#3#4{%
    \expandafter\def\csname #1\endcsname{\@Arrow@type{#2}{#3}{\Arrow@len}{#4}}%
    \expandafter\def\csname long#1\endcsname{\@Arrow@type{#2}{#3}{\long@Arrow@len}{#4}}%
    \expandafter\def\csname @x#1\endcsname##1[##2]{\@xArrow@type{#2}{#3}{#4}{##1}{##2}}%
    \expandafter\def\csname x#1\endcsname##1{%
        \@ifnextchar[ {\csname @x#1\endcsname{##1}}%
                      {\csname @x#1\endcsname{##1}[]}%
    }%
}

\@Arrow@def{varRightarrow}\@Linecap\@Rarrow{1}
\@Arrow@def{varLeftarrow}\@Larrow\@Linecap{1}
\@Arrow@def{varCirclerightarrow}\@Leftcirclecap\@Rarrow{1}
\@Arrow@def{varCircleleftarrow}\@Larrow\@Rightcirclecap{1}
\@Arrow@def{varSquarerightarrow}\@Leftsquarecap\@Rarrow{1}
\@Arrow@def{varSquareleftarrow}\@Larrow\@Rightsquarecap{1}
\@Arrow@def{varRibbonrightarrow}\@Leftribboncap\@Rarrow{1}
\@Arrow@def{varRibbonleftarrow}\@Larrow\@Rightribboncap{1}
\@Arrow@def{roundedarrow}\@Leftcirclecap\@Rightcirclecap{1}
\@Arrow@def{squaredarrow}\@Leftsquarecap\@Rightsquarecap{1}
\@Arrow@def{varrightarrows}\@BigLinecap\@Rightarrows{1.5}
\@Arrow@def{varleftarrows}\@Leftarrows\@BigLinecap{1.5}
\@Arrow@def{varrightleftarrows}\@Leftunderarrow\@Rightoverarrow{1.5}
\@Arrow@def{varleftrightarrows}\@Leftoverarrow\@Rightunderarrow{1.5}
\@Arrow@def{rightPP}\@Circlescap\@Rightcirclecap{1}
\@Arrow@def{leftPP}\@Leftcirclecap\@Circlescap{1}

\def\@multi@Arrow@rule#1#2#3#4{%
    \lower\dimexpr #3 * #4 * \@font@scale * (#1 - 1)\relax
    \vbox{%
        \pdfmsym@repeated{\numexpr #1 - 1\relax}{%
            \hrule width \dimexpr #2pt * #4 * \@font@scale\relax height \dimexpr \vecc@w pt * #4 * \@font@scale\relax depth\z@%
            \kern\dimexpr #3 * #4 * \@font@scale * 2 - \vecc@w pt * #4 * \@font@scale\relax%
        }%
        \hrule width \dimexpr #2pt * #4 * \@font@scale\relax height \dimexpr \vecc@w pt * #4 * \@font@scale\relax depth\z@%
    }%
}

\def\@@@multi@xArrow@type#1#2#3#4#5#6#7#8#9{\mathrel{\arrow@skip%
    {\setbox0=\hbox{$\m@th#9-$}\raise.5\dimexpr\ht0-\dp0-(\vecc@w pt * #7 * \@font@scale)\relax%
    \hbox{$\m@th#9%
        \setbox0=\hbox{$\m@th\@subscriptconv#9#6$}%
        #2{\vecc@w}{\vecc@hw}[#7][#8]%
        \@multi@Arrow@rule{#1}{\xarrow@buffer}{#4pt}{#7}%
        \lower\dimexpr (#4pt * (#1 - 1) + \vecc@skip ex) * #7 * \@font@scale +\ht0\relax
        \vbox{\offinterlineskip%
            \ialign{\hfil##\hfil\cr%
                $\m@th\@subscriptconv#9#5$\cr%
                \noalign{\kern\dimexpr\vecc@skip ex * #7 * \@font@scale\relax}%
                \pdfmsym@repeated{\numexpr #1-1\relax}{%
                    \leaders\hrule height \dimexpr \vecc@w pt * #7 * \@font@scale\relax\hfill\cr%
                    \noalign{\kern\dimexpr #4pt * #7 * \@font@scale * 2 - \vecc@w pt * #7 * \@font@scale\relax}%
                }%
                \leaders\hrule height \dimexpr \vecc@w pt * #7 * \@font@scale\relax\hfill\cr
                \noalign{\kern\dimexpr\vecc@skip ex * #7 * \@font@scale\relax}%
                $\m@th\@subscriptconv#9#6$\cr%
            }%
        }%
        \@multi@Arrow@rule{#1}{\xarrow@buffer}{#4pt}{#7}%
        #3{\vecc@w}{\vecc@hw}[#7][#8]%
    $}}%
    \arrow@skip%
}}

\def\@@multi@xArrow@type#1#2{\@@@multi@xArrow@type#2{\csname vecc@\m@strip#1@sf\endcsname}{\csname vecc@\m@strip#1@s\endcsname}{#1}}
\def\@multi@xArrow@type#1#2#3#4#5#6{\mathpalette\@@multi@xArrow@type{{#1}{#2}{#3}{#4}{#5}{#6}}}

\def\@@@multi@Arrow@type#1#2#3#4#5#6#7#8{\mathrel{\arrow@skip%
    \vcenter{\hbox{$\m@th#8%
        #2{\vecc@w}{\vecc@hw}[#6][#7]%
        \@multi@Arrow@rule{#1}{#4}{#5pt}{#6}%
        #3{\vecc@w}{\vecc@hw}[#6][#7]%
    $}}\arrow@skip%
}}

\def\@@multi@Arrow@type#1#2{\@@@multi@Arrow@type#2{\csname vecc@\m@strip#1@sf\endcsname}{\csname vecc@\m@strip#1@s\endcsname}{#1}}
\def\@multi@Arrow@type#1#2#3#4#5{\mathpalette\@@multi@Arrow@type{{#1}{#2}{#3}{#4}{#5}}}

\def\@multi@Arrow@def#1#2#3#4#5{%
    \expandafter\def\csname #1\endcsname{\@multi@Arrow@type{#2}{#3}{#4}{\Arrow@len}{#5}}%
    \expandafter\def\csname long#1\endcsname{\@multi@Arrow@type{#2}{#3}{#4}{\long@Arrow@len}{#5}}%
    \expandafter\def\csname @x#1\endcsname##1[##2]{\@multi@xArrow@type{#2}{#3}{#4}{#5}{##1}{##2}}%
    \expandafter\def\csname x#1\endcsname##1{%
        \@ifnextchar[ {\csname @x#1\endcsname{##1}}%
                      {\csname @x#1\endcsname{##1}[]}%
    }%
}

\@multi@Arrow@def{varRrightarrow}{3}\@TripleLinecap\@TripleRarrow{1}
\@multi@Arrow@def{varLleftarrow}{3}\@TripleLarrow\@TripleLinecap{1}
\@multi@Arrow@def{varLleftRrightarrow}{3}\@TripleLarrow\@TripleRarrow{1}

\def\accent@skew{.4}
\def\accent@raise{.25}
\def\@@@wide@accent#1#2#3{{%
    \setbox0=\hbox{$\m@th#3#1$}%
    \vbox{\offinterlineskip%
        \ialign{##\cr
            \hbox to \wd0{\hskip\dimexpr \accent@skew ex * \@font@scale\relax%
                \pdf@literal{%
                    q \@pdfmsym@ytrans/
                    \@nopt{\wd0 - \accent@skew ex * \@font@scale} 0 0 1 0 -.5 cm
                    #2
                Q}%
                \hss%
            }\cr\noalign{\kern\accent@raise ex\relax}%
            \box0\cr%
        }%
    }%
}}

\def\@@wide@accent#1#2{\@@@wide@accent#2#1}
\def\@wide@accent#1#2{\mathpalette\@@wide@accent{{#2}{#1}}}

\def\varwidecheck{\@wide@accent{0 1.3 m .5 -.4 l 1 1.3 l 1 1.6 l .5 .3 l 0 1.6 l f}}
\def\varwidehat{\@wide@accent{0 0 m .5 .9 l 1 0 l 1 .3 l .5 1.6 l 0 .3 l f}}
\def\varwidetilde{\@wide@accent{0 0 m .25 1.5 .45 1.1 .5 1 c .55 .9 .75 0 1 1.5 c 1 1.75 l .75 .5 .55 1.4 .5 1.5 c .45 1.6 .25 2 0 .25 c f}}

\def\@skewedlim@op@done#1#2#3#4#5{#1^{\mkern-#2mu#4}_{\mkern-#3mu#5}}
\def\@skewedlim@op@supsub#1#2#3#4_#5{\@skewedlim@op@done{#1}{#2}{#3}{#4}{#5}}
\def\@skewedlim@op@subsup#1#2#3#4^#5{\@skewedlim@op@done{#1}{#2}{#3}{#5}{#4}}
\def\@skewedlim@op@sub#1#2#3#4#5{%
    \@ifnextchar^ {\@skewedlim@op@subsup{#1}{#2}{#3}{#5}}%
                  {\@skewedlim@op@done{#1}{#2}{#3}{}{#5}}%
}
\def\@skewedlim@op@sup#1#2#3#4#5{%
    \@ifnextchar_ {\@skewedlim@op@supsub{#1}{#2}{#3}{#5}}%
                  {\@skewedlim@op@done{#1}{#2}{#3}{#5}{}}%
}
\def\@skewedlim@op@nsup#1#2#3{%
    \@ifnextchar_ {\@skewedlim@op@sub{#1}{#2}{#3}}%
                  {\@skewedlim@op@done{#1}{#2}{#3}{}{}}%
}
\def\@@skewedlim@op#1#2#3{%
    \@ifnextchar^ {\@skewedlim@op@sup{#1}{#2}{#3}}%
                  {\@skewedlim@op@nsup{#1}{#2}{#3}}%
}

\def\@skewedlim@op@limits{%
    \ifx\@skewedlim@char\nolimits%
        \pdfmsym@afterfi{\@@skewedlim@op{\@skewedlim@atom\nolimits}{\@skewedlim@nolim@supskew}{\@skewedlim@nolim@subskew}}%
    \else%
        \pdfmsym@afterfi{\ifx\@skewedlim@char\limits%
            \pdfmsym@afterfi{\@@skewedlim@op{\@skewedlim@atom\limits}{\@skewedlim@lim@supskew}{\@skewedlim@lim@subskew}}%
        \else%
            \pdfmsym@afterfi{\@@skewedlim@op{\@skewedlim@atom\@skewedlim@preflim}{\@skewedlim@preflim@supskew}{\@skewedlim@preflim@subskew}\@skewedlim@char}%
        \fi}%
    \fi%
}

\def\@skewedlim@op#1#2#3#4#5#6#7#8{%
    \def\@skewedlim@atom{#1}%
    \def\@skewedlim@nolim@supskew{#2}\def\@skewedlim@nolim@subskew{#3}%
    \def\@skewedlim@lim@supskew{#4}\def\@skewedlim@lim@subskew{#5}%
    \def\@skewedlim@preflim@supskew{#6}\def\@skewedlim@preflim@subskew{#7}%
    \def\@skewedlim@preflim{#8}%
    \afterassignment\@skewedlim@op@limits%
    \let\@skewedlim@char=%
}

\def\exsym@displaystyle@s{1}                \def\exsym@displaystyle@sf{1}
\def\exsym@textstyle@s{.7}                  \def\exsym@textstyle@sf{7 / 10}
\def\exsym@scriptstyle@s{.7}                \def\exsym@scriptstyle@sf{7 / 10}
\def\exsym@scriptscriptstyle@s{.6}          \def\exsym@scriptscriptstyle@sf{6 / 10}
\def\@@putexsym#1#2#3#4#5#6#7#8{{%
    \setbox0=\hbox{$\m@th#8#1$}%
    \setbox1=\hbox{$\m@th#2{\vecc@w}{\vecc@hw}[#6][#7]$}%
    \kern\dimexpr\wd1 - #5 pt * #6 * \@font@scale\relax\rlap{$\m@th#8#1$}\kern-\dimexpr\wd1 - #5pt * #6 * \@font@scale\relax%
    \raise.5\dimexpr\ht0-\dp0\relax
    \vbox{\hbox{$\m@th#8#2{\vecc@w}{\vecc@hw}[#6][#7]\@@Arrow@rule{\dimexpr\wd0- #5 pt * #6 * \@font@scale * 2\relax}{#4 pt}{#6}#3{\vecc@w}{\vecc@hw}[#6][#7]$}}%
}}

\def\@putexsym#1#2{\@@putexsym#2{\csname exsym@\m@strip#1@sf\endcsname}{\csname exsym@\m@strip#1@s\endcsname}{#1}}
\def\putexsym#1#2#3#4#5{\mathpalette\@putexsym{{#1}{#2}{#3}{#4}{#5}}}

\def\@BigRightcirclecap{\@linehead@type{0 2.5 m 1 2.5 2 1.25 2 0 c 2 -1.25 1 -2.5 0 -2.5 c S}{2}}
\def\@BigLeftcirclecap{\@linehead@type {2 2.5 m 1 2.5 0 1.25 0 0 c 0 -1.25 1 -2.5 2 -2.5 c S}{2}}
\def\@BigRightsquarecap{\@linehead@type {0 2.5 m 1 2.5 l 1 -2.5 l 0 -2.5 l S}{1}}
\def\@BigLeftsquarecap{\@linehead@type{1 2.5 m 0 2.5 l 0 -2.5 l 1 -2.5 l S}{1}}
\def\@BigCirclescap{\@linehead@type{0 -2 m 0 -1 1 0 2 0 c 3 0 4 -1 4 -2 c 4 -3 3 -4 2 -4 c 1 -4 0 -3 0 -2 c
    0 2 m 0 3 1 4 2 4 c 3 4 4 3 4 2 c 4 1 3 0 2 0 c 1 0 0 1 0 2 c S}{4}}

\def\iNint@kern@{\mkern-10mu\mathchoice{\mkern-5mu}{}{}{}}
\def\@iNint#1{\pdfmsym@repeated{\numexpr #1-1\relax}{\int\iNint@kern@}\int}
\def\@oiNint#1{\putexsym{\@iNint{#1}}\@BigLeftcirclecap\@BigRightcirclecap{2.5}{4}}
\def\@biNint#1{\putexsym{\@iNint{#1}}\@BigLeftsquarecap\@BigRightsquarecap{2.5}{4}}

\def\oiNint#1{\@skewedlim@op{\mathop{\@oiNint{#1}}}{-4}{6}{-20}{20}{-4}{6}\nolimits}
\def\biNint#1{\@skewedlim@op{\mathop{\@biNint{#1}}}{-6}{6}{-20}{20}{-6}{6}\nolimits}
\def\iNint#1{\@skewedlim@op{\mathop{\@iNint{#1}}}{0}{0}{-20}{20}{0}{13}\nolimits}

\def\@@putsym#1#2#3{{\setbox0=\hbox{$\m@th#1#2$}\rlap{\hbox to \wd0{\hss$\m@th#1#3$\hss}}}}
\def\@putsym#1#2{\@@putsym#1#2}
\def\putsym#1#2{\mathpalette\@putsym{{#1}{#2}}\mathopen{}#1}

\long\def\@pdf@drawing@macro@def#1#2#3#4#5#6{%
    \hbox to \dimexpr#3 * \@font@scale\relax{\hskip\dimexpr#6 * \@font@scale\relax%
    \vrule width0pt height0pt depth\dimexpr#5 * \@font@scale\relax%
    \vbox to \dimexpr#4 * \@font@scale\relax{\vss\pdf@literal{q \@pdfmsym@trans/ #2 Q}}\hss}%
}

\unless\ifx\pdfxform\undefined 
    \long\def\pdf@drawing@macro@def#1#2#3#4#5#6{{%
        \setbox0=\@pdf@drawing@macro@def{#1}{#2}{#3}{#4}{#5}{#6}%
        \pdfxform0%
        \expandafter\xdef\csname @#1@xformno\endcsname{\the\pdflastxform}%
    }}

    \long\def\pdf@drawing@macro#1#2#3#4#5#6{%
        \pdf@drawing@macro@def{#1}{#2}{#3}{#4}{#5}{#6}%
        \expandafter\edef\csname #1\endcsname{\noexpand\leavevmode\noexpand\pdfrefxform\csname @#1@xformno\endcsname}%
    }

    \def\math@drawing@get#1#2{\expandafter\pdfrefxform\csname @#1@\m@strip#2@xformno\endcsname}
\else
    \long\def\pdf@drawing@macro#1#2#3#4#5#6{%
        \expandafter\def\csname #1\endcsname{\@pdf@drawing@macro@def{#1}{#2}{#3}{#4}{#5}{#6}}%
    }
    \let\pdf@drawing@macro@def=\pdf@drawing@macro

    \def\math@drawing@get#1#2{\csname #1@\m@strip#2\endcsname}
\fi

\long\def\pdf@drawing@math@macro#1#2#3#4#5#6#7#8#9{%
    \pdf@drawing@macro@def{#1@displaystyle}{#2}{#3}{#4}{#5}{#6}%
    \pdf@drawing@macro@def{#1@textstyle}
        {\@secondoftwo#7 0 0 \@secondoftwo#7 0 0 cm #2}
        {#3 * \@firstoftwo#7}
        {#4 * \@firstoftwo#7}
        {#5 * \@firstoftwo#7}
        {#6 * \@firstoftwo#7}%
    \pdf@drawing@macro@def{#1@scriptstyle}
        {\@secondoftwo#8 0 0 \@secondoftwo#8 0 0 cm #2}
        {#3 * \@firstoftwo#8}
        {#4 * \@firstoftwo#8}
        {#5 * \@firstoftwo#8}
        {#6 * \@firstoftwo#8}%
    \pdf@drawing@macro@def{#1@scriptscriptstyle}
        {\@secondoftwo#9 0 0 \@secondoftwo#9 0 0 cm #2}
        {#3 * \@firstoftwo#9}
        {#4 * \@firstoftwo#9}
        {#5 * \@firstoftwo#9}
        {#6 * \@firstoftwo#9}%
    \expandafter\def\csname @#1@\endcsname##1##2{\math@drawing@get{#1}{##1}}%
    \expandafter\def\csname #1\endcsname{\expandafter\mathpalette\csname @#1@\endcsname{}}%
    \expandafter\def\csname @center@#1@\endcsname##1##2{\vcenter{\hbox{\math@drawing@get{#1}{##1}}}}%
    \expandafter\def\csname center@#1\endcsname{\expandafter\mathpalette\csname @center@#1@\endcsname{}}%
}

\def\math@sym@defs{%
    \pdf@drawing@macro{lightning}
        {.86603 -.5 .5 .86603 0 0 cm
        1 J 1 j .6 w
        -3 10 m -3 4.133975 l 0 5.866025 l 0 0 l -1.125 1.5 l 0 0 l 1.125 1.5 l S}
        {4.2pt}{10.5pt}{.5pt}{.9pt}
    \pdf@drawing@macro{(m<3y)}
{.2 w 0 0 0 RG q 1 .75 0 rg 1 J 1 j 12.5 18 m  16 22 l  16.35 22.3 17 19
 17.5 13 c  7.5 18 m  4 22 l  3.65 22.3 3 19 2.5 13 c  B Q q 1 .5 .5 rg 1 J 1 j
 12.2 16.7 m  15.7 20.7 l  16.05 21 16.7 17.7 17.2 11.7 c  7.8 16.7 m  4.3 20.7
 l  3.95 21 3.3 17.7 2.8 11.7 c  f Q q 19.0 10 m 19.0 14.97063 14.97063 19.0 10
 19.0 c 5.02937 19.0 1.0 14.97063 1.0 10 c 1.0 5.02937 5.02937 1.0 10 1.0 c 14.97063 
1.0 19.0 5.02937 19.0 10 c  W n q 1 .75 0 rg 19.0 10 m 19.0 14.97063 14.97063 
19.0 10 19.0 c 5.02937 19.0 1.0 14.97063 1.0 10 c 1.0 5.02937 5.02937 1.0 
10 1.0 c 14.97063 1.0 19.0 5.02937 19.0 10 c  f Q q 1 1 1 rg 7.75 12 m 7.75 12.82843 
6.9665 13.5 6 13.5 c 5.0335 13.5 4.25 12.82843 4.25 12 c 4.25 11.17157 5.0335 
10.5 6 10.5 c 6.9665 10.5 7.75 11.17157 7.75 12 c  15.75 12 m 15.75 12.82843 
14.9665 13.5 14 13.5 c 13.0335 13.5 12.25 12.82843 12.25 12 c 12.25 11.17157
 13.0335 10.5 14 10.5 c 14.9665 10.5 15.75 11.17157 15.75 12 c  B Q q .3 .1 0 rg
7.4 12 m 7.4 12.7732 6.7732 13.4 6 13.4 c 5.2268 13.4 4.6 12.7732 4.6 12 c 4.6
11.2268 5.2268 10.6 6 10.6 c 6.7732 10.6 7.4 11.2268 7.4 12 c  15.4 12 m 15.4
 12.7732 14.7732 13.4 14 13.4 c 13.2268 13.4 12.6 12.7732 12.6 12 c 12.6 11.2268 
13.2268 10.6 14 10.6 c 14.7732 10.6 15.4 11.2268 15.4 12 c  B Q q 0 0 0 rg 7.2 
12 m 7.2 12.66273 6.66273 13.2 6 13.2 c 5.33727 13.2 4.8 12.66273 4.8 12 c 4.8 
11.33727 5.33727 10.8 6 10.8 c 6.66273 10.8 7.2 11.33727 7.2 12 c  15.2 12 m 
15.2 12.66273 14.66273 13.2 14 13.2 c 13.33727 13.2 12.8 12.66273 12.8 12 c 12.8 
11.33727 13.33727 10.8 14 10.8 c 14.66273 10.8 15.2 11.33727 15.2 12 c  f Q q
 1 1 1 rg 1 j 8 21 m 9.5 15 9.5 10 v  9.5 9 8 6.5 v  7 3.5 10 3.75 v  13 
3.5 12 6.5 v  10.5 9 10.5 10 v  10.5 15 12 21 v  B Q q 0 0 0 rg 12.3 6.5 m 12.3
 7.32843 11.27026 8.0 10 8.0 c 8.72974 8.0 7.7 7.32843 7.7 6.5 c 7.7 5.67157 8.72974 
5.0 10 5.0 c 11.27026 5.0 12.3 5.67157 12.3 6.5 c  B Q q 1 j 1 J 10 6.5 m
  10 4.5 l  9.5 4 9 4 8.3 4.2 c  10 4.5 m  10.5 4 11 4 11.7 4.2 c  S Q 8.5 10 m
  8 7 l  11.5 10 m  12 7 l  S Q 19.0 10 m 19.0 14.97063 14.97063 19.0 10 19.0 c
 5.02937 19.0 1.0 14.97063 1.0 10 c 1.0 5.02937 5.02937 1.0 10 1.0 c 14.97063 1.0 
19.0 5.02937 19.0 10 c  S q 1 1 1 rg 1 j 1 J 7 1.5 m  8 0 10 -2 v  11 0 13 1.5
v  11.5 1.2 10.5 1.2 10 1.2 c  7 1.5 m  8.5 1.2 9.5 1.2 10 1.2 c  B* Q}
        {20pt}{22.4pt}{2.30pt}{0pt}
    \pdf@drawing@math@macro{@ndivs}
        {0.4 w 1 j
        0 3 m 5 7 l s
        1.3 w
        2.5 1 0 .1 re
        2.5 5 0 .1 re
        2.5 9 0 .1 re B}
        {5.4pt}{10pt}{0pt}{.2pt}
        {{1}{1}}{{7 / 10}{.7}}{{11 / 20}{.55}}
    \def\ndivs{\mathrel{\@ndivs}}
    \pdf@drawing@math@macro{@divs}
        {1.3 w 1 j
        2.5 1 0 .1 re
        2.5 5 0 .1 re
        2.5 9 0 .1 re B}
        {5.4pt}{10pt}{0pt}{.2pt}
        {{1}{1}}{{7 / 10}{.7}}{{11 / 20}{.55}}
    \def\divs{\mathrel{\@divs}}

    \pdf@drawing@math@macro{@bigforall}
        {.8 w 1 J 1 j
        0 9.5 m 6 -4 l 12 9.5 l
        2.8 4 m 9.2 4 l S}
        {14pt}{10.8pt}{4.8pt}{1pt}
        {{7 / 10}{.7}}{{5 / 10}{.5}}{{35 / 100}{.35}}
    \def\bigforall{\mathop{\@bigforall}}

    \pdf@drawing@math@macro{@bigexists}
        {.8 w 1 J 1 j
        0 9.5 m 8 9.5 l 8 -4.5 l 0 -4.5 l
        1 2.5 m 7.9 2.5 l S}
        {10pt}{10.8pt}{4.8pt}{1pt}
        {{7 / 10}{.7}}{{5 / 10}{.5}}{{35 / 100}{.35}}
    \def\bigexists{\mathop{\@bigexists}}

    \pdf@drawing@math@macro{@smallcircle}
        {.3 w 0 .5 m 0 .75 .25 1 .5 1 c .75 1 1 .75 1 .5 c 1 .25 .75 0 .5 0 c .25 0 0 .25 0 .5 c S}
        {1.6pt}{1.6pt}{.3pt}{.3pt}
        {{1}{1}}{{7 / 10}{.7}}{{6 / 10}{.6}}
    \let\smallcircle=\center@@smallcircle
}

\unless\ifx\pdfxform\undefined
    \def\@@slice#1{%
        \pdfxform\csname #1\endcsname%
        \global\setbox\csname #1\endcsname=\hbox{\pdfrefxform\pdflastxform}%
    }

    \def\@inner@newbox{\alloc@ 4\box \chardef \insc@unt} 

    \def\@slice#1#2#3#4{{%
        \setbox0=\hbox{#1}%
        \expandafter\@inner@newbox\csname #2@L\endcsname%
        \expandafter\@inner@newbox\csname #2@C\endcsname%
        \expandafter\@inner@newbox\csname #2@R\endcsname%
        \global\setbox\csname #2@L\endcsname=\hbox{\copy0 \kern\dimexpr #3\wd0 - \wd0\relax}%
        \global\setbox\csname #2@C\endcsname=\hbox{\kern-#3\wd0 \copy0 \kern\dimexpr #4\wd0 - \wd0\relax}%
        \global\setbox\csname #2@R\endcsname=\hbox{\kern-#4\wd0 \copy0}%
        \@@slice{#2@L}\@@slice{#2@C}\@@slice{#2@R}%
    }}

    \def\@show@slices#1{%
        \hbox{\copy\csname #1@L\endcsname\vrule width .1pt
        \copy\csname #1@C\endcsname\vrule width .1pt
        \copy\csname #1@R\endcsname}
    }

    \def\@wide@operator#1#2#3#4{%
        \@slice{$\m@th\displaystyle#2$}{#1}{#3}{#4}%
        \expandafter\def\csname #1\endcsname{%
            \mathop{\copy\csname #1@L\endcsname \xleaders\copy\csname #1@C\endcsname\hfill \copy\csname #1@R\endcsname}\limits%
        }%
    }

    \@wide@operator{suum}\sum{.52}{.6}
    \@wide@operator{prood}\prod{.48}{.52}
\fi

\catcode`\@=\strudelccode

\pdfmsymsetscalefactor{11}
\makeatletter
\renewcommand{\thediagram}{\Alph{diagram}} %
\renewcommand{\fnum@diagram}{Diagram~\thediagram.}

\makeatother
\graphicspath{ {./figures/} }

\let\originalwidehat\widehat
\makeatletter
\newcommand{\autowidehat}[1]{%
  \settowidth{\@tempdima}{$#1$}%
  \ifdim\@tempdima > 15pt\relax
    \varwidehat{#1}%
  \else
    \originalwidehat{#1}%
  \fi
}
\makeatother

\renewcommand{\widehat}[1]{\autowidehat{#1}}

\def\H{\mathbf{H}}
\def\bfC{\mathbf{C}}
\def\T{\mathbf{T}}
\def\R{\mathbb{R}}
\def\N{\mathbb{N}}
\def\Fp{\mathbb{F}_p}
\def\G{\mathcal{G}}
\def\X{\mathbf{X}}

\def\Q{\mathbb{Q}}
\def\C{\mathbb{C}}
\def\GL{\mathrm{GL}}
\def\SL{\mathrm{SL}}
\def\PSL{\mathrm{PSL}}
\def\Th#1{\left \| {#1} \right \|_{\mathrm{Th}}}
\def\CC{\mathfrak{C}}
\def\FF{\mathfrak{F}}
\def\FFp{\mathfrak{F}_p}
\def\PD{\mathrm{PD}}
\def\tr{\operatorname{tr}}

\def\Ab#1{{#1}^{\mathrm{Ab}}}
\def\ab#1{{#1}^{\mathrm{ab}}}

\def\Hom{\mathrm{Hom}}
\def\Aut{\mathrm{Aut}}
\def\Out{\mathrm{Out}}
\def\Inn{\mathrm{Inn}}
\def\Stab{\mathrm{Stab}}
\def\stab{\Stab}

\def\limi{\varprojlim}

\def\tto{\longrightarrow}

\def\Z{\mathbb{Z}}

\def\Zx{\widehat{\Z}^{\times}}
\def\abs#1{\left |#1 \right |}
\def\Cone{\mathcal{C}}
\def\CFr{\mathcal{C}_{\mathrm{Fr}}}
\def\CTh{\mathcal{C}_{\mathrm{Th}}}
\def\BTh{\mathcal{B}_{\mathrm{Th}}}

\def\psiA{\psi_{\mathrm{A}}}
\def\psiF{\psi_{\mathrm{F}}}
\def\psiAyi{\psi_{\mathrm{A},1}}
\def\psiAer{\psi_{\mathrm{A},2}}
\def\psiFyi{\psi_{\mathrm{F},1}}
\def\psiFer{\psi_{\mathrm{F},2}}

\def\traj{\mathfrak{T}}
\def\Traj{\traj}
\def\Mod{\mathrm{Mod}}

\def\tensor{\widehat{\mathbb{Z}}{\otimes}_{\mathbb{Z}}}
\def\tens{\widehat{\mathbb{Z}}\otimes}
\def\Int{\operatorname{int}}
\def\piorb{\pi_1^{\mathrm{orb}}}

\tikzset{%
  symbol/.style={
    draw=none,
    every to/.append style={
      edge node={node [sloped, allow upside down, auto=false]{$#1$}}
    },
  },
}
\tikzcdset{
  display/.style={
    cells={
     /tikz/font=\everymath\expandafter{\the\everymath\displaystyle},
    },
  },
}

\stackMath
\newcommand\reallywidehat[1]{%
\savestack{\tmpbox}{\stretchto{%
  \scaleto{%
    \scalerel*[\widthof{\ensuremath{#1}}]{\kern.1pt\mathchar"0362\kern.1pt}%
    {\rule{0ex}{\textheight}}
  }{\textheight}%
}{2.4ex}}%
\stackon[-6.9pt]{#1}{\tmpbox}%
}

\newtheorem*{cor3mfd}{\autoref{maincor}}

\makeatletter
\newcommand*\nss[3]{%
  \begingroup
  \setbox0\hbox{$\m@th\scriptstyle\cramped{#2}$}%
  \setbox2\hbox{$\m@th\scriptstyle#3$}%
  \dimen@=\fontdimen8\textfont3
  \multiply\dimen@ by 4             
  \advance \dimen@ by \ht0
  \advance \dimen@ by -1.7\fontdimen17\textfont2
  \@tempdima=\fontdimen5\textfont2  
  \multiply\@tempdima by 4
  \divide  \@tempdima by 5          
  \ifdim\dimen@<\@tempdima
    \ht0=0pt                        
    \@tempdima=\fontdimen5\textfont2
    \divide\@tempdima by 4          
    \advance \dimen@ by -\@tempdima 
    \ifdim\dimen@>0pt
      \@tempdima=\dp2
      \advance\@tempdima by \dimen@
      \dp2=\@tempdima
    \fi
  \fi
  #1_{\box0}^{\box2}%
  \endgroup
}
\makeatother

\date{\today}

\author{Xiaoyu Xu}
\address{Beijing International Center for Mathematical Research, Peking University, Beijing, P.R.China, 100871}
\email{xuxiaoyu@stu.pku.edu.cn}

\title[Profinite rigidity in lattices of $\mathrm{PSL}(2,\mathbb{C})$]{Profinite rigidity in lattices of $\mathrm{PSL}(2,\mathbb{C})$}


\begin{document}

\maketitle

\begin{abstract}
A finitely generated group is profinitely rigid among a class of finitely generated groups if it can be distinguished among this class by its set of finite quotient groups. This paper proves that all lattices in $\mathrm{PSL}(2,\mathbb{C})$ are profinitely rigid among themselves. In addition, for any lattice $\Gamma\le \mathrm{PSL}(2,\mathbb{C})$, it is proven that $\mathrm{Out}(\widehat{\Gamma})\cong \mathrm{Out}(\Gamma)$, where $\widehat{\Gamma}$ denotes the profinite completion of $\Gamma$.  
\end{abstract}

\setcounter{tocdepth}{1}
\tableofcontents 

\section{Introduction}
It is an old and natural question that to what extent a finitely generated residually finite group can be recovered from its finite quotient groups. For a finitely generated group $\Gamma$, we denote
$$
\mathcal{C}(\Gamma)=\{ G\mid G\text{ is a finite quotient group of }\Gamma\},
$$
where $\mathcal{C}(\Gamma)$ is considered  merely  as a set of isomorphism classes of finite groups. 
 
\begin{definition}
Suppose $\Gamma$ belongs to a class $\mathscr{A}$ of finitely generated groups. We say that $\Gamma$ is {\em profinitely rigid} among the class $\mathscr{A}$ if for any $\Delta\in \mathscr{A}$,   $\mathcal{C}(\Delta)=\mathcal{C}(\Gamma)$ implies that $\Delta\cong \Gamma$. 
\end{definition}

For instance, finitely generated abelian groups are profinitely rigid {\em in the absolute sense} \cite[Proposition 3.1]{Rei18}, i.e.\ profinitely rigid among the class of all finitely generated residually finite groups. However, the question of profinite rigidity becomes considerably difficult when the groups being concerned become complicated. For instance, it is still an open question whether finitely generated non-abelian free groups are profinitely rigid in the absolute sence \cite[Question 15]{NRR82}. 

The question becomes more approachable when we restrict ourselves to some specific class of groups arising from low-dimensional topology, where a number of  topological and geometric techniques come into play. In this paper, we focus on lattices in $\PSL_2(\C)$, which are the fundamental groups of orientable complete finite-volume hyperbolic 3-orbifolds.

\begin{mainthm}\label{mainthm1}
Any lattice $\Gamma$ in $\PSL_2(\C)$ is profinitely rigid among the class of  all lattices in $\PSL_2(\C)$.
\end{mainthm}
Bridson and Reid have long conjectured that lattices in $\PSL_2(\C)$ are profinitely rigid in the absolute sense,  e.g.\ \cite[Conjecture 18.2.1]{Bri23} and \cite[Question 4.3]{Rei18}. \autoref{mainthm1} can be seen as a new positive step towards their conjecture. 

Previously, Liu \cite{Liu23} proved that at most finitely many lattices in $\PSL_2(\C)$ share the same set of finite quotient groups. Some infinite families of non-uniform lattices were proven to be profinitely rigid among lattices in $\PSL_2(\C)$, utilizing their specific topological features \cite{BR20,CWX26,Xu24}. Furthermore, a finite list of arithmetic lattice in $\PSL_2(\C)$  were confirmed to be absolutely profinitely rigid \cite{BMRS20,BR22,CW24}, including $\PSL_2(\Z[\omega])$ and the fundamental group of the Weeks manifold, etc. 

The phenomenon of profinite rigidity varies among lattices of Lie groups. Bridson--Conder--Reid \cite{BCR16} proved that lattices of $\PSL_2(\R)$ are profinitely rigid among all lattices of connected Lie groups.  
Stover \cite{Sto19}, however, proved that for every $n\ge 2$, lattices in $\mathrm{PU}(n,1)$   are not profinitely rigid among themselves; in fact, there exist non-commensurable lattices in  $\mathrm{PU}(n,1)$ with the same set of finite quotient groups. 
There are also various 
examples of lattices in  higher-rank (semi)simple Lie groups 
that are non-isomorphic but share the same set of finite quotient groups \cite{ Aka12, KK23,MS10}. 
In the context of 3-manifold geometries, $\PSL_2(\C)$ is sometimes compared  with the solvable Lie group $Sol$, whereas lattices in $Sol$ are not profinitely rigid among themselves \cite{Fun13,Ste72}. 

One significant application of profinite rigidity is to solve the isomorphism problem in a rather simple way. 
Indeed, if a group $\Gamma$ is profinitely rigid within a class $\mathscr{A}$ of finitely presented groups, then there exists an algorithm that, given a finite presentation of any $\Delta \in \mathscr{A}$, decides in finite time whether $\Gamma$ and $\Delta$ are isomorphic; see \cite[Page 5]{BCR16}. This proceeds by simultaneously searching either for a finite group belonging to $\mathcal{C}(\Gamma)\vartriangle\mathcal{C}(\Delta)$, or for an explicit isomorphism between $\Gamma$ and $\Delta$. Consequently, we obtain a new solution to the isomorphism problem among lattices in $\PSL_2(\C)$, which differs from Sela's machine and its generalizations \cite{DG08,DG11,Sel95}, as well as other geometric methods \cite{Mat07,SS14}. 

\begin{maincor}
The isomorphism problem among lattices in $\PSL_2(\C)$ is solvable. 
\end{maincor}

The data of finite quotient groups are encoded in an algebraic construction named profinite completion. 
\begin{definition}\label{DEF}
Let $\Gamma$ be an abstract group. The {\em profinite completion} of $\Gamma$ is a profinite group defined by 
$$
\widehat{\Gamma}= \limi_{N\unlhd_{f.i.} \Gamma} \Gamma/N,
$$
where $N$ ranges through all finite-index normal subgroups of $\Gamma$. 
\end{definition}

\begin{fact}[{\cite{DFPR82}}]\label{FACT}
For two finitely generated groups $\Gamma$ and $\Delta$, $\mathcal{C}(\Gamma)=\mathcal{C}(\Delta)$ if and only if $\widehat{\Gamma}\cong \widehat{\Delta}$. 
\end{fact}

Based on \autoref{FACT}, \autoref{mainthm1} can be reformulated as follows. If $\Gamma$ and $\Delta$ are lattices in $\PSL_2(\C)$  and there exists an isomorphism $\Phi: \widehat{\Gamma}\to \widehat{\Delta}$, then $\Gamma\cong \Delta$. Provided only that $\Gamma$ and $\Delta$ are abstractly isomorphic, one might further ask  what the isomorphism $\Phi$ might be. Our next theorem gives a satisfactory answer to this question. 

For a profinite group $G$, we denote by $\Out(G)=\Aut(G)/\Inn(G)$ its outer automorphism group, in which automorphisms of $G$ are required to be continuous. However, a deep theorem of Nikolov--Segal \cite{NS07} guarantees that whenever $G$ is topologically finitely generated, any automorphism of $G$ as an abstract group is necessarily continuous. This in particular  applies to the profinite completion of a finitely generated group. 

\begin{mainthm}\label{mainthm2}
For any lattice $\Gamma$ in $\PSL_2(\C)$, the homomorphism $\Out(\Gamma)\to \Out(\widehat{\Gamma})$ induced by profinite completion is an isomorphism. 
\end{mainthm}

We point out that when $\Gamma$ is a finitely generated group, $\Out(\widehat{\Gamma})$ is a profinite group; see \cite[Corollary 4.4.4]{RZ10}. Hence, $\Out(\widehat{\Gamma})$ is either  finite  or  uncountable, so one could only expect $\Out(\widehat{\Gamma})\cong \Out(\Gamma)$ when $\Out(\Gamma)$ is finite. However, the finiteness of  $\Out(\Gamma)$ is not a sufficient condition for \autoref{mainthm2}. For instance, $\Out(\Z)=\{\pm1\}$ is finite, whereas $\Out(\widehat{\Z})=\Zx\cong \prod_p \Z_p^\times$ is uncountable. 
More generally, \autoref{mainthm2} could fail for lattices in higher rank Lie groups. The following example was pointed out to the author by Nikolay Nikolov. 
\begin{example}
For $n\ge 3$, consider the lattice $\SL_n(\Z)$ in $\SL_n(\R)$. On one hand, $\Out(\SL_n(\Z))$ is finite. Indeed,   Margulis  superrigidity \cite{Mar91} implies that   $\Out(\SL_n(\Z))\cong \Z/2$ or $\Z/2\times \Z/2$, which is generated by  transpose-inverse  and  additionally the conjugation by an element in $\GL_n(\Z)$ when $n$ is even. 
On the other hand,  $\SL_n(\Z)$ satisfies the congruence subgroup property when $n\ge 3$ \cite{BLS64,Men65}. Consequently,
$$
\widehat{\SL_n(\Z)}\cong \SL_n(\widehat{\Z}) \cong \mathop{{\textstyle \prod}}\limits_{p:\text{ prime}} \SL_n(\Z_p).
$$
An application of \cite{BN26} actually  implies that 
$$
\Out(\widehat{\SL_n(\Z)}) \cong \mathop{{\textstyle \prod}}\limits_{p } \Out(\SL_n(\Z_p)) \cong  \mathop{{\textstyle \prod}}\limits_{p }  (\Z/2)\ltimes (\Z_p^{\times}/(\Z_p^{\times})^n),
$$
where the $\Z/2$ subgroup is generated by transpose-inverse on the $\SL_n(\Z_p)$ factor. In particular, $\Out(\widehat{\SL_n(\Z)})$  is infinite, and $\Out(\widehat{\SL_n(\Z)})\not \cong \Out(\SL_n(\Z))$. 
\end{example}

An implication of \autoref{mainthm2} is that the rigidity phenomenon is preserved under finite extensions of this class of groups. Combining with the profinite classification of 3‑manifold groups, we obtain the following corollary.
\begin{maincor}\label{maincor}
Any finitely generated group that is  virtually a lattice in $\PSL_2(\C)$ is profinitely rigid among all finitely generated virtually 3-manifold groups. 
\end{maincor}
\autoref{maincor} can be compared with the fact that profinite rigidity even fails among the class of virtually cyclic groups \cite{Bau74}. 

\subsection{Ingredients of the proof}


To keep it simple for this subsection, let us focus on   torsion-free uniform lattices in $\PSL_2(\C)$, which are the fundamental groups of closed hyperbolic 3-manifolds. 

\subsubsection{Crossfibered manifolds}

In \cite{Liu23}, Liu 
proved that fibered classes in the first integral cohomology are profinite invariants in the class of hyperbolic 3-manifolds. 
This makes fibered manifolds technically most accessible in this context, where the profinite rigidity of the 3-manifold group can be reformulated as a problem regarding  profinite automorphisms of the fiber surface subgroup. 
On the other hand, there is a pseudo-Anosov suspension flow on the hyperbolic 3-manifold associated to each  fibered class. Liu \cite{Liu23,Liu25} also proved that the homotopy classes of the  periodic trajectories within this pseudo-Anosov  flow are profinitely distinguished. 

The first question  to be dealt with is how these two ingredients   can be related. In this paper, we shall work with a specific class of closed  fibered hyperbolic 3-manifolds called   {\em crossfibered} manifolds. Such a manifold posseses two (necessarily distinct) fibered classes, namely the axial class and the facial class, such that a certain  periodic trajectory of the pseudo-Anosov flow associated to the axial class can be freely homotoped onto the fiber surface dual to the facial class. Building on Agol's  theorems \cite{Ago08,Ago13}, we prove that closed hyperbolic 3-manifolds are virtually  crossfibered (\autoref{thm: virtual crossfibering}). 

Liu's theorems \cite{Liu23,Liu25} imply that crossfibering  is also a profinite invariant among closed hyperbolic 3-manifolds. The advantage of crossfibered manifolds is that any profinite isomorphism between them carries a  profinite isomorphism between the fiber surface subgroups dual to the facial classes, which matches up the procyclic subgroups generated by the loops homotopic to the periodic trajectories. 
This extra piece of information implies many  algebraic properties of the profinite surface group isomorphism. In particular, this construction is used to improve a theorem of Liu \cite[Theorem 1.2]{Liu23}, where we show  in  \autoref{thm: regular} that profinite isomorphisms between hyperbolic 3-manifold groups are regular in the sense of Boileau--Friedl \cite{BF20}. 
\subsubsection{Semiauthenticity}\label{subsubsec: semi}

Next, 
we show   that the profinite isomorphism between the fiber surface subgroups dual to the facial classes satisfies an algebraic condition named {\em virtually hereditary bi-semiauthenticity} (\autoref{lem: technical VHBS}). Given two finitely generated residually finite groups $\Gamma_1$ and $\Gamma_2$, an isomorphism $\phi: \widehat{\Gamma_1}\to \widehat{\Gamma_2}$ is virtually hereditarily  semiauthentic if there exists a finite-index subgroup $\Gamma_1'\le \Gamma_1$ such that for every further finite-index subgroup $\Gamma_1''\le \Gamma_1'$, one can find a semigroup $S\subseteq \Gamma_1''$, which generates $\Gamma_1''$ as a group, such that $\phi(s)$ conjugates into $\Gamma_2$ for any  $s\in S$. An isomorphism  $\phi: \widehat{\Gamma_1}\to \widehat{\Gamma_2}$ is virtually hereditarily bi-semiauthentic if both $\phi$ and $\phi^{-1}$ are virtually hereditarily  semiauthentic. 

 Our proof for this property relies on a geometric construction. 
To outline, we construct another pair of hyperbolic 3-orbifolds that contain our original fiber surfaces as embedded geometrically-finite surfaces, and we also construct a profinite isomorphism between these hyperbolic 3-orbifold groups that carries the prescribed profinite  isomorphism between the surface subgroups. Based on its regularity, we first prove that the profinite isomorphism between the  hyperbolic 3-orbifold groups is virtually hereditarily bi-semiauthentic. 
The conclusion is then deduced using   retraction maps from the 3-orbifolds to the surfaces. 

We breifly sketch the construction of the 3-orbifold. Assume for simplicity that the prescribed periodic trajectory   homotopes to a simple closed curve $\alpha$ on the fiber surface $S$ dual to the facial class, which can always be achieved up to taking a suitable finite cover. The monodromy of the facial class represents a pseudo-Anosov mapping class $f$ on $S$, and a sufficiently large iteration of $f$ produces a simple closed curve $f^n(\alpha)$ that fills the surface $S$  with $\alpha$. We then obtain a hyperbolic link  $\alpha \times\{0\} \cup f^n(\alpha) \times\{\frac{1}{2}\}$ in the 3-manifold $S\times \mathbb{S}^1$, and the hyperbolic 3-orbifold is constructed from $S\times \mathbb{S}^1$ by assigning a  sufficiently small cone angle  to this link. 
\subsubsection{Character variety} 


Finally,  from  its virtually hereditary bi-semi\-authen\-ti\-city, we prove that the  profinite isomorphism between the fiber surface subgroups dual to the facial classes is indeed a genuine one, i.e.\ it arises as the profinite completion of an isomorphism between the abstract surface groups up to differing by an inner automorphism (\autoref{thm: surface genuine}).  

The proof involves a bridge between the profinite completion of a finitely generated group and  its $\SL_2(\C)$-character variety. Specifically, we prove that any bi-semiauthentic isomorphism between the profinite completions of two finitely generated groups induces an algebraic isomorphism between their $\SL_2(\C)$-character varieties (\autoref{cor: bi-semiauthentic induce isomorphism}). In the case of surface groups, a theorem of March\'e--Simon \cite{MS21}  implies that such an algebraic isomorphism is always induced by a homeomorphism between the surfaces. Using techniques of finite covers, we finally recover  a surface homeomorphism from the profinite isomorphism between  the surface subgroups. 

This essentially completes the proof of  \autoref{mainthm1} and \autoref{mainthm2} for the fundamental groups of crossfibered manifolds, and the general case follows with fewer difficulties. 

\subsection{Organization of the paper}
\autoref{sec: Thurston norm} and \autoref{sec: crossfibered} focus on crossfibered manifolds, which serve as  the main target of the paper. We introduce some related concepts  in \autoref{sec: Thurston norm}, and prove the virtual crossfibering theorem (\autoref{thm: virtual crossfibering}) in \autoref{sec: crossfibered}.

\autoref{sec: profinite groups} serves as preliminary materials for profinite groups. In \autoref{sec: inj}, we prove the injectivity of the completion map $\Out(\Gamma)\to \Out(\widehat{\Gamma})$ appearing in \autoref{mainthm2}, and we provide two criteria for its surjectivity in \autoref{sec: surj}.

We briefly introduce the cohomology theory and the Bass--Serre theory for profinite groups in \autoref{sec: cohomology} and \autoref{sec: bstheory}. These serve as preparations for a topological  construction  in \autoref{sec: cut and glue}, where we are able to detect, in the profinite setting, splittings of surface groups arising from cutting along simple closed curves (\autoref{thm: cut and glue}). Using this construction, we prove the regularity of isomorphisms between lattices in $\PSL_2(\C)$ (\autoref{thm: regular}) in  \autoref{sec: regular}.

In \autoref{sec: drilling}, we construct the hyperbolic 3-orbifolds described in \autoref{subsubsec: semi}, and we establish profinite isomorphisms between these 3-orbifolds utilizing once again the profinite Bass--Serre theory. From this construction, we deduce  in \autoref{sec: semiauthentic}  that the profinite surface group isomorphisms being considered  are virtually hereditarily bi-semiauthentic (\autoref{lem: technical VHBS}). 

We then turn to the algebraic aspects of bi-semiauthentic isomorphisms. In  \autoref{sec: 11}, we study the action of  semiauthentic homomorphisms on the $\SL_n$-character variety. This is then applied  to surface groups, where we find out the desired surface homeomorphism in \autoref{sec: realisation}. 

We prove \autoref{mainthm1} and \autoref{mainthm2} for uniform lattices in \autoref{sec: Uniform}. Finally, we finish the proof for non-uniform lattices in  \autoref{sec: non-uniform}, based on Thurston's hyperbolic Dehn surgery theory. 

\subsection{Acknowledgements}
The author would like to express his sincere gratitude to his advisor Yi Liu for discussions and  encouragements throughout this project;  
particularly, the proof of \autoref{thm: virtual crossfibering} benefits from the discussion with Yi Liu. 
Andrei Jaikin-Zapirain and Nikolay Nikolov have previously asked the author about \autoref{mainthm2} and \autoref{maincor} repsectively, 
and he would  like to thank both of them for providing several interesting examples. 
The author would also like to thank Alan Reid and Henry Wilton for their comments on an earlier draft of this paper, and Yaoping Xie for helpful conversations.

\section{Thurston norm and pseudo-Anosov mapping tori}\label{sec: Thurston norm}

\subsection{Fibered classes}
Suppose $M$ is  a compact orientable 3-manifold with empty or toral boundary.  
\begin{definition}\label{def: fibered class}
A cohomology class $\psi\in H^1(M;\Z)$ is called a {\em fibered class} if it is induced by a fiber bundle map $\pi: M\to \mathbb{S}^1$. 
\end{definition}
When $\psi$ is a fibered class, the fiber bundle map $\pi$ satisfying the above property is unique up to isotopy; see \cite[Theorem 4]{Thu86}.   We denote a fiber of $\pi$ as $S_{M,\psi}$, whose isotopy class is uniquely determined by $\psi$. Then, $S_{M,\psi}$ is a compact, orientable,  properly embedded incompressible surface in $M$. We remark that $S_{M,\psi}$  is possibly disconnected when $\psi$ is a non-primitive class.  

Under the identification $H^1(M;\Z)\cong \mathrm{Hom}(\pi_1M,\Z)$, Stallings \cite{Sta62} proved that $\psi\in H^1(M;\Z)\setminus\{0\}$ is a fibered class if and only if $\ker(\psi)$ is   finitely generated. When $\psi$ is a fibered class, the embedding of each connected component $S\subseteq S_{M,\psi}$ into $M$ determines an injective homomorphism $i:\pi_1S \hookrightarrow \pi_1M$ up to an identification of basepoints. The image of $i$ is a normal subgroup of $\pi_1M$ independent with the choice of the connected component and the identifaction of basepoints. In fact, the image of $i$  is exactly $\ker(\psi)$. As a consequence, a free loop $\gamma \subset M$ (also viewed as a conjugacy class in $\pi_1M$) can be freely homotoped onto $S_{M,\psi}$ if and only if $\langle \psi,[\gamma]\rangle =0$, where $[\gamma]\in H_1(M;\Z)$ denotes its homology class. 

\subsection{Thurston norm}
Let $M$ be a compact orientable 3-manifold. We identify $H^1(M;\Z)$ and $H^1(M;\Q)$ with their images in $H^1(M;\R)$ via the inclusions $\Z\hookrightarrow \Q\hookrightarrow \R$, and refer to them as the {\em integral classes} and {\em rational classes} in $H^1(M;\R)$. For a finitely generated abelian group $A$, let $A_{\mathrm{tors}}$ be the torsion subgroup of $A$, and let $A_{\mathrm{free}}= A/A_{\mathrm{tors}}$. Following the same convention in cohomology, we identify $H_1(M;\Z)_{\mathrm{free}}$ and $H_1(M;\Q)$ with their images in $H_1(M;\R)$, and refer to them as the {\em integral classes} and {\em rational classes} in $H_1(M;\R)$. 

\begin{sloppypar}
Thurston \cite{Thu86} introduced a canonical seminorm on   $H^1(M;\R)$, now widely known as the Thurston norm, which we recall as follows. For a compact, orientable, possibly disconnected surface $S$ consisting of connected components $S_1,\cdots, S_n$, we denote $\chi_-(S)=\sum_{i=1}^{n}\mathrm{max}\{-\chi(S_i),0\}.$ 
\end{sloppypar}
\begin{defthm}[{\cite[Theorem 1]{Thu86}}]
The  {\em Thurston norm} on $M$ is the unique seminorm $\Th{\cdot} : H^1(M;\mathbb{R}) \to [0, +\infty)$ on the real vector space $H^1(M;\mathbb{R})$ that is continuous, convex, linear on rays through the origin, and satisfies the following property: for any integral class $\psi \in H^1(M;\mathbb{Z})$, $\|\psi\|_{\text{Th}}$ equals the minimal possible value of $\chi_-(S)$, where $S$ ranges over all properly embedded compact oriented surfaces in $M$ representing the Poincar\'e dual of $\psi$.
\end{defthm}

We remark that although we rely on an orientation of $M$ to make sense of Poincar\'e duality, the Thurston norm on  $M$ is independent of the choice of this orientation.
Before moving on, we recall some concepts in convex geometry.

\begin{definition}
Let $V$ be a finite-dimensional real vector space, with coordinates $(x_1,\cdots, x_n)$. 
Equip $V$ with the standard point-set topology. 
\begin{enumerate}[label=(\arabic*),leftmargin=*]
\item\label{P1} A (convex) {\em polyhedron} $P$ in $V$ is the intersection of a finite number (possibly zero) of closed half-spaces in $V$, i.e.\
$$
P=\{(x_1,\cdots,x_n)\in V\mid a_{i1}x_1+\cdots+a_{in}x_n\ge c_i,\,\forall 1\le i \le m\}
$$
for some $m\ge 0$ and $a_{ij},c_i\in \R$, where we define $P=V$ if $m=0$. 
Throughout this paper, every polyhedron is meant to be convex. 
\item $P$ is called a {\em polyhedral cone} if $c_1=\cdots=c_m=0$ under the hypothesis of~\ref{P1}.
\item The {\em dimension} of a polyhedron $P$ is the minimal dimension of a hyperplane in $V$ that contains $P$. 
\item A {\em face} of a polyhedron $P$ is a convex subset $F\subseteq P$ satisfying the following property: if $u,v\in P$  and there exists $\mu \in (0,1)$ such that $\mu u+(1-\mu)v\in F$, then $u,v\in F$. Note that a face of a polyhedron is again a polyhedron. 
\item A face of $P$ contained in $\partial P$ is abbreviated into {\em a face in $\partial P$}, and a face $F$ in $\partial P$ is {\em top-dimensional} if $\dim F=\dim P-1$. 
\item For a convex subset $X\subseteq V$, the {\em cone} over $X$ is defined as $Cone(X)=\{\lambda x\mid \lambda\ge 0, x\in X\}$. 
\end{enumerate}
\end{definition}


Let $M$ be a compact orientable 3-manifold. Denote $$\BTh(M)=\{\psi\in H^1(M;\R)\mid \Th{\psi}\le 1\}$$ as the {\em Thurston norm unit ball} of $M$.

\begin{proposition}[{\cite[Theorem 2]{Thu86}}]
The Thurston norm unit ball $\BTh(M)$ is a convex polyhedron in $H^1(M;\R)$. 
\end{proposition}

We say that the Thurston norm on $M$ is {\em non-vanishing} if there exists some $\psi\in H^1(M;\R)$ such that $\Th{\psi}\neq 0$. In this case, $\partial \BTh(M)\neq \varnothing$ is the union of finitely many top-dimensional faces in $\partial\BTh(M)$. 
\begin{proposition}[{\cite[Theorem 5]{Thu86}}]\label{thm: fibered}
  If the Thurston norm on $M$ is non-vanishing, then there exists a collection of top-dimensional faces $F_1,\cdots, F_n$ in $\partial\BTh(M)$ such that the set of fibered classes in $H^1(M;\Z)$ equals
$$
\left( \bigcup_{i=1}^{n} \operatorname{int}(Cone(F_i)) \right) \cap  H^1(M;\Z) .
$$ 
\end{proposition}

When the Thurston norm on $M$ is non-vanishing, each $Cone(F_i)$ as appearing in \autoref{thm: fibered} is called a {\em fibered cone} in $H^1(M;\R)$. For any fibered class $\psi\in H^1(M;\Z)$, there is a unique fibered cone $Cone(F_i)$ such that $\psi\in \operatorname{int}(Cone(F_i))$. We denote $\CTh(M,\psi)= Cone(F_i)$. Then, $\CTh(M,\psi)$ is a codimension~0 polyhedral cone.

\begin{remark}
\begin{enumerate}[leftmargin=*, label=(\arabic*)]
\item If the Thurston norm on $M$ is vanishing, then either all nonzero integral cohomology classes are fibered, or all are non-fibered.
\item In some literatures,  $\CTh(M,\psi)$ is taken to be $\operatorname{int}(Cone(F_i))$.  
We remind the readers that $\CTh(M,\psi)$ denotes  the closed polyhedral cone in our notation. 
\end{enumerate}
\end{remark}

In hyperbolic 3-manifolds, which is the main focus of this paper, the Thurston norm is indeed a norm. As such, the above notions for a  non-vanishing Thurston norm applies in our situation.

\begin{proposition}\label{prop: true norm}
If the interior of $M$ admits a complete hyperbolic structure of finite volume, then $\Th{\psi}\neq 0$ whenever $\psi\in H^1(M;\R)\setminus\{0\}$. In this case, $\BTh(M)$ is a compact polyhedron symmetric about the origin.
\end{proposition}

In fact, $M$ is irreducible, $\partial$-irreducible, anannular, and atoroidal. Thus, any properly embedded surface in $M$ that represents a non-trivial homology class in $H_2(M,\partial M;\Z)$ has negative Euler characteristic. Hence, the Thurston norm on $M$ is a norm, see \cite[Theorem 1]{Thu86}.

\subsection{Periodic trajectories}\label{subsec:pertraj}

In this subsection, we focus on fibered hyperbolic 3-manifolds. 

Let $M$ be a closed orientable hyperbolic 3-manifold with a fibered class $\psi \in H^1(M;\Z)$. 
By \autoref{def: fibered class}, $\psi$ is induced by a fiber bundle map $\pi:M\to \mathbb{S}^1$, which gives rise to a monodromy map on the fiber surface $S_{M,\psi}$. Since $M$ is hyperbolic, the monodromy map represents a pseudo-Anosov mapping class, and is hence isotopic to a pseudo-Anosov automorphism $f: S_{M,\psi}\to S_{M,\psi}$.  
Let 
$$
M_f=\frac{S_{M,\psi}\times \R}{(x,s+1)\sim (f(x),s)}
$$
be the mapping torus of $f$, and let 
$$
\begin{tikzcd}[row sep=0cm, column sep=tiny]
 \pi_f: \hspace{-0.3cm} &  M_f \arrow[rrr] & & & \mathbb{S}^1= \R/\Z \\
        &  (x,s) \arrow[rrr, maps to] & & & s\, {\scalebox{1}{$(\mathrm{mod}\,1)$}}
\end{tikzcd} 
$$
be the distinguished fiber  bundle map. 
We can find a fiber-preserving homeomorphism $F:M\to M_f$ such that the following diagram commutes.
\begin{equation*}
\begin{tikzcd}
 M \arrow[rr, "F"] \arrow[rd,"\pi"']  &    &   M_f \arrow[ld, "\pi_f"] \\
                   & \hspace{0.04cm} \mathbb{S}^1 \hspace{-0.05cm} & 
\end{tikzcd}
\end{equation*}

Let 
\begin{equation*}
\begin{tikzcd}[row sep=0cm, column sep=tiny]
\theta_f^t: \hspace{-0.3cm}& M_f \arrow[rrrr] & & & & M_f \\
& (x,s) \arrow[rrrr, maps to] & & & & (x,s+t)
\end{tikzcd}
\end{equation*}
be the suspension flow on the mapping torus $M_f$. We pull back $\theta_f^t$ along $F$ to obtain a one-parameter flow $\Theta^t_\psi= F^{-1} \circ \theta^t_f \circ F$ on $M$. 

We say that two one-parameter flows $\Theta_1^t$ and $\Theta_2^t$ on a manifold $M$ are {\em strictly conjugate} if there exists a homeomorphism $g: M\to M$ isotopic to the identity such that $g\circ \Theta_1^t=\Theta_2^t \circ g$ for all $t$. Fried \cite[Theorem 14.11]{FLP} proved that the strict conjugacy class of $\Theta_\psi^t$ is uniquely determined by the fibered class $\psi$, i.e.\ it is independent with the choice of $\pi$, $f$, and $F$. His proof is based on two facts regarding a closed orientable surface $S$ of genus $g\ge 2$. First, pseudo-Anosov automorphisms of $S$ belonging to the same homotopy class are topologically conjugate by a homeomorphism that is isotopic to the identity \cite[Theorem 12.5]{FLP}. Second, $\mathrm{Homeo}_0(S)$ (the identity component of $\mathrm{Homeo}(S)$) is simply-connected \cite{Ham66}. 

With this justification, we refer to the strict conjugacy class of $\Theta_{\psi}^t$ as the {\em pseudo-Anosov suspension flow associated to $\psi$}.

Suppose $\Theta^t$ is a one-parameter flow on a manifold $M$. A  loop
\begin{equation*}
\gamma: \frac{[0,1]}{0\sim 1}\cong \mathbb{S}^1\to M
\end{equation*}
is called a {\em periodic trajectory} of $\Theta^t$ if there exists $\lambda>0$ such that $\gamma(t)=\Theta^{\lambda t}(\gamma(0))$ for any $t\in [0,1]$. Note that $\gamma$ is not required to be simple.

\begin{definition}
  Suppose $M$ is a closed  orientable  hyperbolic 3-manifold, and  $\psi\in H^1(M;\Z)$ is a fibered class. The set of {\em periodic trajectories associated to $\psi$}, denote by $\mathfrak{T}(M,\psi)$, is the set of free homotopy classes of all  periodic trajectories  of the pseudo-Anosov suspension flow $\Theta_\psi^t$.  $\mathfrak{T}(M,\psi)$ is also viewed as a set of conjugacy classes in $\pi_1M$. 
\end{definition}

\begin{remark}\label{rmk: solely CTh}
Note that the free homotopy classes of the periodic trajectories of a one-parameter flow are determined by its strict conjugacy class, so Fried's result as described above guarantees that $\mathfrak{T}(M,\psi)$ is well-defined by $\psi$. 

In fact, Fried \cite[Theorrem 14.11]{FLP} showed that the pseudo-Anosov suspension flows associated to fibered classes in a same fibered cone are orbit equivalent to each other via homeomorphisms of $M$ that are isotopic to the identity. Thus, the set $\mathfrak{T}(M,\psi)$ is determined solely by the fibered cone $\CTh(M,\psi)$. 
\end{remark}

\begin{definition}
Let $M$ be a closed orientable hyperbolic 3-manifold, and let $\psi\in H^1(M;\Z)$ be a fibered class. The {\em dual Thurston norm cone} (or {\em Fried cone}) of $\psi$ is a subset in $H_1(M;\R)$ defined by
$$
\CFr(M,\psi)=\{ x\in H_1(M;\R)\mid \langle \psi ,x \rangle \ge 0, \, \forall \psi \in \CTh(M,\psi)\}.
$$
\end{definition}
\begin{lemma}\label{lem: codim 0 cone}
$\CFr(M,\psi)$ is a codimension~0 polyhedral cone in $H_1(M;\R)$. 
\end{lemma}
\begin{sloppypar}
\begin{proof}
Suppose that $\CTh(M,\psi)$ is spanned by a top-dimensional face $F$ in $\partial \BTh(M)$. Recall that $\BTh(M)$ is a compact polyhedron (\autoref{prop: true norm}), so $F$ is the convex hull of its vertices, which consist of finitely many points $\phi_1,\cdots, \phi_n$. Then, $\CFr(M,\psi)=\{x\in H_1(M;\R)\mid \langle \phi_i , x \rangle \ge 0, \, \forall 1\le i \le n\}$ is a polyhedral cone in $H_1(M;\R)$. Furthermore, $\CTh(M,\psi)\cap (-\CTh(M,\psi))=\{0\}$  since $\BTh(M)$ is a compact symmetric polyhedron. Thus,  there exists a codimension~$1$ linear subspace $W\subseteq H^1(M;\R)$ such that $\CTh(M,\psi)\setminus\{0\}$ belongs to the interior of a half-space $\mathcal H$ bounded by $W$. We can pick $x_0\in H_1(M;\R)$ such that $\langle \varphi,x_0\rangle =0$ for any $\varphi \in W$, and $\langle \phi,x_0\rangle >0$ for any $\phi\in \mathcal H$. Then, $\langle \phi_i,x_0\rangle >0$ for any $1\le i\le n$, so $\CFr(M,\psi)$ contains an open neighbourhood of $x_0$. Consequently, $\CFr(M,\psi)$ has codimension~$0$ since it has non-empty interior. 
\end{proof}
\end{sloppypar}

Note that periodic trajectories are always homologically non-trivial. 
The Fried cone of $\psi$ characterizes the homological directions of the periodic trajectories associated to $\psi$.

\begin{proposition}[{\cite[Theorem 5.14]{Liu20}}]\label{prop: Fried cone}
Let $M$ be a closed orientable hyperbolic 3-manifold, with $\psi\in H^1(M;\Z)$ being a fibered class. 
\begin{enumerate}[label=(\arabic*), leftmargin=*]
\item\label{Fc1} For any $\gamma \in \mathfrak{T}(M,\psi)$, the homology class $[\gamma]\in H_1(M;\R)$ belongs to $\CFr(M,\psi)$. 
\item\label{Fc2} For any rational class $x\in \operatorname{int}(\CFr (M, \psi))$, there exists  a periodic trajectory $\gamma\in \mathfrak{T}(M,\psi)$ whose homology class $[\gamma]$ is a positive scalar multiple of $x$. 
\end{enumerate}
\end{proposition}

In fact, item \ref{Fc1} is a direct consequence of \cite[Theorem 14.11]{FLP} as explained in \autoref{rmk: solely CTh}. Item \ref{Fc2} is proven in the second paragraph of the proof of {\cite[Theorem 5.14]{Liu20}}, see  subsequent remarks and notes therein for further explanations. 
We remark that $\partial \CFr(M,\psi)$ could contain homology classes of periodic trajectories in $\Traj(M,\psi)$, but it is also possible that some rational direction in $\partial \CFr(M,\psi)$ does not contain the homology class of any periodic trajectory associated to $\psi$. 
\section{Crossfibered manifolds}\label{sec: crossfibered}
We now introduce the class of hyperbolic 3‑manifolds that we will be mainly working with in this paper.
\begin{definition}\label{def: crossfibered}
A closed orientable hyperbolic 3-manifold $M$  is called a {\em crossfibered manifold} if $M$ possesses two primitive fibered classes $\psiA,\psiF \in H^1(M;\Z)$,  and  there exists a periodic trajectory $\gamma$ associated to $\psiA$ such  that  $\gamma$ can be freely homotoped onto the fiber surface dual to $\psiF$. 

The tuple $(M,\psiA,\psiF,\gamma)$ as described is called a {\em crossfibering datum} of $M$, in which $\psiA$ is called the {\em axial fibered class} and $\psiF$ is called the {\em facial fibered class}. 
\end{definition}

Note that the crossfibering datum for a crossfibered manifold may not be unique.  
The purpose of this section is to prove the following virtual crossfibering theorem. 


\begin{theorem}\label{thm: virtual crossfibering}
Any closed hyperbolic 3-manifold has a finite cover which is a crossfibered manifold. 
\end{theorem}

\subsection{Finite cover}
Let $M$ be a closed orientable hyperbolic 3-manifold, and let $p: M' \to M$ be a finite-sheeted cover. Then $p$ induces a surjective linear map $p_\ast: H_1(M';\R)\to H_1(M;\R)$ and a dual injective linear map $p^\ast: H^1(M;\R)\to H^1(M';\R)$, both of which send integral (resp.\  rational) classes to integral (resp.\  rational) classes. 

\begin{proposition}\label{prop: finite cover}
Let $p: M'\to M$ be a finite cover of closed orientable hyperbolic 3-manifolds as described above. 
\begin{enumerate}[label=(\arabic*),leftmargin=*]
\item\label{fcov1} A cohomology class $\psi\in H^1(M;\Z)$ is a fibered class if and only if $p^\ast\psi \in H^1(M';\Z)$ is a fibered class.
\end{enumerate}

In the following, let $\psi \in H^1(M;\Z)$ be a fibered class.
\begin{enumerate}[label=(\arabic*),leftmargin=*,start=2]
\item\label{fcov2} The covering map $p$ projects each periodic trajectory in $\Traj(M',p^\ast\psi)$ to  a periodic trajectory in $\Traj(M,\psi)$. 
\item\label{fcov3} Conversely, every elevation of a periodic trajectory in $\Traj(M,\psi)$ is a periodic trajectory in $\Traj(M',p^\ast\psi)$. 
\item\label{fcov4} $p_\ast (\CFr(M',p^\ast\psi))= \CFr(M,\psi)$. 
\end{enumerate}
\end{proposition}
\begin{proof}
To prove \ref{fcov1}, first suppose that $\psi$ is a fibered class, which is induced by a fiber bundle map $\pi: M\to \mathbb{S}^1$. Then, $p^\ast \psi$ is induced by $\pi \circ p: M' \to \mathbb{S}^1$, which is also a fiber bundle map. Thus, $p^\ast \psi$ is also a fibered class. Conversely, suppose $p^\ast\psi$ is a fibered class. We identify $H^1(M';\Z)$ and $H^1(M;\Z)$ with $\Hom(\pi_1M',\Z)$ and $\Hom(\pi_1M,\Z)$ respectively. Up to a choice  of basepoints, $p_\ast(\pi_1M')$ represents a finite-index subgroup in $\pi_1M$, so $\ker(\pi_1M\xrightarrow{\psi} \Z)$ contains $\ker(\pi_1M' \xrightarrow{p_\ast} \pi_1M \xrightarrow{\psi}\Z)=\ker(\pi_1M' \xrightarrow{p^\ast\psi}  \Z)$ as a finite-index subgroup. Note that $\ker(p^\ast \psi)=\pi_1S_{M',p^\ast \psi}$ is finitely generated, so $\ker(\psi)$ is also finitely generated. Hence, it follows from Stallings's theorem \cite{Sta62} that $\psi$ is also a fibered class.  

For \ref{fcov2} and \ref{fcov3}, it suffices to show that the pseudo-Anosov suspension flow $\Theta_{p^\ast \psi}^t$ on $M'$ is the lift of $\Theta_{\psi}^t$ via the covering map $p$. 
As described in \autoref{subsec:pertraj}, we may identify $M$ with a pseudo-Anosov mapping torus 
$$
M_f=\frac{ S\times \mathbb{R}}{(x,s+1)\sim (f(x),s)}
$$
where $f:S\to S$ is a pseudo-Anosov automorphism on the (possibly disconnected) surface $S$, such that the fibered class $\psi$ is induced by the distinguished fiber bundle map $\pi_f:M_f\to \mathbb{S}^1 $. As such, we take  the suspension flow $\theta^t_f$ on $M_f$ as a representative for the strict conjugacy class of $\Theta^t_\psi$. 

The fibered class $p^\ast \psi\in H^1(M';\Z)$ is induced by the fiber bundle map 
\begin{equation*}
\begin{tikzcd}
\pi':\,M'\arrow[r,"p"] & M=M_f \arrow[r,"\pi_f"] & \mathbb{S}^1.
\end{tikzcd}
\end{equation*}
Let $\widetilde{\theta^t_f}$ be the lift of $\theta^t_f$ onto $M'$ via the covering map $p:M'\to M=M_f$. Let $S'=p^{-1}(S\times \{0\})\subseteq M'$, which is a (possibly disconnected) finite cover of $S$. Then for each $t\in \mathbb{R}$ (also viewed as a point in $\mathbb{S}^1=\R/\Z$), $\widetilde{\nss{\theta}{f}{t}}(S')=p^{-1}(\theta^t_f(S\times \{0\}))=p^{-1}(S\times \{t\})=(\pi')^{-1}(t)$. Thus, $\widetilde{\nss{\theta}{f}{t}}$  defines a suspension flow on the fiber bundle $M'\xrightarrow{\pi'}\mathbb{S}^1$. The return map $\widetilde{\nss{\theta}{f}{1}}:S'\to S'$ is a lift of the return map $f=\theta^1_f:S\to S$ via the covering map $p:S'\to S$. Hence $f'=\widetilde{\nss{\theta}{f}{1}}:S'\to S'$ is also a pseudo-Anosov automorphism on $S'$, and $\widetilde{\nss{\theta}{f}{t}}$ defines a pseudo-Anosov suspension flow on $(M',\pi')$. As explained in \autoref{subsec:pertraj}, $\widetilde{\nss{\theta}{f}{t}}$ is a representative for the strict conjugacy class of $\Theta^t_{p^\ast\psi}$. In other words, $\Theta^t_{p^\ast\psi}$ is the lift of $\Theta^t_\psi$ via $p$, and assertions $\ref{fcov2}$ and $\ref{fcov3}$ follow  directly. 

Now we prove \ref{fcov4}. According to \autoref{lem: codim 0 cone} and \autoref{prop: Fried cone}, $\CFr(M,\psi)$ is the topological closure of $\Lambda=\{\lambda[\gamma]\in H_1(M;\R)\mid \lambda\in [0,+\infty), \gamma \in \Traj( M,\psi)\}$; and similarly, $\CFr(M',p^\ast \psi)$ is the topological closure of $\Lambda'=\{\lambda[\gamma]\in H_1(M';\R)\mid \lambda\in [0,+\infty), \gamma \in \Traj( M',p^\ast \psi)\}$. We claim that $p_\ast(\Lambda')=\Lambda$. 

On one hand, for any $\gamma \in \Traj(M',p^\ast \psi)$, $p_\ast(\gamma)\in \Traj(M,\psi)$ by \ref{fcov2}, and hence $p_\ast(\lambda[\gamma])=\lambda[p_\ast(\gamma)]\in \Lambda$ for any $\lambda\ge 0$. In other words, $p_\ast(\Lambda')\subseteq \Lambda $. On the other hand, for any $\gamma\in \Traj(M,\psi)$, we can pick an elevation $\gamma'$ of $\gamma$ in $M'$, and   \ref{fcov3} asserts that $\gamma'\in \Traj(M',p^\ast \psi)$. 
Suppose that $\gamma'$ is an $n$-fold cover of $\gamma$, where $n\in \N$. Then, $\lambda[\gamma]=p_\ast(\frac{\lambda}n[\gamma'])\in p_\ast(\Lambda')$ for any $\lambda\ge 0$. In other words, $p_\ast(\Lambda')\supseteq \Lambda $. 

Thus, $p_\ast(\Lambda')=\Lambda$. This implies that $p_\ast(\CFr(M',p^\ast\psi))=\CFr(M,\psi)$, since a linear map between finite-dimensional real vector spaces sends any polyhedral cone to a polyhedral cone. 
\end{proof}

\begin{lemma}\label{lem: polyhedral map}
Suppose $V$ and $W$ are finite-dimensional real vector spaces, and $f: V\to W$ is a linear map. Suppose $P\subseteq V$ and $Q\subseteq W$ are codimension~0 polyhedrons such that $f(P)=Q$. Then, for any $q\in \operatorname{int}(Q)$, $f^{-1}(q)\cap \operatorname{int}(P)\neq \varnothing$. 
\end{lemma}
\begin{proof}
Since $Q$ has codimension 0, $Q$ spans $W$ as a vector space, and $f$ is necessarily surjective. Suppose by contradiction that $f^{-1}(q)\cap \operatorname{int}(P) = \varnothing$. Note that $f^{-1}(q)\subseteq V$ is a hyperplane, so by convexity of $\operatorname{int}(P)$, there exists a codimension~1 hyperplane $H\subseteq V$ such that $f^{-1}(q)\subseteq H$ and that $\operatorname{int}(P)$ is contained in the interior of a  halfspace bounded by $H$. In other words, there exist  a surjective linear map $h: V\to \R$ and $c\in \R$ such that  $h(f^{-1}(q))=\{c\}$  and $h(\operatorname{int}(P))\subseteq (c,+\infty)$. As such, $\ker(f)\subseteq \ker(h)$. Since $f$ is surjective, $h$ descends to a surjective linear map $g: W\to \R$ such that $h=g\circ f$. Then, $g(q)=c$. In addition, $h(P)\subseteq [c,+\infty)$ by continuity, so $g(Q)=g(f(P))=h(P)\subseteq [c,+\infty)$. This contradicts with the fact that $q\in \operatorname{int}(Q)$ and $g$ is surjective. 
\end{proof}

\begin{corollary}\label{cor: interior cone}
Suppose $M$ is a closed orientable hyperbolic 3-manifold, and $\psi\in H^1(M;\Z)$ is a fibered class. Let  $p: M'\to M$ be a finite cover. Then, for any $x\in \operatorname{int}(\CFr(M,\psi))$,  there exists $x'\in \operatorname{int}(\CFr(M',p^\ast \psi))$ such that $p_\ast(x')=x$. 
\end{corollary}
\begin{proof}
This corollary is a combination of \autoref{lem: codim 0 cone}, \autoref{prop: finite cover}~\ref{fcov4}, and \autoref{lem: polyhedral map}. 
\end{proof}

\subsection{Virtual specialization}

Nowadays, rich constructions of finite covers of hyperbolic 3-manifolds arise from the technique of virtual specialization. 

\begin{theorem}[{Agol \cite{Ago13}, Wise \cite{Wis21}}]\label{thm: virtually special}
The fundamental group of a complete finite-volume hyperbolic 3-manifold is virtually compact special. 
\end{theorem}

Agol \cite{Ago13} proved \autoref{thm: virtually special} for closed hyperbolic 3-manifolds, building on the works of Kahn--Markovic \cite{KM12} and Bergeron--Wise \cite{BW12}, etc. Wise \cite[Theorem 17.14]{Wis21} proved the theorem for cusped hyperbolic manifolds. 
\autoref{thm: virtually special} resolves the virtual Haken  conjecture, producing plenty of applications in finite covers. 
\begin{proposition}[{\cite[Theorem 9.2]{Ago13}}]\label{thm: virtually fiber}
Let $M$ be a closed hyperbolic 3-manifold.
\begin{enumerate}[label=(\arabic*), leftmargin=*]
\item There exist finite covers of $M$ with arbitrarily large first betti numbers.
\item\label{thmvf2} There exists a finite cover of $M$ which fibers over $\mathbb{S}^1$. 
\end{enumerate}
\end{proposition}

\autoref{thm: virtually special} also implies that the fundamental group of a finite-volume hyperbolic 3-manifold is virtually RFRS \cite[Corollary 2.3]{Ago08}, which, via Agol's criterion, yields a more precise version of \autoref{thm: virtually fiber}~\ref{thmvf2}. 

\begin{proposition}[{\cite[Theorem 5.1]{Ago08}}]\label{thm: fibered closure}
Let $M$ be an orientable complete finite-volume hyperbolic 3-manifold. Then, for any $\psi\in H^1(M;\Z)\setminus\{0\}$, there exists a finite cover $p: M'\to M$ such that $p^\ast\psi$ lies in the closure of a fibered cone in $H^1(M';\Z)$. 
\end{proposition}

\subsection{Virtual crossfibering}We are now ready to prove \autoref{thm: virtual crossfibering}. 

\begin{proof}[Proof of \autoref{thm: virtual crossfibering}]
Let $M$ be a closed hyperbolic 3-manifold. According to \autoref{thm: virtually fiber}, we can find an orientable finite cover $N$ of $M$ which fibers over $\mathbb{S}^1$ such that $b_1(N)\ge 2$. 
Let $\psi\in H^1(N;\Z)$ be a fibered class.  
By \autoref{lem: codim 0 cone}, the Fried cone $\CFr(N,\psi)$ has non-empty interior, so we can find a   rational class $x\in \operatorname{int}(\CFr (N,\psi))$.   
Since $b_1(N)\ge 2$, there exists a non-zero integral cohomology class $\varphi\in H^1(N;\Z)$ such that $\langle \varphi, x \rangle =0$. According to \autoref{thm: fibered closure}, there exists a further finite cover $p:M' \to N$, such that $p^\ast \varphi \in H^1(M';\Z)$ lies in the closure of a fibered cone $\Cone_0\subseteq H^1(M';\R)$. 
We will take the manifold $M'$ constructed in this way as a crossfibered finite cover of $M$. 

Let $\psiA   \in H^1(M';\Z)$ be a  positive rational scaling  of $p^\ast \psi\in H^1(M';\Z)$ such that $\psiA$ is a primitive integral class, which we will take as the axial fibered class for $M'$. Indeed, $\psiA$ is a fibered class according to \autoref{prop: finite cover}~\ref{fcov1}, and $\CFr (M', \psiA)= \CFr (M', p^\ast \psi)$. 

Since $x\in \operatorname{int}(\CFr (N,\psi))$, by \autoref{cor: interior cone},   there exists a homology class $x'\in \operatorname{int}(\CFr (M', \psiA))$ such that $p_{\ast}(x')= x$. Then, $\langle p^\ast \varphi, x'\rangle= \langle \varphi, p_\ast  (x') \rangle =\langle \varphi, x \rangle=0$. In other words, if we view $p^\ast \varphi\in H^1(M';\R)$ as a linear functional on  $H_1(M';\R)$, then $\ker(p^\ast \varphi)$ intersects non-trivally with $\operatorname{int}(\CFr (M',\psiA))$.  Note that the map 
\begin{equation*}
\begin{tikzcd}[row sep=0cm]
H^1(M';\R)-\{0\} \arrow[r] & \mathrm{Gr}_{b_1(M')-1}(H_1(M';\R))  \\
\phi \arrow[r,maps to] & \ker(\phi)  
\end{tikzcd}
\end{equation*}
is continuous, where $\mathrm{Gr}_{b_1(M')-1}(H_1(M';\R)) $ denotes the space of codimension~1 linear subspaces in $H_1(M';\R)$. Also note that rational classes are dense in $H^1(M';\R)$. Thus, there exists a non-zero rational cohomology class $\phi\in \operatorname{int}(\Cone_0)$, belonging to a small neighbourhood of $p^\ast \varphi$, such that $\ker(\phi)$ intersects non-trivially with $\operatorname{int}(\CFr (M', \psiA))$.  Let $\psiF$ be a positive scalar multiple of $\phi$ so that $\psiF$ is   a primitive integral class.  Then, $\psiF\in \operatorname{int}(\Cone_0) \cap H^1(M';\Z)$ is   a fibered class according to \autoref{thm: fibered}, and $\ker(\psiF)=\ker(\phi)$. We will take $\psiF$ as the facial fibered class for $M'$.

\begin{sloppypar}
Since $\ker(\psiF)$ is a rationally defined codimension~1 hyperplane in $H_1(M';\R)$ that intersects non-trivially with the open subset $\operatorname{int}(\CFr (M'; \psiA))\subseteq H_1(M';\R)$, we can find a rational homology class $y\in \ker(\psiF)\cap \operatorname{int}(\CFr (M'; \psiA))$. According to \autoref{prop: Fried cone}~\ref{Fc2}, there exists a periodic trajectory $\gamma \in \mathfrak{T}( M' ,\psiA)$ whose homology class $[\gamma]\in H_1(M';\R)$ is a scalar multiple of $y$. In particular, $\langle \psiF, [\gamma]\rangle=0$.  
Consequently, $\gamma$ can be freely homotoped onto the fiber surface $S_{M',\psiF}$ dual to $\psiF$, and $(M',\psiA,\psiF,\gamma)$ yields a crossfibering datum for $M'$. \qedhere
\end{sloppypar}
\end{proof}

In order that we can take further finite covers of a crossfibered manifold, we state the following proposition. 
\begin{proposition}\label{prop: virc2}
Any finite cover of a crossfibered manifold is also a crossfibered manifold. 
\end{proposition}
\begin{proof}
\begin{sloppypar}
Suppose $M$ is a crossfibered manifold with a crossfibering datum $(M,\psiA,\psiF,\gamma)$, and suppose $p: M'\to M$ is a finite cover. Then, $p^\ast \psiA$ and $p^\ast \psiF$ are fibered classes in $H^1(M';\Z)$ according to \autoref{prop: finite cover}~\ref{fcov1}. Let $ {\psiA'}$ and $ {\psiF'}$ be positive rational scalings of $p^\ast\psiA$ and $p^\ast \psiF$ such that ${\psiA'}$ and $ {\psiF'}$ represent primitive integral classes. Then, $ {\psiA'}$ and $ {\psiF'}$ are also fibered classes. Let $\widetilde{\gamma}$ be an arbitrary elevation of $\gamma$ in $M'$. Then, \autoref{prop: finite cover}~\ref{fcov3} implies that $\widetilde{\gamma}\in \Traj(M',p^\ast\psiA)=\Traj(M', {\psiA'})$. By the homotopy lifting theorem, $\widetilde{\gamma}$ can be freely homotoped onto a component of $p^{-1}(S_{M,\psiF})$, which can be isotoped onto $S_{M', {\psiF'}}$. Hence, $(M',\psiA',\psiF',\widetilde{\gamma})$ is a crossfibering datum for $M'$.  \qedhere
\end{sloppypar}
\end{proof}

\section{Profinite groups}\label{sec: profinite groups}
In addition to the materials covered in this section,   the readers are referred to \cite{RZ10} for a standard reference on profinite groups. 
\subsection{Profinite spaces and profinite groups}
\begin{definition}
\begin{enumerate}[leftmargin=*,label=(\arabic*)]
\item A {\em profinite space} is an inverse limit of finite discrete spaces indexed over a directed partially ordered set. The inverse limit is equipped with the subspace topology inherited from the product topology. 
\item A {\em profinite group} is an inverse limit of finite discrete groups indexed over a directed partially ordered set, which is also equipped with the subspace topology inherited  from the product topology. 
\end{enumerate}
\end{definition}

The category of profinite groups forms a full subcategory of the category of topological groups. By a {\em homomorphism} or an {\em isomorphism} between profinite groups, we always mean a   homomorphism  or an   isomorphism  between them as topological groups, unless otherwise stated. 

A direct characterisation of profinite groups is shown by the following proposition.
\begin{proposition}[{\cite[Theorem 2.1.3]{RZ10}}]\label{PROP: Top}
A topological group $G$ is a profinite group if and only if $G$ is compact, Hausdorff, and totally disconnected. 
\end{proposition}

\subsection{Profinite completion of a group}
In \autoref{DEF}, we have defined the {\em profinite completion} of an abstract group $\Gamma$, namely
$$
\widehat{\Gamma}=\limi_{N\in \mathcal{I}}\Gamma/N,
$$
where $\mathcal{I}$ denotes the directed poset of finite-index normal subgroups in $\Gamma$ ordered by reverse inclusion.

There is a {\em canonical homomorphism} $\iota:\Gamma\to \widehat{\Gamma}$ defined by $$\iota(\gamma)=(\gamma N)_{N\in \mathcal{I}},$$ whose image is dense in $\widehat{\Gamma}$.  Note that $\iota$ is injective if and only if $\Gamma$ is {\em residually finite}, i.e.\ $\bigcap_{N\in \mathcal{I}} N=\{1\}$. 
\begin{convention}\label{conv}
For brevity of notations, when $\Gamma$ is residually finite, we usually identify $\Gamma$ with its image in $\widehat{\Gamma}$ via $\iota$. 
\end{convention}

For any subset $H\subseteq \Gamma$, we denote by $\overline{H}$ the {\em topological closure} of $\iota(H)$ in $\widehat{\Gamma}$. Note that when $H$ is a subgroup of $\Gamma$, $\overline{H}$ is   a closed  subgroup of $\widehat{\Gamma}$.

\begin{lemma}\label{lem: finite index closure intersection}
Let $\Gamma$ be a group, and let $\Gamma_0\le \Gamma$ be a finite-index subgroup. For any subset $H\subseteq \Gamma$, $\overline{H}\cap \overline{\Gamma_0}=\overline{H\cap \Gamma_0}$. 
\end{lemma}
\begin{proof}
One can find a finite-index normal subgroup $N\unlhd \Gamma$ such that $N\subseteq \Gamma_0$. By projecting $\widehat{\Gamma}$ onto $\Gamma/N$, it is easy to verify that  $\widehat{\Gamma}$ is the disjoint union of $\overline{\Gamma_0}$ and $\overline{\Gamma\setminus \Gamma_0}$. Thus, 
$\overline{H}\cap \overline{\Gamma_0}=(\overline{H\cap \Gamma_0}\cup\overline{H\cap(\Gamma\setminus \Gamma_0)})\cap \overline{\Gamma_0}= \overline{H\cap \Gamma_0}$. 
\end{proof}

The profinite completion satisfies the following universal property.  
\begin{proposition}[{\cite[Lemma 3.2.1]{RZ10}}]\label{prop: universal property}
Let $\Gamma$ be an abstract group and let $G$ be a profinite group. Suppose that $\phi:\Gamma\to G$ is a homomorphism. Then, there exists a unique continuous homomorphism $\Phi:\widehat{\Gamma}\to G$ such that the following diagram commutes. 
\begin{equation*}
\begin{tikzcd}
\Gamma \arrow[dr,"\phi"] \arrow[d,"\iota"'] & \\
\widehat{\Gamma}  \arrow[r,"\Phi"'] & G 
\end{tikzcd}
\end{equation*}
\end{proposition}

\subsection{Powers in a profinite group}
 
\begin{example}
By the Chinese remainder theorem, $\widehat{\Z}\cong \prod_p \Z_p$, where $\Z_p$ refers to the $p$-adic integers, and $p$ runs through all prime numbers. 
\end{example}
Let $G$ be a profinite group and let $g\in G$. Applying \autoref{prop: universal property} to the homomorphism $\phi_g:\Z \to G$ that sends $1$  to $g$, we obtain   a unique continuous homomorphism $\Phi_g: \widehat{\Z} \to G$ such that $\Phi_g(1)=g$. The definition of $\Phi_g$ can actually be illustrated as follows. Suppose $G=\limi_{i\in I} G_i$, where each $G_i$ is a finite group, and $g=(g_i)_{i\in I}$. Then, for any $\lambda\in \widehat{\Z}$, $$\Phi_g(\lambda)=\left(g_i^{\lambda (\mathrm{mod}\,{|G_i|})}\right )_{i\in I}\in \limi_{i\in I} G_i.$$ 
$\Phi_g(\lambda)$ is referred to as the {\em $\lambda$-power} of $g$ in $G$, and is usually abbreviated into $g^\lambda$. 

\subsection{The completion functor}
\begin{sloppypar}
The profinite completion is functorial \cite[Lemma 3.2.3]{RZ10}. A homomorphism  $\phi:\Gamma \to \Gamma' $   of abstract groups canonically induces a continuous homomorphism $\widehat{\phi}:\widehat{\Gamma} \to \widehat{\Gamma'}$, which we illustrate as follows.  Let $\mathcal{I}$ and $\mathcal{I}'$ be the directed posets of finite-index normal subgroups of $\Gamma$ and $\Gamma'$. Then,  $\phi^{-1}(\mathcal{I}')$ is a directed subposet  of $\mathcal{I}$, and $\widehat{\phi}$ is defined by composing the following homomorphisms:
\begin{equation*}
\begin{tikzcd}
{\displaystyle \widehat{\phi}: \;\widehat{\Gamma}=\limi_{N\in \mathcal{I}} \Gamma/N} \arrow[r] & {\displaystyle \limi_{N\in \phi^{-1}(\mathcal{I}')} \Gamma/N } \arrow[r,"\phi"] & {\displaystyle \limi_{N'\in \mathcal{I}'} \Gamma'/N'= \widehat{\Gamma'}.}
\end{tikzcd}
\end{equation*}

By construction, the following diagram commutes. 
\begin{equation*}
\begin{tikzcd}
 \Gamma \arrow[r,"\phi"] \arrow[d,"\iota"'] & \Gamma' \arrow[d,"\iota'"] \\ 
\widehat{\Gamma} \arrow[r, "\widehat{\phi}"] & \widehat{\Gamma'}
\end{tikzcd}
\end{equation*}
In particular, $\widehat{\phi}(\widehat{\Gamma})=\overline{\phi(\Gamma)}$ since $\iota(\Gamma)$ is dense in the compact space $\widehat{\Gamma}$, $\widehat{\Gamma'}$ is Hausdorff, and $\widehat{\phi}$ is continuous. 
\end{sloppypar}

We record here some exactness properties of the completion functor. 

\begin{proposition}[{\cite[Proposition 3.2.5]{RZ10}}]\label{prop: right exact}
Given a short exact sequence of abstract groups
\begin{equation*}
\centering
\begin{tikzcd}
1 \arrow[r] & K \arrow[r, "\varphi"] & \Gamma \arrow[r, "\psi"] & Q \arrow[r] & 1,
\end{tikzcd}
\end{equation*}
there is an exact sequence of profinite groups
\begin{equation*}
\centering
\begin{tikzcd}
\widehat K \arrow[r, "\widehat \varphi"] & \widehat \Gamma \arrow[r, "\widehat \psi"] & \widehat Q \arrow[r] & 1.
\end{tikzcd}
\end{equation*}
\end{proposition}

The {\em profinite topology} on an abstract group $\Gamma$ is generated by $\{ \gamma N\mid \gamma\in \Gamma,\,N\in \mathcal{I}\}$ as a basis of open subsets. 

\begin{proposition}[{\cite[Lemma 3.2.6]{RZ10}}]\label{prop: left exact}
Suppose $\varphi:H\to \Gamma$ is an injective homomorphism of abstract groups. Then $\widehat{\varphi}:\widehat{H}\to \widehat{\Gamma}$ is injective if and only if  the profinite topology on $\Gamma$ induces the full profinite topology on $H$ via $\varphi$. 
In this case, we have an isomorphism $\widehat{H}\cong \overline{\varphi(H) }$. 
\end{proposition}
 
A group $\Gamma$ is {\em LERF} if any finitely generated subgroup of $\Gamma$ is closed in the profinite topology of $\Gamma$. This is equivalent to say that $\Gamma$ is residually finite, and for any finitely generated subgroup $H\le \Gamma$, $\overline{H}\cap \Gamma=H$.

\begin{corollary}\label{cor: lerf injective}
Suppose $\Gamma$ is LERF. For any finitely generated subgroup $H\le \Gamma$, the profinite topology of $\Gamma$  induces the full profinite topology on $H$. In particular, the map $\widehat{H}\to \widehat{\Gamma}$ is injective. 
\end{corollary}

For later use, we state another sufficient condition for the injectivity of the profinite completion.
\begin{proposition}[{\cite[Proposition 3]{And74}}]\label{prop: trivial center injective}
Suppose $\Gamma$ is an abstract group, and $H$ is a finitely generated normal subgroup of $\Gamma$. If $\widehat{H}$ has trivial center, then the map $\widehat{H}\to \widehat{\Gamma}$  is injective. 
\end{proposition}

\subsection{Lattice of finite-index subgroups}

In addition to \autoref{FACT} which states that the profinite completion $\widehat{\Gamma}$ encodes the set of finite quotient groups of a finitely generated group $\Gamma$, $\widehat{\Gamma}$ actually encodes the lattice of finite-index subgroups of $\Gamma$ as shown by the following proposition. 

\begin{proposition}[{\cite[Proposition 3.2.2]{RZ10}}]\label{prop: lattice of open subgroup}
Let $\Gamma$ be an abstract group. There is an isomorphism of lattices: 
\begin{equation*}
\begin{tikzcd}[row sep=0cm]
{\{\text{finite-index subgroups of }\Gamma\}} \arrow[r, leftrightarrow] & {\{\text{open subgroups of }\widehat{\Gamma}\}}\\
H \arrow[r,maps to] & \overline{H}\cong \widehat{H}\\ 
\iota^{-1}(U) & U \arrow[l, maps to]
\end{tikzcd}
\end{equation*}
which also matches up the normal subgroups. 
\end{proposition}

\begin{definition}\label{DEF: corresponding}
Let $\Gamma_1$ and $\Gamma_2$ be abstract groups, and let   $\Phi:\widehat{\Gamma_1}\to \widehat{\Gamma_2}$ be an isomorphism. We say that a pair of finite-index subgroups $\Gamma_1'\le \Gamma_1$ and $\Gamma_2'\le \Gamma_2$ form  a {\em $\Phi$-corresponding pair} if $\Phi(\overline{\Gamma_1'})=\overline{\Gamma_2'}$.  In this case, the {\em restriction}  of $\Phi$ yields an isomorphism $\Phi':\widehat{\Gamma_1'}\to \widehat{\Gamma_2'}$. 
\end{definition}

By \autoref{prop: lattice of open subgroup}, the $\Phi$-correspondence yields an isomorphism between the lattices of finite-index subgroups of $\Gamma_1$ and $\Gamma_2$. 

\begin{definition}
Let $\Gamma_1$ and $\Gamma_2$ be abstract groups. We say that an isomorphism $\Phi:\widehat{\Gamma_1}\to \widehat{\Gamma_2}$ {\em virtually} satisfies a property (X) if there exists a $\Phi$-corresponding pair of finite-index subgroups $\Gamma_1'\le \Gamma_1$ and $\Gamma_2'\le \Gamma_2$ such that the restriction $\Phi':\widehat{\Gamma_1'}\to \widehat{\Gamma_2'}$ satisfies property (X). 
\end{definition}

In practice, we often consider a specific cofinal family of finite-index subgroups. 
\begin{definition}
Let $\Gamma$ be a finitely generated group. The {\em $m$-th standard characteristic subgroup} of $\Gamma$ is a finite-index subgroup   defined by
$$
\Gamma^{( m)}=\bigcap_{H\le \Gamma,\,[\Gamma:H]\le m} H.
$$
\end{definition}

According to \autoref{prop: lattice of open subgroup} (and \autoref{lem: finite index closure intersection}), $$\overline{\Gamma^{(m)}}=\bigcap_{U\le_o\widehat{\Gamma},\,[\widehat{\Gamma}:U]\le m} U, $$
where ``$\le_o$'' means open subgroup in $\widehat{\Gamma}$. In particular, $\overline{\Gamma^{(m)}}$ is also a characteristic subgroup of $\widehat{\Gamma}$.  Furthermore, if $\Gamma_1$ and $\Gamma_2$ are finitely generated groups  and $\Phi:\widehat{\Gamma_1}\to \widehat{\Gamma_2}$ is an isomorphism,  then for any $m\in \N$,  $\Gamma_1^{(m)}$ and $\Gamma_2^{(m)}$ form  a $\Phi$-corresponding pair of finite-index subgroups.

When $\Gamma$ is the fundamental group of a manifold $M$ with respect to a   basepoint $x\in M$, finite-index subgroups of $\Gamma$ correspond to pointed finite covers of $M$, i.e.\ finite covers of $M$ marked with a specified preimage of the basepoint $x$. We sometimes omit the basepoint when it is not important, for example when considering finite regular covers of $M$, which correspond to finite-index normal subgroups of $\Gamma$. In particular, the finite regular cover of $M$ corresponding to the standard characteristic subgroup $\Gamma^{(m)}$ is denoted by $M^{(m)}$ and is referred to as the {\em $m$-th standard characteristic cover} of $M$.

\section{Centralizer and outer automorphism}\label{sec: inj}
In this section, we prove the injectivity of the map  $\Out(\Gamma)\to \Out(\widehat{\Gamma})$ appearing in \autoref{mainthm2}. We also include some extra facts that will be used in later sections.  
\subsection{Conjugacy separability}
\begin{definition}
A group $\Gamma$ is {\em conjugacy separable} if for any pair of non-conjugate elements $x,y\in \Gamma$, there exists a finite quotient $q:\Gamma\to Q$ such that $q(x)$ and $q(y)$ are non-conjugate in $Q$. 
\end{definition}
Note that a conjugacy separable group is necessarily residually finite. 
An equivalent definition for conjugacy separability is that two elements $x,y\in \Gamma$ are conjugate in $\Gamma$ if and only if their images are conjugate in $\widehat{\Gamma}$.

\begin{proposition}[{\cite[Theorem 1.3]{HWZ13}}]\label{prop: conjugacy separable}
The fundamental group  of any compact orientable 3-manifold is  conjugacy separable.
\end{proposition}

For a group $G$ and an element $g\in G$, we denote $C_{G}(g) = \{x\in G\mid xg=g x \}$ as the {\em centralizer} of $g$ in $G$. Note that when $G$ is a Hausdorff topological group,
$C_{G}(g)$ is a closed subgroup of $G$.  
The following proposition is due to  Minasyan \cite{Min12}. 
\begin{proposition}[{\cite[Corollary 12.3]{Min12}}]\label{Minasyan}
Suppose that $\Gamma$ is a residually finite group and that every finite-index subgroup of $\Gamma$ is conjugacy separable. Then, for any $\gamma \in \Gamma$, $C_{\widehat{\Gamma}}(\gamma)=\overline{C_{\Gamma}(\gamma)}$. 
\end{proposition}

By a {\em Kleinian group}, we will mean a discrete subgroup of $\PSL_2(\C)$. 

\begin{corollary}\label{cor: centralizer}
Let $\Gamma$ be a finitely generated  torsion-free   Kleinian group. 
Then, $C_{\widehat{\Gamma}}(\gamma)=\overline{C_{\Gamma}(\gamma)}$ for any $\gamma\in \Gamma$.
\end{corollary}
\begin{proof}
Every finite-index subgroup $\Gamma_0$ of $\Gamma$ is a finitely generated  torsion-free   Kleinian group, and is hence the fundamental group of a compact orientable 3-manifold according to Scott's theorem \cite{Sco73}. Thus, $\Gamma_0$ is conjugacy separable according to \autoref{prop: conjugacy separable}, and the corollary follows from \autoref{Minasyan}.
\end{proof}

\subsection{Center of profinite completion} This subsection is devoted to the proof of the following theorem.
\begin{theorem}\label{thm: lattice center free}
Let $\Gamma$ be a finitely generated  non-elementary    Kleinian group.  Then, $\widehat{\Gamma}$ has trivial center. 
\end{theorem}

Note that finitely generated linear groups are residually finite \cite[Теорема VII]{Mal40}, so we may always identify $\Gamma$ with its image in $\widehat{\Gamma}$ as in \autoref{conv}. 
Before the proof, we need some preprations. 

\begin{proposition}[{\cite[Corollary 9.4]{Ago13}}]\label{FK: LERF}
Any finitely generated   Kleinian group is LERF. 
\end{proposition}
We also involve the LERFness of finitely generated Fuchsian groups  in this proposition, which was formerly proven by Scott \cite[Theorem 3.2]{Sco78}.

\begin{lemma}\label{lem: free group lambda}
Let $F$ be a free group over two generators $x$ and $y$, and identify $F$ with its image in $\widehat{F}$. For $\lambda\in \widehat{\Z}$, $x^\lambda y x^{-\lambda}\in F$ if and only if $\lambda \in \Z$. 
\end{lemma}
\begin{proof}
This lemma is well-known, but we would like to include an interesting and self-contained proof. The ``if'' part is clear, and we shall prove the ``only if'' part. 

For any commutative unital ring $R$, let $$\mathrm{H}(R)=\left\{\left.\begin{pmatrix} 1 & a & b \\ 0 & 1 & c \\ 0 & 0 & 1\end{pmatrix}\right| a,b,c\in R\right\}\subseteq \GL_3(R)$$ be the Heisenberg group over $R$. Any ring homomorphism $R\to R'$ induces a group homomorphism $\mathrm{H}(R)\to \mathrm{H}(R')$. In particular, $\mathrm{H}(\widehat{\Z})=\limi \mathrm{H}(\Z/n\Z)$ is a profinite group, and $\mathrm{H}(\Z)$ injects into $\mathrm{H}(\widehat{\Z})$. 

Consider a homomorphism
$\phi: F\to \mathrm{H}(\Z)$ defined by
$$
\phi(x)=\begin{pmatrix} 1 & 1 & 0 \\ 0 & 1 & 0 \\ 0 & 0 & 1\end{pmatrix}\quad \text{and}\quad \phi(y)=\begin{pmatrix} 1 & 0 & 0 \\ 0 & 1 & 1 \\ 0 & 0 & 1\end{pmatrix}.
$$
By \autoref{prop: universal property}, there exists a unique continuous homomorphism $\Phi:\widehat{F}\to \mathrm H(\widehat{\Z})$ such that the following diagram commutes.
\begin{equation*}
\begin{tikzcd}
F \arrow[r,"\phi"] \arrow[d,"\iota"',hook] & \mathrm{H}(\Z) \arrow[d, hook]\\
\widehat{F} \arrow[r,"\Phi"] & \mathrm{H}(\widehat{\Z})
\end{tikzcd}
\end{equation*}
By a direct calculation, 
$$\Phi(x^\lambda y x^{-\lambda})=\begin{pmatrix} 1 & \lambda  & 0 \\ 0 & 1 & 0 \\ 0 & 0 & 1\end{pmatrix}\begin{pmatrix} 1 & 0 & 0 \\ 0 & 1 & 1 \\ 0 & 0 & 1\end{pmatrix}\begin{pmatrix} 1 &-\lambda  & 0 \\ 0 & 1 & 0 \\ 0 & 0 & 1\end{pmatrix}=\begin{pmatrix} 1 & 0 & \lambda \\ 0 & 1 & 1 \\ 0 & 0 & 1\end{pmatrix}.$$
In particular, $\Phi(x^\lambda y x^{-\lambda})\notin \mathrm{H}(\Z)$ when $\lambda\notin\Z$. Therefore, $x^\lambda yx^{-\lambda}\notin F$ when $\lambda\notin \Z$. 
\end{proof}

For a group $G$ and a subgroup $H\le G$, denote $N_{G}(H)=\{g\in G\mid gHg^{-1}=H\}$ as the {\em normalizer} of $H$ in $G$. 

\begin{lemma}\label{lem: normalize}
Let $\Gamma$ be a finitely generated non-elementary torsion-free Kleinian group. Then, $N_{\widehat{\Gamma}}(\Gamma)=\Gamma$. 
\end{lemma}
\begin{proof}
It is clear that $N_{\widehat{\Gamma}}(\Gamma)\supseteq \Gamma$, and we shall prove the inverse. Suppose $\hat{g}\in N_{\widehat{\Gamma}}(\Gamma)$, and let $x\in \Gamma$ be a non-power loxodromic element. Then $x'=\hat{g}x\hat{g}^{-1}\in \Gamma$ is  conjugate to $x$ in $\widehat{\Gamma}$.  Since $\Gamma$ is conjugacy separable by \autoref{prop: conjugacy separable}, we deduce   that $x$ and $x'$ are conjugate in $\Gamma$. Thus, we can pick $\gamma \in \Gamma$ such that $x=\gamma x'\gamma^{-1}$.  If we denote $\hat{h}=\gamma\hat{g}$, then $\hat{h}\in  C_{\widehat{\Gamma}}(x)$. 

According to \autoref{cor: centralizer}, $C_{\widehat{\Gamma}}(x)=\overline{C_{\Gamma}(x)}$. Note that $x$ is a non-power loxodromic element and that $\Gamma$ is torsion-free, so $C_{\Gamma}(x)=\langle x\rangle$. Consequently, $\hat{h}\in \overline{\langle x\rangle}$; in other words, there exists $\lambda\in \widehat{\Z}$ such that $\hat{h}=x^\lambda$. 

Since $\Gamma$ is a non-elementary   Kleinian group, we can find another loxodromic element $y\in \Gamma$   independent with $x$,  such that for some sufficiently large $n\in \N$, $X=x^n$ and $Y=y^n$ generate a free subgroup $F$ over $X$ and $Y$ by the ping-pong lemma. Note that $\Gamma$ is LERF (\autoref{FK: LERF}), so $\overline{F}\cap \Gamma=F$, and $\overline{F}\cong \widehat{F}$ according to \autoref{cor: lerf injective}. 

Recall that both $\hat{g}$ and $\gamma$ belong to $N_{\widehat{\Gamma}}(\Gamma)$, so $\hat{h}=\gamma\hat{g}\in N_{\widehat{\Gamma}}(\Gamma)$, and $X^\lambda=x^{n\lambda}=\hat{h}^n\in N_{\widehat{\Gamma}}(\Gamma)$. Consequently, $X^\lambda Y X^{-\lambda}\in \Gamma$. Note that $X^\lambda$ and $Y$ belong to $\overline{F}\cong \widehat{F}$, so this implies that $X^\lambda Y X^{-\lambda}\in \overline{F}\cap \Gamma=F$. From \autoref{lem: free group lambda}, we deduce that $\lambda\in \Z$. Hence, $\hat{h}=x^\lambda\in \Gamma$, and $\hat{g}= \gamma^{-1}\hat{h} \in \Gamma$. In other words, we have proven that $N_{\widehat{\Gamma}}(\Gamma)\subseteq \Gamma$. 
\end{proof}

We then remove the torsion-free assumption. 

\begin{corollary}\label{cor: normalizer}
Let $\Gamma$ be a finitely generated non-elementary Kleinian group. Then, $N_{\widehat{\Gamma}}(\Gamma)=\Gamma$. 
\end{corollary}
\begin{proof}
It is clear that $N_{\widehat{\Gamma}}(\Gamma)\supseteq \Gamma$, and we prove the inverse. Suppose $\hat{g}\in N_{\widehat{\Gamma}}(\Gamma)$. 

According to  Selberg's lemma \cite[Lemma 8]{Sel62}, we can find a finite-index torsion-free   subgroup $\Gamma_0\le \Gamma$. Then, $\overline{\Gamma_0}\cong \widehat{\Gamma_0}$ is an open subgroup of $\Gamma$ by \autoref{prop: lattice of open subgroup}. Since $\Gamma\le \widehat{\Gamma}$ is a dense subgroup, we have $\widehat{\Gamma}=\Gamma\cdot \overline{\Gamma_0}$. Hence, there exist $\gamma \in \Gamma$ and $\hat{h}\in \overline{\Gamma_0}$ such that $\hat{g}=\gamma \hat{h}$. Consequently, $\hat{h}$ belongs to $N_{\widehat{\Gamma}}(\Gamma)$ since  $\gamma$ does. 

Note that $\hat{h}\overline{\Gamma_0}\hat{h}^{-1}=\overline{\Gamma_0}$ and that $\Gamma_0=\overline{\Gamma_0}\cap \Gamma$. Hence, $$\hat{h}\Gamma_0\hat{h}^{-1}= \hat{h}(\overline{\Gamma_0}\cap \Gamma)\hat{h}^{-1}= \hat{h}\overline{\Gamma_0}\hat{h}^{-1} \cap \hat{h}\Gamma \hat{h}^{-1}= \overline{\Gamma_0}\cap \Gamma = \Gamma_0.$$ In other words, $\hat{h}\in N_{\overline{\Gamma_0}}(\Gamma_0)$. Recall that $\overline{\Gamma_0}\cong \widehat{\Gamma_0}$, so \autoref{lem: normalize} implies that $\hat{h}\in \Gamma_0$. Therefore, $\hat{g}= \gamma\hat{h}\in \Gamma$, and we have proven that $N_{\widehat{\Gamma}}(\Gamma)\subseteq \Gamma$. 
\end{proof}

We can now prove \autoref{thm: lattice center free}. 
\begin{proof}[Proof of \autoref{thm: lattice center free}]
Denote by $Z(G)$ the center of a group $G$. It is clear  from definition that $Z(\widehat{\Gamma})\subseteq N_{\widehat{\Gamma}}(\Gamma)$. Thus, \autoref{cor: normalizer} implies that $Z(\widehat{\Gamma})\subseteq \Gamma$. Hence, $Z(\widehat{\Gamma})\subseteq Z(\Gamma)$. However, any non-elementary Kleinian group has trivial center. Therefore, $Z(\widehat{\Gamma})$ is also trivial. 
\end{proof}

\subsection{Injectivity}

We state a general equivalent condition for the injectivity of the completion map $\Out(\Gamma)\to \Out(\widehat{\Gamma})$. For an arbitary group $G$ and an element $g\in G$, denote by $\Inn_{g}:G\to G$ the {\em inner automorphism}  $\Inn_g(x)=gxg^{-1}$. 
\begin{proposition}\label{prop: injective criterion}
Suppose $\Gamma$ is a residually finite group. The profinite completion map $\Out(\Gamma)\to \Out(\widehat{\Gamma})$  is injective if and only if $N_{\widehat{\Gamma}}(\Gamma)=\Gamma\cdot Z(\widehat{\Gamma})$, where $Z(\widehat{\Gamma})$ denotes the center of $\widehat{\Gamma}$. 
\end{proposition}
\begin{proof}
We first show that this condition is sufficient. It suffices to show that for a given $\varphi\in \Aut(\Gamma)$, if $\widehat{\varphi}\in \Aut(\widehat{\Gamma})$ is an inner-automorphism, then $\varphi$ is itself an inner-automorphism. Suppose $\widehat{\varphi}=\Inn_{\hat{g}}$ for some $\hat{g}\in \widehat{\Gamma}$. By construction of the profinite completion, $\widehat{\varphi}(\Gamma)=\Gamma$,  so  $\hat{g}\in N_{\widehat{\Gamma}}(\Gamma)$. By assumption, there exists $\gamma\in \Gamma$ and $\hat{h}\in Z(\widehat{\Gamma})$ such that $\hat{g}=\gamma\hat{h}$. Then, $\Inn_{\hat{g}}=\Inn_{\gamma}\circ \Inn_{\hat{h}}=\Inn_{\gamma}$, and $\varphi=\widehat{\varphi}|_{\Gamma}$ is an inner automorphism induced by $\gamma$. 

Conversely, we show that this condition is necessary. Suppose that $\Out(\Gamma)\to \Out(\widehat{\Gamma})$ is injective, and $\hat{g}\in N_{\widehat{\Gamma}}(\Gamma)$. Then, $\Inn_{\hat{g}}(\Gamma)=\Gamma$, and $\varphi=\Inn_{\hat{g}}|_{\Gamma}$ defines an automorphism on $\Gamma$. Note that $\Inn_{\hat{g}}=\widehat{\varphi}$ since they coincide on a dense subgroup $\Gamma$ of $\widehat{\Gamma}$. Thus, the injectivity assumption implies that $\varphi$ is an inner automorphism of $\Gamma$. Suppose that $\varphi=\Inn_{\gamma}$ for some $\gamma\in \Gamma$. Then, $\Inn_{\gamma^{-1}\hat{g}}$ restricts to the identity map on $\Gamma$. In other words, $\gamma^{-1}\hat{g}$ belongs to the centralizer of $\Gamma$ in $\widehat{\Gamma}$.  Since $\Gamma$ is dense in $\widehat{\Gamma}$, $\gamma^{-1}\hat{g}\in Z(\widehat{\Gamma})$  by continuity. Hence, $\hat{g}\in \Gamma\cdot Z(\widehat{\Gamma})$. 
\end{proof}

\begin{corollary}\label{cor: completion map injective}
Let $\Gamma$ be a finitely generated non-elementary Kleinian group. Then, the profinite completion map $\Out(\Gamma)\to \Out(\widehat{\Gamma})$  is injective. 
\end{corollary}
\begin{proof}
This follows from combining the criterion \autoref{prop: injective criterion} with \autoref{thm: lattice center free} and \autoref{cor: normalizer}. 
\end{proof}

\section{Genuine isomorphisms}\label{sec: surj}

\begin{definition}
Let $\Gamma_1$ and $\Gamma_2$ be abstract groups. An isomorphism $\Phi:\widehat{\Gamma_1}\to \widehat{\Gamma_2}$ is called {\em genuine} if there exists an isomorphism $f: \Gamma_1 \to \Gamma_2$ and an element $ \hat g\in \widehat{\Gamma_2}$ such that $\Phi=\Inn_{  \hat g}\circ \widehat{f}$. 
In particular, $\Gamma_1$ and $\Gamma_2$ are always meant to be isomorphic when $\Phi$ is genuine. 
\end{definition}

In this section, we provide two criteria that determine whether  such an isomorphism is genuine, both of which will be crucial in proving the surjectivity of the map $\Out(\Gamma)\to \Out(\widehat{\Gamma})$ appearing in \autoref{mainthm2}. 

\subsection{Standard characteristic quotients}\label{sec: Standard characteristic quotients}
The first criterion reduces the problem of genuineness to the standard characteristic quotients. 
Let us introduce some notations to simplify our exposition. 
\begin{definition}\label{def: conjugacy equivalent}
Let $G_1$ and $G_2$ be two arbitrary (abstract or profinite) groups. Two homomorphisms $\phi,\psi: G_1\to G_2$ are {\em conjugacy equivalent}, denoted by $\phi \simeq \psi$, if there exists $g\in G_2$ such that $\phi=\Inn_g\circ \psi$. 
\end{definition}

Suppose  $\Gamma_1$ and $\Gamma_2$ are finitely generated groups, and let $\phi:\widehat{\Gamma_1}\to \widehat{\Gamma_2}$ be a  homomorphism, which might  not necessarily be an isomorphism. Then, for each $m\in \N$, $\phi(\overline{\nss{\Gamma}{1}{(m)}})\subseteq \overline{\nss{\Gamma}{2}{(m)}}$, where $\Gamma_i^{(m)}$ denotes the $m$-th standard characteristic subgroup. 
This is because for any open subgroup $U\le \widehat{\Gamma_2}$, $\phi^{-1}(U)$ is an open subgroup of $\widehat{\Gamma_1}$ with $[\widehat{\Gamma_1}:\phi^{-1}(U)]\le [\widehat{\Gamma_2}: U]$. 
Consequently, $\phi$ descends to a homomorphism 
\begin{equation*}\begin{tikzcd}
\phi^{\flat}_{(m)}: \Gamma_1/\Gamma_1^{(m)}=\widehat{\Gamma_1}/\overline{\nss{\Gamma}{1}{(m)}}\arrow[r] & \widehat{\Gamma_2}/\overline{\nss{\Gamma}{2}{(m)}}=\Gamma_2/\Gamma_2^{(m)}.
\end{tikzcd}\end{equation*}
\begin{lemma}\label{lem: conjugacy equivalence criteria}
Let $\Gamma_1$ and $\Gamma_2$ be finitely generated groups. Suppose $\phi,\psi:\widehat{\Gamma_1}\to \widehat{\Gamma_2}$ are two homomorphisms. Then, $\phi\simeq \psi$ if and only if $\phi^\flat_{(m)}\simeq \psi^\flat_{(m)}$ for every $m\in \N$. 
\end{lemma}
\begin{proof}
First, suppose $\phi\simeq \psi$. Then there exists $\hat g\in \widehat{\Gamma_2}$ such that $\phi=\Inn_{\hat g}\circ \psi$. For each $m\in \N$, let $g_m$ be the image of $\hat g$ in $\Gamma_2/\Gamma_2^{(m)}$. Then by construction, $\phi^\flat_{(m)}=\Inn_{g_m}\circ \psi^\flat_{(m)}$. In other words, $\phi^\flat_{(m)}\simeq \psi^\flat_{(m)}$ for each  $m\in \N$. 

Conversely, suppose that $\phi^\flat_{(m)}\simeq \psi^\flat_{(m)}$ for every $m\in \N$.  For each $m\in \N$, let
$$
X_m=\left\{h\in \Gamma_2/\Gamma_2^{(m)}\mid \phi^\flat_{(m)}=\Inn_{h}\circ \psi^\flat_{(m)} \right\},
$$
which is a non-empty finite set. If we denote by $q_m: \Gamma_2/\Gamma_2^{(m+1)}\to \Gamma_2/\Gamma_2^{(m)}$ the quotient homomorphism, then it is easy to verify that $q_m(X_{m+1})\subseteq X_{m}$. In other words, $(X_m,q_m)_{m\in\N}$ forms an inverse system of non-empty finite sets. Consequently, $\limi_m X_m$ is non-empty by \cite[Proposition 1.1.4]{RZ10}. Pick $\hat g\in \limi_m X_m$. Then, $\hat g\in \limi_m \Gamma_2/\Gamma_2^{(m)}=\widehat{\Gamma_2}$. Note that $\phi=\limi_m \phi^\flat_{(m)}$ and $\psi=\limi_m \psi^\flat_{(m)}$, so by construction, $\phi=\Inn_{\hat g}\circ \psi$. In other words, $\phi \simeq \psi$. 
\end{proof}

Reformulated by \autoref{def: conjugacy equivalent}, an isomorphism $\Phi:\widehat{\Gamma_1}\to \widehat{\Gamma_2}$ is genuine if and only if there exists an isomorphism $f:\Gamma_1\to \Gamma_2$ such that $\Phi \simeq \widehat{f}$.  
Note that when $\Gamma_1$ and $\Gamma_2$ are finitely generated, the isomorphism $f$ sends $\Gamma_1^{(m)}$ to $\Gamma_2^{(m)}$, 
and  $f$ descends to an isomorphism $f^\flat_{(m)}:\Gamma_1/\Gamma_1^{(m)}\to \Gamma_2/\Gamma_2^{(m)}$. It follows from construction that $f^\flat_{(m)}={\widehat{f}}{}^{\,\,\flat}_{(m)}$. Hence, we have  the following corollary of \autoref{lem: conjugacy equivalence criteria}. 

\begin{corollary}\label{cor: genuine criteria 1}
Let $\Gamma_1$ and $\Gamma_2$ be finitely generated groups, and let $\Phi:\widehat{\Gamma_1}\to \widehat{\Gamma_2}$ be an isomorphism. Then, $\Phi$ is genuine if and only if there exists an isomorphism $f:\Gamma_1\to \Gamma_2$ such that $\Phi^\flat_{(m)}\simeq f^\flat_{(m)}$ for every $m\in \N$. 
\end{corollary}

\subsection{Extension of groups}
The next criterion is concerned with group extensions. 
We first point out a standard result. 
\begin{proposition}\label{prop: standard abstract algebra}
Consider two extensions of groups $1\to H  \to G_1\to Q \to 1$ and $1\to H \to G_2\to Q \to 1$.  Let $\theta_i : Q \to \Out(H)$ be the homomorphism  induced by $G_i$ acting via conjugation on $H $. 
Suppose there exist isomorphisms $\psi : H \to H $ and $\phi :Q \to Q $ such that the following diagram commutes. 
\begin{equation*}
\begin{tikzcd}
Q  \arrow[d,"\phi"', "\cong"]  \arrow[r,"\theta_1"] & \Out(H ) \arrow[d,"\cong"',"\psi_\ast"] \\
Q  \arrow[r,"\theta_2"] & \Out(H )
\end{tikzcd}
\end{equation*}
If $H$ has trivial center, then there exists a unique isomorphism $f: G_1\to G_2$ such that the following diagram commutes. 
\begin{equation}\label{equ: group extension commutative diagram}
\begin{tikzcd}
1 \arrow[r] & H  \arrow[r] \arrow[d,"\psi"]  & G_1 \arrow[r] \arrow[d,"f"]  & Q  \arrow[r] \arrow[d,"\phi"] & 1 \\
1 \arrow[r] & H \arrow[r] & G_2 \arrow[r] & Q  \arrow[r] & 1 
\end{tikzcd}
\end{equation}
\end{proposition}
\begin{proof}
We first prove the existence of $f$. 
  Since $H $ has trivial center, we have a short exact sequence $1\to H \to \Aut(H )\to \Out(H )\to 1$. In addition, the action of $G_i$ by conjugation on $H $ induces a commutative diagram between short exact sequences. 
\begin{equation*}
\begin{tikzcd}
1 \arrow[r] & H  \arrow[r] \arrow[d, equal ]  & G_i \arrow[r] \arrow[d ]  & Q  \arrow[r] \arrow[d, "\theta_i" ] & 1 \\
1 \arrow[r] & H  \arrow[r] & \Aut(H ) \arrow[r,"p "] & \Out(H ) \arrow[r] & 1 
\end{tikzcd}
\end{equation*}
Consequently, the homomorphism
\begin{equation*}
\rho_i:\,G_i \tto Q  {\times}_{ {(\theta_i,p)} } \Aut(H )=\{(q,\varphi)\in Q \times \Aut(H ) \mid \theta_i(q)=p (\varphi)\}
\end{equation*}
is an isomorphism. 

Combining them for $i=1,2$, we obtain an isomorphism $f: G_1\to G_2$ that fits into the following commutative diagram. 
\begin{equation*}
\begin{tikzcd}[column sep= large]
 G_1 \arrow[r,"\rho_1","\cong"' ] \arrow[d, dashed,"f"']  & Q \times_{(\theta_1,p)} \Aut(H ) \arrow[d,"\cong"',"\phi\times \psi_\ast"] \\
G_2 \arrow[r,"\rho_2", "\cong"' ] & Q \times_{(\theta_2,p)} \Aut(H ) 
\end{tikzcd}
\end{equation*}
By construction,   the following    diagram commutes. 
\begin{equation*}
\begin{tikzcd}
H \arrow[r,"\cong"] \arrow[d,"\psi"] & \{1\} \times  \Inn(H ) \arrow[r, hook] \arrow[d, "\psi_\ast"', "\cong"]& Q \times_{(\theta_1,p)} \Aut(H )  \arrow[d,"\phi\times \psi_\ast"] \arrow[r, "\text{proj}"] & Q \arrow[d,"\phi"]\\
H  \arrow[r,"\cong"] & \{1\} \times  \Inn(H ) \arrow[r, hook] & Q \times_{(\theta_2,p)} \Aut(H )   \arrow[r, "\text{proj}"] & Q 
\end{tikzcd}
\end{equation*}
Hence, $f$ fits into diagram (\ref{equ: group extension commutative diagram}). 

Next, we prove the uniqueness of $f$. Suppose there exists another isomorphism $f':G_1\to G_2$ that also satisfies the commutative diagram (\ref{equ: group extension commutative diagram}). Let $\xi=f^{-1}\circ f': G_1\to G_1$. Then, the following diagram commutes. 
\begin{equation*}
\begin{tikzcd}
1 \arrow[r] & H \arrow[r] \arrow[d,"\text{id}"]  & G_1 \arrow[r] \arrow[d,"\xi","\cong"']  & Q \arrow[r] \arrow[d,"\text{id}"] & 1 \\
1 \arrow[r] & H  \arrow[r] & G_1  \arrow[r] & Q \arrow[r] & 1 
\end{tikzcd}
\end{equation*}
For any $g\in G_1$, $g$ and $\xi(g)$ project  to the same element in $Q $. Thus, $g^{-1}\xi(g)\in H $. Moreover, for any $h\in H $,
\begin{equation*}
\begin{aligned}
\left(g^{-1}\xi(g)\right) h= g^{-1} \xi(gh)= g^{-1} \xi(ghg^{-1}) \xi(g) \stackrel{\star}{=} g^{-1} (ghg^{-1}) \xi(g)= h\left (g^{-1}\xi(g)\right ),
\end{aligned}
\end{equation*}
where the marked equality follows from the facts that $ghg^{-1}\in H $ and that $\xi$ restricts to the identity map on $H $. Consequently, $g^{-1}\xi(g)$ belongs to the center of $H $, which is trivial. In other words, $\xi:G_1\to G_1$ is the identity map, and hence  $f'=f$. 
\end{proof}

\begin{remark}
In general, the obstruction for the existence of the isomorphism $f$ belongs to the group cohomology $H^2(Q ;Z(H ))$, and all isomorphisms satisfying (\ref{equ: group extension commutative diagram}) are then classified by the derivations $\mathrm{Der}(Q ,Z(H ))$. Both of them are trivial in our case since $H $ has trivial center.  
\end{remark}

We then introduce an extension lemma which is essential to our proof. 
\begin{lemma}\label{lem: extension lemma}
Let $H$ and $Q$ be two abstract groups, and consider two extensions 
\begin{equation*}
\begin{tikzcd}[column sep=0.65cm]
1 \arrow[r] & H \arrow[r,"u_1"] & G_1 \arrow[r,"v_1"] & Q \arrow[r] & 1 
\end{tikzcd}\quad\text{and}\quad
\begin{tikzcd}[column sep=0.65cm]
1 \arrow[r] & H \arrow[r,"u_2"] & G_2 \arrow[r,"v_2"] & Q \arrow[r] & 1 .
\end{tikzcd}
\end{equation*}
Assume that the following three statements hold. 
\begin{enumerate}[leftmargin=*, label=(\alph*)]
\item\label{extasmpA} $H$ is finitely  generated. 
\item\label{extasmpB} Both $H$ and $\widehat{H}$ have trivial center. 
\item\label{extasmpC} The completion map $\Out(H)\to \Out(\widehat{H})$ is injective. 
\end{enumerate}
Suppose there exist isomorphisms $\Psi: \widehat{H}\to \widehat{H}$, $F: \widehat{G_1}\to \widehat{G_2}$, and $\Phi: \widehat{Q}\to \widehat{Q}$ that fit into the following commutative diagram. 
\begin{equation*}
\begin{tikzcd}
  \widehat H \arrow[r,"\widehat{u_1}"] \arrow[d,"\Psi"] & \widehat{G_1} \arrow[r,"\widehat{v_1}"] \arrow[d,"F"]  & \widehat Q   \arrow[d,"\Phi"]  \\
  \widehat H \arrow[r,"\widehat{u_2}"] & \widehat{G_2} \arrow[r,"\widehat{v_2}"] & \widehat Q  
\end{tikzcd}
\end{equation*}
Assume that $\Phi$ is the profinite completion of an isomorphism $\phi: Q\to Q$, and $\Psi$ is genuine.  Then, $F$ is also genuine; in particular, $G_1\cong G_2$. 
\end{lemma}
\begin{proof}
According to \autoref{prop: right exact} and \autoref{prop: trivial center injective},    assumptions \ref{extasmpA} and \ref{extasmpB} imply that  the profinite completion  induces  short exact sequences  
\begin{equation*}
\begin{tikzcd}[column sep=0.65cm]
1 \arrow[r] & \widehat H \arrow[r,"\widehat{u_1}"] & \widehat{G_1} \arrow[r,"\widehat{v_1}"] & \widehat Q \arrow[r] & 1 
\end{tikzcd}\quad\text{and}\quad
\begin{tikzcd}[column sep=0.65cm]
1 \arrow[r] & \widehat H \arrow[r,"\widehat{u_2}"] & \widehat{G_2} \arrow[r,"\widehat{v_2}"] & \widehat Q \arrow[r] & 1 .
\end{tikzcd}
\end{equation*}

By assumption, we can find $\hat h \in \widehat{H}$  and an isomorphism $\psi: H\to H$ such that $\Psi=\Inn_{\hat h}\circ \widehat{\psi}$. Let $F'=\Inn_{\widehat{u_2}(\hat h)^{-1}} \circ F: \widehat{G_1} \to \widehat{G_2}$ be an isomorphism, which  is conjugacy equivalent with $F$. Since $\widehat{v_2}(\widehat{u_2}(\hat h))=1$, the following diagram commutes. 
\begin{equation}\label{equ: completion ses commutative F}
\begin{tikzcd}
1 \arrow[r] & \widehat H \arrow[r,"\widehat{u_1}"] \arrow[d,"\widehat{\psi}"',"\cong"] & \widehat{G_1} \arrow[r,"\widehat{v_1}"] \arrow[d,"F'"',"\cong"]  & \widehat Q \arrow[r]  \arrow[d,"\widehat{\phi}"',"\cong"] & 1 \\
1 \arrow[r] & \widehat H \arrow[r,"\widehat{u_2}"] & \widehat{G_2} \arrow[r,"\widehat{v_2}"] & \widehat Q \arrow[r] & 1 
\end{tikzcd}
\end{equation}

Let $\iota: Q\to \widehat{Q}$ be the canonical homomorphism. Let $\theta_i: Q\to \Out(H)$ be the homomorphism induced by $G_i$ acting via conjugation on $H$, and let $\Theta_i: \widehat{Q}\to \Out(\widehat{H})$ be the homomorphism induced by  $\widehat{G_i}$ acting via conjugation on $\widehat{H}$. Then, the five blocks \circledtext{1}$\,$---$\,$\circledtext{5} in the following diagram commute. 
\begin{equation}\label{equ: compd diag}
\begin{tikzcd}[row sep=tiny, column sep=tiny]
Q \arrow[dddddd, "\phi"'] \arrow[rrrrdd, "\iota"] \arrow[rrrrrrrr, "\theta_1"] & &   &  &                                                                 &     &                                                     &     & \hspace{-1mm}\Out(H) \arrow[dddddd, "\psi_\ast"] \arrow[lldd, hook,"\widehat{\cdot}"'] \\ & &
                                                                           &     &                                                                 & \hspace{2.75mm}{\scalebox{0.9} {\circledtext{1}}}  \hspace{-2.75mm} &                                                     &     &                                                        \\ & &
                                                                           &     & \widehat{Q} \arrow[rr, "\Theta_1"] \arrow[dd, "\widehat{\phi}"'] &     & \Out(\widehat{H}) \hspace{-2mm}\arrow[dd, "\widehat{\psi}_\ast"] &     &                                                        \\ &
                                                                           & \hspace{-1mm}{\scalebox{0.9}{\circledtext{2}}} \hspace{2mm} &                                                                 & & \hspace{2.75mm}{\scalebox{0.9}{\circledtext{5}}}  \hspace{-2.75mm}  &                                                     & \hspace{2mm}{\scalebox{0.9}{\circledtext{4}}}\hspace{-2mm}  &                                                        \\ & & 
                                                                           &     & \widehat{Q} \arrow[rr, "\Theta_2"]                              &     & \Out(\widehat{H}) \hspace{-2mm}                                  &     &                                                        \\ & &
                                                                           &     &                                                                 & \hspace{2.75mm}{\scalebox{0.9}{\circledtext{3}}}  \hspace{-2.75mm}  &                                                     &     &                                                        \\
Q \arrow[rrrruu, "\iota"] \arrow[rrrrrrrr, "\theta_2"]               &      &     &     &                                                                 &     &                                                     &     & \hspace{-1mm}\Out(H) \arrow[lluu, hook,"\widehat{\cdot}"']                            
\end{tikzcd}
\end{equation}
In fact, blocks \circledtext{1}$\,$---$\,$\circledtext{4} follow from the construction of the profinite completion of a homomorphism, and block \circledtext{5} follows from diagram (\ref{equ: completion ses commutative F}). 

From this, one can verify that the outermost square of (\ref{equ: compd diag}) commmutes. Indeed, for any $q\in Q$,
$$
\widehat{\psi_\ast(\theta_1(q))}=\widehat{\psi}_\ast (\Theta_1(\iota(q)))= \Theta_2(\widehat{\phi}(\iota(q)))=\widehat{\theta_2(\phi(q))},
$$
from which one deduce $\psi_\ast(\theta_1(q))=\theta_2(\phi(q))$ using assumption \ref{extasmpC}. 

Recall that $H$ has trivial center (assumption \ref{extasmpB}), so we can  apply  \autoref{prop: standard abstract algebra} to the two extensions $1\to H \to G_1\to Q\to 1$ and $1\to H \to G_2\to Q \to 1$, and we obtain an isomorphism $f: G_1\to G_2$ that fits into  the following commutative diagram.
\begin{equation}\label{equ: abstract ses}
\begin{tikzcd}
1 \arrow[r]  & H \arrow[r,"u_1"] \arrow[d,"\psi"',"\cong"] & G_1 \arrow[r,"v_1"] \arrow[d,"f"',"\cong" ] & Q \arrow[r] \arrow[d,"\phi"',"\cong"]  & 1 \\
1 \arrow[r] & H \arrow[r,"u_2"]  & G_2 \arrow[r,"v_2"] & Q \arrow[r] & 1 
\end{tikzcd}
\end{equation}
In particular, $G_1\cong G_2$. 

Now, we apply the profinite completion functor to diagram (\ref{equ: abstract ses}), and obtain a commutative diagram with rows exact. 
\begin{equation}\label{equ: completion ses commutative hatf}
\begin{tikzcd}
1 \arrow[r] & \widehat H \arrow[r,"\widehat{u_1}"] \arrow[d,"\widehat{\psi}"',"\cong"] & \widehat{G_1} \arrow[r,"\widehat{v_1}"] \arrow[d,"\widehat{f}"',"\cong"]  & \widehat Q \arrow[r]  \arrow[d,"\widehat{\phi}"',"\cong"] & 1 \\
1 \arrow[r] & \widehat H \arrow[r,"\widehat{u_2}"] & \widehat{G_2} \arrow[r,"\widehat{v_2}"] & \widehat Q \arrow[r] & 1 
\end{tikzcd}
\end{equation}

We now view $\widehat{G_1}$ and $\widehat{G_2}$ as extensions of $\widehat{H}$ by $\widehat{Q}$ as abstract groups, and view isomorphisms between them as isomorphisms of abstract groups. Since $\widehat{H}$ has trivial center (assumption \ref{extasmpB}), we can apply \autoref{prop: standard abstract algebra} to  deduce that the isomorphism $\widehat{G_1}\to \widehat{G_2}$ (as abstract groups) satisfying diagram  (\ref{equ: completion ses commutative F}) is unique. Comparing (\ref{equ: completion ses commutative F}) with (\ref{equ: completion ses commutative hatf}), we finally conclude that $F'=\widehat{f}$, i.e.\ $F'$ is genuine. Recall that $F$ is conjugacy equivalent with $F'$, so $F$ is also genuine.  
\end{proof}

We record a straightforward lemma used for the corollary immediately below. 
\begin{lemma}\label{lem: subgroup also genuine}
Let $\Gamma_1$ and $\Gamma_2$ be finitely generated groups, and let $\Phi:\widehat{\Gamma_1}\to \widehat{\Gamma_2}$ be a genuine isomorphism. Then, for any $\Phi$-corresponding pair of finite-index subgroups $\Gamma_1'\le \Gamma_1$ and $\Gamma_2'\le \Gamma_2$, the restriction   $\Phi':\widehat{\Gamma_1'}\to \widehat{\Gamma_2'}$ is also genuine. 
\end{lemma}
\begin{proof}
By assumption, there exists $\hat g\in \widehat{\Gamma_2}$ and an isomorphism $f:\Gamma_1\to \Gamma_2$ such that $\Phi=\Inn_{\hat g} \circ \widehat f$. Recall that $\overline{\Gamma_2'}$ is an open subgroup in $\widehat{\Gamma_2}$ and $\iota_2(\Gamma_2)$ is a dense subgroup in $\widehat{\Gamma_2}$, where $\iota_2:\Gamma_2\to \widehat{\Gamma_2}$ denotes the canonical homomorphism. Thus, $\widehat{\Gamma_2}=  \overline{\Gamma_2'} \cdot \iota_2(\Gamma_2)$, and we can find $\hat{h}\in \overline{\Gamma_2'}   $ and $\gamma \in \Gamma_2$ such that $\hat{g}=\hat{h} \cdot \iota_2(\gamma)$. We denote $f_0= \Inn_{\gamma}\circ f: \Gamma_1\to \Gamma_2$. By construction, $\Phi=\Inn_{\hat h} \circ \widehat{f_0}$. 

Since $\Phi(\overline{\Gamma_1'})=\overline{\Gamma_2'}$ and $\hat{h}\in \overline{\Gamma_2'}$, we deduce that $\widehat{f_0}(\overline{\Gamma_1'})=\overline{\Gamma_2'}$. According to \autoref{prop: lattice of open subgroup}, $f_0(\Gamma_1')=\Gamma_2'$. We denote by $f_0':\Gamma_1'\to \Gamma_2'$ the isomorphism which is the restricition of $f_0$. By construction, $\Phi'=\Inn_{\hat h} \circ \widehat{f_0'}$, where $\hat{h}\in \widehat{\Gamma_2'}\cong \overline{\Gamma_2'}$. In other words, $\Phi'$ is also genuine. 
\end{proof}

We now state and prove a useful corollary  of  \autoref{lem: extension lemma}.

\begin{sloppypar}
\begin{corollary}\label{cor: virtually genuine}
Let $\Gamma_1$ and $\Gamma_2$ be finitely generated non-elementary Klei\-nian groups, and let $\Phi:\widehat{\Gamma_1}\to \widehat{\Gamma_2}$ be an isomorphism. If $\Phi$ is virtually genuine, then $\Phi$ is genuine. In particular, $\Gamma_1\cong \Gamma_2$. 
\end{corollary}
\end{sloppypar}
\begin{proof}
According to \autoref{lem: subgroup also genuine}, by taking possibly smaller subgroups, we can find a $\Phi$-corresponding pair of finite-index normal subgroups $\Gamma_1'\unlhd \Gamma_1$ and $\Gamma_2'\unlhd \Gamma_2$ such that the restriction $\Phi': \widehat{ \Gamma_1'} \to \widehat{\Gamma_2'}$ is also genuine. In particular, $\Gamma_1'\cong \Gamma_2'$.  We then obtain a commutative diagram between short exact sequences. 
\begin{equation*}
 \begin{tikzcd}
 1 \arrow[r] & \widehat{\Gamma_1'} \arrow[r, hook] \arrow[d,"\Phi'"',"\cong"] & \widehat{\Gamma_1} \arrow[r, two heads] \arrow[d,"\Phi"', "\cong"] & \Gamma_1/\Gamma_1'=\widehat{ \Gamma_1/\Gamma_1'} \arrow[r] \arrow[d,"\phi"',"\cong"] & 1 \\
 1 \arrow[r] & \widehat{\Gamma_2'} \arrow[r, hook] & \widehat{\Gamma_2} \arrow[r, two heads] & \Gamma_2/\Gamma_2' =\widehat{ \Gamma_2/\Gamma_2'} \arrow[r] & 1
\end{tikzcd}
\end{equation*}
We now apply \autoref{lem: extension lemma} to the two extensions $1 \to \Gamma_1' \to \Gamma_1 \to \Gamma_1/ \Gamma_1'\to 1 $ and $1 \to \Gamma_2' \to \Gamma_2 \to \Gamma_2/ \Gamma_2'\to 2 $. 
Recall that $\Gamma_1'$ and $\Gamma_2'$ are also finitely generated non-elementary Kleinian groups. Thus, $\Gamma_1',\Gamma_2',\widehat{\Gamma_1'},\widehat{\Gamma_2'}$ have trivial center according to \autoref{thm: lattice center free}, and the completion  maps $\Out(\Gamma_1')\to \Out(\widehat{\Gamma_1'})$ and $\Out(\Gamma_2')\to \Out(\widehat{\Gamma_2'})$  are injective according to \autoref{cor: completion map injective}. Also note that $\phi$ is trivially the profinite completion of a group  isomorphism since $\Gamma_1/\Gamma_1'$ and $\Gamma_2/\Gamma_2'$ are already finite groups. Therefore, all assumptions for \autoref{lem: extension lemma} are satisfied, and  \autoref{lem: extension lemma} implies that $\Phi$ is genuine. 
In particular, $\Gamma_1\cong \Gamma_2$. 
\end{proof}

\section{Cohomology theory of profinite groups}\label{sec: cohomology}
\subsection{Definitions}
A standard reference for the cohomology theory of profinite groups can be found at  \cite[Chapter 6]{RZ10}. In this subsection, we omit most of the technical details, and present a brief introduction. 

Let $G$ be a profinite group. Define $\CC(G)$ to be the category of profinite $\widehat{\Z}[\![G]\!]$-modules, and define $\FF(G)$ to be the category of finite discrete  $\widehat{\Z}[\![G]\!]$-modules. 


For  $N \in \CC(G)$, the {\em $n$-th profinite homology} of $G$ with coefficients in $N$ is a profinite abelian group denoted by $\H_n(G;N)$. For  $M\in \FF(G)$, the {\em $n$-th profinite cohomology} of $G$ with coefficients in $M$ is a discrete torsion abelian group denoted by $\H^n(G;M)$. Note that the profinite cohomology is usually defined for coefficients in discrete $G$-modules; however,  only finite   $G$-modules are needed in this paper. The readers are referred to \cite[Chapter 6]{RZ10} for  concrete definitions in various equivalent forms. 

Profinite homology and cohomology satisfy the following functoriality. 
\begin{enumerate}[leftmargin=*,label=(\arabic*)]
\item Let $\phi:G_1\to G_2$ be a homomorphism of profinite groups. For $N\in \CC(G_2)$, there is a natural homomorphism $\phi_\ast:\H_\ast(G_1;N)\to \H_\ast(G_2;N)$. For $M\in \FF(G_2)$, there is a natural homomorphism $\phi^\ast : \H^\ast (G_2;M)\to \H^\ast (G_1;M)$. In addition, when $\phi$ is an isomorphism,   both $\phi_\ast$ and $\phi^\ast$ are isomorphisms of $\widehat{\Z}$-modules. 
\item For any homomorphism $\varphi:N_1\to N_2$   in $\CC(G)$ , there is a natural homomorphism $\varphi_\ast : \H_\ast(G;N_1)\to \H_\ast(G;N_2)$. For any homomorphism $\psi:M_1\to M_2$ in $\FF(G)$, there is a natural homomorphism $\psi^\ast: \H^\ast(G;M_1)\to \H^\ast (G;M_2)$. 
\item Suppose $\Gamma$ is an abstract group, and $G$ is a profinite group. Let $f:\Gamma\to G$ be a homomorphism, and any profinite/discrete $\widehat{\Z}[\![G]\!]$-module is viewed as an abstract $\widehat{\Z} \Gamma $-module via $f$. For $N\in \CC(G)$ and $M\in \FF(G)$, there are natural homomorphisms of abstract $\widehat{\Z}$-modules  $f_\ast:H_\ast(\Gamma;N)\to \H_\ast(G;N)$ and $f^\ast: \H^\ast(G;M)\to H^\ast(\Gamma;M)$, where $H^\ast$ and $H_\ast$ denote the usual group (co)homology.   
\end{enumerate}

\begin{lemma}[{\cite[Lemma 6.8.6]{RZ10}}]\label{lem: abelianization}
Let $\widehat{\Z}\in \CC(G)$ be equipped with trivial $G$-action. There is a canonical isomorphism $\H_1(G;\widehat{\Z})\cong \Ab{G}$, where $\Ab{G}=G/\overline{[G,G]}$ is the profinite abelianization. 
\end{lemma}

\subsection{Cohomological goodness}

The concept of cohomological goodness was first introduced by Serre in \cite{Ser01}. For any abstract group $\Gamma$, it is clear that the category $\FF(\widehat{\Gamma})$ can be identified with $\FF(\Gamma)$, the category of finite discrete $\Gamma$-modules.  

\begin{definition}
An abstract group $\Gamma$ is {\em cohomologically good} if for any $M\in \FF(\Gamma)$, the map $\iota^\ast:\H^\ast(\widehat{\Gamma};M)\to H^\ast (\Gamma;M)$ induced by the canonical homomorphism $\iota:\Gamma\to \widehat{\Gamma}$ is an isomorphism. 
\end{definition}

With extra assumptions on $\Gamma$, cohomological goodness has a dual homological version. 

\begin{proposition}\label{prop: homologically good}
Let $\Gamma$ be an abstract group of type $\mathrm{FP}_\infty$. Suppose  that $\Gamma$ is  cohomologically good. Then, for any $M\in \FF(\Gamma)$, the map $\iota_\ast: H_\ast(\Gamma;M)\to \H_\ast(\widehat{\Gamma};M)$ induced by the canonical homomorphism  is also an isomorphism. 
\end{proposition}
\begin{proof}
This involves  Pontryagin duality. For any abstract abelian group $A$, denote  $A^\ast= \mathrm{Hom}(A,\mathbb{Q}/\Z)$. When $A$ is finite, $A^\ast$ is abstractly isomorphic with $A$. In particular, $M^\ast$ equipped with the inverted $\Gamma$-action can also be viewed as a finite $\Gamma$-module. 
From the pairings $ H_\ast(\Gamma;M)\times   H^\ast(\Gamma;M^\ast)\to \mathbb{Q}/\Z$ and $ \H_\ast(\widehat \Gamma;M)\times   \H^\ast(\widehat \Gamma;M^\ast)\to \mathbb{Q}/\Z$, we obtain homomorphisms $\varphi: H_\ast(\Gamma;M)\to H^\ast(\Gamma;M^\ast)^\ast$ and $\Phi: \H_\ast(\widehat \Gamma;M)\to \H^\ast(\widehat \Gamma;M^\ast)^\ast$ that fit into the following commutative diagram 
$$\begin{tikzcd}
H_\ast(\Gamma;M) \arrow[r,"\varphi"] \arrow[d,"\iota_\ast"'] & H^\ast(\Gamma;M^\ast)^\ast \arrow[d,"(\iota^\ast)^\ast"] \\ 
\H_\ast(\widehat \Gamma;M) \arrow[r,"\Phi"] &\H^\ast(\widehat \Gamma;M^\ast)^\ast
\end{tikzcd}$$

It follows from \cite[Proposition 6.3.6]{RZ10} that $\Phi$ is an isomorphism, and it follows from \cite[Proposition VI.7.1]{Brown} that $\varphi$ is  an isomorphism when $\Gamma$ has type $\mathrm{FP}_\infty$. In addition, $(\iota^\ast)^\ast$ is an isomorphism since $\Gamma$ is cohomologically good. Therefore, $\iota_\ast: H_\ast(\Gamma;M)\to \H_\ast(\widehat{\Gamma};M)$ is also an isomorphism. 
\end{proof}

\begin{proposition}[{\cite[Proposition 5.19]{Xu25}}]\label{prop: Zhat homology}
Let $\Gamma$ be an abstract group of type $\mathrm{FP}_\infty$ which is  cohomologically good. Let $\widehat{\Z}\in \CC(\widehat{\Gamma})$ be equipped with trivial $\widehat{\Gamma}$-action. Then, there is an isomorphism of $\widehat{\Z}$-modules 
$$
\tensor H_\ast(\Gamma;\Z)\xrightarrow{\;\cong\;} H_\ast(\Gamma;\widehat{\Z}) \xrightarrow{\;\cong\;} \H_\ast(\widehat{\Gamma};\widehat{\Z}). 
$$
\end{proposition}

Most abstract groups we encounter in this paper are cohomologically good. 
\begin{proposition}[{\cite{Cav12}}]\label{prop: good}
Finitely generated Kleinian groups are cohomologically good. 
\end{proposition}
\begin{proof}
In fact, all finitely  generated 3-manifold groups are cohomologically good according to \cite[Corollary 3.5.1]{Cav12}, whose proof is based on the virtual fibering theorem (\autoref{thm: fibered closure}). This implies that   torsion-free Kleinian groups are cohomologically good. The proposition then follows from the fact that a group $\Gamma$ is cohomologically good if some finite-index  subgroup of $\Gamma$ is cohomologically good; see \cite[Lemma 3.4.4]{Cav12}.
\end{proof}

Cohomological goodness implies some basic profinite properties of lattices in $\PSL_2(\C)$. 

\begin{proposition}\label{prop: torsion free}
Let $\Gamma_1$ and $\Gamma_2$ be lattices in $\PSL_2(\C)$. Suppose that $\widehat{\Gamma_1}\cong \widehat{\Gamma_2}$. Then, $\Gamma_1$ is torsion-free if and only if $\Gamma_2$ is torsion-free; and $\Gamma_1$ is a uniform lattice if and only if $\Gamma_2$ is a uniform lattice. 
\end{proposition}
\begin{proof}
First, it follows from \cite[Corollary 2.16]{BCR16} that $\Gamma_i$ is torsion-free if and only if $\widehat{\Gamma_i}$ is torsion-free, since $\Gamma_i$ is residually finite, cohomologically good, and has finite cohomological dimension over $\Z$ when it is torsion-free. Thus,  $\Gamma_1$ is torsion-free if and only if $\Gamma_2$ is torsion-free. 

 Second, $\Gamma_i$ is a uniform lattice if and only if  for any torsion-free finite-index subgroup $\Gamma_i'\le \Gamma$, $\H^3(\widehat{\Gamma_i'};\Fp)\cong H^3(\Gamma_i';\Fp)\cong \Fp$. Once we fix an isomorphism $\phi:\widehat{\Gamma_1}\to \widehat{\Gamma_2}$, the $\phi$-correspondence between the finite-index subgroups of $\Gamma_1$ and $\Gamma_2$   matches up the torsion-free ones by the previous paragraph.  For any $\phi$-corresponding pair  of  torsion-free finite-index subgroups $\Gamma_1'\le \Gamma_1$ and $\Gamma_2'\le \Gamma_2$, $\H^3(\widehat{\Gamma_1'};\Fp)\cong \H^3(\widehat{\Gamma_2'};\Fp)$. Hence, $\Gamma_1$ is a uniform lattice if and only if $\Gamma_2$ is a uniform lattice. 
\end{proof}

\subsection{Cap product}
The cap product in profinite cohomology theory was established by A. Pletch \cite{Ple77,PleI}. 

\begin{proposition}[{\cite{Ple77,PleI}}]
Let $G$ be a profinite group.  For any $A,B\in \FF(G)$, there is a family of well-defined $\widehat{\Z}$-bilinear maps called {\em cap products}:
$$
-\cap - :\,\H_k(G,A)   \times\H^l(G,B)\longrightarrow  \H_{k-l}(G,A\otimes_{\Z} B), 
$$
where $A\otimes_{\Z}B \in \mathfrak F (G)$ is equipped with diagonal $G$-action. 
\end{proposition}
\begin{proof}
 \cite[Definition and Theorem 8.1]{Ple77} defined, for any compact $\Z_p[\![G]\!]$-module $A$ and any finite $\Z_p[\![G]\!]$-module $B$, the cap products
$$
-\cap - :\,\H_k(G,A)   \times\H^l(G,B)\longrightarrow  \H_{k-l}(G,A\otimes_{\Z_p} B). 
$$ 
These are defined on the level of chain complexes of the standard resolutions by 
\begin{equation*}
\begin{tikzcd}[row sep=0.01cm]
\bfC_k(G,A)\times \bfC^l(G,B) \arrow[r,"-\cap -"] & \bfC_{k-l}(G,A\otimes_{\Z_p} B)\\
(  a\otimes (g_0,\cdots, g_k), \sigma ) \arrow[r, maps to]& (a \otimes \sigma(g_0,\cdots,g_l) ) \otimes (g_{l},\cdots,g_k).
\end{tikzcd}
\end{equation*}

For general finite $G$-modules $A$ and $B$, cap products can be defined by decomposing them into their $p$-primary components. Let $A_p$ and $B_p$ be the abelian subgroups of $A$ and $B$ consisting of all elements with order a power of $p$. Then, $A_p$ and $B_p$ are preserved by the $G$-action, and are hence finite $\Z_p[\![G]\!]$-modules. The structure of finite abelian groups implies that $A=\oplus _ p A_p$ and $B= \oplus _ p B_p$. In addition, $A\otimes_{\Z}B= \oplus _p (A\otimes_{\Z}B)_p= \oplus_p (A_p\otimes_{\Z_p} B_p)$, where only finitely many of these terms are non-trivial. 

Then, the cap product $$
-\cap - :\,\left (\oplus_p \H_k(G,A_p)  \right )\times \left (\oplus_p \H^l(G,B_p)\right )\longrightarrow   \oplus_p \H_{k-l}(G,A_p\otimes_{\Z_p}B_p)
$$
is defined by a direct sum of the cap products on the $p$-primary components. The map is $\widehat{\Z}$-bilinear since it is $\Z_p$ bilinear on each $p$-primary component.  
\end{proof}

\begin{remark}
In fact, $A$ in the above definition is allowed to be taken from $\mathfrak{C}(G)$, but finite modules are sufficient for this paper.  
\end{remark}

The cap product enjoys the following naturality.
\begin{proposition}[{\cite[Proposition 8.5]{Ple77}}]\label{nat1}
Suppose $\phi:G_1\to G_2$ is a homomorphism of profinite groups, and $A,B\in \FF(G_2)$. Then, the following diagram commutes.
\begin{equation*}
\begin{tikzcd}[column sep=tiny]
\H_k(G_1,A)  \arrow[d,"\phi_\ast"] \arrow[r, symbol=\times] & \H^l(G_1,B) \arrow[rrrr,"-\cap-"]  &  & & & \H_{k-l}(G_1,A\otimes_{\Z}B) \arrow[d,"\phi_\ast"]\\
\H_k(G_2,A)   \arrow[r, symbol=\times]  & \H^l(G_2,B) \arrow[rrrr,"-\cap-"] \arrow[u,"\phi^\ast"'] & & &  & \H_{k-l}(G_2,A\otimes_{\Z}B) 
\end{tikzcd}
\end{equation*} 
In other words, for  $c\in \H_k(G_1,B)$ and $\alpha \in \H^l(G_2,A)$, $\phi_\ast( c\cap\phi^\ast(\alpha))=\phi_\ast(c)\cap \alpha $.  
\end{proposition}

\begin{proposition}\label{nat2}
Suppose $\Gamma$ is an abstract group, $G$ is a profinite group, and $\varphi:\Gamma \to G$ is a homomorphism. For $A,B\in \FF(G)$, the following diagram commutes. 
\begin{equation*}
\begin{tikzcd}[column sep=tiny]
H_k(\Gamma,A)  \arrow[d,"\varphi_\ast"]  \arrow[r, symbol=\times] & H^l(\Gamma,B) \arrow[rrrrr,"-\cap-"] &  & & & & H_{k-l}(\Gamma,A\otimes_{\Z}B) \arrow[d,"\varphi_\ast"]\\
\H_k(G,A)  \arrow[r, symbol=\times]   & \H^l(G,B)\arrow[u,"\varphi^\ast"']  \arrow[rrrrr,"-\cap-"] & & &  & & \H_{k-l}(G,A\otimes_{\Z}B) 
\end{tikzcd}
\end{equation*} 
\end{proposition}
 \begin{proof}
The profinite version of cap product is defined on the level of chain complexes of standard resolutions by analogy with the usual cap product for (co)homo\-logy of abstract groups. Hence, the diagram actually commutes on the level of chain complexes. 
\end{proof}

\section{Profinite Bass--Serre theory}\label{sec: bstheory}
\subsection{Graph of groups}
In this paper, we focus on   finite graphs of groups, and we usually omit the notation ``finite'' when finiteness is clear. 

\begin{definition}\label{def: finite oriented graph}
A {\em finite oriented  graph} $X$ is a finite set $X=V(X)\sqcup E(X)$, with $V(X)$ and $E(X)$  denoting the set of vertices and edges respectively, which is equipped with two {\em relation maps} $d^+,d^-: E(X)\to V(X)$ assigning the positive and negative endpoints of the edges. We always require that $X$ is connected as a 1-complex, though we omit the word connected from the notation. 
\end{definition}

\begin{definition}
A {\em graph of groups} over the finite oriented graph $X$, denoted by $(G,X)$, consists of the following data. 
\begin{enumerate}[leftmargin=*, label=(\arabic*)]
\item Each $x\in X$ is associated with a group $G_x$. 
\item For each $e\in E(X)$, there are injective homomorphisms $\varphi^+_e:G_e\to G_{d^+(e)}$ and $\varphi^-_e:G_e\to G_{d^-_e}$. 
\end{enumerate}
\end{definition}

\begin{definition}\label{DEF: Fundamental group of graph of group}
Let $(G,X)$ be a finite graph of groups, and let $T$ be a maximal subtree of $X$. Denote by $F_{E}$ the free group over $E(X)$ with generators $t_e, \, e\in E(X)$. The {\em fundamental group} of $(G,X)$ with respect to $T$ is defined as
$$
\pi_1(G,X,T)=\left ((\ast_{v\in V(\Gamma)}G_v)\ast F_{E}\right )/N,
$$ 
where $N$ is the normal subgroup generated by $$\{t_e\mid e\in E(T)\}\cup \{\varphi_e^{+}(g)^{-1}t_e\varphi_e^{-}(g) t_e^{-1}\mid e\in E(X),\,g\in G_e\}.$$
\end{definition}

It follows from \cite[Proposition 20]{Ser80} that up to isomorphism, $\pi_1(G,X,T)$ is independent with the choice of the maximal subtree $T$. Thus, we may omit the notation $T$ and simply denote the fundamental group as $\pi_1(G,X)$ when there is no need for a precise maximal subtree. 

For instance, when $X$ consists of one vertex $v$ and one edge $e$, $\pi_1(G,X)$ is an HNN-extension of $G_v$ along $G_e$. When $X$ consists of two vertices $v$ and $w$ and an edge $e$ connecting them, $\pi_1(G,X)$ is an amalgamation of $G_v$ and $G_w$ along $G_e$. In general, $\pi_1(G,X)$ is constructed via a combination of amalgamations and HNN-extensions. 

Geometrically, the fundamental group of a graph of group encodes the fundamental group of a graph of spaces. Indeed, suppose that for each $x\in X$, $Y_x$ is a CW-complex such that $\pi_1Y_x=G_x$. For $e\in E$, let $\phi_e^{\pm}:Y_e\to Y_{d^{\pm}(e)}$ be cellular maps  that induce,  up to homotopy, the group homomorphisms $\varphi_e^{\pm}: \pi_1Y_e\to \pi_1Y_{d^{\pm}(e)}$. Then, we can construct a CW-complex $$Y=\left(\bigsqcup_{e\in E(X)} Y_e\times[-1,+1]\right) \bigcup_{\phi_e^{\pm}} \left(\bigsqcup _{v\in V(X)} Y_v\right).$$ It follows  that $\pi_1(Y)\cong \pi_1(G,X)$.
 
\subsection{Graph of profinite groups}

This theory also extends to profinite groups. In addition to the contents of this subsection, we refer the readers to \cite{Rib17} for a full account of profinite Bass--Serre theory. 

\begin{definition}
A {\em graph of profinite groups} over the finite oriented graph $X$, denoted by $(\G,X)$, consists of the following data. 
\begin{enumerate}[leftmargin=*, label=(\arabic*)]
\item Each $x\in X$ is associated with a profinite group $\G_x$. 
\item For each $e\in E(X)$, there are injective continuous homomorphisms $\varphi^{+}_e:\G_e\to \G_{d^{+}(e)}$ and  $\varphi^{-}_e:\G_e\to \G_{d^{-}(e)}$. 
\end{enumerate}
\end{definition}

For profinite groups $\G_1$ and $\G_2$, let $\G_1\amalg G_2$ denote their free profinite product.

\begin{definition}\label{DEF: Profinite fundamental group of graph of profinite group}
Let $(\G,X)$ be a finite graph of profinite groups, and let $T$ be a maximal subtree of $X$. Denote by $\mathcal F_{E}$ the free profinite group over $E(X)$ with generators $t_e, \, e\in E(X)$. The {\em profinite fundamental group} of $(\G,X)$ with respect to $T$ is defined as
$$
\Pi_1(\G,X,T)=\left ((\amalg_{v\in V(\Gamma)}\G_v)\amalg \mathcal{F}_{E}\right )/\mathcal N,
$$ 
where $\mathcal N$ is the closed normal subgroup generated by $$\{t_e\mid e\in E(T)\}\cup \{\varphi_e^{-}(g)^{-1}t_e\varphi_e^{+}(g) t_e^{-1}\mid e\in E(X),\,g\in \G_e\}.$$
\end{definition}

$\Pi_1(\G,X,T)$ is a profinite group, and is irrelevant, up to isomorphism, with the choice of the maximal subtree $T$ as shown by \cite[Theorem 6.2.4]{Rib17}. Similarly, we denote the profinite fundamental group by $\Pi_1(\G,X)$ if the clarification of the maximal subtree is not essential.

\begin{definition}
A finite graph of profinite groups $(\G,X)$ is called \textit{injective} if for any $v\in V(\Gamma)$, the apparent homomorphism $ \G_v\to  \Pi_1(\G,X)$ is injective.
\end{definition}

Note that there exist non-injective examples in the profinite settings as constructed by \cite{Rib73}, though injectivity always holds 
in the classical settings, see \cite[Proposition 6.2.1]{Geo07}.

One can relate a graph of abstract groups with a graph of profinite groups as shown by the following construction. 
Let $(G,X)$ be a finite graph of groups, and suppose that for any $e\in E(\Gamma)$, the profinite topology of $G_{d^{\pm}(e)}$ induces the full profinite topology on $G_e$ through $\varphi^{\pm}_e$. Then $(\widehat{G},X)$ is a finite graph of profinite groups, where the group associated to $x\in X$ is replaced by the profinite completion $\widehat{G_x}$, and the homomorphisms are also replaced by the profinite completions $\widehat{\varphi^{\pm}_e}:\widehat{ G_e}\to \widehat{  G_{d^{\pm}(e)}}$, which  are  injective according to \autoref{prop: left exact}. 

From \autoref{DEF: Fundamental group of graph of group} and \autoref{DEF: Profinite fundamental group of graph of profinite group}, one establishes a group homomorphism $\xi: \pi_1(G,X)\to \Pi_1(\widehat{G},X)$ through the canonical homomorphisms $G_x\to \widehat{G_x}$, which has dense image in $\Pi_1(\widehat{G},X)$. 
According to \autoref{prop: universal property}, there is a unique continuous homomorphism $\Xi: \widehat{\pi_1(G,X)}\to \Pi_1(\widehat{G},X)$ such that $\xi=\Xi\circ \iota$, where $\iota: \pi_1(G,X)\to  \widehat{\pi_1(G,X)}$ denotes  the canonical homomorphism.

\begin{definition}
A graph of groups $(G,X)$ is called \textit{adequate} if the following two properties hold.
\begin{enumerate}[leftmargin=*,label=(\arabic*)]
\item $\pi_1(G,X)$ is residually finite.
\item For each $x\in X$, the profinite topology of $\pi_1(G,X)$ induces the full profinite topology on   $G_x$.
\end{enumerate}
\end{definition}

\begin{proposition}[{\cite[Theorem 6.5.3]{Rib17}}]\label{THM: Efficient completion}
Let $(G,X)$ be an adequate finite graph of groups. Then, $\Xi:\widehat{\pi_1(G,X)}\to \Pi_1(\widehat{G},X)$ is an isomorphism, and $(\widehat{G},X)$ is injective. 
\end{proposition}

\begin{remark}
A more familiar concept for the graph of groups $(G,X)$ is {\em efficiency}, which, in addition to adequacy, requires that each $G_x$ is separable in $\pi_1(G,X)$. For  application to \autoref{THM: Efficient completion}, we do not need this  extra condition, although all adequate examples appearing in this paper are indeed efficient. 
\end{remark}

\subsection{Profinite Bass--Serre tree}

\begin{definition}
A {\em profinite graph} $\X$ consists of a profinite space $\X$, a closed subspace $V(\X)\subseteq \X$ called the set of {\em vertices}, and two continuous maps $\partial^{+},\partial^{-}:\X\to V(\X)$ called {\em relation maps} such that $\partial^{\pm}$ restrict to the identity map on $V(\X)$. We denote $E(\X)=\X\setminus V(\X)$ as the set of {\em edges}. 
\end{definition}
 
A profinite graph in this paper is always oriented by $\partial^{\pm}$. It can always be viewed as an oriented combinatorial graph $(V(\X),E(\X))$ by forgetting the topology; or as an unoriented combinatorial graph by forgetting the orientation simultaneously. 

We say that a profinite group $G$ {\em acts} on a profinite graph $\X$ if   $G$, viewed as an abstract group, admits a combinatorial left action   on the oriented combinatorial graph $\X$, and  the action map $G\times \X\to \X$ is continuous.  The quotient space $\X/G$ is again a profinite graph. 

\begin{definition}
A profinite graph $\X$ is a {\em profinite tree} if the following chain complex is exact:  
\begin{equation*}
\centering
\begin{tikzcd}
0 \arrow[r] & {\widehat{\Z}[\![{(\X/V(\X),\ast)}]\!]} \arrow[r,"\partial"] & \widehat{\Z}[\![{V(\X)} ]\!]\arrow[r,"\epsilon"] & \widehat{\Z} \arrow[r] & 0 .
\end{tikzcd}
\end{equation*}
\end{definition}

Indeed, a profinite tree, if viewed as an unoriented combinatorial graph, does not contain circuts, see \cite[Exercise 2.4.4]{Rib17}. 

Given a finite graph of profinite groups $(\G,X)$, there is a standard way to construct a profinite tree on which $\Pi_1(\G,X)$ acts. For simplicity, we focus on the case that $(\G,X)$ is injective. In this case, for each $v\in V(X)$, we can identify $\G_v$ with its image in $\Pi_1(\G,X)$. For each $e\in E(X)$, we identify $\G_e$ with its image in $\Pi_1(\G,X)$ through the injective map $\G_e\xrightarrow{\varphi^{-}_e} \G_{d^{-}(e)} \to \Pi_1(\G,X)$. Then, for any $x\in X$, $\G_x$ is a closed subgroup in $\Pi_1(\G,X)$, and the left coset space $\Pi_1(\G,X)/\G_x$ is a profinite space.   
We construct a profinite graph $\T$ by setting
\begin{equation*}
V(\T)=\bigsqcup_{v\in V(X)} \Pi_1(\G,X)/\G_v,\;E(\T)=\bigsqcup_{e\in E(X)} \Pi_1(\G,X)/\G_e,\; \T=V(\T)\sqcup E(\T),
\end{equation*}
to which we equip the disjoint union topology. It suffices to define the relation maps on $E(\T)$ by
\begin{equation*}
\partial^-(h\G_e)=h\G_{d^{-}(e)},\;\partial^+(h\G_e)=ht_e\G_{d^{+}(e)}, \quad e\in E(X),\,h\in \Pi_1(\G,X).
\end{equation*}
One can check from \autoref{DEF: Profinite fundamental group of graph of profinite group} that the relation maps are well-defined, and it is clear that the relation maps are continuous. Thus, $\T$ is a profinite graph. 
In fact, it is proven in \cite[Corollary 6.3.6]{Rib17}  that $\T$ is  a profinite tree. 

\begin{definition}
Given an injective finite graph of profinite groups $(\G,X)$, the profinite tree $\T$ as constructed above is called the {\em profinite Bass--Serre tree} associated to $(\G,X)$. 
\end{definition}

The profinite version of Bass--Serre theory can be concluded by the following proposition. 
\begin{proposition}\label{PROP: Bass-Serre}
Let $(\G,X)$ be an injective finite graph of profinite groups. Then, $\Pi_1(\G,X)$ acts on the profinite Bass--Serre tree $\T$ by left multiplication (which is obviously continuous). The quotient graph $  \T/\Pi_1(\G,X)$ equals $X$. The vertex or edge stabilizers of the action are conjugates of $\G_x$ in $\Pi_1(\G,X)$. 
\end{proposition}

\subsection{Profinite groups acting on profinite trees}

In this subsection, we collect some useful results regarding vertex stabilizers for a profinite group acting on a profinite tree.

\begin{proposition}[{\cite[Proposition 2.5]{ZM89}}]\label{prop: quotient tree}
Suppose a profinite group $G$ acts on a profinite tree $\T$. If $H$ is a closed subgroup of $G$ topologically generated by vertex stabilizers, i.e.\ $H=\overline{\left\langle \Stab_H(v):\,v\in V(\T)\right\rangle }$, then $ \T/H$ is a profinite tree. 
\end{proposition}
\begin{remark}\label{rmk: quotient action}
Under the assumption of \autoref{prop: quotient tree}, suppose further that $H$ is a closed normal subgroup of $G$. Then, the quotient group $G/H$ naturally acts on $\T/H$, and it is clear by definition of the quotient topologies that the map $G/H\times \T/H \to \T/H$ is continuous. 
\end{remark}

The next proposition was first proven as a step in \cite[Lemma 2.3]{Zal91}, see also \cite[Lemma 1.3]{WZ19}. We include a quick proof for convenience. 

\begin{proposition}\label{prop: subgroup not stabilizer}
Suppose that a profinite group $G$ acts on a profinite tree $\T$ and
does not fix any vertex. Then, there exists an open normal subgroup $H$ of $G$ that is not topologically  generated by vertex stabilizers.
\end{proposition}
\begin{sloppypar}
\begin{proof}
Suppose by contrary that every open normal subgroup $H$ of $G$ is topologically generated by vertex stabilizers, i.e.\  $H= \overline{\left\langle \Stab_H(v):\,v\in V(\T)\right\rangle }$. Then 
by \autoref{prop: quotient tree}, $\T/H$ is a profinite tree for each open normal subgroup  $H$, and the finite group $G/H$   acts on $\T/H$ by \autoref{rmk: quotient action}. According to \cite[Theorem 2.10]{ZM89}, any finite group acting on a profinite tree fixes a vertex. So if we let $X_H$ be the set of vertices in $\T/H$ fixed by $G/H$, then $X_H$ is a non-empty closed subset of $V(\T/H)$. Let $Y_H$ be the preimage of $X_H$ in $V(\T)$. Then, $Y_H$ is  a non-empty  closed subset of  $V(\T)$. Let $\mathcal{H}$ be the directed poset of open normal subgroups of $G$, ordered by reverse inclusion. Note that $Y_H=\left\{ v\in V(\T)\mid G\cdot v \subseteq H\cdot v\right\}$, so $Y_H\subseteq Y_{H'}$ if $H\subseteq H'$. Therefore, $\bigcap_{H\in \mathcal{H}} Y_H=\limi_{  \mathcal{H}} Y_H$ is non-empty. 

On the other hand, $\bigcap_{H\in \mathcal{H}} (H\cdot v)= v$. This is  because for any $w\neq v\in V(\T)$, $\{g\in G\mid g\cdot v= w\}$ is a closed subset in $G$ not containing $1$, while $\mathcal{H}$ is an open neighbourhood basis for $1$. Thus, all vertices in $\bigcap_{H\in \mathcal{H}} Y_H=\left\{ v\in V(\T)\mid G\cdot v \subseteq \bigcap_{H\in \mathcal{H}}(H\cdot v)\right\}$ are fixed by $G$, yielding a contradiction. 
\end{proof}
\end{sloppypar}

\begin{proposition}\label{prop: surface group fix vertex}
Let $\Sigma$ be a closed orientable surface. Suppose $\widehat{\pi_1\Sigma}$ acts on a profinite tree $\T$ with trivial edge stabilizers. Then, it fixes a vertex in $\T$. 
\end{proposition}
\begin{proof}
The conclusion is trivial when $\Sigma$ has genus $0$, since in this case $\widehat{\pi_1\Sigma}$ is the trivial group. Now we assume $\Sigma$ has positive genus. Then, $\pi_1\Sigma$ is a $\mathrm{PD}^2$ group (of type $\mathrm{FP}_\infty$) since it is the fundamental group of a closed, orientable, aspherical, smooth manifold. Moreover, $\pi_1\Sigma$ is cohomologically good by \autoref{prop: good}. Thus, according to \cite[Theorem 4.1]{KZ08} (see also \cite[Theorem 1.10]{WZ19}), $\widehat{\pi_1\Sigma}$ is a profinite $\mathrm{PD}^2$ group at every prime $p$. Meanwhile, the edge stabilizers of $\widehat{\pi_1\Sigma}$ acting on $\T$ are trivial, and have  cohomological dimension $0$. Thus, by \cite[Theorem 1.7]{WZ19}, $\widehat{\pi_1\Sigma}$ fixes a vertex in $\T$.
\end{proof}

\section{Automorphism of profinite surface group I: cut and glue}\label{sec: cut and glue}
In Sections \ref{sec: cut and glue}, \ref{sec: drilling}, and \ref{sec: realisation}, we prove three theorems concerning isomorphisms between the profinite completions of surface groups. These arise from topological or geometric constructions, with the latter ones depending more or less on the former ones.  

We point out that if the fundamental groups of two closed surfaces have isomorphic profinite completions, then the two surfaces are homeomorphic, since they can be distinguished by their first homologies. Thus, while we are essentially studying an automorphism of a profinite surface group, we often formulate it in the sense of an isomorphism between two profinite surface groups, simply for the sake of clarity in the proof as well as in the applications. 

\subsection{Cutting along curves}
In this section, we focus on the following setting. For $i=1,2$, 
let $\Sigma_i$ be a closed orientable surface, and let $\alpha_i\subseteq \Sigma_i$ be an oriented non-separating simple closed curve. Let $F_i=\Sigma_i \setminus n(\alpha_i)$, where $n(\alpha_i)$ denotes an open regular neighbourhood of $\alpha_i$. Then, $F_i$ is  an essential subsurface of $\Sigma$, and we fix a basepoint $x_i\in F_i$. We identify $\pi_1(F_i,x_i)$ as a subgroup of $\pi_1(\Sigma_i, x_i)$. Choose an oriented  simple loop $\gamma_i\subseteq \Sigma_i$ based at $x_i$ such that $\gamma_i$ intersects with $\alpha_i$ transversely at one point, and  $\gamma_i$ is also viewed as an element in $\pi_1(\Sigma_i,x_i)$. Then, $\pi_1(\Sigma_i,x_i)$ is generated by $\pi_1(F_i,x_i)$ and $\gamma_i$. In fact, there exist conjugacy representatives $a_{i,+},a_{i,-}\in \pi_1(F_i,x_i)$ for the two peripheral loops of $F_i$, oriented along $\alpha_i$, such that 
\begin{equation}\label{equ: HNN1c}
\pi_1(\Sigma_i, x_i )=\langle \pi_1(F_i,x_i),\gamma_i\mid \gamma_i a_{i,+}\gamma_i^{-1}=a_{i,-}\rangle,
\end{equation}
 see \autoref{fig: HNN1}. This yields a splitting of $\pi_1(\Sigma_i, x_i)$ as an HNN-extension of $\pi_1(F_i,x_i)$. 

\begin{figure}[ht!]
\centering
\includegraphics[width=11cm]{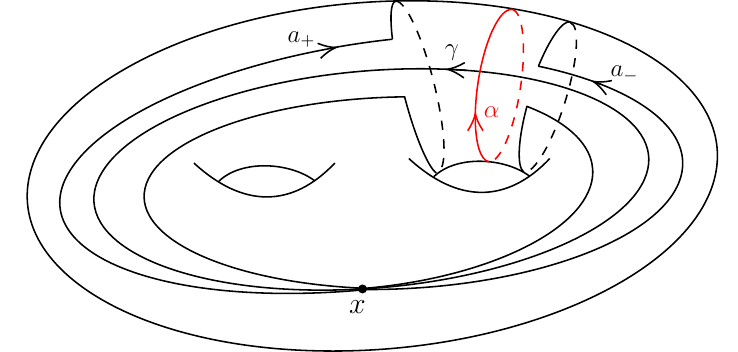}
\caption{The HNN-structure of $\pi_1(\Sigma,x)$ (subscript $i$ omitted)}
\label{fig: HNN1}
\end{figure}

To simplify notations, we will omit the basepoints  in the fundamental groups. However, the readers should always keep in mind that our fundamental groups are taken with respect to the fixed basepoints $x_i$. We set down some more notations. Let $a_i\in \pi_1\Sigma_i$ be a conjugacy representative for the free homotopy class of the oriented curve $\alpha\subseteq \Sigma_i$. In fact, $a_i$  is conjugate to  $a_{i,\pm}$ in $\pi_1\Sigma_i$. In addition, since $\pi_1 \Sigma_i $ is LERF (\autoref{FK: LERF}), we can identify $\widehat{\pi_1 F_i }$ with the subgroup $\overline{\pi_1 F_i }\le \widehat{\pi_1 \Sigma_i }$ according to \autoref{cor: lerf injective}.

We can now state the main theorem of this section.  
\begin{theorem}\label{thm: cut and glue}
Let $\Sigma_1$ and $\Sigma_2$ be closed orientable surfaces defined above, and let $\phi:\widehat{\pi_1 \Sigma_1 }\to \widehat{\pi_1 \Sigma_2 }$ be an isomorphism. Assume that the genus of $\Sigma_1$ and $\Sigma_2$ is at least $2$. Suppose there exists $\mu\in\Zx$ such that $\phi(a_1)$  is conjugate to  $a_2^\mu$ in $\widehat{\pi_1 \Sigma_2 }$. Then, there exists an isomorphism $\psi: \widehat{\pi_1F_1}\to \widehat{\pi_1F_2}$ satisfying the following three properties. 
\begin{enumerate}[leftmargin=*,label=(\arabic*)]
\item\label{cag1} There exists $\hat{g} \in \widehat{\pi_1\Sigma_2}$ such that the following diagram commutes.   
\begin{equation*}
\begin{tikzcd}[row sep=large]
{\widehat{\pi_1 F_1 }} \arrow[d, "\psi"',"\cong"] \arrow[rr, hook] &                                                             & {\widehat{\pi_1 \Sigma_1 }} \arrow[d, "\phi"',"\cong"] \\
{\widehat{\pi_1 F_2 }} \arrow[r, hook]                     & {\widehat{\pi_1 \Sigma_2 }} \arrow[r, "\Inn_{\hat{g}}","\cong"'] & {\widehat{\pi_1 \Sigma_2 }}                  
\end{tikzcd}
\end{equation*}
\item\label{cag2} Adopting $\hat{g}$ from \ref{cag1}, there exists $\hat\delta\in \widehat{\pi_1 F_2 }$ such that $\phi(\gamma_1)= \hat{g} (\gamma_2^{\pm1}\hat\delta)\hat{g}^{-1}$.
\item\label{cag3} 
$\psi(a_{1,\pm})$  is conjugate to   $a_{2,\pm}^\mu$   or $a_{2,\mp}^\mu$  in $\widehat{\pi_1 F_2 }$. 
\end{enumerate}
\end{theorem}

In short, such an isomorphism $\phi$ always arises from an isomorphism $\psi:\widehat{\pi_1 F_1 }\to \widehat{\pi_1 F_2 }$ via an HNN-extension. 

\begin{remark}
\autoref{thm: cut and glue} is not true when $\Sigma_1$ and $\Sigma_2$ are tori. In that case, one may identify $\widehat{\pi_1 \Sigma_1 }$ and $\widehat{\pi_1 \Sigma_2 }$ with $\widehat{\Z}^2$, such that  $a_1$ and $a_2$ are identified with $(1,0)$ while $\gamma_1$ and $\gamma_2$ are identified with $(0,1)$. For any $\lambda \in\Zx$, there is an isomorphism $\phi: \widehat{\Z}^2 \to \widehat{\Z}^2$ defined by $\phi(x,y)=(\mu x,\lambda y)$. In case $\lambda \neq \pm 1$, property \ref{cag2} in \autoref{thm: cut and glue} does not hold. 
\end{remark}



\subsection{The Bass--Serre structure}
We will start with some preparations before the  proof of \autoref{thm: cut and glue}. 

Our group presentation (\ref{equ: HNN1c}) for $\pi_1\Sigma_i$ identifies it as the fundamental group of a graph of groups, which we clarify as follows. Let $X$ be the finite oriented graph with one vertex $v$ and one edge $e$, where $d^+(e)=d^- (e)=v$, see \autoref{fig: graph1}. Let $(G_i,X)$ be a graph of groups over $X$ such that $G_{i,v}= \pi_1F_i$ and  $G_{i,e}=\Z$ with generator $s_i$, and that the homomorphisms $\varphi_{i,e}^\pm: G_{i,e} \to G_{i,v}$ send $s_i$ to $a_{i,\pm}$. Then, we have a clear identification $\pi_1\Sigma_i= \pi_1(G_i,X)$ by \autoref{DEF: Fundamental group of graph of group}. 

\begin{figure}[ht!]
\centering
\includegraphics[width=3cm]{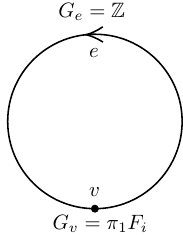}
\caption{The graph $X$ for the splitting of $\pi_1\Sigma_i$}
\label{fig: graph1}
\end{figure}

\begin{lemma}
$(G_i,X)$ is adequate. 
\end{lemma}
\begin{proof}
This follows as a combination of the   facts that $\pi_1(G_i,X)=\pi_1\Sigma_i$ is LERF (\autoref{FK: LERF}), and that the groups $G_{i,v}$ and $G_{i,e}$ are finitely generated. 
\end{proof}

\begin{sloppypar}
Therefore, according to \autoref{THM: Efficient completion}, we can identify $\widehat{\pi_1\Sigma_i}$ with $\Pi_1(\widehat{G_i},X)$, which can be viewed as   a profinite HNN-extension of $\widehat{\pi_1F_i}$. Let $\T_i$ be the profinite Bass--Serre tree associated to $(\widehat{G_i},X)$, whose vertices consist of cosets of $\widehat{\pi_1F_i}$ and whose edges consist of cosets of $\overline{\langle a_{i,-}\rangle}$. Then, $\widehat{\pi_1\Sigma_i}$ acts on $\T_i$. Furthermore, $\widehat{\pi_1\Sigma_1}$ acts on $\T_2$ via $\phi$, and  $\widehat{\pi_1\Sigma_2}$ acts on $\T_1$ via  $\phi^{-1}$. 
\end{sloppypar} 

\subsection{The action on the profinite Bass--Serre tree}

We say that a subgroup $H$ in a group $G$ is {\em malnormal} if for any $g\in G\setminus H$, $gHg^{-1}\cap H$ is trivial. 

\begin{lemma}\label{lem: malnormal}
The subgroup $\overline{\langle a_i \rangle}$ in $\widehat{\pi_1\Sigma_i}$ is malnormal. 
\end{lemma}
\begin{proof}
Indeed, $\pi_1\Sigma_i$ is word-hyperbolic and is the fundamental group of a compact virtually special cube
complex, $\langle a_i \rangle$ is a quasi-convex subgroup of $\pi_1\Sigma_i$ and is malnormal in $\pi_1\Sigma_i$ since $a_i$, represented by a simple closed curve, is a primitve hyperbolic element. Then, it follows from \cite[Theorem 3.3]{WZ17} that $\overline{\langle a_i \rangle}$ is malnormal in $\widehat{\pi_1\Sigma_i}$.
\end{proof}
 Note that \autoref{lem: malnormal} does not hold for tori.

\begin{definition}
Suppose a profinite group $G$ acts on a profinite tree $\T$. We say that the action is {\em 1-acylindrical} if the following property holds. Whenever a non-trivial element $g\in G$ fixes two distinct vertices $u,v\in V(\T)$, then $u$ and $v$ are connected by an edge. 
\end{definition}

\begin{lemma}\label{lem: acylindrical}
The action of $\widehat{\pi_1\Sigma_i}$ on $\T_i$ is 1-acylindrical. 
\end{lemma}
\begin{proof}
Suppose by contrary that a non-trivial element $g\in \widehat{\pi_1\Sigma_i}$ fixes two vertices $u,v\in V(\T_i)$ that are not connected by an edge. Let $\X=[u,v]$ be the chain connecting $u$ and $v$, which is the minimal subtree of $ \T_i $ containing $u$ and $v$, see \cite[Proposition 2.4.9]{Rib17}. According to \cite[Corollary 4.1.6]{Rib17}, $g$ fixes every element in $E(\X)$. By construction of the profinite Bass--Serre tree, $E(\T_i)$ is a closed subset of $\T_i$, so $E(\X)=E(\T_i)\cap \X$ is also a closed subset of $\X$. Hence, according to \cite[Proposition 2.1.6(c)]{Rib17}, there exist $e_u,e_v\in E(\X)$ such that $u$ is an endpoint of $e_u$ and $v$ is an endpoint of $e_v$. Since $u$ and $v$ are not connected by an edge, $e_u\neq e_v$. Therefore, $g$ fixes two distinct edges in $\T_i$. 

Suppose that the two edges $e_u$ and $e_v$ correspond to the cosets $h_u \overline{\langle a_{i,-}\rangle}$ and $h_v\overline{\langle a_{i,-}\rangle}$, where $h_u,h_v\in\widehat{\pi_1\Sigma_i}$ and $h_v^{-1}h_u\notin \overline{\langle a_{i,-}\rangle}$. Then, $g\in \Stab(e_u)\cap \Stab(e_v)= h_u\overline{\langle a_{i,-}\rangle}h_u^{-1}\cap h_v\overline{\langle a_{i,-} \rangle} h_v^{-1}$. Thus, we reach a contradiction with \autoref{lem: malnormal}, since $\overline{\langle a_{i,-} \rangle}$, being a conjugate of $\overline{\langle a_i \rangle}$,  is malnormal in $\widehat{\pi_1\Sigma_i}$.  
\end{proof}

\begin{lemma}\label{lem: FA edge stab}
Let the subgroup  $ \widehat{\pi_1F_1} \leq \widehat{\pi_1\Sigma_1}$ act  on $\T_2$ via $\phi$. For any $e\in E(\T_2)$,  the  stabilizer $\stab_{\widehat{\pi_1F_1}}(e)$ is either trivial or a conjugate of $\overline{\langle a_{1,+}\rangle}$ or $\overline{\langle a_{1,-} \rangle }$ within $\widehat{\pi_1F_1}$. Conversely, any conjugate of  $\overline{\langle a_{1,+}\rangle}$ or $\overline{\langle a_{1,-} \rangle }$ in $\widehat{\pi_1F_1}$ fixes an edge in $\T_2$. 
\end{lemma}
 
\begin{proof}
Note that $\Stab_{\widehat{\pi_1\Sigma_2}}(e)$ is a conjugate of $\overline{\langle a_{2,-}\rangle }$ in $\widehat{\pi_1\Sigma_2}$. Since $\phi$ sends $\overline{\langle a_{1,-}\rangle }$ to a conjugate of $\overline{\langle a_{2,-}\rangle }$,  $\Stab_{\widehat{\pi_1\Sigma_1}}(e)$ is a conjugate of $\overline{\langle a_{1,-}\rangle }$. Now, let us also consider the action of $\widehat{\pi_1\Sigma_1}$ on $\T_1$. By \autoref{PROP: Bass-Serre}, there exists an edge $\widetilde{e}\in E(\T_1)$   such that $\Stab_{\widehat{\pi_1\Sigma_1}}(\widetilde{e})=\Stab_{\widehat{\pi_1\Sigma_1}}(e)$.  
Let $v_1\in V(\T_1)$ be the vertex corresponding to the coset $\widehat{\pi_1F_1}$ itself. Then, $\Stab_{\widehat{\pi_1\Sigma_1}}(v_1)=\widehat{\pi_1F_1}$. 
Consequently, $$\stab_{\widehat{\pi_1F_1}}(e) =\widehat{\pi_1F_1} \cap  \Stab_{\widehat{\pi_1\Sigma_1}}(\widetilde{e})=\Stab_{\widehat{\pi_1\Sigma_1}}(v_1) \cap \Stab_{\widehat{\pi_1\Sigma_1}}(\widetilde{e}). $$

There are two possiblities to be considered. First, if $v_1$ is not an endpoint of $\widetilde{e}$. Then, at least one endpoint $u$ of $\widetilde{e}$ is not connected with $v_1$ by  an edge, since $\T_1$ has no circuits. Note that $\Stab_{\widehat{\pi_1\Sigma_1}}(\widetilde{e}) \subseteq \Stab_{\widehat{\pi_1\Sigma_1}}(u)$. Thus, $\stab_{\widehat{\pi_1F_1}}(e)\subseteq \Stab_{\widehat{\pi_1\Sigma_1}}(u) \cap \Stab_{\widehat{\pi_1\Sigma_1}}(v_1) $ which is trivial according to \autoref{lem: acylindrical}. 

\begin{sloppypar}
Second, if $v_1$ is an endpoint of $\widetilde{e}$, then $\Stab_{\widehat{\pi_1\Sigma_1}}(v_1) \supseteq  \Stab_{\widehat{\pi_1\Sigma_1}}(\widetilde{e})$, so  $\stab_{\widehat{\pi_1F_1}}(e)= \Stab_{\widehat{\pi_1\Sigma_1}}(\widetilde{e})$. There are two subcases. On one hand, if $\partial^-(e)=v_1$, then $e$ corresponds to a coset $h\overline{\langle a_{1,-}\rangle}$ such that $h\in \widehat{\pi_1F_1}$. In this case, $\stab_{\widehat{\pi_1F_1}}(e)= h\overline{\langle a_{1,-}\rangle}h^{-1}$ is a conjugate of $\overline{\langle a_{1,-}\rangle}$ in $\widehat{\pi_1F_1}$. On the other hand, if $\partial^+(e)=v_1$, then $e$ corresponds to a coset $h\overline{\langle a_{1,-}\rangle}$ such that $h\gamma_{1}\in \widehat{\pi_1F_1}$. In this case, $\stab_{\widehat{\pi_1F_1}}(e)= h \overline{\langle a_{1,-}\rangle}h^{-1}=(h\gamma_1) \overline{\langle a_{1,+} \rangle } (h\gamma_1)^{-1}$ is a conjugate of $\overline{\langle a_{1,+}\rangle}$ in $\widehat{\pi_1F_1}$. Thus, we have proven the first statement in this lemma. 
\end{sloppypar}

Conversely, any conjugate of $\overline{\langle a_{1,+}\rangle}$ or $\overline{\langle a_{1,-} \rangle }$  in $\widehat{\pi_1 F_1 }$ is a conjugate of $\overline{\langle a_{1,-} \rangle }$ in $\widehat{\pi_1 \Sigma_1 }$, which is mapped by $\phi$ to a conjugate of $\overline{\langle a_{2,-} \rangle }$ in $\widehat{\pi_1 \Sigma_2 }$. Hence, it fixes an edge in  $\T_2$ by \autoref{PROP: Bass-Serre}.  
\end{proof}

Next, we consider the setting of a finite cover. Let $F_1'$ be a finite regular cover of $F_1$, and let $x_1'\in F_1'$ be a preimage of $x_1$. Then, we can identify $\pi_1(F_1',x_1')$ as a normal subgroup of $\pi_1(F_1,x_1)$ via the  covering map. For brevity, we also omit the basepoint $x_1'$ in the following, while the readers may  keep in mind that the basepoint $x_1'$ is fixed. We also identify $\widehat{\pi_1 F_1' } $ with the open normal subgroup $\overline {\pi_1 F_1' }\subseteq \widehat{\pi_1 F_1 }$, see \autoref{prop: lattice of open subgroup}. Note that $F_1'$ is an orientable surface with boundary. Let $b_1,\cdots, b_m\in \pi_1 F_1' $ be conjugacy representatives of the free homotopy classes of the peripheral loops of $F_1'$,  equipped with orientations lifted from $\partial F_1$. 

\begin{lemma}\label{lem: finite cover setting surface}
Let the subgroup  $ \widehat{\pi_1 F_1 '} \leq \widehat{\pi_1 F_1 } \leq  \widehat{\pi_1 \Sigma_1 }$ act  on $\T_2$ via $\phi$. For any $e\in E(\T_2)$,  the  stabilizer $\stab_{\widehat{\pi_1F_1'}}(e)$ is either trivial or a conjugate of a certain $\overline{\langle b_j\rangle}$ within $\widehat{\pi_1 F_1' }$. Conversely, any conjugate of  $\overline{\langle b_j\rangle}$ in $\widehat{\pi_1 F_1' }$ fixes an edge in $\T_2$. 
\end{lemma}
\begin{proof}
Note that $\stab_{\widehat{\pi_1F_1'}}(e)=\widehat{\pi_1 F_1' }\cap \stab_{\widehat{\pi_1F_1}}(e)$. According to \autoref{lem: FA edge stab}, $\stab_{\widehat{\pi_1F_1}}(e)$ is either trivial or a conjugate of $\overline{\langle a_{1,\pm}\rangle}$. In the former case, $\stab_{\widehat{\pi_1F_1'}}(e)$ is also trivial.  In the latter case, assume that $\stab_{\widehat{\pi_1F_1}}(e)= \hat{h} \overline{\langle a_{1,\epsilon}\rangle} \hat{h}^{-1}$, where $\hat{h} \in \widehat{\pi_1 F_1 }$ and $\epsilon\in \{+,-\}$. Since $\widehat{\pi_1 F_1' }$ is an open subgroup of $\widehat{\pi_1 F_1 }$ and $\pi_1 F_1 $ is a dense subgroup of $\widehat{\pi_1 F_1 }$,  we have $\widehat{\pi_1 F_1 }=\widehat{\pi_1 F_1' }\cdot \pi_1 F_1 $. Thus, we can find $\hat{h}'\in \widehat{\pi_1 F_1' }$ and $\beta \in \pi_1 F_1 $ such that $\hat{h}=\hat{h}'\beta$. Then, 
\begin{equation*}
\begin{aligned}
\stab_{\widehat{\pi_1F_1'}}(e)=&\widehat{\pi_1 F_1' } \cap \hat{h} \overline{\langle a_{1,\epsilon}\rangle} \hat{h}^{-1}= \hat{h}'(\widehat{\pi_1 F_1' } \cap  \beta\overline{\langle a_{1,\epsilon}\rangle }\beta^{-1})(\hat{h}')^{-1} \\
 =& \hat{h}'(\widehat{\pi_1 F_1' } \cap  \overline{\beta\langle a_{1,\epsilon}\rangle \beta^{-1}})(\hat{h}')^{-1}=\hat{h}'(\overline{\pi_1 F_1'  \cap  \beta\langle a_{1,\epsilon}\rangle \beta^{-1}})(\hat{h}')^{-1} ,
\end{aligned}
\end{equation*} 
where the last equality follows from  \autoref{lem: finite index closure intersection}. Note that $\pi_1 F_1'  \cap  \beta\langle a_{1,\epsilon}\rangle \beta^{-1}$ is  a peripheral subgroup of $\pi_1 F_1' $, i.e.\ there exists $1\le j \le m$ such that $\pi_1 F_1'  \cap  \beta\langle a_{1,\epsilon}\rangle \beta^{-1}$  is conjugate to  $\langle b_j \rangle$ in $\pi_1 F_1' $. Consequently, $\stab_{\widehat{\pi_1F_1'}}(e)$  is conjugate to  $\overline{\langle b_j\rangle}$ in $\widehat{\pi_1 F_1' }$. 

Conversely, any conjugate of $\overline{\langle b_j\rangle}$ in $\widehat{\pi_1 F_1' }$ is contained in a conjugate of $\overline{\langle a_{1,\pm}   \rangle}$ in $ \widehat{\pi_1 F_1 }$, which fixes an edge in $\T_2$  according to \autoref{lem: FA edge stab}. 
\end{proof}

\subsection{Relating two profinite Bass--Serre trees}
\begin{lemma}\label{lem: F fix vertex}
The subgroup $\widehat{\pi_1 F_1 }\le \widehat{\pi_1 \Sigma_1 }$, when acting on $\T_2$ via $\phi$, fixes a vertex $v_2\in V(\T_2)$. 
\end{lemma}
\begin{sloppypar}
\begin{proof}
We assume by contradiction that $\widehat{\pi_1 F_1 }$ does not fix a vertex in $\T_2$. Then, \autoref{prop: subgroup not stabilizer} implies that there is an open normal subgroup $U$ of $\widehat{\pi_1 F_1 }$ which  is not topologically  generated by vertex stabilizers. According to \autoref{prop: lattice of open subgroup}, $U$ is the closure of a finite-index normal subgroup $H$ in $\pi_1 F_1 $. Then, $H$ corresponds to a finite regular cover $(F_1',x_1')$ of $(F_1,x_1)$, where we identify $\pi_1(F_1',x_1')$ with $H$. We omit the basepoints, and  follow the notations in    \autoref{lem: finite cover setting surface} for $\pi_1 F_1'  $. 

Let $N=\langle \!\langle b_1,\cdots, b_m\rangle \! \rangle$  be the normal subgroup of $\pi_1 F_1' $  generated by the peripheral elements. Then, $\overline{N} \unlhd \widehat{\pi_1 F_1' }$ is a closed normal subgroup topologically generated by the conjugates of $b_j$, which, according to \autoref{lem: finite cover setting surface},  are edge stabilizers  (and hence vertex stabilizers) of $\widehat{\pi_1 F_1' }$ acting on $\T_2$. By \autoref{prop: quotient tree}, $  \T_2/\overline{N}$ is a profinite tree, and the   quotient   group $\widehat{\pi_1 F_1' } / \overline{N}$ acts on $  \T_2/\overline{N}$  according to \autoref{rmk: quotient action}. 

We claim that the action of $\widehat{\pi_1 F_1' } / \overline{N}$ on $  \T_2/ \overline{N} $ has trivial edge stabilizers.  Suppose $\bar{e}\in E(  \T_2/\overline {N})$, and let $e\in E(\T_2)$ be a pre-image of $\bar{e}$. Then, $$\Stab_{\widehat{\pi_1F_1'}/\overline{N}}(\bar{e})= (\overline{N}\cdot \Stab_{\widehat{\pi_1F_1'}}(e))/\overline{N}.$$ According to \autoref{lem: finite cover setting surface}, $\Stab_{\widehat{\pi_1F_1'}}(e)\subseteq \overline{N}$. Hence, $ \Stab_{\widehat{\pi_1F_1'}/\overline{N}}(\bar{e}) $ is trivial. 

On the other hand, we claim that $\widehat{\pi_1 F_1' }/\overline{N}$ is isomorphic to $\widehat{\pi_1\Sigma}$ for some closed orientable surface $\Sigma$. Indeed, let $\Sigma$ be the closed surface obtained through capping off each boundary component of $F_1'$ with a disk. Then, van-Kampen's theorem implies that $\pi_1\Sigma \cong \pi_1F_1'/ N$, where we choose $x_1'$ to be the basepoint of $\Sigma$. Thus, according to \autoref{prop: right exact}, $\widehat{\pi_1 F_1' }/\overline{N} \cong \widehat{\pi_1 F_1' /N} \cong \widehat{\pi_1\Sigma}$.

Then, according to  \autoref{prop: surface group fix vertex}, $\widehat{\pi_1 F_1' }/\overline{N}$ fixes a vertex $\bar{v} \in V(  \T_2/\overline{N})$. Let $v\in V(\T_2)$ be a pre-image of $\bar v$. Then, $\widehat{\pi_1 F_1' }\cdot v= \overline{N}\cdot v$, so $\widehat{\pi_1 F_1' }= \overline{N}\cdot \Stab_{\widehat{\pi_1F_1'}}(v)$. Recall that $\overline{N}$ is topologically generated by vertex stabilizers of $\widehat{\pi_1 F_1' }$ acting on $\T_2$, so  $\widehat{\pi_1 F_1' }=U$ is also topologically generated by vertex stabilizers. This yields a contradiction with our assumption. 
\end{proof}
\end{sloppypar}

The proof of \autoref{lem: F fix vertex} is actually inspired by the proof of \cite[Lemma 3.4]{WZ19}. 
By a symmetric proof, we also have:
\begin{lemma}\label{lem: converse}
The subgroup $\widehat{\pi_1 F_2 }\leq \widehat{\pi_1 \Sigma_2 }$, when acting on $\T_1$ via $\phi^{-1}$, fixes a vertex in $V(\T_2)$. 
\end{lemma}

\begin{lemma}\label{lem: stab dt}
Let $v_2\in V(\T_2)$ be given by \autoref{lem: F fix vertex}. Then, $\phi(\widehat{\pi_1 F_1 })= \Stab_{\widehat{\pi_1\Sigma_2}}(v_2)$. 
\end{lemma}
\begin{proof}
\autoref{lem: F fix vertex} implies that $\phi(\widehat{\pi_1 F_1 })\subseteq  \Stab_{\widehat{\pi_1\Sigma_2}}(v_2)$, so it suffices to show that $\phi^{-1} (\Stab_{\widehat{\pi_1\Sigma_2}}(v_2)) \subseteq  \widehat{\pi_1 F_1 }$. By \autoref{PROP: Bass-Serre}, $\Stab_{\widehat{\pi_1\Sigma_2}}(v_2)$ is a conjugate of $\widehat{\pi_1 F_2 }$ in $\widehat{\pi_1 \Sigma_2 }$, so \autoref{lem: converse} implies that $\Stab_{\widehat{\pi_1\Sigma_2}}(v_2)$, when acting on $\T_1$ via $\phi^{-1}$, fixes a vertex $v_1'$ in $\T_1$. Let $v_1\in V(\T_1)$ be the vertex corresponding to the coset $\widehat{\pi_1 F_1 }$. We claim that $v_1'=v_1$. 

Suppose by contradiction that $v_1'\neq v_1$. Recall that $\phi(\widehat{\pi_1 F_1 })\subseteq  \Stab_{\widehat{\pi_1\Sigma_2}}(v_2)$, so the subgroup $\widehat{\pi_1 F_1 }=\phi^{-1}(\phi(\widehat{\pi_1 F_1 }))\subseteq \phi^{-1}(\Stab_{\widehat{\pi_1\Sigma_2}}(v_2))$, when acting on $\T_1$,  fixes $v_1'$. Note that $\widehat{\pi_1 F_1 }$ also fixes $v_1$.  When $v_1'\neq v_1$,  \cite[Corollary 4.1.6]{Rib17} implies that $\widehat{\pi_1 F_1 }$ fixes an edge in $\T_1$. However, the edge stabilizers of $\widehat{\pi_1 \Sigma_1 }$ acting on $\T_1$ are abelian since they are conjugate to $\overline{\langle a_{1,-}\rangle }\cong \widehat{\Z}$, while $\widehat{\pi_1 F_1 }$ is a non-abelian free profinite group  since the genus of $\Sigma_1$ is at least 2. 
This yields a contradiction. Hence, we have proven that $v_1'=v_1$. 

Therefore, $\phi^{-1}(\Stab_{\widehat{\pi_1\Sigma_2}}(v_2)) \subseteq \Stab_{\widehat{\pi_1\Sigma_1}}(v_1)= \widehat{\pi_1 F_1}$. 
\end{proof}

\begin{lemma}\label{lem: unoriented isomorphism}
There exists an isomorphism $\Phi: \T_1\to \T_2$ as unoriented combinatorial graphs that is  equivariant with respect to  $\phi: \widehat{\pi_1 \Sigma_1 }\to \widehat{\pi_1 \Sigma_2 }$. 
\end{lemma}
\begin{proof}
Let $v_1\in V(\T_1)$ be the vertex corresponding to the coset $\widehat{\pi_1 F_1 }$, and let $v_2\in V(\T_2)$ be given by \autoref{lem: F fix vertex}. Define $\Phi_V: V(\T_1)\to V(\T_2)$ by $\Phi(g\cdot v_1)=\phi(g) \cdot v_2$ for any $g\in \widehat{\pi_1 \Sigma_1 }$. Indeed, $\Phi_V $ is a well-defined bijection, since $\widehat{\pi_1 \Sigma_i }$ acts transitively on $V(\T_i)$ and $\phi(\stab_{\widehat{\pi_1\Sigma_1}}(v_1))= \stab_{\widehat{\pi_1\Sigma_2}}(v_2)$ by \autoref{lem: stab dt}. Clearly, $\Phi_V$ is $\phi$-equivariant. 

To show that $\Phi_V$ yields a $\phi$-equivariant  isomorphism $\Phi: \T_1\to \T_2$ as unoriented combinatorial graphs, it suffices to show that for any $w,w'\in V(\T_1)$, $w$ and $w'$ are joined by an edge if and only if $\Phi_V(w)$ and $\Phi_V(w')$ are joined by an edge, since $\T_1$ and $\T_2$ have no loops or multiple edges. We claim that $w$ and $w'$ are joined by an edge if and only if $\Stab_{\widehat{\pi_1\Sigma_1}}(w)\cap \Stab_{\widehat{\pi_1\Sigma_1}}(w')$ is non-trivial. On one hand, if they are joined by an edge $e\in E(\T_1)$, then the non-trivial subgroup $\Stab_{\widehat{\pi_1\Sigma_1}}(e)$ stabilizes both $w$ and $w'$. On the other hand, if they are not joined by an edge, then \autoref{lem: acylindrical} implies that $\Stab_{\widehat{\pi_1\Sigma_1}}(w)\cap \Stab_{\widehat{\pi_1\Sigma_1}}(w')$ is trivial. Similarly, $\Phi_V(w)$ and $\Phi_V(w')$ are joined by an edge if and only if $\Stab_{\widehat{\pi_1\Sigma_2}}(\Phi_V(w))\cap \Stab_{\widehat{\pi_1\Sigma_2}}(\Phi_V(w'))$ is non-trivial. The conclusion then follows from the $\phi$-equivariance of $\Phi_V$, which implies that $\phi(\Stab_{\widehat{\pi_1\Sigma_1}}(w))= \Stab_{\widehat{\pi_1\Sigma_2}}(\Phi_V(w))$ and $\phi(\Stab_{\widehat{\pi_1\Sigma_1}}(w'))= \Stab_{\widehat{\pi_1\Sigma_2}}(\Phi_V(w'))$. 
\end{proof}

We now present the proof for \autoref{thm: cut and glue}.

\begin{sloppypar}
\begin{proof}[Proof of \autoref{thm: cut and glue}]
Let $\Phi:\T_1\to \T_2$ be the $\phi$-equivariant isomorphism of unoriented combinatorial graphs given by \autoref{lem: unoriented isomorphism}. Since $a_{i,-}$  is conjugate to  $a_i$, we can assume that $\phi(a_{1,-})=\hat{h}a_{2,-}^{\mu} \hat{h}^{-1}$ for some $\hat{h} \in \widehat{\pi_1 \Sigma_2 }$.  Consider the edges $e=\overline{\langle a_{1,-}\rangle }\in E(\T_1)$ and $e'=\overline{\langle a_{2,-}\rangle }\in E(\T_2)$. 
Then, 
\begin{equation*}
\Stab_{\widehat{\pi_1\Sigma_2}}(\hat h   e')=\hat{h} \overline{\langle a_{2,-}\rangle } \hat{h}^{-1}=\phi(\overline{\langle a_{1,-}\rangle })=\phi(\Stab_{\widehat{\pi_1\Sigma_1}}(e))=\Stab_{\widehat{\pi_1\Sigma_2}}(\Phi(e)),
\end{equation*}
where the last equality follows from the $\phi$-equivariance of $\Phi$. 
According to \autoref{lem: malnormal}, this implies $\Phi(e)=\hat h  e'$. 

Since $\Phi$ is an isomorphism of unoriented combinatorial graphs, there are two possible cases. 
\begin{itemize}
\item Case (i): $\Phi(\partial^{-}(e))=\partial^{-}(\hat h  e')$ and $\Phi(\partial^{+}(e))=\partial^{+}(\hat h  e')$. 
\item Case (ii): $\Phi(\partial^{-}(e))=\partial^{+}(\hat h  e')$ and $\Phi(\partial^{+}(e))=\partial^{-}(\hat h  e')$. 
\end{itemize}

To simplify notations, we denote
\begin{equation*}
\begin{gathered}
v^-=\partial^{-}(e)=\widehat{\pi_1 F_1 }\in V(\T_1),\;v^+=\partial^+(e)= \gamma_1\widehat{\pi_1 F_1 }\in V(\T_1) \\
w^-=\partial^{-}(\hat h e')= \hat{h}\widehat{\pi_1 F_2 }\in V(\T_2),\;w^+=\partial^{+}(\hat h  e')=\hat{h}\gamma_{2}\widehat{\pi_1 F_2 }\in V(\T_2).
\end{gathered}
\end{equation*}

In case (i), $\phi$ sends $\Stab_{\widehat{\pi_1\Sigma_1}}(v^-)=\widehat{\pi_1 F_1 }$ to $\Stab_{\widehat{\pi_1\Sigma_2}}(w^-)=\hat{h}\widehat{\pi_1 F_2 }\hat{h}^{-1}$. 
Thus, we obtain an isomorphism $\psi: \widehat{\pi_1 F_1 } \to \widehat{\pi_1 F_2 }$ by  setting 
$$ \psi= {\left.\left(\Inn_{\hat{h}}^{-1}\circ \phi\right)\right|}_{\widehat{\pi_1F_1}}\,.$$
Then, property \ref{cag1} stated in  the theorem holds if we take $\hat{g}=\hat{h}$.
In addition, $$\phi(\gamma_1)w^-=\Phi(\gamma_1v^-)=\Phi(v^+)=w^+=\hat{h}\gamma_2\hat{h}^{-1}w^-.$$ Hence, $\hat{h}\gamma_2^{-1}\hat{h}^{-1}\phi(\gamma_1) \in \Stab_{\widehat{\pi_1\Sigma_2}}(w^-)= \hat{h} \widehat{\pi_1 F_2 }\hat{h}^{-1}$. In other words, there exists $\hat\delta\in \widehat{\pi_1 F_2 }$ such that $$\phi(\gamma_1)= \hat{h} \gamma_2\hat{\delta} \hat{h}^{-1}=\hat{g} \gamma_2 \hat{\delta} \hat{g}^{-1},$$ and property \ref{cag2} in the statement holds. Finally, $\psi(a_{1,-})=a_{2,-}^\mu$ by construction;  and 
\begin{equation*}
\begin{aligned}
\psi(a_{1,+})=&\hat{h}^{-1} \phi(a_{1,+}) \hat{h}=\hat{h}^{-1} \phi(\gamma_1^{-1}a_{1,-} \gamma_1)\hat{h}\\=&\hat{\delta}^{-1} \gamma_2^{-1} \hat{h}^{-1} \phi(a_{1,-}) \hat{h} \gamma_2 \hat{\delta}
=\hat{\delta}^{-1} \gamma_2^{-1} a_{2,-}^\mu \gamma_2 \hat{\delta} = \hat{\delta}^{-1} a_{2,+}^{\mu} \hat{\delta},
\end{aligned}
\end{equation*}
 which is a conjugate of $a_{2,+}^{\mu} $ within $\widehat{\pi_1 F_2 }$. Thus, property \ref{cag3} also holds. 

In case (ii), $\phi$ sends $\Stab_{\widehat{\pi_1\Sigma_1}}(v^-)= \widehat{\pi_1 F_1 }$ to $\Stab_{\widehat{\pi_1\Sigma_2}}(w^+)= \hat{h} \gamma_2 \widehat{\pi_1 F_2 } \gamma_2^{-1} \hat{h}^{-1}$. 
We then obtain an isomorphism $\psi:\widehat{\pi_1 F_1 }\to \widehat{\pi_1 F_2 }$ by setting 
$$\psi=\left.\left( \Inn^{-1}_{\hat{h}\gamma_2}\circ \phi \right) \right|_{\widehat{\pi_1F_1}}\,. $$
Then, property \ref{cag1} holds if we take $\hat{g}=\hat{h}\gamma_2\in \widehat{\pi_1 \Sigma_2 }$. 
In addition, $$\phi(\gamma_1)w^+=\Phi(\gamma_1v^-)=\Phi(v^+)= w^-= \hat{h} \gamma_2^{-1} \hat{h}^{-1} w^+.$$ 
Hence, $\hat{h}\gamma_2 \hat{h}^{-1} \phi(\gamma_1) \in \Stab_{\widehat{\pi_1\Sigma_2}}(w^+)= \hat{h} \gamma_2 \widehat{\pi_1 F_2 } \gamma_2^{-1} \hat{h}^{-1}$. In other words, there exists $\hat{\delta}\in \widehat{\pi_1 F_2 }$ such that $$\phi(\gamma_1)= \hat{h} \hat \delta \gamma_2^{-1} \hat{h}^{-1}= \hat{g}\gamma_2^{-1} \hat \delta \hat{g}^{-1}.$$ Thus, property \ref{cag2} holds. Finally, $$\psi(a_{1,-})= \gamma_2^{-1} \hat{h}^{-1} \phi(a_{1,-}) \hat{h} \gamma_2= \gamma_2^{-1} a_{2,-}^{\mu} \gamma_2= a_{2,+}^{\mu},$$ and 
$$
\psi(a_{1,+})= \gamma_2^{-1} \hat{h}^{-1} \phi(\gamma_1^{-1} a_{1,-} \gamma_1) \hat{h} \gamma_2= \hat{\delta}^{-1} \hat{h}^{-1} \phi(a_{1,-}) \hat{h} \hat{\delta}=   \hat{\delta}^{-1} a_{2,-}^{\mu} \hat{\delta}
$$
is a conjugate of $a_{2,-}^{\mu}$ within $\widehat{\pi_1 F_2 }$. Thus, property \ref{cag3} also holds. 
\end{proof}
\end{sloppypar}

\begin{remark}\label{rmk: orientation reversal}
\begin{enumerate}[leftmargin=*, label=(\arabic*)]
\item\label{rmkor1} Since $\widehat{\pi_1 \Sigma_i }$ acts transitively on $E(\T_i)$, the isomorphism $\Phi$ either preserves the orientations of all the edges or reverses the orientations of all the edges. This orientation reversal could possibly happen, even when $\phi$ is the profinite completion of an isomorphism between $\pi_1 \Sigma_1 $ and $\pi_1 \Sigma_2 $. One may consider for example an orientation reversing homeomorphism of the surface $\Sigma$ that preserves the oriented simple closed curve $\alpha$. This homeomorphism switches the two sides of $\alpha$, yielding an orientation reversal on the HNN-structure. 
\item According to \cite[Lemma 3.12]{Xu25A}, $\Phi$ descends  to a congruent isomorphism between the two graphs of profinite groups $(\widehat{G_1},X)$ and $(\widehat{G_2},X)$, up to possibly reorienting the edges. This would imply items \ref{cag1} and \ref{cag3} of \autoref{thm: cut and glue}. However, we  provide a concrete proof here for \autoref{thm: cut and glue} to obtain the precise conjugator in item \ref{cag2} of \autoref{thm: cut and glue}. 
\item The proof of \autoref{lem: unoriented isomorphism} (and the extension to a congruent isomorphism) also works for non-separating simple closed curves, and even for multicurves on $\Sigma_i$, where the HNN-extension is replaced by amalgamations or general splittings. We only prove the case of non-separating simple closed curves here for simplicity, as this is the only case that will be used later. 
\end{enumerate}
\end{remark}

\subsection{The second homology}
In this subsection, we introduce an application of \autoref{thm: cut and glue}. 

\begin{sloppypar}
Let $\Sigma$ be a closed orientable surface with positive genus. Then, $\Sigma$ is aspherical, and we can identify $H_\ast(\Sigma;\Z)$ with the group homology $H_\ast(\pi_1\Sigma;\Z)$.  
According to \autoref{prop: Zhat homology} and \autoref{prop: good},   we have canonical isomorphisms $\tensor H_\ast(\Sigma;\Z)\cong H_\ast(\pi_1\Sigma;\widehat{\Z})\cong \H_\ast(\widehat{\pi_1\Sigma};\widehat{\Z})$.
\end{sloppypar}  
\begin{proposition}\label{prop: surface homology coefficient}
Follow the same assumption of \autoref{thm: cut and glue}. Choose  arbitrary orientations on $\Sigma_1$ and $\Sigma_2$, and let $[\Sigma_i]\in H_2(\Sigma_i;\Z)$ be the corresponding fundamental classes.  Then, the isomorphism
$$
\phi_\ast: \tensor H_2(\Sigma_1;\Z) \cong \H_2(\widehat{\pi_1 \Sigma_1 };\widehat\Z)\longrightarrow \H_2(\widehat{\pi_1 \Sigma_2 };\widehat\Z)\cong\tensor H_2(\Sigma_2;\Z)
$$
sends $1\otimes[\Sigma_1]$ to $\pm\mu \otimes [\Sigma_2]$. 
\end{proposition}

\begin{proof}
The proof invokes Poincar\'e duality, which we explain as follows. Let $\Sigma$ be one of $\Sigma_1$ and $\Sigma_2$. Poincar\'e duality gives an isomorphism
\begin{equation*}
\begin{tikzcd}[column sep=large]
{\PD:\; H^1(\Sigma;\Z) } \arrow[r,"{  [\Sigma] \cap -}"] & {H_1(\Sigma;\Z).}
\end{tikzcd}
\end{equation*}
By a slight abuse of notation, its inverse map  is also denoted by $\PD: H_1(\Sigma;\Z)\to H^1(\Sigma;\Z)$. For any oriented loops $\alpha,\beta$ on $\Sigma$, let $[\alpha],[\beta]\in H_1(\Sigma;\Z)$ denote their homology classes. Then $\langle \PD[\alpha],[\beta]\rangle=I( \beta , \alpha )$, where $I(-,-)$ denotes the algebraic intersection number. 

There is also a  mod $ n$ version of Poincar\'e duality. For any class $[\omega]$ in $H_\ast(\Sigma;\Z)$ or $H^\ast(\Sigma;\Z)$, let $[\omega]_n$ denote its image in  $H_\ast(\Sigma;\Z/n)$ or $H^\ast(\Sigma;\Z/n)$. Then, by a further abuse of notation, we have an isomorphism
\begin{equation*}
\begin{tikzcd}[column sep=large]
\PD:\; {H^1(\Sigma;\Z/n)} \arrow[r,"{ [\Sigma]_n  \cap -}"] & {H_1(\Sigma;\Z/n).}
\end{tikzcd}
\end{equation*}
The inverse of this map is also denoted by $\PD: H_1(\Sigma;\Z/n)\to H^1(\Sigma;\Z/n)$. It is clear that $\PD([\omega]_n)=(\PD[\omega])_n$ for any $[\omega]\in H_1(\Sigma;\Z)$ or $H^1(\Sigma;\Z)$, so we simply denote it by $\PD[\omega]_n$.  For the oriented loops $\alpha,\beta$ as above, $\langle \PD[\alpha]_n,[\beta]_n\rangle =I(\beta,\alpha)\pmod n$. 

Now let us return to the proof. 
Recall that $\phi_\ast: \tens H_2(\Sigma_1;\Z)\to \tens H_2(\Sigma_2;\Z)$ is an isomorphism of $\widehat{\Z}$-modules, so there exists $\kappa\in \Zx$ such that $\phi_\ast(1\otimes [\Sigma_1])=\kappa\otimes [\Sigma_2]$, and it suffices to prove $\kappa=\pm \mu$. 

For any $n\in \N$, $\phi$ induces isomorphisms 
$$
\phi_\ast: H_\ast(\Sigma_1;\Z/n) \cong \H_\ast(\widehat{\pi_1\Sigma_1};\Z/n) \tto \H_\ast(\widehat{\pi_1\Sigma_2};\Z/n) \cong  H_\ast(\Sigma_2;\Z/n)
$$
and 
$$
\phi^\ast:H^\ast(\Sigma_2;\Z/n)\cong \H^\ast(\widehat{\pi_1\Sigma_2};\Z/n) \tto \H^\ast(\widehat{\pi_1\Sigma_1};\Z/n) \cong H^\ast(\Sigma_1;\Z/n) . 
$$
In particular, $\phi_\ast([\Sigma_1]_n)=\kappa\cdot [\Sigma_2]_n$ since the following diagram commutes. 
\begin{equation}\label{diagmoq}
\begin{tikzcd}[column sep=0.61cm]
H_\ast(\Sigma_1;\Z) \arrow[r] \arrow[d] & \H_\ast(\widehat{\pi_1\Sigma_1};\widehat{\Z}) \arrow[r,"\phi_\ast"] \arrow[d] & \H_\ast(\widehat{\pi_1\Sigma_2};\widehat{\Z}) \arrow[d] & H_\ast(\Sigma_2;\Z) \arrow[l] \arrow[d]\\
H_\ast(\Sigma_1;\Z/n) \arrow[r, "\cong" ]   & \H_\ast(\widehat{\pi_1\Sigma_1};\Z/n) \arrow[r,"\phi_\ast"]   & \H_\ast(\widehat{\pi_1\Sigma_2};\Z/n)  & H_\ast(\Sigma_2;\Z/n) \arrow[l,"\cong"']  
\end{tikzcd}
\end{equation}

We  claim that $\phi^\ast(\kappa^{-1}\mu \cdot\PD[\alpha_2]_n)= \PD[\alpha_1]_n$.  
 Combining \autoref{nat1} with \autoref{nat2}, we have the following  commutative diagram. 
\begin{equation*}
\begin{tikzcd}[column sep=0.24cm]
{H_2(\Sigma_1;\Z/n)} \arrow[r, symbol=\times]                                                               \arrow[d, "\phi_\ast"',"\cong"]   & {H^1(\Sigma_1;\Z/n)} \arrow[rrrr, "\cap"]              &  &  &  & {H_1(\Sigma_1;\Z/n)} \arrow[d, "\phi_\ast","\cong"']               \\
{H_2(\Sigma_2;\Z/n)} \arrow[r, symbol=\times]                                                                   & {H^1(\Sigma_2;\Z/n)}\arrow[rrrr, "\cap"]      \arrow[u,"\phi^\ast","\cong"']          &  &  &  & {H_1(\Sigma_2;\Z/n)}  
\end{tikzcd}
\end{equation*}
 As a consequence, for any $[\omega]\in H^1(\Sigma_2;\Z/n)$, 
$$
\phi_\ast(\PD(\phi^\ast[\omega]))= \phi_\ast( [\Sigma_1]_n\cap\phi^\ast[\omega] )=  \phi_\ast([\Sigma_1]_n)\cap[\omega]=\kappa [\Sigma_2]_n\cap[\omega] = \kappa \cdot \PD[\omega]. 
$$
If we take $[\omega]=\kappa^{-1}\mu\cdot \PD[\alpha_2]_n$, then $\phi_\ast(\PD(\phi^\ast(\kappa^{-1}\mu\cdot \PD[\alpha_2]_n))) = \mu [\alpha_2]_n$.  
Note that  \autoref{lem: abelianization} together with diagram (\ref{diagmoq}) implies that  $\phi_\ast([\alpha_1]_n)=\mu [\alpha_2]_n$. 
 Hence,
$$
\phi^\ast(\kappa^{-1}\mu\cdot \PD[\alpha_2]_n)=\PD(\phi_\ast^{-1}(\mu[\alpha_2]_n))=\PD[\alpha_1]_n.
$$ 


In addition,   \autoref{thm: cut and glue}~\ref{cag2} implies that there exist  $\epsilon\in \{\pm 1\}$ and $\hat{\delta}\in \widehat{\pi_1F_2}$ such that $\phi(\gamma_1)$  is conjugate to  $\gamma_2^\epsilon \hat{\delta}$ in $\widehat{\pi_1\Sigma_2}$. By 
 \autoref{lem: abelianization},   $\phi_\ast:\H_1(\widehat{\pi_1\Sigma_1};\widehat{\Z})\to \H_1(\widehat{\pi_1\Sigma_2};\widehat{\Z})$ sends $[\gamma_1]$ to $\epsilon[\gamma_2]+[\hat{\delta}]$, where $[\hat{\delta}]$ belongs to the image of $  \H_1(\widehat{\pi_1F_2};\widehat{\Z})$ in  $\H_1(\widehat{\pi_1\Sigma_2};\widehat{\Z})$. Let $[x_n]$ be the image of $[\hat{\delta}]$ in $\H_1(\widehat{\pi_1\Sigma_2};\Z/n)\cong H_1(\Sigma_2;\Z/n)$. Then, $\phi_\ast([\gamma_1]_n)=\epsilon[\gamma_2]_n+[x_n]$, and $[x_n]$ belongs to the image of $$\mathrm{incl}_\ast: H_1(F_2;\Z/n) \cong \H_1(\widehat{\pi_1F_2};\Z/n)\tto \H_1(\widehat{\pi_1\Sigma_2};\Z/n) \cong H_1(\Sigma_2;\Z/n). $$
Since the loop $\alpha_2$ is disjoint with the subsurface $F_2$, the restriction of $\PD[\alpha_2]$ on $H^1(F_2;\Z)$ is zero. 
In particular,  $\langle \PD[\alpha_2]_n, [x_n]\rangle =0$. 

We  consider the pairings between these (co)homology classes. For any $n\in \N$, 
\begin{equation*}
\begin{aligned}
I(\gamma_1,\alpha_1)\,(\mathrm{mod}\,n)=&\langle \PD[\alpha_1]_n,[\gamma_1]_n\rangle =\langle \phi^\ast (\kappa^{-1}\mu \cdot \PD[\alpha_2]_n),[\gamma_1]_n\rangle\\ 
= &  \langle\kappa^{-1}\mu \cdot \PD[\alpha_2]_n, \phi_\ast([\gamma_1]_n)\rangle =\kappa^{-1}\mu \langle \PD[\alpha_2]_n, \epsilon [\gamma_2]_n+[x_n]\rangle\\
=&\kappa^{-1}\mu \langle \PD[\alpha_2]_n, \epsilon [\gamma_2]_n \rangle = \kappa^{-1}\mu \epsilon \cdot I(\gamma_2,\alpha_2)\,(\mathrm{mod}\,n).
\end{aligned}
\end{equation*}
Recall that $\gamma_i$ intersects with $\alpha_i$ transversely at one point, so  $I(\gamma_i,\alpha_i) \in \{\pm1\}$. Therefore,  
$$\kappa= \left (\epsilon \cdot I(\gamma_1,\alpha_1) \cdot  I(\gamma_2,\alpha_2)\right )\mu \pmod n, $$ 
%
where $\epsilon \cdot I(\gamma_1,\alpha_1) \cdot  I(\gamma_2,\alpha_2)\in \{\pm1\}$ is independent with $n$. 
In other words,  $\kappa=\pm \mu \in \widehat{\Z}$, which finishes the proof. 
\end{proof}

\begin{remark}
The $\pm$-sign appears in \autoref{prop: surface homology coefficient} not only because of the arbitrary  choices of orientations on $\Sigma_1$ and $\Sigma_2$, but also because of the possible orientation reversal in the HNN-extension described in \autoref{rmk: orientation reversal}~\ref{rmkor1}, which yields the possibilities of  $\pm$-signs in \autoref{thm: cut and glue}~\ref{cag2}. 
\end{remark}

Using techniques of finite covers, we deduce the following corollary from  \autoref{prop: surface homology coefficient}. 
\begin{corollary}\label{cor: surface coefficient}
Let $S_1$ and $S_2$ be closed oriented surfaces with genera at least $2$, and let $\phi: \widehat{\pi_1S_1}\to \widehat{\pi_1S_2}$ be an isomorphism. 
Suppose $b_1\in \pi_1S_1,\, b_2\in \pi_1S_2$ are non-trivial elements, and there exists $\mu\in \Zx$ such that $\phi(b_1)$  is conjugate to  $b_2^\mu$ in $\widehat{\pi_1S_2}$. 
Then, 
$$
\phi_\ast: \tensor H_2(S_1;\Z) \cong \H_2(\widehat{\pi_1S_1};\widehat \Z)\tto \H_2(\widehat{\pi_1S_2};\widehat \Z)\cong\tensor H_2(S_2;\Z)
$$
sends $1\otimes[S_1]$ to $\pm\mu \otimes [S_2]$, where $[S_1]$ and $[S_2]$ denote the fundamental classes representing their orientations. 
\end{corollary}

We start the proof with a well-known lemma. 
\begin{lemma}\label{lem: Scott simple cover}
Let $S$ be a closed orientable surface, and let $\beta$ be a homotopically non-trivial loop on $S$. Then, there exists a standard characteristic cover $S^{(m)}$ of $S$ such that every elevation of $\beta$ in $S^{(m)}$ can be freely homotoped to  a non-separating simple closed curve. 
\end{lemma}
We note that the property also holds for any $S^{(m')}$ where $m'\ge m$, since the standard characteristic cover $S^{(m')}\to S$ factors through $S^{(m)}$, and any elevation of a non-separating simple closed curve on $S^{(m)}$ is a non-separating simple closed curve on $S^{(m')}$. 
\begin{proof}
According to Scott's LERFness theorem \cite{Sco78}, there exists a finite cover $S'$ of $S$ such that a certain elevation of $\beta$ can be freely homotoped to a simple closed curve $\beta'$ in $S'$. If $\beta'$ is separating, then we can take a further two-fold cover $S''$ of $S'$ in which $\beta'$ lifts to a non-separating simple closed curve $\beta''$. If $\beta'$ is already non-separating, then we denote $S''=S'$ and $\beta''=\beta'$. We then take a standard characteristic cover $S^{(m)}$ that factors through $S''$; for instance, $m=[S'':S']$.  Then, any elevation of $\beta''$ in $S^{(m)}$ is a non-separating simple closed curve. Moreover, $S^{(m)}$ is a regular cover of $S$, so the free homotopy class of any elevation of $\beta$ differs from an elevation of $\beta''$ by a deck transformation. In other words, every elevation of $\beta$ in $S^{(m)}$ is  freely homotopic to  a non-separating simple closed curve.
\end{proof}

\begin{proof}[Proof of \autoref{cor: surface coefficient}]
For $i=1,2$, let $\beta_i$ be an oriented loop on $S_i$ whose free homotopy class is represented by the conjugacy class of $b_i\in \pi_1S_i$. According to \autoref{lem: Scott simple cover}, there exists a sufficiently large $m\in \N$, consistent for $i=1,2$, such that every elevation of $\beta_i$ in $S_i ^{(m)}$ can be freely homotoped to a non-separating simple closed curve. We denote $\phi':\widehat{\pi_1\nss{S}{1}{(m)}}\to \widehat{\pi_1\nss{S}{2}{(m)}}$ as the restriction of $\phi$ to the $\phi$-corresponding pair of finite-index subgroups $\pi_1S_1^{(m)}$ and $\pi_1S_2^{(m)}$. 

 Note that $S_i ^{(m)}$ is a finite regular cover of $S_i$, so all elevations of $\beta_i$ in $S_i ^{(m)}$ are $n_i$ fold covers of $\beta_i$. Group theoretically, $n_i$ is the smallest positive integer such that $b_i^{n_i}\in \pi_1  \nss{S}{i}{(m)} $, and the conjugacy class of $b_i^{n_i}$ in $\pi_1  \nss{S}{i}{(m)} $ represents the free homotopy class of a certain elevation of $\beta_i$.  Now observe that $b_1^n \in \pi_1 S_1 ^{(m)}$ is equivalent to $b_1^n \in \overline{\pi_1 \nss{S}{1}{(m)}}$. Since $ \overline{\pi_1 \nss{S}{2}{(m)}  }   $ is a normal subgroup  in $\widehat{\pi_1 S_2 }$, this is in turn equivalent to $b_2^{n\mu} \in\overline{\pi_1 \nss{S}{2}{(m)} }$. Note that $\mu\in \Zx$, so this is further equivalent to $(b_2^{n\mu})^{\mu^{-1}}=b_2^n\in\overline{\pi_1 \nss{S}{2}{(m)} }$, and finally  equivalent  to $b_2^n \in \pi_1S_2^{(m)}$. We therefore conclude that $n_1=n_2$.

Denote $a_1=b_1^{n_1}\in \pi_1S_1^{(m)}$, whose conjugacy class represents a non-separating simple closed curve on $S_1^{(m)}$. 
Suppose that $\phi(b_1)= \hat{g} b_2^\mu \hat{g}^{-1}$ for some $\hat{g}\in \widehat{\pi_1S_2}$. Since $\overline{\pi_1\nss{S}{2}{(m)}}$ is an open subgroup  in $\widehat{\pi_1S_2}$  and $\pi_1S_2$ is a dense subgroup in $\widehat{\pi_1S_2}$, we can find $\hat{h}\in \overline{\pi_1\nss{S}{2}{(m)}}$ and $\gamma\in \pi_1S_2$ such that $\hat{g}=\hat{h}\gamma$. 
We denote $a_2= \gamma b_2^{n_2} \gamma^{-1}$, which belongs  to $\pi_1\nss{S}{2}{(m)}$ by our assumption. 
The conjugacy class of $a_2$ in $\pi_1\nss{S}{2}{(m)}$   represents the free homotopy class of a certain elevation of $\beta_2$, which is also a non-separating simple closed curve in $\nss{S}{2}{(m)}$.  

By construction, $\phi'(a_1)=\hat{h} a_2^\mu \hat{h}^{-1}$, where the conjugator $\hat{h}$ belongs to $\widehat{\pi_1\nss{S}{2}{(m)}}$. Then, \autoref{prop: surface homology coefficient} implies that 
$$
\phi_\ast': \tensor H_2(S_1 ^{(m)};\Z) \cong \H_2(\widehat{\pi_1 \nss{S}{1}{(m)}};\widehat\Z)\longrightarrow \H_2(\widehat{\pi_1 \nss{S}{2}{(m)}};\widehat\Z)\cong\tensor H_2(S_2 ^{(m)};\Z)
$$
sends $1\otimes[S_1 ^{(m)}]$ to $\pm\mu \otimes [S_2 ^{(m)}]$, where $[S_i ^{(m)}]$ are fundamental classes representing the orientations lifted from $S_i$. 

By construction, the following diagram commutes
\begin{equation*}
\begin{tikzcd}[row sep=huge,column sep=small]
\tensor H_2(S_1 ^{(m)};\Z) \arrow[r, "\cong"] \arrow[d, "{p_1}_\ast"'] & \H_2(\widehat{\pi_1 \nss{S}{1}{(m)} };\widehat\Z) \arrow[r, "\phi'_\ast"] \arrow[d, "\mathrm{incl}_\ast"'] & \H_2(\widehat{\pi_1 \nss{S}{2}{(m)} };\widehat\Z) \arrow[d, "\mathrm{incl}_\ast"] & \tensor H_2(S_2 ^{(m)};\Z) \arrow[l, "\cong"'] \arrow[d, "{p_2}_\ast"] \\
\tensor H_2(S_1;\Z) \arrow[r, "\cong"]                               & \H_2(\widehat{\pi_1 S_1 };\widehat\Z) \arrow[r, "\phi_\ast"]                                        & \H_2(\widehat{\pi_1 S_2 };\widehat\Z)                                      & \tensor H_2(S_2;\Z) \arrow[l, "\cong"']                             
\end{tikzcd}
\end{equation*}
where $p_i:S_i^{(m)}\to S_i$ denotes the covering map. 

Let $n=[S_1^{(m)}:S_1]=[ S_2^{(m)}:S]$ be the covering degree of $p_i$. Then, the vertical maps  induced by the coverings send $1\otimes[S_i^{(m)}]$ to $n\otimes [S_i]$. As a consequence, $\phi_\ast$ sends $n\otimes [S_1]$ to $\pm n\mu \otimes [S_2]$. Therefore, $\phi_\ast$ sends $1\otimes [S_1]$ to $\pm  \mu \otimes [S_2]$ since $\phi_\ast$ is an isomorphism of $\widehat{\Z}$-modules, and $\widehat{\Z}$ is torsion free. 
\end{proof}

\section{Regularity}\label{sec: regular}
For an abstract group $\Gamma$, let $\ab{\Gamma}=\Gamma/[\Gamma,\Gamma]$ denote its abelianization; and for a profinite group $G$, let $\Ab{G}=G/\overline{[G,G]}$ denote its profinite abelianization. 
For a finitely generated group $\Gamma$, there are canonical isomorphisms  
$$
\tensor \ab{\Gamma}\cong \widehat{\ab{\Gamma}}\cong \Ab{\widehat{\Gamma}};  
$$
see \cite[Lemmas 2.8 and 2.9]{Xu25A}. 

The notion of regularity was first introduced by Boileau and Friedl in \cite{BF20}. 
\begin{definition}
Let $\Gamma_1$ and $\Gamma_2$ be finitely generated groups. An isomorphism $\Phi: \widehat{\Gamma_1}\to \widehat{\Gamma_2}$ is {\em regular} if the induced isomorphism $\Ab{\Phi}:   \widehat{\Gamma_1}{}^{\mathrm{Ab}}\to  \widehat{\Gamma_2}{}^{\mathrm{Ab}}$ is the profinite completion of an isomorphism $f :  \Gamma_1 ^{\mathrm{ab}}\to   \Gamma_2 ^{\mathrm{ab}}$. 
\end{definition}

Note that the isomorphism $\Phi$ is always regular when $b_1(\Gamma_1)=b_1(\Gamma_2)=0$. Hence we shall mainly focus on the case that  $b_1(\Gamma_1)=b_1(\Gamma_2)>0$.

The purpose of this section is to prove the following theorem. 

\begin{theorem}\label{thm: regular}
Suppose $\Gamma_1$ and $\Gamma_2$ are lattices in $\PSL_2(\C)$, and $\Phi: \widehat{\Gamma_1}\to \widehat{\Gamma_2}$ is an isomorphism. Then, $\Phi$ is regular. 
\end{theorem}

\autoref{thm: regular} can be viewed as a first step in proving the surjectivity in \autoref{mainthm2}, since any genuine isomorphism $\Phi: \widehat{\Gamma_1}\to \widehat{\Gamma_2}$ is always regular. 

\begin{remark}
\autoref{thm: regular} for non-uniform lattices was previously proven by the author in \cite[Theorem 1.4]{Xu25}. The proof we present here is independent of that one, and in fact, provides an alternate proof for the earlier result. However, for convenience, we mainly focus on the uniform case here, and some intermediate techniques in \cite{Xu25} will be applied  to reduce the non-uniform case to the uniform case.
\end{remark}

Upon submitting this paper, the author learnt that \autoref{thm: regular} has also been proven by Hanany \cite{Hanany} recently, using number-theoretical methods. Our proof, however, uses geometric methods, where some ingredients also play a role in the subsequent proofs. 

\subsection{Liu's alignment theorems} 

For simplicity, let us focus on torsion-free lattices in $\PSL_2(\C)$, which are fundamental groups of orientable complete finite-volume hyperbolic 3-manifolds. In a series of works \cite{Liu23,Liu25}, Liu established several key properties for isomorphisms between the profinite completions of these groups. 

\begin{proposition}[{\cite[Theorems 1.2 and 1.3]{Liu23}}]\label{prop: Liu Zx-regular}
Suppose $M_1$ and $M_2$ are orientable complete finite-volume hyperbolic 3-manifolds. Let   $\Phi: \widehat{\pi_1M_1}\to \widehat{\pi_1M_2}$ be an isomorphism. Then the following properties hold. 
\begin{enumerate}[leftmargin=*, label=(\arabic*)]
\item\label{LiuReg1} The induced isomorphism 
$$
\hspace{3mm}\Ab{\Phi}: \tensor H_1(M_1;\Z) \cong \widehat{\pi_1M_1}{}^{\mathrm{Ab}} \tto  \widehat{\pi_1M_2}{}^{\mathrm{Ab}} \cong \tensor H_1(M_2;\Z) \hspace{-3mm}
$$
can be decomposed as $\mu \otimes f $, where $\mu\in \Zx$ denotes the scalar multiplication in $\widehat{\Z}$, and $f: H_1(M_1;\Z)\to H_1(M_2;\Z)$ is an isomorphism. 
\item\label{LiuReg2} The dual isomorphism of $f$, namely
\begin{equation*}
\hspace{5mm} f ^\ast: H^1(M_2;\Z)=\Hom(H_1(M_2;\Z),\Z) \to \Hom(H_1(M_1;\Z),\Z) =  H^1(M_1;\Z), \hspace{-4mm}
\end{equation*}
 preserves the Thurston norm and the fibered classes. 
\end{enumerate}
\end{proposition}

We will refer to $(\mu,f)$ appearing in \autoref{prop: Liu Zx-regular}~\ref{LiuReg1} as a  {\em homological datum} of $\Phi$. When $b_1(M_1)=b_1(M_2)>0$, there are exactly two possibilities of $(\mu,f)$, which differ from each other by a sign change. Reformulated by these notations, $\Phi$ is regular if and only if $\mu=\pm1$.  


We make some additional clarifications for \autoref{prop: Liu Zx-regular}~\ref{LiuReg2}. When $M_1$ and $M_2$ are non-compact, they are so called {\em cusped hyperbolic manifolds}. Any cusped hyperbolic manifold is homeomorphic to the interior of a compact manifold with toral boundary, which is  obtained by truncating its cusps. Since we are only concerned with the fundamental group, we do not distinguish between this compact manifold and its interior. As such, our definitions for Thurston norm and fibered classes from \autoref{sec: Thurston norm} apply for $M_1$ and $M_2$. 

In particular, \autoref{prop: Liu Zx-regular}~\ref{LiuReg2} implies that fiberedness is a profinite invariant for finite-volume hyperbolic 3-manifolds, see also \cite{JZ20}. 

\begin{remark}
In fact, Liu \cite[Lemma 7.1]{Liu25} proved that $\mu^2=1$, based on his positive solution to McMullen's conjecture \cite{Liu20}. This implies that each $p$-adic factor of $\mu$ belongs to $\{\pm1\}$, but it does not directly imply that $\mu=\pm1$ in $\widehat{\Z}$. 
Our proof for regularity does not rely on this result, though some other techniques in \cite{Liu20,Liu25} are applied. 
\end{remark}

Let us now focus on closed hyperbolic manifolds. Using the profinite correspondence of twisted Alexander polynomials, Liu \cite{Liu25} proved the following proposition.

\begin{proposition}[{\cite[Lemma 5.2]{Liu25}}]\label{prop: match up trajectory}
Let $M_1$ and $M_2$ be closed  orientable   hyperbolic 3-manifolds. Suppose $\Phi: \widehat{\pi_1M_1}\to \widehat{\pi_1M_2}$ is an isomorphism with a homological datum $(\mu,f)$.  Let $\psi_1\in H^1(M_1;\Z)$ be a fibered class. Then $\psi_2=(f^\ast)^{-1} (\psi_1) \in H^1(M_2;\Z)$  is also a fibered class; and 
 there exists a bijection
$$
\tau: \Traj(M_1,\psi_1) \to \Traj(M_2,\psi_2)
$$
such that for every $\gamma_1 \in \Traj(M_1,\psi_1)$, if $g_1 \in \pi_1M_1$ and $g_2 \in \pi_1M_2$ are representatives of the conjugacy classes $\gamma_1$ and $\tau(\gamma_1)$ respectively, then $\Phi(g_1)$ is conjugate to $g_2^\mu$ in $\widehat{\pi_1M_2}$.
\end{proposition}

In addition to fibering, \autoref{prop: match up trajectory} implies that crossfibering is also a profinite invariant for closed hyperbolic 3-manifolds, which we explain as follows.

\begin{sloppypar}
\begin{proposition}\label{prop: detect crossfibering}
Suppose that $M_1$ and $M_2$ are closed orientable hyperbolic 3-manifolds, and $\Phi: \widehat{\pi_1M_1}\to \widehat{\pi_1M_2}$ is an isomorphism. Then, $M_1$ is a crossfibered manifold if and only if $M_2$ is a crossfibered manifold. 

Moreover, suppose $(\mu,f)$ is a homological datum of $\Phi$, and   $(M_1,\psi_{\mathrm{A},1},\psi_{\mathrm{F},1},\gamma_1)$ is a crossfibering datum of $M_1$. Then, there exists  a crossfibering datum $(M_2,\psi_{\mathrm{A},2},\psi_{\mathrm{F},2},\gamma_2)$ of $M_2$ such that $\psi_{\mathrm{A},1}=f^\ast(\psi_{\mathrm{A},2})$, $\psi_{\mathrm{F},1}=f^\ast(\psi_{\mathrm{F},2})$, and for conjugacy representatives $g_i\in \pi_1M_i$ of $\gamma_i$, $\Phi(g_1)$ is conjugate to $g_2^\mu$ in $\widehat{\pi_1M_2}$. 
\end{proposition}
\end{sloppypar}

\begin{proof}
By symmetry, it suffices to assume  that $M_1$ is crossfibered and prove that $M_2$ is also crossfibered. Let $(\mu,f)$ denote a homological datum of $\Phi$.  We fix a crossfibering datum $(M_1,\psi_{\mathrm{A},1},\psi_{\mathrm{F},1},\gamma_1)$ of $M_1$. Then, \autoref{prop: Liu Zx-regular} implies that $(f^{\ast})^{-1}(\psi_{\mathrm{A},1})$ and $(f^{\ast})^{-1}(\psi_{\mathrm{F},1})$ are also  primitive  fibered classes, which we define as $\psi_{\mathrm{A},2}$ and $\psi_{\mathrm{F},2}$ respectively. Let $\tau:\Traj(M_1,\psi_{\mathrm{A},1}) \to \Traj(M_2,\psi_{\mathrm{A},2})$ be the bijection given by \autoref{prop: match up trajectory}, and let $\gamma_2=\tau(\gamma_1)\in  \Traj(M_2,\psi_{\mathrm{A},2})$. Then, for conjugacy representatives $g_i\in \pi_1M_i$ of $\gamma_i$, $\Phi(g_1)$ is conjugate to $g_2^\mu$ in $\widehat{\pi_1M_2}$. 

Let $[\gamma_i]\in H_1(M_i;\Z)$ denote their homology classes. We first prove that $f([\gamma_1])=[\gamma_2]$. Indeed, $\Ab{\Phi}:\widehat{\Z}\otimes H_1(M_1;\Z) \to \widehat{\Z} \otimes H_1(M_2;\Z)$ sends $1\otimes [\gamma_1] $ to $\mu \otimes [\gamma_2]$, since the conjugation does not affect the abelianization. Since $\mu \in \Zx$, we can recover $f$ by $\mu^{-1}\cdot \Ab{\Phi}$. Then, $f$ sends $[\gamma_1]$ to $[\gamma_2]$. 

To show that $(M_2,\psi_{\mathrm{A},2},\psi_{\mathrm{F},2},\gamma_2)$ is a crossfibering datum of $M_2$, it suffices to show that $\gamma_2$ freely homotopes onto $S_{M_2,\psi_{\mathrm{F},2}}$, i.e.\ $\langle \psi_{\mathrm{F},2}, [\gamma_2]\rangle =0$. This is because
\begin{equation*}
\langle \psi_{\mathrm{F},2}, [\gamma_2]\rangle= \langle \psi_{\mathrm{F},2}, f([\gamma_1])\rangle = \langle  f^\ast (\psi_{\mathrm{F},2}),  [\gamma_1] \rangle = \langle \psi_{\mathrm{F},1} , [\gamma_1]\rangle =0. \qedhere
\end{equation*}
\end{proof}

\begin{convention}\label{conv: crossfibering 1}
 We will refer to $(M_1,\psi_{\mathrm{A},1},\psi_{\mathrm{F},1},\gamma_1)$ and $(M_2,\psi_{\mathrm{A},2},\psi_{\mathrm{F},2},\gamma_2)$ appearing in \autoref{prop: detect crossfibering} as a {\em $(\Phi,\mu,f)$-corresponding pair} of crossfibering data. 
\end{convention}

\subsection{Proof of regularity}

We shall first prove \autoref{thm: regular} for the special case that $\Gamma_1$ and $\Gamma_2$ are fundamental groups of crossfibered manifolds. The proof is based on \autoref{cor: surface coefficient}. 
\begin{lemma}\label{lem: crossfibered regular}
Suppose $M_1$ and $M_2$ are crossfibered manifolds, and $\Phi:\widehat{\pi_1M_1}\to \widehat{\pi_1M_2}$ is an isomorphism. Then, $\Phi$ is regular. 
\end{lemma}
\begin{proof}
Let $(\mu,f)$ be a homological datum of $\Phi$. It suffices to prove that $\mu=\pm1$.  Let $(M_1,\psiAyi,\psiFyi, \gamma_1)$ and $(M_2,\psiAer,\psiFer,\gamma_2)$ be a $(\Phi,\mu,f)$-corresponding pair of crossfibering data given by \autoref{prop: detect crossfibering}. Let $g_i\in \pi_1M_i$ be a conjugacy representative of $\gamma_i\in \Traj(M_i,\psi_{\mathrm{A},i})$. We may possibly compose $\Phi$ with an inner automorphism of $\widehat{\pi_1M_2}$, which does not affect the abelianization $\Ab{\Phi}$ (and hence the homological datum),  so that we may assume without loss of generality that $\Phi(g_1)=g_2^\mu$. Denote by $S_{i}=S_{M_i,\psi_{\mathrm{F},i}}$ the fiber surface dual to $\psi_{\mathrm{F},i}$. 

Let us fix orientations on $M_1$ and $M_2$, and denote the corresponding fundamental classes by $[M_i]\in H_3(M_i;\Z)$. Fix appropriate orientations on $S_i$, whose fundamental classes are denoted by $[S_i]\in   H_2(M_i;\Z)$, such that $ [M_i]\cap\psi_{\mathrm{F},i}=[S_i]$, where ``$\cap$'' denotes the cap product.  We now study how the isomorphism $\Phi$ matches up $[M_i]$, $[S_i]$ and $\psi_{\mathrm{F},i}$. Following the notation in \autoref{prop: surface homology coefficient}, for any $n\in \N$ and any $[\omega]$ in $H^\ast(M_i;\Z)$ or $H_\ast(M_i;\Z)$, let $[\omega]_n$ denote its image in $H^\ast(M_i;\Z/n)$ or $H_\ast(M_i;\Z/n)$.

Firstly, $\Phi$ induces an isomorphism $$\Phi_\ast: \tensor H_3(M_1;\Z)\cong \H_3(\widehat{\pi_1M_1};\widehat{\Z}) \tto \H_3(\widehat{\pi_1M_2};\widehat{\Z}) \cong \tensor H_3(M_2;\Z).$$
In particular, there exists $\lambda_3\in \Zx$ such that $\Phi_\ast(1\otimes [M_1])=\lambda_3\otimes [M_2]$. According to \cite[Lemma 5.3]{Liu25} (see also \cite[Proposition 6.5]{Xu25}), $\lambda_3=\pm \mu^3\in \Zx$. 

Secondly, since $\Ab{\Phi}$ decomposes as $\mu\otimes f$, we have the following commutative diagram. 
\begin{equation*}
\begin{tikzcd}[column sep=0.4cm]
\widehat{\pi_1M_1} \arrow[d, "\Phi"'] \arrow[rrrr, two heads] & &  &  & \widehat{\pi_1M_1}{}^{\mathrm{Ab}} \arrow[d, "\Ab{\Phi}"'] \arrow[r, symbol=\cong] & \tensor H_1(M_1;\Z) \arrow[rrrr, "{1\otimes \psi_{\mathrm{F},1}}", two heads] \arrow[d, "\mu\otimes f"] &  & & & \widehat{\Z} \arrow[d, "\mu"] \\
\widehat{\pi_1M_2} \arrow[rrrr, two heads]                    & & &  & \widehat{\pi_1M_2}{}^{\mathrm{Ab}} \arrow[r,symbol=\cong]                         & \tensor H_1(M_2;\Z) \arrow[rrrr, "{1\otimes \psi_{\mathrm{F},2}}", two heads]                           &  & & & \widehat{\Z}                 
\end{tikzcd}
\end{equation*}
Composing with the quotient map $\widehat{\Z}\to \Z/n\Z$, we have the following commutative diagram for each $n\in \N$:  
\begin{equation}\label{diag: regular in mu}
\begin{tikzcd}
\widehat{\pi_1M_1} \arrow[d, "\Phi"',"\cong"] \arrow[r, "{\widehat{\psi_{\mathrm{F},1}}}", two heads] & \widehat{\Z} \arrow[d, "\mu"',"\cong" ] \arrow[r, two heads] & \Z/n \arrow[d, "\mu"',"\cong"] \\
\widehat{\pi_1M_2} \arrow[r, "{\widehat{\psi_{\mathrm{F},2}}}", two heads]                    & \widehat{\Z} \arrow[r, two heads]                  & \Z/n                 
\end{tikzcd}
\end{equation}
where we view $\psi_{\mathrm{F},i}$ as elements in $\Hom(\pi_1M_i,\Z)$. By \cite[Lemma 6.8.1]{RZ10}, we can identify $\H^1(\widehat{\pi_1M_i};\Z/n)$ with $\Hom(\widehat{\pi_1M_i},\Z/n)$. Then, the above commutative diagram implies that for each $n\in \N$, 
$$
\Phi^\ast: H^1(M_2;\Z/n)\cong \H^1(\widehat{\pi_1M_2};\Z/n) \tto \H^1(\widehat{\pi_1M_1};\Z/n)  \cong H^1(M_1;\Z/n)
$$
sends $[\psi_{\mathrm{F},2}]_n$ to $\mu[\psi_{\mathrm{F},1}]_n$. 

Thirdly, the fibrations induce  short exact sequences
\begin{equation*}
\begin{tikzcd}
1 \arrow[r] & \pi_1S_i \arrow[r,"\mathrm{incl}"] & \pi_1M_i \arrow[r,"\psi_{\mathrm{F},i}"] & \Z \arrow[r] & 1,
\end{tikzcd}
\end{equation*}
where one may choose basepoints belonging to $S_i$ for justification. According to \cite[p.16, Exercise 2]{Ser01} (or alternatively, \autoref{prop: trivial center injective} and \autoref{thm: lattice center free}), we have the following short exact sequence: 
\begin{equation*}
\begin{tikzcd}
1 \arrow[r] & \widehat{\pi_1S_i} \arrow[r] & \widehat{\pi_1M_i} \arrow[r,"\widehat{\psi_{\mathrm{F},i}}"] & \widehat{\Z} \arrow[r] & 1. 
\end{tikzcd}
\end{equation*}
In other words, $\ker(\widehat{\psi_{\mathrm{F},i}})=\overline {\pi_1S_i}$ is exactly $\widehat{\pi_1S_i}$. The left block of diagram (\ref{diag: regular in mu}) implies that $\Phi(\ker(\widehat{\psi_{\mathrm{F},1}}))=\ker(\widehat{\psi_{\mathrm{F},2}})$, so we obtain a commutative diagram of short exact sequences
\begin{equation}\label{equ: ses commutative crossfibering}
\begin{tikzcd}
1 \arrow[r] & \widehat{\pi_1S_1} \arrow[r] \arrow[d, "\phi"',"\cong"] & \widehat{\pi_1M_1} \arrow[r] \arrow[d, "\Phi"',"\cong"] & \widehat{\Z} \arrow[d, "\mu"',"\cong"] \arrow[r] & 1 \\
1 \arrow[r] & \widehat{\pi_1S_2} \arrow[r]                    & \widehat{\pi_1M_2} \arrow[r]                    & \widehat{\Z} \arrow[r]                   & 1
\end{tikzcd}
\end{equation}
where $\phi:\widehat{\pi_1S_1}\to \widehat{\pi_1S_2}$ is the restriction of $\Phi$. 

The isomorphism $\phi$ induces an isomorphism
$$
\phi_\ast: \tensor H_2(S_1;\Z)\cong \H_2(\widehat{\pi_1S_1};\widehat{\Z}) \tto \H_2(\widehat{\pi_1S_2};\widehat{\Z}) \cong \tensor H_2(S_2;\Z), 
$$
so there exists $\lambda_2\in \Zx$  such that $\phi_\ast(1\otimes [S_1])=\lambda_2\otimes [S_2]$. 
By construction, $g_1\in \pi_1S_1$ and $g_2\in \pi_1S_2$ are non-trivial elements  and $\Phi(g_1)=g_2^\mu$, so $\phi(g_1)=g_2^\mu$. Therefore,   \autoref{cor: surface coefficient} implies that  $\lambda_2=\pm \mu$. In addition, the following diagram commutes. 
\begin{equation*}
\begin{tikzcd}[column sep=0.3cm]
\tensor H_2(S_1;\Z) \arrow[r,symbol=\cong] & \H_2(\widehat{\pi_1S_1};\widehat{\Z}) \arrow[d, "\phi_\ast"'] \arrow[rrrr] &  &  &  & \H_2(\widehat{\pi_1M_1};\widehat{\Z}) \arrow[r,symbol=\cong] \arrow[d, "\Phi_\ast"] & \tensor H_2(M_1;\Z) \\
\tensor H_2(S_2;\Z) \arrow[r,symbol=\cong] & \H_2(\widehat{\pi_1S_2};\widehat{\Z}) \arrow[rrrr]                         &  &  &  & \H_2(\widehat{\pi_1M_2};\widehat{\Z}) \arrow[r,symbol=\cong]                        & \tensor H_2(M_2;\Z)
\end{tikzcd}
\end{equation*}
Hence, $\Phi_\ast$ also sends $1\otimes [S_1]$ to $\lambda_2\otimes[S_2]$. 

Now, we consider the cap products between them. 
For any $n\in \N$, combining \autoref{nat1} and \autoref{nat2}, we have the following commutative diagram. 
\begin{equation*}
\begin{tikzcd}[column sep=tiny]
{H_3( M_1;\Z/n)} \arrow[r, symbol=\times]                                                                \arrow[d, "\cong"']& {H^1( M_1;\Z/n)} \arrow[rrrr, "\cap"]                &  &  &  & {H_2( M_1;\Z/n)} \arrow[d, "\cong"]               \\
{\H_3(\widehat{\pi_1 M_1};\Z/n)} \arrow[r, symbol=\times]         \arrow[d, "\Phi_\ast"']                & {\H^1(\widehat{\pi_1 M_1};\Z/n)} \arrow[rrrr, "\cap"] \arrow[u, "\cong"]  &  &  &  & {\H_2(\widehat{\pi_1 M_1};\Z/n)} \arrow[d, "\Phi_\ast"] \\
{\H_3(\widehat{\pi_1 M_2};\Z/n)}\arrow[r, symbol=\times] & {\H^1(\widehat{\pi_1 M_2};\Z/n)} \arrow[u, "\Phi^\ast"]  \arrow[d, "\cong"'] \arrow[rrrr, "\cap"]                         &  &  &  & {\H_2(\widehat{\pi_1 M_2};\Z/n)}                        \\
{H_3( M_2;\Z/n)} \arrow[r, symbol=\times]                                                               \arrow[u, "\cong"]  & {H^1( M_2;\Z/n)} \arrow[rrrr, "\cap"]                &  &  &  & {H_2( M_2;\Z/n)} \arrow[u, "\cong"']             
\end{tikzcd}
\end{equation*}
Recall that $[M_i]_n\cap [\psi_{\mathrm{F},i}]_n=[S_i]_n$. Hence, 
\begin{equation*}
\begin{aligned}
\lambda_3[S_2]_n=&\lambda_3[M_2]_n\cap [\psi_{\mathrm{F},2}]_n =\Phi_\ast([M_1]_n)\cap  [\psi_{\mathrm{F},2}]_n = \Phi_\ast([M_1]_n\cap\Phi^\ast([\psi_{\mathrm{F},2}]_n ) )\\  =& \Phi_\ast([M_1]_n\cap \mu [\psi_{\mathrm{F},1}]_n )= \Phi_\ast (\mu[S_1]_n)= \lambda_2\mu[S_2]_n. 
\end{aligned}
\end{equation*}
Note that $[S_2]\in H_2(M_2;\Z)$ is a primitive class in the free abelian group $ H_2(M_2;\Z)$, so the universal coefficient theorem implies that $[S_2]_n\in H_2(M_2;\Z/n)$ has order $n$. As a consequence, $\lambda_3\equiv \lambda_2\mu\,(\mathrm{mod}\,{n})$ for every $n\in \N$. In other words, $\lambda_3=\lambda_2\mu $ in $\widehat{\Z}$. 


Recall that $\lambda_3=\pm \mu^3$ and $\lambda_2=\pm \mu$ in $\widehat{\Z}$, so $\mu^3=\pm \mu^2$ in $\widehat{\Z}$. Since $\mu\in \Zx$, we deduce that $\mu=\pm1$. Therefore, $\Phi$ is regular. 
\end{proof}

\begin{convention}\label{conv: crossfibering 2}
According to \autoref{lem: crossfibered regular}, if $M_1$ and $M_2$ are crossfibered manifolds and $\Phi: \widehat{\pi_1M_1}\to \widehat{\pi_1M_2}$ is an isomorphism, there is a preferred choice of its homological datum $(\mu,f)$, in which $\mu=1$ and $\Ab{\Phi}=\widehat{f}$. By this convention, we shall abbreviate a $(\Phi,1,f)$-corresponding pair of crossfibering data into a {\em $\Phi$-corresponding pair} of crossfibering data.   
\end{convention}

In order to address the problem of finite covers, we point out the following lemma. 
\begin{lemma}[{\cite[Sub-lemma 1]{Xu25}}]\label{virtually regular}
Suppose $\Gamma_1$ and $\Gamma_2$ are finitely generated groups. An isomorphism $\Phi:\widehat{\Gamma_1}\to \widehat{\Gamma_2}$ is regular if it is virtually regular. 
\end{lemma}

We can now prove the complete version of \autoref{thm: regular}. 

\begin{proof}[Proof of  \autoref{thm: regular}]
First, consider the case that $\Gamma_1$ is a torsion-free uniform lattice. Then $\Gamma_2$ is also a torsion-free uniform lattice by \autoref{prop: torsion free}. Combining  the virtual crossfibering theorem (\autoref{thm: virtual crossfibering} and \autoref{prop: virc2}) with  \autoref{lem: crossfibered regular}, we deduce that $\Phi$ is virtually regular. Hence, $\Phi$ is regular according to \autoref{virtually regular}.

Next, according to \cite[Corollary 8.8]{Xu25}, the case that $\Gamma_1$ (and hence $\Gamma_2$) is   a torsion-free non-uniform lattice follows from the torsion-free uniform case, using the technique  of  hyperbolic Dehn fillings. 
Finally, the case with torsions follows from the torsion-free case by  Selberg's lemma  and \autoref{virtually regular}. 
\end{proof}

\section{Automorphism of profinite surface group II: drilling}\label{sec: drilling}
\subsection{Drilling $\Sigma\times \mathbb{S}^1$}\label{subsection: drill}
In this section, we consider the following setting. For $i=1,2$, let $\Sigma_i$ be a  closed orientable surface, and let $\alpha_i,\alpha_i'$ be two non-separating simple closed curves on $\Sigma_i$.  We identify $\mathbb{S}^1$ with $[0,1]/0\sim 1$ as before. The manifold $\Sigma_i\times \mathbb{S}^1$ contains a  link $L_i=\alpha_i\times\{0\}\cup \alpha_i'\times\{\frac{1}{2}\}$,  and we denote by  $$M_i=\left(\Sigma_i\times \mathbb{S}^1\right)\setminus n(L_i)$$  the exterior of $L_i$. 

Denote $F_i=\Sigma_i\setminus n(\alpha_i)$ and $F_i'=\Sigma_i \setminus n(\alpha_i')$. Let $M_{i,u}=M_i\cap \Sigma_i\times [0,\frac{1}{2}]$ and $M_{i,v}=M_{i}\cap \Sigma_i\times [\frac{1}{2},1]$. Note that both $M_{i,u}$ and $M_{i,v}$ are homeomorphic to $\Sigma_i\times I$. 
The manifold $M_i$ is obtained by gluing $M_{i,u}$ and $M_{i,v}$ along $F_i\times\{0\}$ and $F_{i}'\times \{\frac{1}{2}\}$, which yields a graph-of-space structure on $M_i$, see \autoref{fig: HNN2}. From this, we can derive a group  presentation of $\pi_1 M_i $, which we clarify as follows. 

Fix a point $x_i\in \Sigma_i\setminus (n(\alpha_i)\cup n(\alpha_i'))$, and let $\overline{x}_i=(x_i,0)$ be the basepoint of $M_i$.  
Let $t_i=x_i\times \mathbb{S}^1$ be a loop based at $\overline{x}_i$, oriented along $\mathbb{S}^1$ that goes from $0$ to $1$, and $t_i$ is viewed as an element in $\pi_1(M_i,\overline{x}_i)$. 
Then, $\pi_1(M_i,\overline{x}_i)$ is generated by $\pi_1(M_{i,u},\overline{x}_i)$, $\pi_1(M_{i,v},\overline{x}_i)$,  and $t_i$. 
There are apparent isomorphisms $\pi_1(M_{i,u},\overline{x}_i)\cong \pi_1(\Sigma_i,x_i)$ and  $\pi_1(M_{i,v},\overline{x}_i)\cong \pi_1(\Sigma_i,x_i)$ induced by projecting onto the $\Sigma_i$ factor. We  identify $\pi_1(\Sigma_i,x_i)$ with $\pi_1(M_{i,u},\overline{x}_i) $, and we denote by $\pi_1(\Sigma_i,x_i)^\star$ a copy of $\pi_1(\Sigma_i,x_i)$ which we identify with  $\pi_1(M_{i,v},\overline{x}_i) $. 
The copy of an element $h\in \pi_1(\Sigma_i,x_i)$ in $\pi_1(\Sigma_i,x_i)^\star$ is denoted by $h^\star$. 
Similar to \autoref{sec: cut and glue}, we identify $\pi_1(F_i,x_i)$ and $\pi_1(F_i',x_i)$ as subgroups in $\pi_1(\Sigma_i,x_i)$. 
By van-Kampen's theorem, we derive the following presentation for $\pi_1(M_i,\overline{x}_i)$. 
\begin{equation}\label{presentation}
\pi_1(M_i,\overline{x}_i)=\Big \langle \pi_1(\Sigma_i,x_i),\pi_1(\Sigma_i,x_i)^\star , t_i \,\Big |\, \, {\scalebox{0.9}{$ \begin{gathered} g^\star=g \;\left(g\in \pi_1(F_i,x_i)\right),\\ t_ih^\star t^{-1}_i=h\;\left(h\in \pi_1(F_i',x_i)\right)\end{gathered}$}} \Big\rangle.
\end{equation}

\begin{figure}[ht!]
\centering
\includegraphics[width=12cm]{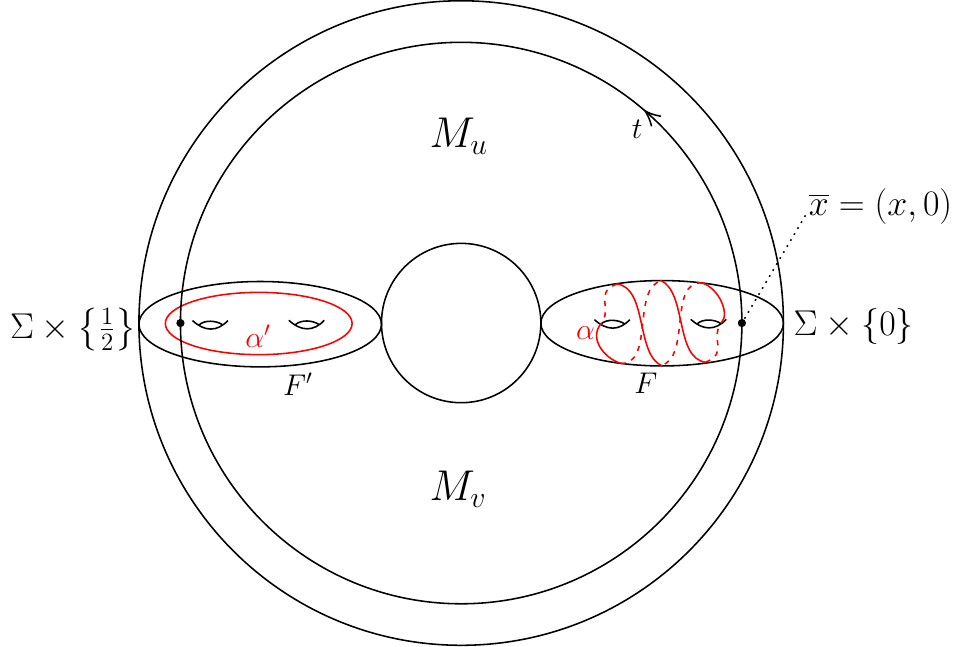}
\caption{The splitting of $\pi_1(M_i,\overline{x}_i)$ (subscript $i$ omitted)}
\label{fig: HNN2}
\end{figure}

The {\em meridians} of the link $L_i$ refer to the (free homotopy classes of) simple loops on $\partial M_i$ that bound a disk  in $\overline{n}(L_i) $. We will equip the meridians with  arbitrary orientations, and the orientations do not affect subsequent arguments.  
Now, we find conjugacy representatives in the presentation (\ref{presentation}) for the meridians of  $L_i$. Fix an oriented simple loop $\gamma_i$ (resp.\  $\gamma_i'$) on $\Sigma_i$ based at $x_i$ which intersects with $\alpha_i$ (resp.\  $\alpha_i'$) transversely at one point. We also view  $\gamma_i$  and  $\gamma_i'$  as  elements in $\pi_1(\Sigma_i,x_i)$. Then  from \autoref{fig: HNN2}, we can see that 
$$
m_i=\gamma_i^\star \gamma_i^{-1}
$$
is a conjugacy representative for 
the meridian of the $\alpha_i\times\{0\}$ component; and
$$
m_i'=t_i (\gamma_i')^\star t_i^{-1} (\gamma_i')^{-1}
$$
is a conjugacy representative for the meridian of the $\alpha_i'\times\{\frac{1}{2}\}$ component. 

The inclusion map $M_i\hookrightarrow \Sigma_i\times \mathbb{S}^1$ induces a surjective homomorphism $$q_i:  \pi_1(M_i,\overline{x}_i)\tto \pi_1(\Sigma_i \times \mathbb{S}^1,\overline{x}_i)=\pi_1(\Sigma_i,x_i)\times \langle t_i \rangle.$$ 
To be specific, $q_i$ sends the subgroups $\pi_1(\Sigma_i,x_i)$ and $\pi_1(\Sigma_i,x_i)^\star$ in $\pi_1(M_i,\overline{x_i})$ identically to the subgroup $\pi_1(\Sigma_i,x_i) $ in $\pi_1(\Sigma_i\times \mathbb{S}^1,\overline{x_i})$, and it sends $t_i$ to $t_i$.  According to van-Kampen's theorem, $\ker(q_i)= \langle \!\langle m_i,m_i'\rangle\!\rangle $. 

This can also be explained by the group presentation (\ref{presentation}). 
For brevity, we will omit the basepoints in the fundamental groups, but we remind the readers that our basepoints $x_i$ for $\Sigma_i$ and $F_i$, and  $\overline{x_i}$ for $M_i$ are fixed throughout this section. 
Indeed, $\pi_1 \Sigma_i $ is generated  by $\pi_1 F_i $ and $\gamma_i$,  and  $\pi_1 \Sigma_i $ is also  generated  by $\pi_1 F_i' $ and $\gamma_i'$. Thus, 
\begin{equation*}
\begin{aligned}
&\pi_1 M_i \big/\langle\!\langle m_i,m_i'\rangle \!\rangle \\
=&
\left \langle \pi_1 \Sigma_i ,{\pi_1 \Sigma_i} ^\star , t_i \left |\, {\scalebox{0.9}{$\begin{gathered} g^\star=g \;(g\in \pi_1 F_i) ,\;  \gamma_i^\star=\gamma_i, \; t_ih^\star t^{-1}_i=h\;(h\in \pi_1 F_i') ,\;t_i(\gamma_{i}')^\star t_i^{-1}=\gamma_i'\end{gathered}$}}  \right.  \right \rangle \\
= &  \left \langle \pi_1 \Sigma_i ,{\pi_1 \Sigma_i }^\star , t_i \left |\, {\scalebox{0.9}{$ \begin{gathered} g^\star=g \; (g\in \pi_1 \Sigma_i) ,\; t_ih^\star t^{-1}_i=h \;(h\in \pi_1 \Sigma_i )\end{gathered}$}}\right. \right\rangle  \\
 =  & \pi_1 \Sigma_i \times \langle t_i\rangle. 
\end{aligned}
\end{equation*}

The profinite completion of $q_i$  is denoted by  $$\widehat{q_i}: \widehat{\pi_1 M_i }\tto \widehat{\pi_1 \Sigma_i }\times \widehat{\langle t_i\rangle } \cong \widehat{\pi_1 \Sigma_i }\times \widehat{\Z}. $$
According to \autoref{prop: right exact}, $\ker(\widehat{q_i})=\overline{\langle \! \langle m_i,m_i'\rangle \! \rangle}$. 

Throughout this section, let $a_i,a_i'\in \pi_1 \Sigma_i $   be conjugacy representatives for the free homotopy classes of $\alpha_i$ and $\alpha_i'$, equipped with arbitrary orientations.

\subsection{The drilling theorem}
Under the  setting of \autoref{subsection: drill}, we now state the main theorem of this section. 
\begin{theorem}\label{thm: drilling}
Suppose $\phi:\widehat{\pi_1 \Sigma_1 }\to \widehat{\pi_1 \Sigma_2 }$ is  an isomorphism. 
Suppose there exist $\lambda,\mu\in \Zx$ such that $\phi(a_1)$  conjugates with  $a_2^\lambda$, and $\phi(a_1')$ conjugates with  $(a_2')^\mu$ in $\widehat{\pi_1 \Sigma_2 }$. Assume that the genus of $\Sigma_1$ and $\Sigma_2$ is  at least $2$. Then the following statements hold. 
\begin{enumerate}[label=(\arabic*),leftmargin=*]
\item\label{drl1} There exists an isomorphism $\Phi: \widehat{\pi_1 M_1 }\to \widehat{\pi_1  M_2 }$.  
\item\label{drl3} The isomorphism $\Phi$ fits into the following commutative diagram:   
\begin{equation*}
\begin{tikzcd}[column sep=large]
{\widehat{\pi_1 M_1 }} \arrow[d, "\widehat{q_1}"', two heads] \arrow[r, "\Phi","\cong"'] & {\widehat{\pi_1 M_2 }} \arrow[d, "\widehat{q_2}", two heads] \\
{\widehat{\pi_1 \Sigma_1 }\times\widehat{\Z}} \arrow[r, "\phi\times I","\cong"']          & {\widehat{\pi_1 \Sigma_2 }\times\widehat{\Z}}                          
\end{tikzcd}
\end{equation*}
where $I:\widehat{\Z}\to \widehat{\Z}$ denotes  the isomorphism that sends  $t_1$ to $t_2$. 
\item\label{drl2}   
$\Phi(m_1)$  conjugates with  $m_2 ^{\pm1 } $, and $\Phi(m_1')$  conjugates with   $(m_2')^{\pm 1}$ in $\widehat{\pi_1 M_2 }$.
\end{enumerate}
\end{theorem}
\begin{remark}
In fact,   \autoref{prop: surface homology coefficient} implies that  $\lambda=\pm \mu$, but our proof for \autoref{thm: drilling} does not rely on this fact. 
\end{remark}


We start with some preparations for  \autoref{thm: drilling}. 
The group presentation (\ref{presentation}) identifies $\pi_1 M_i $ as the fundamental group of a graph of groups, which we explain as follows. Let $X$ be the finite oriented graph with two vertices $u,v$ and two edges $e,e'$, where $d^+(e)=d^-(e')=u$ and $d^-(e)=d^+(e')=v$, see \autoref{fig: graphX}. For $i=1,2$, let $(G_i,X)$ be a graph of groups over $X$ such that 
\begin{equation*}
 G_{i,u}=\pi_1 \Sigma_i  , \; G_{i,v}={\pi_1 \Sigma_i} ^\star ,  \;  G_{i,e}=\pi_1 F_i  , \; G_{i,e'}=\pi_1 F_i' ,
\end{equation*} 
and that 
\begin{equation*}
\begin{gathered}
\varphi_e^{+} :\pi_1 F_i \to \pi_1 \Sigma_i , \; \varphi_e^{-} :\pi_1 F_i \to {\pi_1 \Sigma_i} ^\star,\;
\varphi_{e'}^{+} :\pi_1 F_i' \to {\pi_1 \Sigma_i }^\star ,\; \varphi_{e'}^{-} :\pi_1 F_i' \to \pi_1 \Sigma_i  
\end{gathered}
\end{equation*}
are the inclusion maps. If we take $T=\{u,v,e\}$ as the maximal subtree of $X$, then \autoref{DEF: Fundamental group of graph of group} yields an  identification $\pi_1 M_i =\pi_1(G_i,X,T)$.

\begin{figure}[ht!]
\centering
\includegraphics[width=4cm]{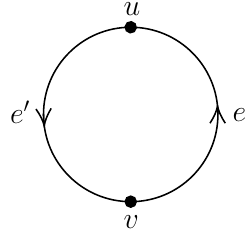}
\caption{The graph $X$ for the splitting of $\pi_1(M_i,\overline{x}_i)$}
\label{fig: graphX}
\end{figure}

\begin{lemma}\label{lem: adequate2}
The graph of group $(G_i,X)$ described above is adequate. 
\end{lemma}
\begin{proof}
The fact that $\pi_1(G_i,X)\cong \pi_1 M_i $ is residually finite is proven by \cite{Hem16}. In fact, all finitely generated 3-manifold groups are residually finite. In addition,  the profinite topology of $\pi_1(G_i,X)$ induces the full profinite topology on $G_{i,u}$ and $G_{i,v}$, since these subgroups are retracts of $ \pi_1 M_i  $ through  projecting onto the $\Sigma_i$ factor. Finally, the profinite topology of $G_{i,u}$ (resp.\  $G_{i,v}$) induces the full profinite topology on $G_{i,e'}$ (resp.\  $G_{i,e}$) via $\varphi^{-}_{e'}$ (resp.\  $\varphi^{-}_{e}$) because $G_{i,u}$ (resp.\  $G_{i,v}$) is LERF according to  \autoref{FK: LERF}. Thus, $(G_i,X)$ is adequate.  
\end{proof}

\begin{sloppypar}
Therefore, according to \autoref{THM: Efficient completion}, we can identify $\widehat{\pi_1 M_i}$ with $\Pi_1(\widehat{G_i},X,T)$. 
Following the notations in (\ref{presentation}), we use ${\widehat{\pi_1 \Sigma_i }}{}^\star $ to denote a copy of $ {\widehat{\pi_1 \Sigma_i }}$; and for any element   $h\in \widehat{\pi_1 \Sigma_i }$, $h^\star$ denotes its copy in  ${\widehat{\pi_1 \Sigma_i }}{}^\star $. 
Then, by \autoref{DEF: Profinite fundamental group of graph of profinite group}, 
\begin{equation}\label{profinite presentation}
\widehat{\pi_1 M_i }= \left(\widehat{\pi_1 \Sigma_i }\amalg {\widehat{\pi_1 \Sigma_i }}{}^\star \amalg \widehat{\Z}\right)/\mathcal{N}_i\,,
\end{equation}
where  $\widehat{\Z}$ is generated by $t_i$, and $\mathcal{N}_i$ is the closed normal subgroup generated by $\{g^{-1}g^\star \mid  g\in \widehat{\pi_1 F_i } =\overline{\pi_1F_i}\}\cup \{h^{-1} t_ih^{\star} t_i^{-1} \mid h\in \widehat{\pi_1 F_i' }=\overline{\pi_1F_i'}\}$. 
\end{sloppypar}

\def\gesi{\gamma_2^{\prime\, \sigma}}
\begin{proof}[Proof of \autoref{thm: drilling}]
We emphasize that all fundamental groups we refer to in this proof are based on the presentations given in (\ref{presentation}) and (\ref{profinite presentation}).  
According to \autoref{thm: cut and glue}, there exist $\hat{g},\hat{g}'\in \widehat{\pi_1\Sigma_2}$, $\hat{\delta}\in \widehat{\pi_1F_2}$,  $\hat{\delta}' \in \widehat{\pi_1F_2'}$, and $\epsilon,\sigma \in \{\pm 1\}$ such that 
\begin{equation*}
\phi(\widehat{\pi_1F_1})= \hat{g} \cdot  \widehat{\pi_1F_2} \cdot \hat{g}^{-1},\;\phi(\widehat{\pi_1F_1'})= \hat{g}'\cdot \widehat{\pi_1F_2'} \cdot (\hat{g}')^{-1}, 
\end{equation*}
and that
\begin{equation*}
\phi(\gamma_1)=\hat{g} \gamma_2^{\epsilon} \hat{\delta} \hat{g}^{-1}, \; \phi(\gamma_1')=\hat{g}' \gesi \hat{\delta}' (\hat{g}')^{-1}.
\end{equation*}
Define $\phi_0=\Inn_{\hat{g}}^{-1}\circ \phi: \widehat{\pi_1\Sigma_1}\to  {\widehat{\pi_1\Sigma_2}}$, and let $\tilde{g}=\hat{g}^{-1}\hat{g}'$. Then, 
$$
\phi_0(\widehat{\pi_1F_1})=  \widehat{\pi_1F_2}  ,\;\phi_0(\widehat{\pi_1F_1'})= \tilde{g} \cdot \widehat{\pi_1F_2'} \cdot \tilde{g}^{-1},
$$
and
$$
\phi_0(\gamma_1)=  \gamma_2^{\epsilon} \hat{\delta}  , \; \phi_0(\gamma_1')=\tilde{g} \gesi \hat{\delta}' \tilde{g}^{-1} .
$$
Let $\phi_0^\star: \widehat{\pi_1\Sigma_1}{}^\star\to \widehat{\pi_1\Sigma_2}{}^\star$ be the copy of $\phi_0$.  

Using free profinite products, we can construct a homomorphism $$\widetilde{\Phi_0}: \widehat{\pi_1\Sigma_1} \amalg \widehat{\pi_1\Sigma_1}{}^\star \amalg \widehat{\Z} \to \widehat{\pi_1\Sigma_2} \amalg \widehat{\pi_1\Sigma_2}{}^\star \amalg \widehat{\Z} $$ by setting 
$ 
\widetilde{\Phi_0}(h)= \phi_0(h)  
$ and $\widetilde{\Phi_0}(h^\star)=  \phi_0^\star(h^\star)$ for $h\in \widehat{\pi_1\Sigma_1}$, and $\widetilde{\Phi_0}(t_1)=\tilde{g} t_2 (\tilde{g}^\star)^{-1}$. It can be easily verified,  through constructing its inverse,  that $\widetilde{\Phi_0}$ is an isomorphism. In addition, $\widetilde{\Phi_0}(\mathcal{N}_1)=\mathcal{N}_2$ since 
$$
\widetilde{\Phi_0} (\{g^{-1}g^\star \mid  g\in \widehat{\pi_1 F_1 }\})= \{\mathpzc{g}^{-1}\mathpzc{g}^\star \mid  \mathpzc{g}\in \widehat{\pi_1 F_2 }\}
$$
and
$$
\widetilde{\Phi_0} (\{h^{-1} t_1h^{\star} t_1^{-1} \mid h\in \widehat{\pi_1 F_1' } \} ) = \tilde{g} \cdot  \{\mathpzc{h}^{-1} t_2\mathpzc{h}^{\star} t_2^{-1} \mid \mathpzc{h}\in \widehat{\pi_1 F_2' } \} \cdot  \tilde{g}^{-1}.
$$
Hence, $\widetilde{\Phi_0}$ descends to an isomorphism
$$
\Phi_0: \widehat{\pi_1M_1} = (\widehat{\pi_1\Sigma_1} \amalg \widehat{\pi_1\Sigma_1}{}^\star \amalg \widehat{\Z})/\mathcal{N}_1\to (\widehat{\pi_1\Sigma_2} \amalg \widehat{\pi_1\Sigma_2}{}^\star \amalg \widehat{\Z})/\mathcal{N}_2 = \widehat{\pi_1M_2}.
$$

By construction of $\Phi_0$ and $q_i$,  the following diagram commutes
\begin{equation*}
\begin{tikzcd}[column sep=large]
{\widehat{\pi_1M_1}} \arrow[d, "\widehat{q_1}"', two heads] \arrow[r, "\Phi_0" ] & {\widehat{\pi_1M_2}} \arrow[d, "\widehat{q_2}", two heads] \\
{\widehat{\pi_1\Sigma_1}\times\widehat{\Z}} \arrow[r, "\phi_0\times I" ]          & {\widehat{\pi_1\Sigma_2}\times\widehat{\Z}}                          
\end{tikzcd}
\end{equation*}
where $I:\widehat{\Z}\to \widehat{\Z}$   sends  $t_1$ to $t_2=\tilde{g}t_2\tilde{g}^{-1}$. 
We alternatively consider the isomorphism $$\Phi= \Inn_{\hat{g}} \circ \Phi_0: \widehat{\pi_1M_1}\tto \widehat{\pi_1M_2} .$$
 Note that $\widehat{q_2}$ maps $\hat{g}\in   \widehat{\pi_1M_2}$ to $\hat{g} \in   \widehat{\pi_1\Sigma_2} \times \widehat{\Z}$, and  that $\phi=\Inn_{\hat{g}}\circ \phi_0$. Hence, $\Phi$ fits into the following commutative diagram.  
\begin{equation*}
\begin{tikzcd}[column sep=large]
{\widehat{\pi_1M_1}} \arrow[d, "\widehat{q_1}"', two heads] \arrow[r, "\Phi " ] & {\widehat{\pi_1M_2}} \arrow[d, "\widehat{q_2}", two heads] \\
{\widehat{\pi_1\Sigma_1}\times\widehat{\Z}} \arrow[r, "\phi \times I" ]          & {\widehat{\pi_1\Sigma_2}\times\widehat{\Z}}                          
\end{tikzcd}
\end{equation*}
In other words,   property \ref{drl3} is  satisfied.

To prove property \ref{drl2}, it suffices to  compute $\Phi_0(m_1)$ and $\Phi_0(m_1')$. 
First, note that $\hat{\delta}^\star = \hat{\delta}$ in $\widehat{\pi_1M_2}$  since $\hat{\delta}\in \widehat{\pi_1F_2}$. Thus, 
$$
\Phi_0(m_1)=\phi_0(\gamma_1)^\star \phi_0(\gamma_1)^{-1}=(\gamma_2^{ \epsilon})^{\star}  \hat{\delta}^\star  \hat{\delta}^{-1} \gamma_2^{-\epsilon}= (\gamma_2^\epsilon)^\star \gamma_2^{-\epsilon}. 
$$
In particular, $\Phi_0(m_1)$  equals $m_2=\gamma_2^\star \gamma_2^{-1}$ when $\epsilon=1$, or   conjugates with  $m_2^{-1} = \gamma_2 (\gamma_2^{-1})^\star$ when $\epsilon=-1$. Consequently, $\Phi(m_1)= \hat{g} \Phi_0(m_1) \hat{g}^{-1}$ is also conjugate to $m_2^{\pm1}$. 

Next, note that $t_2  \hat{\delta}^{\prime \star} t_2^{-1} = \hat{\delta}'$ in $\widehat{\pi_1M_2}$ since $\hat{\delta}'\in \widehat{\pi_1F_2'}$. Hence, 
\begin{equation*}
\begin{aligned}
\Phi_0(m_1')  
= & \Phi_0(t_1) \,\phi_0(\gamma_1')^\star \, \Phi_0(t_1)^{-1} \,\phi_0(\gamma_1')^{-1}\\
= & \tilde{g} \, t_2 (\tilde{g}^\star)^{-1}\, \tilde{g}^\star \, (\gesi)^\star \, \hat{\delta}^{\prime \star} \, (\tilde{g}^\star)^{-1}\,  \tilde{g}^\star \,  t_2^{-1} \, \tilde{g}^{-1} \, \tilde{g}\,  (\hat{\delta}')^{-1} \,(\gamma_2')^{-\sigma} \,\tilde{g}^{-1} \\ 
= & \tilde{g}\, t_2 \,(\gesi)^\star \,\hat{\delta}^{\prime \star}\, t_2^{-1}\, (\hat{\delta}')^{-1}\,(\gamma_2')^{-\sigma}\, \tilde{g}^{-1}  \\
= & \tilde{g} \cdot t_2  (\gesi)^\star t_2^{-1} (\gamma_2')^{-\sigma} \cdot \tilde{g}^{-1}.
\end{aligned}
\end{equation*}
Therefore, $\Phi_0(m_1')$  is conjugate to  $m_2'=t_2 (\gamma_2')^\star t_2^{-1} (\gamma_2')^{-1}$ when $\sigma=1$, or  is  conjugate to  $(m_2')^{-1}= \gamma_2'  t_2 (\gamma_2^{\prime \, -1})^{\star} t_2^{-1}  $ when $\sigma=-1$. 
As a consequence, $\Phi(m_1')= \hat{g} \Phi_0(m_1') \hat{g}^{-1}$ is also conjugate to $(m_2')^{\pm1}$, and property \ref{drl2} also holds. 
\end{proof}

\begin{sloppypar}
\begin{remark}
\begin{enumerate}[leftmargin=*, label=(\arabic*)]
\item In fact, $(\phi,\phi^\star,\psi,\psi')$ yields a congruent isomorphism between $(\widehat{G_1},X)$ and $(\widehat{G_2},X)$, from which there is a standard procedure that produces the isomorphism $\Phi_0:\Pi_1(\widehat{G_1},X)\to \Pi_1(\widehat{G_2},X)$, see \cite[Proposition 3.11]{Xu25A}. We display the concrete construction for $\Phi_0$ here for the  preciseness of  property \ref{drl2} in the theorem. 
\item The proof of \autoref{thm: drilling} also works for drilling separating simple closed curves, as well as for drilling any number of curves lying on distinct levels of surfaces in $\Sigma_i\times \mathbb{S}^1$. 
\end{enumerate}
\end{remark}
\end{sloppypar}

\subsection{Hyperbolicity of the drilled manifold}
We follow the same setting of \autoref{subsection: drill}. 
\begin{theorem}\label{thm: drilled hyperbolic}
Suppose that the genus of $\Sigma_i$ is at least $2$, and that $\alpha_i\cup \alpha_i'$ fills the surface $\Sigma_i$. Then, the interior of $M_i$ admits a complete finite-volume hyperbolic structure. 
\end{theorem}
\begin{proof}
For brevity, we omit the subscript $i$ in the proof. Note that the manifold $M$ is irreducible, since the ambient manifold $\Sigma\times\mathbb{S}^1$ is irreducible, and both components of the link $L$ are homotopically   non-trivial. This also implies that   $M$ is  $\partial$-irreducible, since $\partial M$ consists of two tori. Thus, according to Thurston's theorem \cite{Thu82,ThuI} (see also \cite{Mor84}), $M$ admits a geometric decomposition. Note that $M$ is not a Seifert fibered space, since it contains a closed incompressible surface $\Sigma\times \{\frac{1}{4}\}$ of genus at least $2$ while it has non-empty boundary. Therefore, $\operatorname{int}(M)$ admits a complete finite-volume hyperbolic structure if and only if $M$ is geometrically atoroidal, i.e.\ any embedded incompressible torus in $M$ is isotopic to $\partial M$. 

To show that $M$ is geometrically atoroidal, suppose $T\subseteq M$ is an embedded incompressible torus. For brevity of notation, we abbreviate the original notation $F\times\{0\}$ into $F$, and also $F'\times\{\frac{1}{2}\}$ into $F'$, which should not cause a confusion. By a minor isotopy, we may assume that $T$ is transversal to $F\cup F'$.  Then, $T\cap (F\cup F')$ is a finite collection of simple closed curves. Varying $T$ within its isotopy class, we may assume that $T\cap (F\cup F')$ has the minimal number of components. 

A standard argument shows that each component of $T\cap (F\cup F')$ is homotopically non-trivial in $M$. In fact, if such a component $c$ is null-homotopic, then it is null-homotopic when projected onto the $\Sigma$ factor, and hence it bounds a disk in $F\cup F'$. We may choose $c$ to be an inner-most component so that the disk $D$ bounded by $c$ in $F\cup F'$ contains no other components in $T\cap (F\cup F')$. Note that $T$ is incompressible, so $c$ is also null-homotopic in $T$, and it also bounds a disk $D'$ in $T$. $D\cup_c D'$ yields an embedded $2$-sphere in $M$, so it bounds a $3$-ball $B$ in $M$ since $M$ is irreducible. By construction, $T$ does not intersect $\operatorname{int}(B)$. 
Then, one can push $D'$ along $\operatorname{int}(B)$ towards $D$, and then slightly push it away from $F\cup F'$. This yields an isotopy of $T$ in $M$ such that $T\cap (F\cup F')$ has strictly fewer components, namely the component $c$ is removed. This contradicts with the minimality of $\#(T\cap (F\cup F'))$. 

In particular, each component of $T\cap (F\cup F')$ is  homotopically non-trivial in $T$. Consequently, $T\cap (F\cup F')$ is a collection of parallel essential simple closed curves on $T$, which are pairwise homotopic. Projecting these components onto the $\Sigma$ factor of $ \Sigma\times \mathbb{S}^1$ yields a collection of pairwise homotopic essential simple closed curves on $\Sigma$, whose homotopy class is denoted by $\beta$. Since $\alpha\cup \alpha'$ fills $\Sigma$, $\beta$ has non-zero geometric intersection number with at least one of $\alpha$ and $\alpha'$. Without loss of generality, assume that $\beta$ has non-zero geometric intersection number  with $\alpha$. By construction, every component of $T\cap F$ is disjoint with $\alpha$, yet belongs to the homotopy class of $\beta$ in $\Sigma$. This is impossible unless $T\cap F=\varnothing$. 

Consequently, $T\cap F=\varnothing$, and there exists $0<\epsilon<\!<\frac{1}{2}$ such that $T\subseteq M\cap (\Sigma \times[\epsilon,1-\epsilon])$. Denote $N=M\cap (\Sigma \times[\epsilon,1-\epsilon])$, 
and we can identify $N$ with $\Sigma \times I\setminus n(\alpha'\times\{\frac{1}{2}\})$ by an apparent homeomorphism. Then, $T$ is an incompressible torus in $N$. 

$N$ has a clear JSJ-decomposition, which we illustrate as follows. 
Pick a tubular neighbourhood $\tilde{n}(\alpha')$ of $\alpha'$ on $\Sigma$ which is slightly larger than the drilled one, and denote the two boundary curves of  $ \overline{\tilde{n}}(\alpha')$ as $c_1$ and $c_2$. Let $F''= \Sigma \setminus  \tilde{n} (\alpha')$, which is homeomorphic to $F'$, and let $P=(c_1\cup c_2)\times I$. Then, $N$ can be cut along $P$ to obtain two pieces: $X=F''\times I$ and $Y=\overline{\tilde{n}}(\alpha')\times I \setminus n(\alpha'\times\{\frac{1}{2}\})$, see \autoref{fig: JSJ-N}. In fact,  $Y$ is  homeomorphic to $\mathbb{T}^2\times I$. 

\begin{figure}[ht!]
\centering
\includegraphics[width=9.5cm]{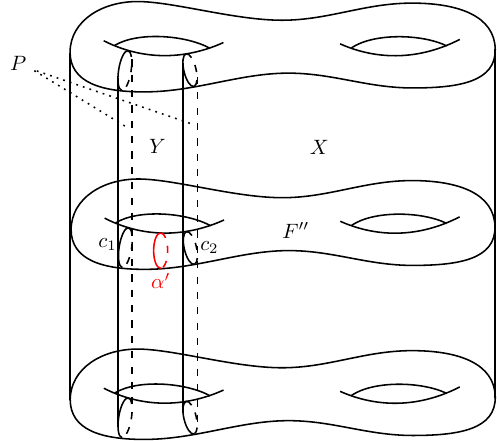}
\caption{The JSJ-decomposition of $N$}
\label{fig: JSJ-N}
\end{figure}

We may assume that $T$ is transversal to $P$.  By varying $T$ within its isotopy class, we assume that $T\cap P$ has the minimal number of components. Since both $T$ and $P$ are $\pi_1$-injective in $N$, using the irreducibility of $N$ and the same argument as above, we can also deduce that each component of $T\cap P$ is homotopically non-trivial.

We claim that $T\cap P=\varnothing$. 
Once $T\cap P\neq \varnothing$,  $T\cap P$ is a collection of parallel essential simple closed curves on $T$, which decomposes $T$ into essential annuli $A_1,\cdots,A_{2n}$ properly embedded in $(X,P)$ and $(Y,P)$ alternately. However,   every $\pi_1$-injective properly embedded annulus
$$
(A,\partial A)=(\mathbb{S}^1\times I ,\mathbb{S}^1\times \partial I ) \lhook\joinrel\longrightarrow (X,P) = (F''\times I,(\partial F'') \times I)
$$
is parallel (relative $\partial A$) to $P$. This is because   $F''$ has positive genus, which implies that $\partial A$  consists of two parallel curves on one component of $(\partial F'') \times I$. 
Thus, we can pick an innermost annulus $A_i$ in $X$ (i.e.\ $T$ does not intersect the interior of the annulus bounded by $\partial A_i$ in $P$), push $A_i$ towards $P$ and then slightly off $P$ via an isotopy, thereby reducing $\#(T\cap P)$ by $2$. 
Hence, $T\cap P$ is in fact empty. 

Therefore, either $T\subseteq X$ or $T\subseteq Y$. However, $X=F''\times I$ does not contain   incompressible tori since $\pi_1(X)$ is a free group. Hence, $T\subseteq Y$. Recall that $Y\cong \mathbb{T}^2\times I$, where one component of $\partial Y$ is given by $\partial \overline{n}(\alpha'\times\{\frac{1}{2}\})$. Any embedded incompressible torus in $ \mathbb{T}^2\times I$ is parallel to both boundary components. Hence, $T$ is isotopic within $Y$ to $\partial \overline{n}(\alpha'\times\{\frac{1}{2}\})$, which is the component of $\partial N$ that belongs to $\partial M$.  To sum up, we have proven that $T$ is isotopic to $\partial M$  in $M$, which completes the proof. 
\end{proof}

\begin{remark}
Under the setting of \autoref{thm: drilling}, if $\alpha_i\cup \alpha_i'$ fills $\Sigma_i$, then \autoref{thm: drilled hyperbolic} combined with \autoref{thm: regular} actually implies that $\mu,\lambda \in \{\pm1\}$. More generally, if $\alpha_i$ and $\alpha_i'$ have non-zero geometric intersection number, then one may follow the  proof of  \autoref{thm: drilled hyperbolic} to show that $\partial M_i$ belongs to a hyperbolic JSJ-piece of $M_i$. In this case,  \cite[Theorem 9.2]{Xu25} also implies that $\mu,\lambda\in \{\pm 1\}$. However, when $\alpha_i$ and $\alpha_i'$ are disjoint, $\mu$ (and accordingly, $\lambda=\pm \mu$) may range through any element in $\Zx$. Examples can be constructed from  the exotic automorphisms of $\widehat{F_2}$ introduced in \cite[Proposition 1.6]{Iha94}. 
\end{remark}

\subsection{Hyperbolic Dehn surgery theory}\label{subsec: hyperbolic Dehn surgery}

An application of \autoref{thm: drilled hyperbolic} is  to invoke Thurston's hyperbolic Dehn surgery theory to $M_i$. We will be utilizing the orbifold version of this theory, of which we give a brief introduction.

Let $M$ be a compact orientable 3-manifold with toral boundary, whose boundary components are denoted by $\partial _1 M ,\cdots, \partial _k M$. For any non-zero homology classes $d_i \in H_1(\partial_i M;\Z)$, we can construct a closed 3-orbifold $M_{d_1,\cdots, d_k}$, called the {\em Dehn filling} of $M$ along $d_1,\cdots, d_k$, via the following procedure. 

Let $n_i\in \N$ such that $d_i=n_ic_i$, where $c_i\in H_1(\partial_iM;\Z)$ is a primitive class. The underlying space of $M_{d_1,\cdots, d_k}$ is a closed 3-manifold obtained from $M$ by gluing $k$ solid tori $V_1,\cdots, V_k$ to $\partial M$, such that the meridian  of $V_i$ is glued to the slope corresponding to $c_i$ on $\partial_iM$. 
The singluar locus of $M_{d_1,\cdots, d_k}$ is the union of the core curves of $V_i$ where $n_i>1$, and  the orbifold structure of $M_{d_1,\cdots, d_k}$  is defined by assigning cone angle $\frac{2\pi}{n_i}$ to the core curve of each $V_i$. 

In general, Dehn fillings could provide compact orbifolds with toral boundary if we leave some of the boundary components unfilled. For the purpose of this paper, we only focus on closed orbifolds.  We remark that $M_{d_1,\cdots, d_k}$ is a manifold if and only if each $d_i$ is a primitive homology class, in which case $n_i=1$. 

The orbifold fundamental group $\piorb M_{d_1,\cdots, d_k} $ can be computed using the Seifert--van-Kampen theorem for orbifolds.
\begin{lemma}[{\cite[Theorem 4.7.1]{Choi}}]\label{orbifold fundamental group}
For each $1\le i \le k$,  let $\gamma_i\in \pi_1 M $ be a conjugacy representative for the closed curve on $\partial_iM$ representing the homology class of $d_i\in H_1(\partial_iM;\Z)$. Then,
\begin{equation*}
\piorb M_{d_1,\cdots, d_k} \cong \pi_1 M /\langle \! \langle \gamma_1 ,\cdots, \gamma_m \rangle \! \rangle. 
\end{equation*}
\end{lemma}

The hyperbolic Dehn surgery theorem that will be needed in this paper can be summarized as follows.
\begin{theorem}[{\cite[Theorem 5.8.2]{Thurston}}]\label{thm: hyperbolic dehn surgery}
Suppose $\operatorname{int}(M)$ admits a complete  hyperbolic structure.
 Then, there exist finite subsets $Z_i \in H_1(\partial_iM;\Z)$ for $1\le i\le k$,  such that for any non-zero homology classes $d_i\in H_1(\partial_iM;\Z)\setminus Z_i$, the orbifold $M_{d_1,\cdots, d_k}$ admits a complete hyperbolic structure. 
\end{theorem}

The hyperbolic structure on $ M_{d_1,\cdots, d_k} $ carries the cone angles on its singular locus. In fact, this hyperbolic structure is obtained by deforming the hyperbolic structure on  $\operatorname{int}(M)$.
Using the construction of developing maps, Thurston proved that hyperbolic orbifolds are {\em good}. We state the following proposition for complete hyperbolic orbifolds.
\begin{proposition}[{\cite[Proposition 13.3.2]{Thurston}}]\label{prop: orbifold good}
Let $O$ be an orientable 3-orbifold without boundary that admits a complete hyperbolic structure. Then, there exists a discrete group of orientation-preserving isometries $\Gamma \le \mathrm{Isom}^+(\mathbb{H}^3)$ such that $O=\mathbb{H}^3/\Gamma$. In particular, $\Gamma$ is the orbifold fundamental group $\piorb (O)$. 
When $O$ has finite hyperbolic volume, $\Gamma$ is a lattice in $\PSL_2(\C)$. Furthermore,  when $O$ is closed, $\Gamma$ is a uniform lattice in $\PSL_2(\C)$. 
\end{proposition}

In this paper, we will be focusing on closed hyperbolic Dehn fillings of $M$. In these cases, $M_{d_1,\cdots,d_k}$ is a closed hyperbolic 3-orbifold, and $\piorb M_{d_1,\cdots,d_k} $ is a uniform lattice in $\PSL_2(\C)$. 

Now, let us return to the setting of \autoref{subsection: drill}. In the following proposition, we also omit the basepoints in the fundamental groups for brevity.

\begin{proposition}\label{prop: drilling lattice surjection}
Following the   setting of \autoref{subsection: drill}, 
suppose that for $i=1,2$, the genus of $\Sigma_i$ is   at least $2$ and  $\alpha_i\cup \alpha_i'$ fills $\Sigma_i$. Suppose $\phi:\widehat{\pi_1\Sigma_1}\to \widehat{\pi_1\Sigma_2}$ is an isomorphism such that $\phi(a_1)$ conjugates with $a_2$ and $\phi(a_1')$ conjugates with $a_2'$ in $\widehat{\pi_1\Sigma_2}$. 
Then, there exist uniform lattices $\Gamma_1$ and $\Gamma_2$ in $\PSL_2(\C)$, with surjective homomorphisms $p_1: \Gamma_1\twoheadrightarrow \pi_1\Sigma_1$ and $p_2: \Gamma_2\twoheadrightarrow \pi_1\Sigma_2$, and an isomorphism $\widetilde{\phi}: \widehat{\Gamma_1}\to \widehat{\Gamma_2}$, such that the following diagram commutes. 
\begin{equation*}
\begin{tikzcd}
\widehat{\Gamma_1} \arrow[r,"\widetilde{\phi}","\cong"'] \arrow[d,"\widehat{p_1}"', two heads] & \widehat{\Gamma_2} \arrow[d,"\widehat{p_2}", two heads]\\
\widehat{\pi_1\Sigma_1} \arrow[r,"\phi","\cong"'] & \widehat{\pi_1\Sigma_2}
\end{tikzcd}
\end{equation*} 
\end{proposition}
\begin{proof}
We denote by $[m_i]$ and $[m_i']$ the homology classes of the meridians of $L_i$ in $H_1(\partial M_i;\Z)$. According to \autoref{thm: drilled hyperbolic}, $\operatorname{int}(M_1)$ and $\operatorname{int}(M_2)$ admit complete hyperbolic structures. Then, according to \autoref{thm: hyperbolic dehn surgery}, there exists a sufficiently large $n\in \N$ such that the orbifolds
\begin{equation*}
O_1=(M_1)_{n[m_1],n[m_1']}\quad \text{and}\quad O_2=(M_2)_{n[m_2],n[m_2']}
\end{equation*}
admit complete hyperbolic structures. By \autoref{orbifold fundamental group} and \autoref{prop: orbifold good}, 
\begin{equation*}
\Gamma_1:=\piorb O_1 = \pi_1M_1/\langle\!\langle m_1^n,m_1^{\prime n}\rangle\!\rangle \;\; \text{and}\;\; \Gamma_2:=\piorb O_2 = \pi_1M_2/\langle\!\langle m_2^n,m_2^{\prime n}\rangle\!\rangle 
\end{equation*}
are uniform lattices in $\PSL_2(\C)$. 
We denote $r_i: \pi_1M_i\to \Gamma_i$ as the quotient maps. 

Recall that we have defined surjective homomorphisms $q_i:\pi_1M_i\to \pi_1(\Sigma_i\times \mathbb{S}^1) =\pi_1\Sigma_i\times \Z$, where $\ker(q_i)=\langle \! \langle m_i,m_i\rangle\!\rangle$. Since $\langle \! \langle m_i,m_i\rangle\!\rangle\supseteq \langle\!\langle m_i^n,m_i^{\prime n}\rangle\!\rangle$,  $q_i$   factors through $\Gamma_i$. In other words, there are  unique surjective  homomorphisms $f_i:  \Gamma_i \to \pi_1\Sigma_i\times \Z$ such that $q_i=f_i\circ r_i$. Then, we define $p_i: \Gamma_i \to \pi_1\Sigma_i$ to be the projection of $f_i$ onto the first factor, which is also surjective.  

We will now construct the isomorphism $\widetilde{\phi}: \widehat{\Gamma_1}\to \widehat{\Gamma_2}$. 

With the given isomorphism $\phi:\widehat{\pi_1\Sigma_1}\to \widehat{\pi_1\Sigma_2}$, let $\Phi: \widehat{\pi_1M_1}\to \widehat{\pi_1M_2}$ be the isomorphism constructed by \autoref{thm: drilling}~\ref{drl1}. Then,   \autoref{thm: drilling}~\ref{drl2} guarantees that  $\Phi(\overline{\langle\!\langle m_1^n,m_1^{\prime n}\rangle\!\rangle })= \overline{\langle\!\langle m_2^n,m_2^{\prime n}\rangle\!\rangle }$. 
By  \autoref{prop: right exact}, 
$\widehat{r_i}$ is surjective with $\ker(\widehat{r_i})=\overline{\langle\!\langle m_i^n,m_i^{\prime n}\rangle\!\rangle }$. 
Hence, $\Phi$ descends to an isomorphism $\widetilde{\phi}: \widehat{\Gamma_1}\to \widehat{\Gamma_2}$  such that the following diagram commutes. 
\begin{equation*}
\begin{tikzcd}
\widehat{\pi_1M_1} \arrow[r,"\Phi","\cong"'] \arrow[d,two heads, "\widehat{r_1}"'] & \widehat{\pi_1M_2} \arrow[d,two heads, "\widehat{r_2}"]\\
\widehat{\Gamma_1} \arrow[r,"\widetilde{\phi}","\cong"'] & \widehat{\Gamma_2}
\end{tikzcd}
\end{equation*}

Combining with \autoref{thm: drilling}~\ref{drl3}, we deduce the following commutative diagram. 
\begin{equation*}
\begin{tikzcd}[row sep=small,column sep=small]
\widehat{\pi_1M_1} \arrow[dd, "\Phi"'] \arrow[rr, "\widehat{r_1}", two heads] &     & \widehat{\Gamma_1} \arrow[dd, "\widetilde{\phi}"'] \arrow[rr, "\widehat{f_i}", two heads] &     & \widehat{\pi_1\Sigma_1}\times \widehat{\Z} \arrow[dd, "\phi\times I"'] \arrow[rr, two heads] &     & \widehat{\pi_1\Sigma_1} \arrow[dd, "\phi"'] \\
                                                                              & \circledtext{1} &                                                                                           & \circledtext{2} &                                                                                              & \circledtext{3} &                                            \\
\widehat{\pi_1M_2} \arrow[rr, "\widehat{r_2}", two heads]                     &     & \widehat{\Gamma_2} \arrow[rr, "\widehat{f_i}", two heads]                                 &     & \widehat{\pi_1\Sigma_2}\times \widehat{\Z} \arrow[rr, two heads]                             &     & \widehat{\pi_1\Sigma_2}                   
\end{tikzcd}
\end{equation*}
Indeed, block \circledtext{2}  is deduced  from \autoref{thm: drilling}~\ref{drl3}, 
 and block \circledtext{3} is obtained by projecting  to the first factor. 
Combining blocks \circledtext{2} and \circledtext{3} then yields the desired commutative diagram stated in the proposition. 
\end{proof}

\section{Semiauthentic homomorphisms}\label{sec: semiauthentic}

In this section, we introduce a notion called semiauthenticity, which is a rather special property for homomorphisms between  the profinite completions of finitely generated residually finite groups. We also provide some examples of semiauthentic homomorphisms that will be used in subsequent sections. 


\begin{definition}\label{def: semiauthentic homomorphism}
Let $\Gamma_1$ and $\Gamma_2$ be two finitely generated residually finite groups, and we adopt \autoref{conv}. A homomorphism $\Phi: \widehat{\Gamma_1}\to \widehat{\Gamma_2}$ is {\em semiauthentic} if there exists a semigroup $S\subseteq \Gamma_1$, which generates $\Gamma_1$ as a group, such that $\Phi(s)$ conjugates into $\Gamma_2$ for each $s\in S$. 
\end{definition}

\begin{remark}
\begin{enumerate}[leftmargin=*, label=(\arabic*)]
\item In \autoref{def: semiauthentic homomorphism}, we only require the existence of such a semigroup $S$, and the definition does not specify a particular semigroup. 
\item We do not claim a correspondence between $S$ and a certain semigroup in $\Gamma_2$. Indeed, the asserted conjugation from $\Phi(S)$ to $\Gamma_2$ is merely pointwise. 
\item The semigroup $S$ is not required to be finitely generated. 
\end{enumerate}
\end{remark}

We introduce two additional terminologies elaborated from semiauthenticity.  
\begin{definition}
Let $\Gamma_1$ and $\Gamma_2$ be finitely generated residually finite groups, and let $\Phi: \widehat{\Gamma_1}\to \widehat{\Gamma_2}$ be an isomorphism. 
\begin{enumerate}[leftmargin=*, label=(\arabic*)]
\item $\Phi$ is {\em bi-semiauthentic} if both $\Phi$ and $\Phi^{-1}$ are semiauthentic. 
\item $\Phi$ is {\em hereditarily bi-semiauthentic} if for any $\Phi$-corresponding pair of finite-index subgroups $\Gamma_1'\le \Gamma_1$ and $\Gamma_2'\le \Gamma_2$, the restriction $\Phi':\widehat{\Gamma_1'}\to \widehat{\Gamma_2'}$ is bi-semiauthentic. 
\end{enumerate}
\end{definition}

We remark that $\Phi$ being (hereditarily) bi-semiauthentic does not necessarily imply $\Gamma_1\cong \Gamma_2$, see \autoref{rmk: sol counterexample} for an example.

\subsection{Fibered hyperbolic 3-manifolds}
The notion of semiauthenticity originates from the following example, and most of its ideas have appeared in \cite{Liu25}. 
\begin{proposition}\label{prop: fibered semiauthentic}
Let $M_1$ and $M_2$ be two closed orientable hyperbolic 3-manifolds that fiber over $\mathbb{S}^1$. Then, any isomorphism $\widehat{\pi_1M_1}\to \widehat{\pi_1M_2}$ is semiauthentic. 
\end{proposition}

Some preparations are needed before proving \autoref{prop: fibered semiauthentic}. First, we recall a key construction due to Liu \cite{Liu20}. Let us introduce some terminologies for it. 

\begin{sloppypar}
\begin{definition}
Let $T$ be a finite oriented graph (\autoref{def: finite oriented graph}). For $u, v \in V(T)$, a {\em dynamical path} from $u$ to $v$ is a concatenation of vertices $u = w_0,\allowbreak w_1, \cdots,\allowbreak  w_n = v$ by oriented edges $e_1, \cdots, e_n$ such that $d^-(e_i) = w_{i-1}$ and $d^+(e_i) = w_i$. We allow $n=0$ in the definition of a dynamical path. The graph $T$ is called an {\em irreducible directed graph} if for every $u, v \in V(T)$, there exists a dynamical path from $u$ to $v$. A dynamical path from $v$ to itself is called a {\em dynamical cycle} based at $v$. When discussing the topology of $T$, we view $T$ as a $1$-complex.
\end{definition}
\end{sloppypar}

\begin{proposition}[{\cite{Liu20}}]\label{prop: dynamical graph}
Suppose $M$ is a closed orientable hyperbolic 3-manifold, and $\psi\in H^1(M;\Z)$ is a fibered class. Then, there exists an irreducible directed graph $T$ and a $\pi_1$-surjective map $q:T\to M$   that maps every dynamical cycle in $T$ to the free homotopy class of some periodic trajectory associated to $\psi$.  
\end{proposition}
\begin{proof}
As in \autoref{subsec:pertraj},  we may identify $M$ as a pseudo-Anosov mapping torus $M_{f}$, so that $\psi$ corresponds to the specified fibered class. The graph $T$ is constructed as the transition graph $T_{f,\mathcal{R}}$ associated to a Markov partition $\mathcal{R}$ of $S_{M,\psi}$ with respect to the pseudo-Anosov automorphism $f$. The map $q$ is constructed from  a specified homotopy equivalence from $T$ to the associated flow-box complex $X_{f,\mathcal{R}}$ by composition with the zipping map $X_{f,\mathcal{R}}\to M_f$. The readers are referred to \cite[Section 5]{Liu20} for a detailed description. 

The irreducibility of $T$ is proven as \cite[Lemma 5.6]{Liu20}; the $\pi_1$-surjectivity  of $q$  is proven as \cite[Lemma 5.3]{Liu20}; and the fact that $q$ sends every dynamical cycle in $T$  to the free homotopy class of a periodic trajectory in $ \Traj(M,\psi)$ follows as a combination of \cite[Lemma 5.3 and Lemma 5.7]{Liu20}. 
\end{proof}

\begin{lemma}\label{lem: irreducible semigroup}
Let $T$ be an irreducible directed graph, and let $v$ be a vertex of $T$.
Then, 
$$
\mathcal{DC}_v=\left\{ [\alpha]\in \pi_1(T,v)\mid \alpha\text{ is a dynamical cycle based at }v\right \}
$$
forms a semigroup in $\pi_1(T,v)$ which generates $\pi_1(T,v)$ as a group.  
\end{lemma}
\autoref{lem: irreducible semigroup} appears as a step in \cite[Lemma 4.1]{Liu25}. For completeness, we present a quick proof.

\begin{proof}
It is clear from definition that $\mathcal{DC}_v$ forms a semigroup, since a concatenation of dynamical cycles is again a dynamical cycle. We now prove that $\mathcal{DC}_v$ generates $\pi_1(T,v)$ as a group.  Let $\gamma$ be any cycle based at $v$. Then, $\gamma$ can be decomposed into a concatenation of paths $\gamma=\xi_1\eta_1^{-1}\xi_2\eta_2^{-1}\cdots\xi_n\eta_n^{-1}$,  where $\eta_i^{-1}$ denotes the reversal of $\eta_i$, such that   $\xi_i$ and $\eta_i$ are  dynamical paths (allowing single points).  Let $u_i$ be the initial of $\xi_i$, and let $w_i$ be the terminal of $\xi_i$. By the irreducibility of $T$, we can find dynamical paths $\alpha_i$ from $v$ to $u_i$, and dynamical paths $\beta_i$ from $w_i$ to $v$. Then, $\gamma$ is homotopic relative endpoints to 
\begin{equation*}
\begin{aligned}
&\xi_1(\beta_1\beta_1^{-1}) \eta_1 ^{-1} (\alpha_2^{-1} \alpha_2) \xi_2 (\beta_2 \beta_2^{-1}) \eta_2^{-1}\cdots \xi_n (\beta_n \beta_n^{-1}) \eta_n^{-1}\\=&(\xi_1\beta_1) (\alpha_2 \eta_1 \beta_1)^{-1} (\alpha_2 \xi_2 \beta_2) (\alpha_3\eta_2 \beta_2)^{-1} \cdots (\alpha_n \xi_n \beta_n) (\eta_n\beta_n)^{-1}. 
\end{aligned}
\end{equation*}
Note that  $\alpha_i\xi_i\beta_i$, $\alpha_{i+1}\eta_i\beta_i$,  $\xi_1\beta_1$, and $\eta_n\beta_n$  are  dynamical cycles based at $v$. Hence, $[\gamma]\in \pi_1(T,v)$ belongs to the group generated by $\mathcal{DC}_v$.  Since any element in $\pi_1(T,v)$ can be represented by a cycle based at $v$, we conclude that  $\mathcal{DC}_v$ generates $\pi_1(T,v)$ as a group. 
\end{proof}

\begin{corollary}\label{lem: semigroup trajectory}
Suppose $M$ is a closed orientable hyperbolic 3-manifold, and $\psi\in H^1(M;\Z)$ is a fibered class. Then, there exists a semigroup $S\subseteq \pi_1M$ that generates $\pi_1M$ as a group, with the property that every element $s \in S$ belongs to the conjugacy class of some periodic trajectory associated to $\psi$.
\end{corollary}
\begin{proof}
From \autoref{prop: dynamical graph}, we have an irreducible directed graph $T$ and a specified $\pi_1$-surjective map $q: T\to M$. Let $v$ be an arbitrary vertex of $T$, and let $\mathcal{DC}_v$ be the semigroup defined in \autoref{lem: irreducible semigroup}. Denote $x=q(v)\in M$, and let $q_\ast: \pi_1(T,v)\to \pi_1(M,x)$ be the surjective  homomorphism induced by $q$. Then, $q_\ast(\mathcal{DC}_v)$ is a semigroup in $\pi_1(M,x)$ which generates $\pi_1(M,x)$ as a group, and every element in $q_\ast(\mathcal{DC}_v)$ belongs to the conjugacy class of some $\gamma \in \Traj(M,\psi)$ according to \autoref{prop: dynamical graph}. 
\end{proof}

We remark that the semigroup $S$ in \autoref{lem: semigroup trajectory} admits many choices, and it  might not be uniquely determined by the fibered class $\psi$. 

\begin{proof}[Proof of \autoref{prop: fibered semiauthentic}] 
Let $\Phi: \widehat{\pi_1M_1}\to \widehat{\pi_1M_2}$ be an isomorphism. 
Pick a fibered class $\psi_1\in H^1(M_1;\Z)$, and let $S\subseteq \pi_1M_1$ be the semigroup given by \autoref{lem: semigroup trajectory}. In particular, $S$ generates $\pi_1M_1$ as a group. 
It suffices to show that $\Phi(s)$ conjugates into the subgroup $\pi_1M_2\le \widehat{\pi_1M_2}$ for any $s\in S$. 

Recall that  every $s\in S$ belongs to the conjugacy class of some periodic trajectory in $\Traj(M_1,\psi)$. Since  $\Phi$ is regular by \autoref{thm: regular},  \autoref{prop: match up trajectory} actually implies that for any element $s\in \pi_1M_1$ belonging to the conjugacy class of a periodic trajectory $\gamma\in \Traj(M_1,\psi)$, $\Phi(s)$  conjugates into   $\pi_1M_2$. Indeed, $\Phi(s)$ belongs to the conjugacy class of a corresponding  periodic trajectory in $\widehat{\pi_1M_2}$. This completes the proof. 
\end{proof}


\begin{corollary}\label{cor: hyperbolic bisemiauthentic}
Let $M_1$ and $M_2$ be two closed orientable hyperbolic 3-manifolds that fiber over $\mathbb{S}^1$. Then, any isomorphism $\Phi:\widehat{\pi_1M_1}\to \widehat{\pi_1M_2}$ is hereditarily bi-semiauthentic. 
\end{corollary}
\begin{proof}
Applying  \autoref{prop: fibered semiauthentic} to $\Phi$ and $\Phi^{-1}$, we deduce that $\Phi$ is bi-semi\-au\-then\-tic. In addition, any finite-index subgroup of $\pi_1M_i$ is also the fundamental group of a closed fibered hyperbolic 3-manifold according to \autoref{prop: finite cover}~\ref{fcov1}. Therefore, the restriction of $\Phi$ to any $\Phi$-corresponding pair of finite-index subgroups is also bi-semiauthentic. 
\end{proof}

\begin{remark}\label{rmk: sol counterexample}
There exist 3-manifolds admitting $Sol$-geometry whose fundamental groups are non-isomorphic but have isomorphic profinite completions, see examples constructed by Stebe \cite{Ste72} and Funar \cite{Fun13}. These examples arise from Anosov torus bundles over $\mathbb{S}^1$ with monodromies $A,B\in \SL_2(\Z)$ such that $A$ is  conjugate to $B$ in $\GL_2(\widehat{\Z})$, but $A$ is not conjugate to $B^{\pm 1}$ in $\GL_2(\Z)$. In fact, the isomorphism  between $\widehat{\pi_1M_A}$ and $\widehat{\pi_1M_B}$ (induced by the conjugator in $\GL_2(\widehat{\Z})$) is (hereditarily)  bi-semiauthentic, which can be proven using the same argument as   \autoref{prop: fibered semiauthentic}. Consequently, \autoref{cor: bi-semiauthentic induce isomorphism} in the next section implies that $M_A$ and $M_B$ have isomorphic $\SL_n(\C)$-character varieties, yet $\pi_1M_A\not \cong \pi_1M_B$. 
\end{remark}

\subsection{Surfaces with filling curves}
The specific example of semiauthentic homomorphisms that we consider in this paper is the following. 
\begin{proposition}\label{prop: filling curves virtually hereditarily bi-semiauthentic}
Let $\Sigma_1$ and $\Sigma_2$ be closed orientable surfaces with genera at least $2$. For $i=1,2$, suppose that $\alpha_i$ and $\alpha_i'$ are non-separating simple closed curves on $\Sigma_i$ such that $\alpha_i\cup \alpha_i'$ fills $\Sigma_i$. Let $a_i,a_i'\in \pi_1\Sigma_i$ be conjugacy representatives for the free homotopy classes of $\alpha_i$ and $\alpha_i'$, equipped with arbitrary orientations. Suppose $\phi: \widehat{\pi_1\Sigma_1}\to \widehat{\pi_1\Sigma_2}$ is an isomorphism, such that $\phi(a_1)$  conjugates with $a_2$, and $\phi(a_1')$  conjugates with $a_2'$ in $\widehat{\pi_1\Sigma_2}$. Then, $\phi$ is virtually hereditarily bi-semiauthentic. 
\end{proposition}

To prove \autoref{prop: filling curves virtually hereditarily bi-semiauthentic}, we need the following lemma. 
 
\begin{lemma}\label{lem: commutative semiauthentic}
Let $\Gamma_1,\Gamma_2,\Delta_1,\Delta_2$ be finitely generated residually finite groups. Suppose $p_1: \Gamma_1\to \Delta_1$, $p_2: \Gamma_2\to \Delta_2$,  $\Phi: \widehat{\Gamma_1}\to\widehat{ \Gamma_2}$ and $\phi: \widehat{\Delta_1} \to \widehat{\Delta_2}$ are homomorphisms of abstract groups and profinite groups respectively, which fit into the following commutative diagram. 
\begin{equation*}
\begin{tikzcd}
\widehat{\Gamma_1} \arrow[r,"\Phi"] \arrow[d,"\widehat{p_1}"'] & \widehat{\Gamma_2} \arrow[d,"\widehat{p_2}"] \\ 
\widehat{\Delta_1} \arrow[r,"\phi"] & \widehat{\Delta_2}
\end{tikzcd}
\end{equation*}
If $p_1$ is surjective and $\Phi$ is semiauthentic, then $\phi$ is also  semiauthentic. 
\end{lemma}
\begin{proof}
We adopt \autoref{conv} throughout the proof. Let $R\subseteq \Gamma_1$ be a semigroup that generates $\Gamma_1$ as a group, such that $\Phi(r)$ conjugates into the subgroup $\Gamma_2\le \widehat{\Gamma_2}$ for each $r\in R$. We then let $S=p_1(R)\subseteq \Delta_1$, which is a semigroup in $\Delta_1$. Since $p_1$ is surjective, $S$ generates $\Delta_1$ as a group. In addition, for any element $s\in S$, we can  pick $r\in R$ such that $p_1(r)=s$. Assume that $\Phi(r)=g\widetilde{r}g^{-1}$ for some $\widetilde{r}\in \Gamma_2$ and $g\in \widehat{\Gamma_2}$.  Then, $$\phi(s)=\phi(\widehat{p_1}(r))=\widehat{p_2}(\Phi(r))=\widehat{p_2}(g) \cdot  p_2(\widetilde{r})  \cdot \widehat{p_2}(g)^{-1}  $$
is conjugate  into the subgroup $p_2(\Gamma_2)\le \Delta_2 \le \widehat{\Delta_2}$. In other words, $\phi$ is semiauthentic. 
\end{proof}

\begin{corollary}\label{cor: comdig semiauthentic}
Let $\Gamma_1,\Gamma_2,\Delta_1,\Delta_2$ be finitely generated residually finite groups. Suppose $p_1: \Gamma_1\to \Delta_1$, $p_2: \Gamma_2\to \Delta_2$ are surjective homomorphisms, and  $\Phi: \widehat{\Gamma_1}\to\widehat{ \Gamma_2}$, $\phi: \widehat{\Delta_1} \to \widehat{\Delta_2}$  are isomorphisms of profinite groups, which fit into the following commutative diagram. 
\begin{equation*}
\begin{tikzcd}
\widehat{\Gamma_1} \arrow[r,"\Phi","\cong"'] \arrow[d,"\widehat{p_1}"',two heads] & \widehat{\Gamma_2} \arrow[d,"\widehat{p_2}", two heads] \\ 
\widehat{\Delta_1} \arrow[r,"\phi","\cong"'] & \widehat{\Delta_2}
\end{tikzcd}
\end{equation*}
If $\Phi$ is hereditarily bi-semiauthentic, then $\phi$ is also hereditarily bi-semiauthentic.  
\end{corollary}
\begin{proof}
First, it is easy to see that $\phi$ is bi-semiauthentic by applying \autoref{lem: commutative semiauthentic} to the following two commutative diagrams. 
\begin{equation*}
\begin{tikzcd}
\widehat{\Gamma_1} \arrow[r,"\Phi" ] \arrow[d,"\widehat{p_1}"',two heads] & \widehat{\Gamma_2} \arrow[d,"\widehat{p_2}" ] \\ 
\widehat{\Delta_1} \arrow[r,"\phi" ] & \widehat{\Delta_2}
\end{tikzcd} \quad \quad \quad
\begin{tikzcd}
\widehat{\Gamma_2} \arrow[r,"\Phi^{-1}" ] \arrow[d,"\widehat{p_2}"',two heads] & \widehat{\Gamma_1} \arrow[d,"\widehat{p_1}"] \\ 
\widehat{\Delta_2} \arrow[r,"\phi^{-1}" ] & \widehat{\Delta_1}
\end{tikzcd}
\end{equation*} 

Moreover, suppose $\Delta_1'\le \Delta_1$ and $\Delta_2'\le \Delta_2$ form a $\phi$-corresponding pair of finite-index subgroups, and let $\phi': \widehat{\Delta_1'}\to \widehat{\Delta_2'}$ be the restriction of $\phi$. 
For $i=1,2$, let $\Gamma_i'=p_i^{-1}(\Delta_i')\le \Gamma_i$, which is a finite-index subgroup in $\Gamma_i$. Note that $\overline{\Gamma_i'}=\widehat{p_i}^{-1}( \overline{\Delta_i'})$. Thus, $\Phi (\overline{\Gamma_1'})= \overline{\Gamma_2'}$. In other words, $\Gamma_1'$ and $\Gamma_2'$ form  a pair of $\Phi$-corresponding finite-index subgroups. Let $\Phi': \widehat{\Gamma_1'}\to \widehat{\Gamma_2'}$ be the restriction of $\Phi$. Then, $\Phi'$ is bi-semiauthentic and the following diagram commutes. 
\begin{equation*}
\begin{tikzcd}
\widehat{\Gamma_1'} \arrow[r,"\Phi'","\cong"'] \arrow[d,"\widehat{p_1}"',two heads] & \widehat{\Gamma_2'} \arrow[d,"\widehat{p_2}", two heads] \\ 
\widehat{\Delta_1'} \arrow[r,"\phi'","\cong"'] & \widehat{\Delta_2'}
\end{tikzcd}
\end{equation*}
By the reasoning of the previous paragraph, $\phi'$ is also bi-semiauthentic. In other words, we have proven that $\phi$ is hereditarily bi-semiauthentic. 
\end{proof}

We can now prove \autoref{prop: filling curves virtually hereditarily bi-semiauthentic} based on the drilling construction established in \autoref{sec: drilling}. 

\begin{proof}[Proof of \autoref{prop: filling curves virtually hereditarily bi-semiauthentic}]
According to \autoref{prop: drilling lattice surjection},  there exist two uniform lattices in $\PSL_2(\C)$, denoted by $\Gamma_1$ and $\Gamma_2$, with surjective homomorphisms $p_1: \Gamma_1\twoheadrightarrow \pi_1\Sigma_1$ and $p_2: \Gamma_2\twoheadrightarrow \pi_1\Sigma_2$ and an isomorphism $\Phi: \widehat{\Gamma_1}\to \widehat{\Gamma_2}$, such that the following diagram commutes. 
\begin{equation}\label{equ: comdiag12.9}
\begin{tikzcd}
\widehat{\Gamma_1} \arrow[r,"\Phi","\cong"'] \arrow[d,"\widehat{p_1}"', two heads] & \widehat{\Gamma_2} \arrow[d,"\widehat{p_2}", two heads]\\
\widehat{\pi_1\Sigma_1} \arrow[r,"\phi","\cong"'] & \widehat{\pi_1\Sigma_2}
\end{tikzcd}
\end{equation}

According to  Selberg's lemma and  the virtual fibering theorem  (\autoref{thm: virtually fiber}), there exists a $\Phi$-corresponding pair of finite-index subgroups $\Gamma_1'\le \Gamma_1$ and $\Gamma_2'\le \Gamma_2$ which are fundamental groups of closed fibered hyperbolic 3-manifolds. 
Let $\Phi':\widehat{\Gamma_1'}\to \widehat{\Gamma_2'}$ be the restriction of $\Phi$. Then, \autoref{cor: hyperbolic bisemiauthentic} implies that $\Phi'$ is hereditarily bi-semiauthentic. 

For $i=1,2$, let $\Delta_i=p_i(\Gamma_i')\le \pi_1\Sigma_i$. Then, $\Delta_i$ is a finite-index subgroup of $\pi_1\Sigma_i$ since $p_i$ is surjective. Moreover, $\widehat{p_i}$ is a closed map, so $\overline{\Delta_i}= \widehat{p_i}(\overline{\Gamma_i'}) $. Consequently, $\phi(\overline{\Delta_1})= \overline{\Delta_2}$ according to the commutative diagram (\ref{equ: comdiag12.9}). In other words, $\Delta_1\le \pi_1\Sigma_1$ and $\Delta_2\le \pi_1  \Sigma_2$ form  a $\phi$-corresponding pair of finite-index subgroups. Let $\phi':\widehat{\Delta_1}\to \widehat{\Delta_2}$ be the restriction of $\phi$. Then, the following diagram commutes. 
\begin{equation*}
\begin{tikzcd}
\widehat{\Gamma_1'} \arrow[r, "\Phi'","\cong"'] \arrow[d,"\widehat{p_1}"',two heads] & \widehat{\Gamma_2'} \arrow[d,"\widehat{p_2}",two heads] \\
\widehat{\Delta_1}  \arrow[r, "\phi'","\cong"']  & \widehat{\Delta_2}
\end{tikzcd}
\end{equation*}

According to \autoref{cor: comdig semiauthentic}, we deduce that $\phi'$ is hereditarily bi-semiauthentic. As a consequence, $\phi$ is virtually hereditarily bi-semiauthentic. 
\end{proof}

\subsection{Pseudo-Anosov mapping classes}
In this subsection, we introduce a technical application of  \autoref{prop: filling curves virtually hereditarily bi-semiauthentic} that better fits in our proof for the main theorems.

\begin{sloppypar}
Let us first recall some notations. For an abstract group $\Gamma$, we denote by~$\,\widehat{\cdot}: \Out(\Gamma)\to \Out(\widehat{\Gamma})$ the profinite completion map. Let $G_1$ and $G_2$ be arbitrary (abstract or profinite) groups. If  $\phi:G_1\to G_2$ is an isomorphism, we denote by $\phi_\ast :\Out(G_1)\to \Out(G_2)$ the isomorphism that sends $[f]\in \Out(G_1)$ to $[\phi f \phi^{-1}] \in \Out(G_2)$. 

In addition, if $\Sigma$ is a closed orientable surface, we denote by $\Mod^{\pm}(\Sigma)$ the extended mapping class group. When  $\Sigma$ has positive genus, the Dehn--Nielsen--Baer theorem \cite[Theorem 8.1]{FM11} implies an isomorphism
\begin{equation}\label{equ: DNB}
\Mod^{\pm}(\Sigma) \tto \Out(\pi_1\Sigma). 
\end{equation}
\end{sloppypar}

\begin{lemma}\label{lem: technical VHBS}
Suppose $\Sigma_1$ and $\Sigma_2$ are closed orientable surfaces with genera at least $2$. Let $\phi: \widehat{\pi_1\Sigma_1}\to \widehat{\pi_1\Sigma_2}$ be an isomorphism. 
 We make the following two assumptions. 
\begin{enumerate}[label=(\alph*), leftmargin=*]
\item There exist non-trivial elements $b_1\in \pi_1\Sigma_1$ and $b_2\in \pi_1\Sigma_2$ such that $\phi(b_1)=b_2$.
\item\label{12.16-b} There exist pseudo-Anosov mapping classes $[f_1]\in \Mod(\Sigma_1)$ and $[f_2]\in \Mod(\Sigma_2)$, which correspond to outer automorphism classes $[F_1]\in \Out(\pi_1\Sigma_1)$ and $[F_2]\in \Out(\pi_1\Sigma_2)$ via (\ref{equ: DNB}), such that $\phi_\ast[\widehat{F_1}]=[\widehat{F_2}]$. 
\end{enumerate}
Then, $\phi$ is virtually hereditarily bi-semiauthentic. 
\end{lemma}
\begin{proof}
For $i=1,2$, let $\beta_i \subset  \Sigma_i$ be an oriented  loop whose free homotopy class is represented by the conjugacy class of $b_i\in \pi_1\Sigma_i$. 
According to \autoref{lem: Scott simple cover}, we can find a sufficiently large $m\in \N$, consistent for $i=1,2$, such that all the elevations of $\beta_i$ in the standard characteristic cover $\Sigma_i^{(m)}$ are freely homotopic to non-separating simple closed curves. Since $\phi(b_1)=b_2$, the images of $b_1$ and $b_2$ in the quotient group $\pi_1\Sigma_1/\pi_1\Sigma_1^{(m)}\cong \pi_1\Sigma_2/ \pi_1\Sigma_2^{(m)}$ have the same order, which we denote as $k\in \N$. Then, every elevation of $\beta_i$ in $\Sigma_i^{(m)}$ is a  $k$-fold cover of $\beta_i$. 
We denote $a_i= b_i^k \in  \pi_1\Sigma_i^{(m)}$, and $a_i$ is a conjugacy representative of a particular elevation $\alpha_i\subset \Sigma_i^{(m)}$ of $\beta_i$. Now let $\phi': \widehat{\pi_1\nss{\Sigma}{1}{(m)}} \to \widehat{\pi_1\nss{\Sigma}{2}{(m)}}$  be the restriction of $\phi$ to the $\phi$-corresponding pair of standard characteristic subgroups. Then, $\phi'(a_1)=a_2$. 

We denote   $F_i\in \Aut(\pi_1\Sigma_i)$ as a representative of the outer automorphism class $[F_i]\in \Out(\pi_1\Sigma_i)$. Then, $F_i$ preserves the standard characteristic subgroup $\pi_1\nss{\Sigma}{i}{(m)}$, and we denote   $F_i'\in \Aut(\pi_1\nss{\Sigma}{i}{(m)})$ as the restriction of $F_i$ onto $\pi_1\nss{\Sigma}{i}{(m)}$. The outer automorphism class $[F_i']\in \Out(\pi_1\nss{\Sigma}{i}{(m)})$ corresponds via (\ref{equ: DNB}) to a lift of the mapping class $[f_i]$, which we denote as $[f_i']\in \Mod(\nss{\Sigma}{i}{(m)})$. In particular, $[f_i']$ is a pseudo-Anosov mapping class. 

By assuption \ref{12.16-b}, there exists an element $\hat{g}\in \widehat{\pi_1\Sigma_2}$ such that 
\begin{equation}\label{equ: Inn hat g S}
\Inn_{\hat{g}}\circ \widehat{F_2}=\phi \circ  \widehat{F_1} \circ \phi^{-1}  \in \Aut(\widehat{\pi_1\Sigma_2}).
\end{equation}
 Since $\overline{\pi_1\nss{\Sigma}{2}{(m)}}$ is an open subgroup of $\widehat{\pi_1\Sigma_2}$  and $\pi_1\Sigma_2$ is a dense subgroup of $\widehat{\pi_1\Sigma_2}$, we can find $\hat{g}'\in\overline{\pi_1\nss{\Sigma}{2}{(m)}}$ and  $h\in \pi_1\Sigma_2$ such that $\hat{g}=\hat{g}'h$. Let $T \in \Aut(\pi_1\Sigma_2^{(m)})$ be the restriction of $\Inn_h\in \Aut(\pi_1\Sigma_2)$ onto the standard characteristic subgroup. 
Indeed, the outer automorphism class   $[T]\in \Out(\pi_1 {\Sigma}_{2}^{(m)})$ corresponds via (\ref{equ: DNB}) to a deck transformation $[\tau]\in \Mod( {\Sigma}_{2}^{(m)}) $ of the standard characteristic cover. Hence, $[f_2'']=[\tau]\cdot[f_2']\in \Mod( {\Sigma}_{2}^{(m)}) $ is also a lift of the mapping class $[f_2]$, and $[f_2'']$  is also a pseudo-Anosov mapping class.  We denote $F_2''= T\circ F_2' \in \Aut(\pi_1 {\Sigma}_{2}^{(m)})$, and  the  outer automorphism class of $F_2''$ corresponds to $[f_2'']$ via (\ref{equ: DNB}). 

Restricting (\ref{equ: Inn hat g S}) to the standard characteristic subgroups, we have
$$
\Inn_{\hat{g}'} \circ \widehat{F_2''}= \Inn_{\hat{g}'} \circ \widehat{T} \circ \widehat{F_2'} = \phi'\circ \widehat{F_1'} \circ (\phi')^{-1} \in \Aut(\widehat{\pi_1\nss{\Sigma}{2}{(m)}}) .
$$
In other words, $\phi'_\ast [\widehat{F_1'}]=[\widehat{F_2''}]   \in \Out (\widehat{\pi_1\nss{\Sigma}{2}{(m)}}) $. Then, for any $n\in \N$,  $$\phi'_\ast [\widehat{F_1^{\prime \,  n}}]= [\widehat{ F_2 ^{\prime \prime \, n}}]\in \Out (\widehat{\pi_1\nss{\Sigma}{2}{(m)}}).$$ Recall that $\phi'(a_1)=a_2$, where $a_i\in \pi_1\Sigma_2 ^{(m)}$. Thus, $\phi'(F_1^{\prime\,  n}( a_1)) $ is conjugate to $F_2 ^{\prime \prime\,   n}(a_2)   $ in $\widehat{\pi_1\nss{\Sigma}{2}{(m)}}$. 

Note that $F_1^{\prime\, n}(a_1)$ is a conjugacy representative for the free homotopy class of $[f_1']^n(\alpha_1)$, and $F_2^{\prime \prime  \, n}(a_1)$ is a conjugacy representative for the free homotopy class of $[f_2'']^n(\alpha_2)$, both of which are non-separating simple closed curves since $\alpha_1$ and $\alpha_2$ are. 
According to \cite[Proposition 4.6]{MM99}, pseudo-Anosov mapping classes act as loxodromic elements on the curve complex. As such, one can find a sufficiently large $n\in \N$ such that both pairs $(\alpha_1,[f_1']^n(\alpha_1))$ and $(\alpha_2,[f_2'']^n(\alpha_2))$ have distance at least $3$ on their corresponding curve complexes. In other words, $\alpha_1\cup [f_1']^n(\alpha_1)$ fills $\Sigma_1^{(m)}$ and $\alpha_2\cup [f_2'']^n(\alpha_2)$ fills $\Sigma_2^{(m)}$. Then,  \autoref{prop: filling curves virtually hereditarily bi-semiauthentic} implies that $\phi'$ is virtually hereditarily bi-semiauthentic. Therefore, $\phi$ is also virtually hereditarily bi-semiauthentic. 
\end{proof}

\section{Character variety}\label{sec: 11}

In this section, we study the algebraic features of semiauthentic homomorphisms. These are reflected by the relation between the profinite completion and the character variety of a finitely generated group.

\subsection{Preliminaries}
Let $F$ be an algebraically closed field of characteristic zero, and let $\Gamma$ be a finitely generated group. For any $n\in \N$, the {\em $\SL_n(F)$-representation variety} of $\Gamma$ is an affine algebraic set over $F$ defined as $$R_n(\Gamma,F)=\Hom(\Gamma,\SL_n(F)).$$

In fact, if $\gamma_1,\cdots, \gamma_m$ are generators of $\Gamma$, then $R_n(\Gamma,F)$ can be identified as an algebraic set in $F^{mn^2}$, so that the coordinate ring $F[R_n(\Gamma,F)]$ is generated by the matrix entries of the images of $\gamma_i$ as an $F$-algebra. The reductive algebraic group $\SL_n(F)$ acts by conjugation on $R_n(\Gamma,F)$ algebraically. 
According to \cite[Theorem 1.1]{Mum94}, we can define a universal categorical quotient
$$
X_n(\Gamma,F)=R_n(\Gamma,F)/\!/\SL_n(F)= \mathrm{MaxSpec}\left(F[R_n(\Gamma,F)]^{\SL_n(F)}\right),
$$
where $F[R_n(\Gamma,F)]^{\SL_n(F)}$ denotes the subring of $F[R_n(\Gamma,F)]$ fixed by the $\SL_n(F)$-action.  \cite[Theorem 1.1]{Mum94} ensures that $X_n(\Gamma,F)$ is an affine algebraic set, i.e.\ $F[X_n(\Gamma,F)]=F[R_n(\Gamma,F)]^{\SL_n(F)}$ is  finitely generated as an $F$-algebra. We refer to $X_n(\Gamma,F)$   as the  {\em $\SL_n(F)$-character variety} of $\Gamma$, although it might be reducible. 

\begin{convention}\label{convention: character variety}
We omit the subscript $n$ when $n=2$, and omit the specification for the field $F$ when $F=\C$. 
\end{convention}

For each $\gamma\in \Gamma$, one can define a trace function $\tr_\gamma:R_n(\Gamma,F)\to F$ by $\tr_\gamma(\rho)=\tr(\rho(\gamma))$, which belongs to $F[R_n(\Gamma,F)]$. Clearly, $\tr_\gamma$ is fixed by the $\SL_n(F)$-action, and hence $\tr_\gamma\in F[X_n(\Gamma,F)]$. The trace functions satisfy certain trace relations. In fact, given a group $\Gamma$, there is a set $\mathcal{Z}_n(\Gamma)\subseteq \Q[\tr_\gamma:\gamma\in \Gamma]$, that formally collects all the $\SL_n$-trace relations in $\Gamma$, see \cite[Section 4]{DCP17} for its construction. 

\begin{theorem}[{\cite[Theorems 1.10 and 1.12]{DCP17}}]\label{thm: generator of character variety}
Let $\Gamma$ be a finitely generated group. For any algebraically closed field $F$ of characteristic zero,
\begin{equation}\label{equ: F generators}
F[X_n(\Gamma,F)]=F\left[\tr_\gamma:\gamma\in \Gamma\right]/\left(\mathcal{Z}_n(\Gamma)\right),
\end{equation}
where $( \mathcal{Z}_n(\Gamma) ) $ denotes the ideal generated by $\mathcal{Z}_n(\Gamma)$. 
\end{theorem}

We emphasize that $\mathcal{Z}_n(\Gamma)$ is independent with the field $F$ as long as $\mathrm{char}\,F=0$. Consequently, for any prescribed embedding $F\hookrightarrow F'$ of algebraically closed fields with zero characteristics, we have 
\begin{equation}\label{equ: F tensor}
F'[X_n(\Gamma,F')]=F'\otimes_F F[X_n(\Gamma,F)].
\end{equation}

According to \autoref{thm: generator of character variety}, a point in  $ X_n(\Gamma,F)=\mathrm{MaxSpec}(F[X_n(\Gamma,F)])$ can be equivalently viewed as a function $\chi:\Gamma\to F$, by evaluations at $\tr_\gamma$,   that satisfies all trace relations in $\mathcal{Z}_n(\Gamma)$.  Such a function $\chi$ is called an {\em $\SL_n(F)$-character} of $\Gamma$.  
The inclusion of $F$-algebras $F[R_n(\Gamma,F)]^{\SL_n(F)}\hookrightarrow F[R_n(\Gamma,F)]$ induces an algebraic map
$$
p: R_n(\Gamma,F)\to X_n(\Gamma,F).
$$
The map $p$ is surjective, see \cite[Theorem 1.28]{LM85}. The image of a representation $\rho\in R_n(\Gamma,F)$ under $p$ corresponds to the function $\chi_\rho: \Gamma\to F$ defined by $\chi_\rho(\gamma)=\tr(\rho(\gamma))$, which is called the {\em character} of $\rho$. Indeed, the preimage of each point $x\in X_n(\Gamma,F)$ consists of all $\SL_n(F)$-representations of $\Gamma$ with the same character.

According to Malcev's theorem \cite[Теорема VII]{Mal40}, any finitely generated subgroup of $\SL_n(F)$ is residually finite. In other words, any $\SL_n(F)$-representation of a finitely generated group $\Gamma$ factors through $\Gamma/K$, where $K$ is the intersection of all finite-index subgroups in $\Gamma$. As a consequence, the quotient map $\Gamma\to \Gamma/ K$ induces algebraic isomorphisms $R_n(\Gamma/K,F)\cong R_n(\Gamma,F)$ and $X_n(\Gamma/K,F)\cong X_n(\Gamma,F)$. The group $\Gamma/K$ can be viewed as the maximal residually finite quotient of $\Gamma$.  Hence, in our context, it suffices to study the $\SL_n$-representation theory of finitely generated residually finite groups.

\subsection{Semiauthentic homomorphism}

Fix $\overline{\Q}$ to be an algebraic closure of $\Q$. 
The purpose of this section is to prove the following theorem. 
\begin{theorem}\label{thm: semiauthentic induce Qbar}
Suppose that $\Gamma_1$ and $\Gamma_2$ are finitely  generated residually finite groups, and that $\Phi: \widehat{\Gamma_1}\to \widehat{\Gamma_2}$ is a semiauthentic homomorphism. Then, for each $n\in \N$, $\Phi$ canonically induces an algebraic map
$$
\Phi^\ast: X_n(\Gamma_2,\overline{\Q})\to X_n(\Gamma_1,\overline{\Q}), 
$$
and the induced homomorphism of $\overline{\Q}$-algebras
$$
\Phi^{\sharp}: \overline{\Q}[X_n(\Gamma_1,\overline{\Q})] \to \overline{\Q}[X_n(\Gamma_2,\overline{\Q})], 
$$
is a polynomial map with coefficients in $\Q$ in terms of the generators in (\ref{equ: F generators}). 
\end{theorem}

\begin{corollary}\label{cor: semiauthentic algebraic map}
Suppose that $\Gamma_1$ and $\Gamma_2$ are finitely  generated residually finite groups, and that $\Phi: \widehat{\Gamma_1}\to \widehat{\Gamma_2}$ is a semiauthentic homomorphism.  For any algebraically closed field $F$ of characteristic zero, $\Phi$ canonically induces an algebraic map
$
\Phi^\ast: X_n(\Gamma_2,F) \to X_n(\Gamma_1,F).
$
\end{corollary}
\begin{proof}

Choose an arbitrary embedding $\overline{\Q}\hookrightarrow F$. By (\ref{equ: F tensor}), $$F[X_n(\Gamma_i,F)]=F\otimes_{\overline{\Q}} \overline{\Q}[X_n(\Gamma_i,\overline{\Q})].$$ Let $\Phi^{\sharp}: \overline{\Q}[X_n(\Gamma_1,\overline{\Q})] \to \overline{\Q}[X_n(\Gamma_2,\overline{\Q})]$ be the homomorphism given by \autoref{thm: semiauthentic induce Qbar}. Then, we can define a homomorphism of $F$-algebras by
\begin{equation*}
\begin{tikzcd}[column sep=large]
\Phi^\sharp_F: F[X_n(\Gamma_1,F)] \arrow[r,"\operatorname{id}_F\otimes \Phi^\sharp"] &    F[X_n(\Gamma_2,F)]. 
\end{tikzcd}
\end{equation*}
Since $\Phi^\sharp$ is a polynomial map with coefficients in $\Q$ in terms of the generators in (\ref{equ: F generators}), $\Phi^\sharp_F$ is independent with the choice of the embedding $\overline{\Q}\hookrightarrow F$. Consequently, $\Phi^\sharp_F$ induces a canonical algebraic map between the maximal ideal spectra
$$
\Phi^\ast: X_n(\Gamma_2,F)\to X_n(\Gamma_1,F), 
$$
which is independent with the embedding $\overline{\Q}\hookrightarrow F$. 
\end{proof}

\subsection{Semigroups}
Before the proof of \autoref{thm: semiauthentic induce Qbar}, we need the following lemma, which is essentially \cite[Theorem 3.15]{DCP17}. However, we provide a down-to-earth proof  here to avoid extra machinery. 

\begin{lemma}\label{lem: semigroup generator}
Let $\Gamma$ be a finitely generated group. Suppose $S\subseteq \Gamma$ is a semigroup that generates $\Gamma$ as a group. Then, 
$$
\Q[\tr_\gamma:\gamma\in \Gamma]/(\mathcal{Z}_n(\Gamma))
$$
is generated by $\{\tr_s\mid s\in S\}$ as a $\Q$-algebra.
\end{lemma}
\begin{proof}
Let $F$ be any algebraically closed field with zero characteristic.   Since tensor product with a $\Q$-vector space is exact, we have an injective homomorphism of $\Q$-algebras: 
$$
\Q[\tr_\gamma:\gamma\in \Gamma]/(\mathcal{Z}_n(\Gamma))\tto F\otimes_\Q\left(\Q[\tr_\gamma:\gamma\in \Gamma]/(\mathcal{Z}_n(\Gamma))\right) = F[X_n(\Gamma,F)]. 
$$
To prove the lemma, it suffices to show that each $\tr_\gamma \in F[X_n(\Gamma,F)]$ belongs to the $\Q$-subalgebra of $F[X_n(\Gamma,F)]$ generated by $\{\tr_s\mid s\in S\}$. 

Since $S$ generates $\Gamma$ as a group, any element $\gamma\in \Gamma$ can be expressed as $$\gamma=s_1t_1^{-1}s_2t_2^{-1}\cdots s_mt_m^{-1},$$ where $s_i,t_i\in S$. Let $S'\subseteq S$ be the semigroup generated by $s_1,t_1,\cdots , s_m,t_m$. We actually prove that $\tr_\gamma$ belongs to the $\Q$-subalgebra generated by $\{\tr_s\mid s\in S'\}$.

For any matrix $A\in \SL_n(F)$, we record its characteristic polynomial by
$$
P_A(\lambda)=\det(\lambda I-A)=\lambda^n-\sigma_{1}(A) \lambda^{n-1}+\cdots + (-1)^{n-1}\sigma_{n-1}(A) \lambda +(-1)^n.
$$
Each $\sigma_{j}(A)$ is a polynomial in the matrix entries of $A$. Thus, for each $g\in \Gamma$, we can define $\sigma_{j,g}\in F[R_n(\Gamma,F)]$ by a polynomial map $\sigma_{j,g}(\rho)=\sigma_j(\rho(g))$ for $\rho \in R_n(\Gamma,F)$. 
Clearly, $\sigma_{j,g}$ is invariant under the $\SL_n(F)$-action, 
so $\sigma_{j,g}\in F[X_n(\Gamma,F)]$. Note that $\sigma_{1,g}=\tr_g$. 
Using Newton's identities, each $\sigma_j(A)$ can be expressed as a polynomial of $\tr(A),\tr(A^2),\cdots , \tr(A^j)$ with coefficients in $\Q$. Therefore, each $\sigma_{j,g}$ belongs to the $\Q$-subalgebra of $F[X_n(\Gamma,F)]$ generated by $\tr_{g},\cdots, \tr_{g^{j}}$. 

For $A\in \SL_n(F)$, we define another relavent polynomial
$$
Q_A(\lambda)=(-1)^{n-1} \lambda^{n-1}+(-1)^{n-2}\sigma_1(A)\lambda^{n-2}+\cdots +\sigma_{n-1}(A).
$$
In fact, $Q_A(\lambda)=(-1)^{n-1}\lambda^{-1}(P_A(\lambda)-(-1)^n)$. According to the  Cayley--Hamilton theorem, $A^{-1}=Q_A(A)$.

Let $\rho: \Gamma\to \SL_n(F)$ be an arbitrary representation. Then,
$$
\rho(\gamma)=\rho(s_1)\rho(t_1)^{-1}\cdots \rho(s_m)\rho(t_m)^{-1}= \rho(s_1)Q_{\rho(t_1)}(\rho(t_1))\cdots \rho(s_m)Q_{\rho(t_m)}(\rho(t_m)). 
$$
Thus, $\rho(\gamma)$ can be expressed as a polynomial of $$\rho(s_1),\rho(t_1),\cdots, \rho(s_m),\rho(t_m)$$ with coefficients in $$\Z\left [\sigma_{j}(\rho(t_i)):\, 1\le i \le m, \,1\le j \le n-1 \right].$$ 
Consequently, $\tr(\rho(\gamma))$ can be expressed as a polynomial of 
$$
\{\tr(\rho(s))\mid s\in S'\}\cup \{ \sigma_{j}(\rho(t_i))\mid  1\le i \le m, \,1\le j \le n-1 \}
$$
with fixed coefficients in $\Z$ irrelavent with $\rho$. 
In turn, $\tr_\gamma\in F[X_n(\Gamma,F)]$ belongs to the $\Z$-subalgebra of $F[X_n(\Gamma,F)]$  generated by $$\{\tr_s\mid s\in S'\}\cup \{\sigma_{j,t_i}\mid 1\le i\le m, 1\le j\le n-1\}.$$ Hence, $\tr_\gamma $ belongs to the $\Q$-subalgebra of $F[X_n(\Gamma,F)]$  generated by $\{\tr_s\mid s\in S'\}$ since every $\sigma_{j,t_i}$ does, and the proof is complete. 
\end{proof}

\begin{remark}
\autoref{lem: semigroup generator} does not hold if we switch from the $\SL_n$-character variety to the $\GL_n$-character variety, although the coordinate ring of the $\GL_n$-character variety of $\Gamma$ is also generated by the trace functions evaluated at all elements of $\Gamma$. Indeed, expressing $\tr(\rho(s)^{-1})$ in terms of the  matrix entries of $\rho(s)$ involves placing $\det(\rho(s))$ in the denominator, which is no longer a polynomial in the case of $\GL_n$-representations. 
\end{remark}

\begin{corollary}\label{cor: generate Falgebra}
Let $F$ be an algebraically closed field of characteristic zero. 
Suppose $\Gamma$ is a finitely generated group, and   $S\subseteq \Gamma$ is a semigroup that generates $\Gamma$. Then, $\{\tr_s\mid s\in S\}$ generates $F[X_n(\Gamma,F)]$ as an $F$-algebra. 
\end{corollary}
\begin{proof}
This is a straightforward combination of \autoref{lem: semigroup generator} and \autoref{thm: generator of character variety}. 
\end{proof}

\subsection{The $\mathfrak p$-adic construction}
In this subsection, we prove \autoref{thm: semiauthentic induce Qbar} utilizing the technique of $\mathfrak p$-adic construction, which relates the $\SL_n$-representation theory with the profinite completion. Similar ideas have appeared in \cite{BMRS20} and \cite{Liu25}. 

\begin{proof}[Proof of \autoref{thm: semiauthentic induce Qbar}]
We first construct a set-theoretic  map $\Phi^\ast: X_n(\Gamma_2,\overline{\Q})\to X_n(\Gamma_1,\overline{\Q})$, and then we show that this map is algebraic and satisfies the asserted condition. 

Let $x_2\in X_n(\Gamma_2,\overline{\Q})$ be any point, 
and let $\chi_2:\Gamma_2\to \overline{\Q}$ be the $\SL_n(\overline{\Q})$-character corresponding to $x_2$. 
By the surjectivity of $R_n(\Gamma_2,\overline{\Q})\to X_n(\Gamma_2,\overline{\Q})$, we can pick a representation $\rho_2:\Gamma_2\to \SL_n(\overline{\Q})$ whose character equals $\chi_2$. Then, there exists a subfield $K\subseteq \overline{\Q}$, which is a finite extension over $\Q$, such that the image of $\rho_2$ is contained in $\SL_n(K)$. Indeed, if $\gamma_1,\cdots, \gamma_m$  are generators of $\Gamma_2$, then $K$ can be taken as the subfield generated by the matrix entries of $\rho_2(\gamma_1),\cdots, \rho_2(\gamma_m)$. Let $\mathcal{O}_K$ be the ring of integers of $K$. Then, there exists a finite prime $\mathfrak p$ in $\mathcal{O}_K$ such that 
the image of $\rho_2$ is contained in $\SL_n(\mathcal{O}_{K,\mathfrak {p}})$, where $\mathcal{O}_{K,\mathfrak {p}}$ denotes the $\mathfrak {p}$-localization of $\mathcal{O}_K$. 
Indeed, it suffices to choose $\mathfrak p$ outside finitely many primes that appear as denominators in the matrix entries of $\rho_2(\gamma_1),\cdots,\rho_2(\gamma_m)$. 

Let $\widehat{\mathcal{O}}_{K,\mathfrak p}$ be the $\mathfrak p$-adic completion of $\mathcal{O}_K$. 
The inclusion of rings $ \mathcal{O}_{K,\mathfrak p} \hookrightarrow \widehat{\mathcal{O}}_{K,\mathfrak p}$ induces an inclusion of groups $\SL_n(\mathcal{O}_{K,\mathfrak p})\hookrightarrow \SL_n(\widehat{\mathcal{O}}_{K,\mathfrak p})$. Note that $\SL_n(\widehat{\mathcal{O}}_{K,\mathfrak p})=\limi_i \SL_n(\mathcal{O}_K/ \mathfrak{p}^i)$ is a profinite group. Thus, \autoref{prop: universal property} implies the existence of a unique continuous homomorphism $P_2: \widehat{\Gamma_2} \to  \SL_n(\widehat{\mathcal{O}}_{K,\mathfrak p})$ such that the following diagram commutes. 
\begin{equation*}
\begin{tikzcd}
\Gamma_2 \arrow[r,"\rho_2"] \arrow[d, hook] & \SL_n( {\mathcal{O}}_{K,\mathfrak p}) \arrow[d,hook]\\
\widehat{\Gamma_2} \arrow[r,"P_2"] & \SL_n(\widehat{\mathcal{O}}_{K,\mathfrak p})
\end{tikzcd}
\end{equation*}

Let $K_{\mathfrak p}$ be the  field of fractions of $\widehat{\mathcal{O}}_{K,\mathfrak p}$. Then, $K$ embeds into $K_{\mathfrak p}$ in the apparent way. We now obtain an $\SL_n(K_{\mathfrak p})$-representation of $\Gamma_1$ defined by 
\begin{equation*}
\begin{tikzcd}
\rho_1:\Gamma_1 \arrow[r, hook] & \widehat{\Gamma_1} \arrow[r,"\Phi"] & \widehat{\Gamma_2} \arrow[r,"P_2"] &  \SL_n(\widehat{\mathcal{O}}_{K,\mathfrak p}) \arrow[r, hook] & \SL_n(K_{\mathfrak{p}}).
\end{tikzcd}
\end{equation*}
Let $\chi_1: \Gamma_1\to K_{\mathfrak{p}}$ be the character of $\rho_1$. 
Provided that $\Phi$ is semiauthentic, we prove the following  two assertions. 
\begin{enumerate}[leftmargin=*,label=(\roman*),align=left,widest=ii]
\item\label{assA} The function $\chi_1$ satisfies all trace relations in $\mathcal{Z}_n(\Gamma_1)$, and takes value in $K$.
\end{enumerate}

Recall that $K$ is a prescribed subfield of $\overline{\Q}$, so composing with the inclusion map $K\hookrightarrow \overline{\Q}$, $\chi_1$ can be viewed as a map from $\Gamma$ to $\overline{\Q}$, and hence an $\SL_n(\overline{\Q})$-character of $\Gamma_1$. 
\begin{enumerate}[leftmargin=*,label=(\roman*),align=left,widest=ii,resume]
\item\label{assB} The map $\chi_1:\Gamma_1\to \overline{\Q}$ is uniquely determined by $x_2\in X_n[\Gamma_2,\overline{\Q}]$, and is independent with the choice of $\rho_2$, $K$, and $\mathfrak p$. 
\end{enumerate}


Firstly, the character of any $\SL_n$-representation of $\Gamma_1$ over a field of characteristic zero satisfies all trace relations in $\mathcal{Z}_n(\Gamma)$. This can also be seen by taking an algebraic closure $\overline{K_{\mathfrak p}}$ of $K_{\mathfrak p}$, viewing $\rho_1$ as an $\SL_n(\overline{K_{\mathfrak p}})$-representation of $\Gamma_1$, and applying \autoref{thm: generator of character variety}. 
Now, let $S\subseteq \Gamma_1$ be a semigroup that generates $\Gamma_1$  such that for each $s\in S$, $\Phi(s)$ is conjugate within $\widehat{\Gamma_2}$ to some element $\widetilde{s}\in \Gamma_2$. 
Then, for each $s\in S$,
\begin{equation}\label{equ: chis}
\chi_1(s)=\tr(P_2(\Phi(s)))=\tr(P_2(\widetilde{s}))=\tr(\rho_2(\widetilde{s}))\in \mathcal{O}_{K,\mathfrak p}\subseteq K.
\end{equation}
To show that $\chi_1(\gamma)\in K$ for each $\gamma\in \Gamma_1$, we apply \autoref{lem: semigroup generator}. Indeed, \autoref{lem: semigroup generator} implies that for each $\gamma \in \Gamma_1$,   there exists a polynomial $f_{\gamma}\in \Q[S]$, where $\Q[S]$ refers to the polynomial ring over $S$ as formal variables, such that 
$$
\tr_\gamma=f_\gamma(\tr_s : s\in S)\in \Q[\tr_\gamma:\gamma\in \Gamma] /(\mathcal{Z}_n(\Gamma)).
$$ 
Since $\chi_1$ satisfies all trace relations in  $\mathcal{Z}_n(\Gamma)$, we have 
$$
\chi_1(\gamma)=f_\gamma(\chi_1(s):s\in S).
$$
Recall that each $\chi_1(s)$ belongs to $K$ and the coefficients of $f_\gamma$ belong  to $\Q$. Hence, $\chi_1(\gamma)\in K$ for every $\gamma\in \Gamma_1$, and assertion \ref{assA} is proven. 

Moreover, (\ref{equ: chis}) implies that 
\begin{equation}\label{equ: chisqbar}
\chi_1(s)=\chi_2(\widetilde{s})\in \overline{\Q}
\end{equation}
 for each $s\in S$, which is independent with the choice of $\rho_2$, $K$, and $\mathfrak{p}$. Note that the polynomial $f_\gamma$ only relies on $\Gamma_1$ and $S$, so $\chi_1(\gamma)=f_\gamma(\chi_2(\widetilde s):s\in S)\in \overline{\Q}$ is also independent with the choice of $\rho_2$, $K$, and $\mathfrak{p}$. Hence, assertion \ref{assB} is also proven. 

According to assertion  \ref{assA}, $\chi_1$ extends to a homomorphism of $\overline{\Q}$-algebras
$$
\psi: \overline{\Q}[X_n(\Gamma_1,\overline{\Q})] = \overline{\Q}[\tr_\gamma:\gamma\in \Gamma_1]/(\mathcal{Z}_n(\Gamma_1)) \tto \overline{\Q},
$$
defined by $\psi(\tr_\gamma)=\chi_1(\gamma)$. 
The kernel of $\psi$ is a maximal ideal in $ \overline{\Q}[X_n(\Gamma_1,\overline{\Q})]$, which is in turn a point $x_1\in X_n(\Gamma_1,\overline{\Q})$. By assertion \ref{assB}, the point $x_1$ is uniquely determined by $x_2\in X_n(\Gamma_2,\overline{\Q})$.  

\begin{sloppypar}
We thus obtain a canonical set-theoretic map. 
\begin{equation*}
\begin{tikzcd}[row sep=0cm, column sep=0.3cm]
\Phi^\ast: \hspace{-0.4cm} & {X_n(\Gamma_2,\overline{\Q})} \arrow[rrr] & & & {X_n(\Gamma_1,\overline{\Q})}\\
& x_2 \arrow[rrr,maps to] & & & x_1 
\end{tikzcd}
\end{equation*}
We now prove that $\Phi^\ast$ is an algebraic map. Indeed, $\Phi^\ast$ induces a homomorphism of $\overline{\Q}$-algebras 
$$
\Phi^\bullet: \mathrm{Map}(X_n(\Gamma_1,\overline{\Q}),\overline{\Q})\tto \mathrm{Map}(X_n(\Gamma_2,\overline{\Q}),\overline{\Q}),
$$
where $\mathrm{Map}(X_n(\Gamma_i,\overline{\Q}),\overline{\Q})$ refers to the $\overline{\Q}$-algebra consisting of all  $\overline{\Q}$-valued functions on $X_n(\Gamma_1,\overline{\Q})$. 
Since $\overline{\Q}$ is algebraically closed, $\overline{\Q}[X_n(\Gamma_i,\overline{\Q})]$ embeds into $ \mathrm{Map}(X_n(\Gamma_i,\overline{\Q}),\overline{\Q})$ as a $\overline{\Q}$-subalgebra. To show that $\Phi^\ast$ is algebraic, we need to  show that 
\begin{equation}\label{equ: phibullet}
\Phi^\bullet( \overline{\Q}[X_n(\Gamma_1,\overline{\Q})])\subseteq  \overline{\Q}[X_n(\Gamma_2,\overline{\Q})] .
\end{equation} 

Let $S\subseteq \Gamma_1$ be the semigroup as above. Recall by  \autoref{cor: generate Falgebra} that $\{\tr_s\mid s\in S\}$ generates $\overline{\Q}[X_n(\Gamma_1,\overline{\Q})]$ as a $\overline{\Q}$-algebra,  so (\ref{equ: phibullet}) holds if and only if $\Phi^\bullet(\tr_s)\in  \overline{\Q}[X_n(\Gamma_2,\overline{\Q})]$ for each $s\in S$. Indeed, according to (\ref{equ: chisqbar}), for any $x_2\in X_n(\Gamma_2,\overline{\Q})$, 
$$
\Phi^\bullet(\tr_s)(x_2)= \tr_s(\Phi^\ast(x_2))=\tr_s(x_1)=\chi_1(s)=\chi_2(\widetilde{s})=\tr_{\widetilde{s}}(x_2). 
$$
Therefore, $\Phi^\bullet(\tr_s)=\tr_{\widetilde{s}}\in \overline{\Q}[X_n(\Gamma_2,\overline{\Q})]$, and (\ref{equ: phibullet}) holds. Consequently, $\Phi^\bullet$ restricts to a homomorphism of $\overline{\Q}$-algebras
$$
\Phi^\sharp: \overline{\Q}[X_n(\Gamma_1,\overline{\Q})]\to \overline{\Q}[X_n(\Gamma_2,\overline{\Q})]. 
$$
By construction, the map between the maximal ideal spectra induced by $\Phi^\sharp$ is exactly $\Phi^\ast$. Consequently, $\Phi^\ast$ is an algebraic map. 
\end{sloppypar}

In addition, since $\Phi^\sharp(\tr_s)=\tr_{\widetilde{s}}$, \autoref{lem: semigroup generator} implies that for each $\gamma\in \Gamma_1$, $\Phi^\sharp(\tr_\gamma)$ belongs to the $\Q$-subalgebra of $ \overline{\Q}[X_n(\Gamma_2,\overline{\Q})]$ generated by $\{\tr_{\widetilde{s}}\mid s\in S\}$. In fact, $\Phi^\sharp(\tr_\gamma)= f_\gamma (\tr_{\widetilde{s}}:s\in S)$, with $f_\gamma \in \Q[S]$ defined above. Thus, $\Phi^\sharp$ is a polynomial map with coefficients in $\Q$ in terms of the standard generators in (\ref{equ: F generators}). 
\end{proof}

We remind the readers that the construction for $\Phi^\ast$ (and hence  $\Phi^\sharp$) only relies on the existence of a semigroup $S\subseteq \Gamma_1$ satisfying \autoref{def: semiauthentic homomorphism} that guarantees the   assertions \ref{assA} and \ref{assB}, and is independent with the specific choice of $S$, since the representation $\rho_1$ constructed in  the proof is independent with $S$. 

\begin{remark}\label{rmk: construction brief}
The algebraic map given by \autoref{thm: semiauthentic induce Qbar} and \autoref{cor: semiauthentic algebraic map} can be briefly summarized as follows. Suppose as before that $S\subseteq \Gamma_1$ is a semigroup that generates $\Gamma_1$ such that for each $s\in S$, $\Phi(s)$ conjugates to an element $\widetilde{s}\in \Gamma_2 $ within $\widehat{\Gamma_2}$. Then, the assignment
$$\begin{tikzcd}
\tr_s\in F[X_n(\Gamma_1,F)] \arrow[r,maps to] & \tr_{\widetilde{s}}\in F[X_n(\Gamma_2,F)], \quad s\in S 
\end{tikzcd}
$$
extends to a homomorphism of $F$-algebras $F[X_n(\Gamma_1,F)]\to F[X_n(\Gamma_2,F)]$, which is unique according to \autoref{cor: generate Falgebra}. This homomorphism is what we define to be $\Phi^\sharp$, and is in fact  independent with the choice of $S$. 
\end{remark}

We illustrate the algebraic maps given by \autoref{thm: semiauthentic induce Qbar} and \autoref{cor: semiauthentic algebraic map} in two simplest examples that will be used in subsequent sections. 

\begin{example}\label{example}
Let $\Gamma_1$ and $\Gamma_2$ be finitely generated residaully finite groups, and let $F$ be an algebraically closed field of characteristic zero. 
\begin{enumerate}[leftmargin=*, label=(\arabic*)]
\item\label{exchar1} 
Suppose $f: \Gamma_1\to \Gamma_2$ is a homomorphism. Then, $\widehat{f}: \widehat{\Gamma_1}\to \widehat{\Gamma_2}$ is semiauthentic, and it follows from construction that 
$\widehat{f}^\ast: X_n(\Gamma_2,F)\to X_n(\Gamma_1,F)$ is exactly the map $f^\ast$. 
\item\label{exchar2} 
Suppose   $\Phi:\widehat{\Gamma_1}\to \widehat{\Gamma_1}$ is an inner automorphism. Then, $\Phi$ is semiauthentic, and $\Phi^\ast: X_n(\Gamma_1 ,F)\to X_n(\Gamma_1 ,F)$ is the identity map. 
\end{enumerate}

In fact, both cases can be verified using \autoref{rmk: construction brief} by taking the semigroup $S$ to be the entire group $\Gamma_1$. 
\end{example}

\subsection{Composition of semiauthentic homomorphisms}
We point out that a composition of semiauthentic homomorphisms is not necessarily a semiauthentic homomorphism. However, this composition does induce an algebraic map between the character varieties by composing the algebraic maps induced by each semiauthentic homomorphism. 
In this subsection, we explain that this construction is canonical. 


We start with a technical lemma that describes the map constructed in \autoref{thm: semiauthentic induce Qbar} and \autoref{cor: semiauthentic algebraic map} in an abstract way which is independent with a specific $\mathfrak p$-adic construction.

\begin{lemma}\label{lem: technical composition}
 Suppose $\Gamma_1,\Gamma_2$ are finitely generated residually finite groups, and $\Phi: \widehat{\Gamma_{1}}\to \widehat{\Gamma_2}$ is a semiauthentic homomorphism. Let $F$ be an algebraically closed  field of characteristic zero, and let $x\in X(\Gamma_2,F)$. 
We make the following assumptions.
\begin{enumerate}[label=(\alph*), leftmargin=*]
\item The $\SL_n(F)$-character $\chi: \Gamma_2 \to F$ corresponding to $x$ takes value in a prescribed subfield $K\subseteq F$.
\item Let $L$ be an extension of $K$ (and we fix $K$ as the embedded subfield of $L$) such that  there exists a representation $\rho: \Gamma_2 \to \SL_n(L)$ whose character $\chi_\rho : \Gamma_2 \to L$ takes value in $K$ and equals $\chi$. 
\item There exists a representation $P: \widehat{\Gamma_2}\to \SL_n(L)$, where  $\originalwidehat{\Gamma_2}$ is viewed as an abstract group, such that  the following diagram commutes. 
\begin{equation*}
\begin{tikzcd}
\Gamma_2 \arrow[dr,"\rho"] \arrow[d, hook]& \\
\originalwidehat{\Gamma_2} \arrow[r,"P"] & \SL_n(L)
\end{tikzcd}
\end{equation*}
\end{enumerate}

Then, the representation
\begin{equation*}
\begin{tikzcd}[column sep=0.75cm]
\theta: \Gamma_1 \arrow[r, hook] & \widehat{\Gamma_1} \arrow[r, "\Phi"] & \widehat{\Gamma_2} \arrow[r, "P"] & \SL_n(L)
\end{tikzcd}
\end{equation*}
satisfies the following properties. 
\begin{enumerate}[leftmargin=*,label=(\arabic*)]
\item\label{property1} The character $\chi_\theta: \Gamma_1\to L$  takes value in the prescribed subfield  $K$. 
\end{enumerate}

Thus, $\chi_\theta$ can be viewed as an $\SL_n(F)$-character of $\Gamma_1$ via the inclusion map $K\hookrightarrow F$.
\begin{enumerate}[leftmargin=*,label=(\arabic*),resume]
\item\label{property2} 
The  $\SL_n(F)$-character $\chi_\theta$ corresponds exactly to the point $\Phi^\ast(x) \in X(\Gamma_1,F)$. 
\end{enumerate}
\end{lemma}

\begin{proof}
This follows from the same  proof of assertion \ref{assA} in \autoref{thm: semiauthentic induce Qbar}. Suppose $S\subseteq \Gamma_1$ is a semigroup which generates $\Gamma_1$, such that for each $s\in S$, $\Phi (s)$   conjugates within $\widehat{\Gamma_2}$ to an element $\widetilde{s}\in \Gamma_2$. Then, for each $s\in S$, $\chi_\theta(s)=\chi_P(\Phi (s))=\chi_\rho(\widetilde s)\in K$. According to \autoref{lem: semigroup generator}, the trace relations for $\chi_\theta$ imply that $\chi_\theta(\gamma)\in K$ for any $\gamma \in \Gamma_1$, which proves property \ref{property1}. 

In addition, \autoref{rmk: construction brief} implies that the $\SL_n(F)$-character corresponding to the point $\Phi ^\ast(x)\in X_n(\Gamma_1,F)$ is the unique map $\chi_1: \Gamma_1\to F$ satisfying all trace relations in $\mathcal{Z}_n(\Gamma_1)$ such that $\chi_1(s)=\chi(\widetilde{s})$ for each $s\in S$. 
Note that $\chi_\theta: \Gamma_1\to F$ does satisfy these   conditions. Hence, property \ref{property2} also holds.
\end{proof}

\begin{proposition}\label{prop: composition}
Let $\Gamma_1$ and $\Gamma_2$ be finitely generated residaully finite groups, and let $F$ be a algebraically closed field of characteristic zero. 
Suppose there are two sequences of finitely generated residually finite groups
$$
\Gamma_1=A_0,\,A_1,\,\cdots,\, A_{k-1},\,A_k=\Gamma_2 \quad\text{and}\quad \Gamma_1=B_0,\,B_1,\,\cdots, B_{l-1},\,B_l=\Gamma_2  ,
$$
together with semiauthentic homomorphisms
$$
\Phi_i: \widehat{A_{i-1}}\tto \widehat{A_{i}} \quad\text{and}\quad \Psi_j: \widehat{B_{j-1}}\tto \widehat{B_{j}}.
$$
If the two homomorphisms $$\Phi_k\circ\cdots \circ \Phi_1:\widehat{\Gamma_1}\tto \widehat{\Gamma_2} \quad\text{and}\quad \Psi_l\circ\cdots \circ \Psi_1: \widehat{\Gamma_1}\tto \widehat{\Gamma_2} $$
are identical to each other, then the two algebraic maps 
$$
\Phi_1^\ast\circ\cdots\circ \Phi_k^\ast:X_n(\Gamma_2,F)\to X_n(\Gamma_1,F) \;\,\text{and}\;\, \Psi_1^\ast\circ\cdots\circ \Psi_l^\ast:X_n(\Gamma_2,F)\to X_n(\Gamma_1,F)
$$
are also identical to each other. 
\end{proposition}
 
\begin{proof}
Let $\overline{\Q}\hookrightarrow F$ be any embedding. Then, \autoref{cor: semiauthentic algebraic map} implies that the map 
$$
\Phi_1^\ast\circ\cdots\circ \Phi_k^\ast: \;X_n(\Gamma_2,F)=X_n(\Gamma_2,\overline{\Q}) {\times}_{{\textstyle\mathstrut}{\scalebox{0.6}{$\overline{\Q}$}}} F \tto X_n(\Gamma_1,\overline{\Q}){\times}_{{\textstyle\mathstrut}{\scalebox{0.6}{$\overline{\Q}$}}}  F = X_n(\Gamma_1,F)
$$
is uniquely determined by
$ 
\Phi_1^\ast\circ\cdots\circ \Phi_k^\ast:X_n(\Gamma_2,\overline{\Q})\to X_n(\Gamma_1,\overline{\Q})
$ through an extension of scalars, and is independent with the embedding $\overline{\Q}\hookrightarrow F$. The same statement holds for $\Psi_1^\ast \circ \cdots \circ \Psi_l^\ast$. Thus, it suffices to prove the proposition assuming $F=\overline{\Q}$. 

Let $x\in X_n(\Gamma_2,\overline{\Q})$ be any point. Let $\chi: \Gamma_2\to \overline{\Q}$ be the corresponding character, and  let $\rho_2:\Gamma_2\to \SL_n(\overline{\Q})$ be a representation whose character equals $\chi$. Following the procedure in the proof of \autoref{thm: semiauthentic induce Qbar}, we can fix a subfield $K\subseteq \overline{\Q}$ with finite degree over $\Q$ and a finite prime  $\mathfrak p$ in $\mathcal{O}_K$ such that the image of $\rho_2$ belongs to $\SL_n(\mathcal{O}_{K,\mathfrak p})$. Then using  \autoref{prop: universal property}, we obtain a homomorphism $P_2: \widehat{\Gamma_2} \to \SL_n(\widehat{\mathcal{O}}_{K,\mathfrak p})$ that fits into the following commutative diagram. 
\begin{equation*}
\begin{tikzcd}
\Gamma_2 \arrow[d,hook] \arrow[r,"\rho_2"] & \SL_n(\mathcal{O}_{K,\mathfrak p}) \arrow[d, hook] \\ 
\widehat{\Gamma_2} \arrow[r, "P_2" ] & \SL_n( \widehat{\mathcal{O}}_{K,\mathfrak p})
\end{tikzcd}
\end{equation*}


The field $K_{\mathfrak{p}}$ is an extension of $K$, where $K$ is viewed as a fixed subfield of  $K_{\mathfrak p}$ via the apparent inclusion. We also embed $\mathcal{O}_{K,\mathfrak p}$ and $\widehat{\mathcal{O}}_{K,\mathfrak p}$ into $K_{\mathfrak p}$, so that $\rho_2$ and $P_2$ are viewed as $\SL_n(K_{\mathfrak p})$-representations of abstract groups. Then, we obtain an $\SL_n(K_{\mathfrak p})$-representation  of $\Gamma_1$ as follows. 
\begin{equation}\label{equ: composite construction}
\begin{tikzcd}[column sep=0.8cm]
                                  &                                                      & A_1 \arrow[d, hook]               &                          \cdots             & A_{k-1} \arrow[d, hook]               &  \Gamma_2 \arrow[d, hook] \arrow[rd, "\rho_2"] &                        \\
\theta:\;\Gamma_1 \arrow[r, hook] & \widehat{\Gamma_1}  \arrow[r, "\Phi_1"] & \widehat{A_1} \arrow[r, "\Phi_2"] & \, \cdots \, \arrow[r, "\Phi_{k-1}"] & \widehat{A_{k-1}} \arrow[r, "\Phi_k"] &  \widehat{\Gamma_2} \arrow[r, "P_2"] & \SL_n(K_{\mathfrak p})
\end{tikzcd}
\end{equation}
Applying \autoref{lem: technical composition} a total of $k$ times to the construction (\ref{equ: composite construction}), where $K_{\mathfrak p}$ plays the role of $L$ in \autoref{lem: technical composition}, we deduce that the character $\chi_\theta$ takes value in the subfield $K$; and $\chi_\theta$ viewed as an $\SL_n(\overline{\Q})$-character corresponds to the point $\Phi_1^\ast \circ \cdots \circ \Phi_k^\ast (x)\in X(\Gamma_1,\overline{\Q})$.

Similarly, the character of the representation
\begin{equation*}
\begin{tikzcd}[column sep=0.8cm]
\theta':\;\Gamma_1 \arrow[r, hook] & \widehat{\Gamma_1} \arrow[r, "\Psi_1"] & \widehat{B_1} \arrow[r, "\Psi_2"] & \cdots \arrow[r] & \widehat{B_{l-1}} \arrow[r, "\Psi_l"] & \widehat{\Gamma_2} \arrow[r, "P_2"] & \SL_n(K_{\mathfrak p})
\end{tikzcd}
\end{equation*}
takes value in $K$; and $\chi_{\theta'}$ viewed as an $\SL_n(\overline{\Q})$-character corresponds to the point $\Psi_1^\ast \circ \cdots \circ \Psi_l^\ast (x)\in X(\Gamma_1,\overline{\Q})$. 

However, $\Phi_k\circ \cdots \circ \Phi_1=\Psi_l\circ \cdots \circ \Psi_1$, so $\theta=\theta'$ as $\SL_n(K_{\mathfrak p})$-representations, and $\chi_\theta=\chi_{\theta'}$ as $\overline{\mathbb Q}$-valued functions. Consequently, $\Phi_1^\ast \circ \cdots \circ \Phi_k^\ast (x)= \Psi_1^\ast \circ \cdots \circ \Psi_l^\ast (x)\in X(\Gamma_1,\overline{\Q})$. 
\end{proof}

\begin{remark}
In fact, \autoref{prop: composition} establishes the functoriality of the construction given by \autoref{thm: semiauthentic induce Qbar} and \autoref{cor: semiauthentic algebraic map}. Let $\mathscr{C}$ be a category  whose objects consist of finitely  generated residually finite groups, and whose set of morphisms $\mathrm{Mor}_{\mathscr{C}}(\Gamma_1,\Gamma_2)$ consists of  homomorphisms $\widehat{\Gamma_1}\to \widehat{\Gamma_2}$ which are  compositions of semiauthentic homomorphisms. Then for any algebraically closed field $F$ of characteristic zero, we have defined a sequence of contravariant functors from $\mathscr{C}$ to the category of affine algebraic sets over $F$, which send  a finitely  generated residually finite group to its $\SL_n(F)$-character variety. 
\end{remark}

\begin{corollary}\label{cor: bi-semiauthentic induce isomorphism}
Let  $\Gamma_1$ and $\Gamma_2$ be finitely generated residually finite groups, and let $\Phi: \widehat{\Gamma_1}\to \widehat{\Gamma_2}$   be a bi-semiauthentic isomorphism. Then, for any algebraically closed  field $F$ of characteristic zero, $\Phi$ induces an algebraic isomorphism $\Phi^\ast: X_n(\Gamma_2,F)\to X_n(\Gamma_1,F)$. 
\end{corollary}
\begin{proof}
For $i=1,2$, the identity homomorphism $\widehat{\Gamma_i}\to \widehat{\Gamma_i}$ induces the identity map $X_n(\Gamma_i,F)\to X_n(\Gamma_i,F)$; see \autoref{example}. Hence, applying \autoref{prop: composition} to the compositions $\Phi\circ \Phi^{-1}$ and $\Phi^{-1}\circ \Phi$, we deduce that the algebraic maps $\Phi^\ast: X_n(\Gamma_2,F)\to X_n(\Gamma_1,F)$ and $(\Phi^{-1})^\ast: X_n(\Gamma_1,F)\to X_n(\Gamma_2,F)$ are inverses of each other. In particular, $\Phi^\ast $ is an algebraic isomorphism. 
%
\end{proof}

\section{Automorphism of profinite surface group III: realization}\label{sec: realisation}
In this section, we apply the construction in \autoref{sec: 11} to study automorphisms of profinite surface groups. This section is devoted to the following theorem.

\begin{theorem}\label{thm: surface genuine}
Let $\Sigma_1$ and $\Sigma_2$ be closed orientable surfaces, and let $\phi: \widehat{\pi_1\Sigma_1}\to \widehat{\pi_1\Sigma_2}$ be an isomorphism. Assume that the genus of $\Sigma_1$ and $\Sigma_2$ is at least $2$. If $\phi$ is  virtually  hereditarily bi-semiauthentic, then $\phi$ is genuine.
\end{theorem}

Although we have stated it in an abstract and  general form, \autoref{thm: surface genuine}  is indeed intended to deal with the concrete  examples provided by \autoref{prop: filling curves virtually hereditarily bi-semiauthentic} and \autoref{lem: technical VHBS}. 
Also note that \autoref{thm: surface genuine} is obviously true when $\Sigma_1$ and $\Sigma_2$ are tori, since  bi-semiauthenticity implies genuiness for the profinite completions of finitely generated abelian groups. 

\subsection{Automorphism of $\SL_2(\C)$-character variety}

\begin{sloppypar}
 The proof of \autoref{thm: surface genuine} builds upon deep results concerning the $\SL_2(\C)$-character varieties of surfaces. For a closed orientable surface $\Sigma$, we abbreviate the $\SL_2(\C)$-character variety of its fundamental group into  $X(\Sigma)$; see \autoref{convention: character variety}.  
When the genus $g(\Sigma)\ge 2$, $X(\Sigma)$ is an irreducible affine variety of dimension $6g(\Sigma)-6$; see  \cite{Rap96} for a full account. 

\end{sloppypar}

For any finitely generated group $\Gamma$, the group 
$$
H^1(\Gamma;\Z/2\Z)\rtimes \Out(\Gamma)
$$
acts on $X(\Gamma)$ by algebraic automorphisms. To be specific, $f\in \Aut(\Gamma)$ acts algebraically on $X(\Gamma)$ by 
$
(f^{-1})^\ast : X(\Gamma)\to X(\Gamma),
$ 
where inner automorphisms act trivially on $X(\Gamma)$. This descends to an action of $\Out(\Gamma)$ on $X(\Gamma)$. In addition, $  H^1(\Gamma;\Z/2\Z)\cong \Hom(\Gamma,\{\pm1\})$ acts algebraically on $R(\Gamma)$ by multiplication with a  central representation $\Gamma \to \{\pm I\}$. The action  is invariant under conjugation by $\SL_2(\C)$, and  descends to an algebraic automorphism on $X(\Gamma)$. In other words, $\epsilon\in \Hom(\Gamma,\{\pm1\})$ sends an $\SL_2(\C)$-character $\chi:\Gamma\to \C$ to $\epsilon\cdot \chi:\Gamma\to \C$. These two actions are compatible with the $\Out(\Gamma)$-action on $H^1(\Gamma;\Z/2\Z)$, so we obtain  a homomorphism $
H^1(\Gamma;\Z/2\Z)\rtimes \Out(\Gamma) \to \Aut_{\text{alg}}(X(\Gamma))$. 

When $\Gamma$ is the fundamental group of a closed orientable surface $\Sigma$ with positive genus, $\Out(\Gamma)$ can be identified with the extended mapping class group $\Mod^\pm (\Sigma)$ according to the Dehn--Nielsen--Baer theorem; see (\ref{equ: DNB}). Thus, we have a homomorphism
$$
H^1(\Sigma;\Z/2\Z)\rtimes \Mod^\pm (\Sigma) \tto \Aut_{\text{alg}}(X(\Sigma)).
$$

We will apply the following theorem of March\'e and Simon. 
\begin{sloppypar}
\begin{theorem}[{\cite{MS21}}]\label{thm: MS theorem}
When the genus $g(\Sigma)\ge 3$, the so defined homomorphism $H^1(\Sigma;\Z/2\Z)\rtimes \Mod^\pm (\Sigma) \to \Aut_{\mathrm{alg}}(X(\Sigma))$ is an isomorphism. 
\end{theorem}
\end{sloppypar}
In particular, the subgroup $\Mod^\pm(\Sigma)$ is the stabilizer of the $\SL_2(\C)$-character of the trivial representation. 
\begin{remark}
When $g(\Sigma)=2$, \cite{MS21} proved that this homomorphism is surjective, with kernel being the hyperelliptic involution. We will avoid this case using \autoref{cor: virtually genuine}. 
\end{remark}

\subsection{Finite cover of surfaces}
The major step to \autoref{thm: surface genuine} is the following lemma.

\begin{lemma}\label{lem: major step to realisation}
Let $\Sigma$ be a closed orientable surface with genus $g(\Sigma)\ge 3$, and let $\phi: \widehat{\pi_1\Sigma}\to \widehat{\pi_1\Sigma}$ be an isomorphism. Suppose that the restriction of $\phi$ to the closure of each standard characteristic subgroup is bi-semiauthentic. Then, $\phi$ is genuine. 
\end{lemma}
\begin{proof}
First of all, the restriction of $\phi$ to the first standard characteristic subgroup, which is $\phi$ itself, is bi-semiauthentic. According to \autoref{cor: bi-semiauthentic induce isomorphism}, $\phi$ induces an algebraic automorphism  $\phi^\ast: X(\Sigma)\to X(\Sigma)$. By construction, $\phi^\ast$ fixes the character of the trivial representation; this can be verified using \autoref{lem: technical composition} in which one takes $K=L=\Q$ and takes $\rho$ and $P$ to be the trivial representations. Hence, according to \autoref{thm: MS theorem}, $\phi^\ast: X(\Sigma)\to X(\Sigma)$ is induced by the pull-back of a homeomorphism $h: \Sigma\to \Sigma$. To be concrete, we take an automorphism $f\in \Aut(\pi_1\Sigma)$ belonging to the outer automorphism class that corresponds to  $[h]\in \Mod^{\pm}(\Sigma)$ via (\ref{equ: DNB}), and $\phi^\ast$ is consistent with $f^\ast: X(\Sigma)\to X(\Sigma)$.  

Next, for each $m\in \N$,  consider the $m$-th standard characteristic subgroup of $\pi_1\Sigma$, which corresponds to the $m$-th standard characteristic cover $\Sigma^{(m)}$ of $\Sigma$. We denote by $p:  {\pi_1\Sigma^{(m)}} \to  \pi_1\Sigma $ the inclusion map, which is induced by the covering map $\Sigma^{(m)}\to \Sigma$. We denote the restriction of $\phi$ to $\overline{\pi_1\Sigma^{(m)}}$ as $\widetilde{\phi}: \widehat{\pi_1\Sigma^{(m)}}\to \widehat{\pi_1\Sigma^{(m)}}$, and we abbreviate the quotient group $\pi_1\Sigma/\pi_1\Sigma^{(m)}$ into $Q_m$. Then, we have the following commutative diagram
\begin{equation}\label{equ: standard characteristic surface}
\begin{tikzcd}
1 \arrow[r] & \widehat{\pi_1\Sigma^{(m)}} \arrow[r, hook,"\widehat{p}"] \arrow[d,"\widetilde{\phi}"',"\cong"] & \widehat{\pi_1\Sigma} \arrow[r, two heads] \arrow[d,"\phi"',"\cong"] & Q_m \arrow[r] \arrow[d,"\phi^\flat_{(m)}"',"\cong"] & 1 \\
1 \arrow[r] & \widehat{\pi_1\Sigma^{(m)}} \arrow[r, hook,"\widehat{p}"]   & \widehat{\pi_1\Sigma} \arrow[r, two heads]   & Q_m \arrow[r]  & 1
\end{tikzcd}
\end{equation}
where the rows are exact; see \autoref{sec: Standard characteristic quotients} for the notation of $\phi^\flat_{(m)}$. 

By assumption, $\widetilde{\phi}$ is bi-semiauthentic. Again by  \autoref{cor: bi-semiauthentic induce isomorphism}, $\widetilde{\phi}$ induces an algebraic automorphism  $\widetilde{\phi}^\ast: X(\Sigma^{(m)})\to X(\Sigma^{(m)})$, which fixes the character of the trivial representation. Note that $g(\Sigma^{(m)})\ge g(\Sigma)\ge 3$. Hence, \autoref{thm: MS theorem} again implies that $\widetilde{\phi}^\ast$ is induced by the pull-back of a homeomorphism $\widetilde{h}: \Sigma^{(m)} \to \Sigma^{(m)}$. As above, we take an automorphism $\widetilde{f}\in \Aut(\pi_1\Sigma^{(m)})$ belonging to the outer automorphism class corresponding to  $[\widetilde{h}]$, and $\widetilde{\phi}^\ast$ is identical with $\widetilde{f}^\ast: X(\Sigma^{(m)})\to X(\Sigma^{(m)})$. 

\begin{claim}\label{clm1}
Let $\theta:Q_m \to \Out(\pi_1\Sigma^{(m)})$ be the homomorphism induced by $\pi_1\Sigma$ acting via conjugation on the normal subgroup $\pi_1\Sigma^{(m)}$.  
Then, the following diagram commutes.
\begin{equation}\label{equ: Qm surface}
\begin{tikzcd}
Q_m \arrow[r,"\theta"]  \arrow[d,"\phi^\flat_{(m)}"'] & \Out(\pi_1\Sigma^{(m)}) \arrow[d,"\widetilde{f}_\ast"]\\
Q_m \arrow[r,"\theta"] & \Out(\pi_1\Sigma^{(m)})
\end{tikzcd}
\end{equation}
\end{claim}

\begin{proof}[Proof of \autoref{clm1}]
Let $q_1\in Q_m$ be an arbitrary element, and we denote $q_2=\phi^\flat_{(m)}(q_1)$ for brevity. Recall that $Q_m=\pi_1\Sigma/\pi_1\Sigma^{(m)}$, so we can choose preimages $g_1,g_2\in \pi_1\Sigma$ for $q_1,q_2$ respectively. For $i=1,2$, the automorphism
\begin{equation*}
\Theta_i={\Inn_{g_i}|}_{\pi_1\Sigma^{(m)}} \in \Aut(\pi_1\Sigma^{(m)})
\end{equation*}
is a representative for the outer automorphism class $\theta(q_i)$, and  it extends to an automorphism $\widehat{\Theta_i}$ of $\widehat{\pi_1\Sigma^{(m)}}$ via profinite completion. In addition, the commutative diagram (\ref{equ: standard characteristic surface}) implies that $g_2^{-1}\phi(g_1)\in \widehat{\pi_1\Sigma^{(m)}}$, and therefore,  $$\widehat{\Theta_2}{}^{-1}\circ \widetilde{\phi} \circ \widehat{\Theta_1} \circ \widetilde{\phi}^{-1}: \widehat{\pi_1\Sigma^{(m)}}\to   \widehat{\pi_1\Sigma^{(m)}}$$
is an inner automorphism of  $\widehat{\pi_1\Sigma^{(m)}}$ induced by the conjugator $g_2^{-1}\phi(g_1)$. In particular, 
$$(\widehat{\Theta_2}{}^{-1}\circ \widetilde{\phi} \circ \widehat{\Theta_1} \circ \widetilde{\phi}^{-1})^{\ast}:X(\Sigma^{(m)})\to   X(\Sigma^{(m)})$$
is the identity map according to \autoref{example}~\ref{exchar2}. Note that $\widehat{\Theta_1}$ and $\widehat{\Theta_2}$ are bi-semiauthentic by \autoref{example}~\ref{exchar1}, and $\widetilde{\phi}$ is bi-semiauthentic by our assumption. Hence, we can invoke \autoref{prop: composition} to decompose this composition of maps. Specifically, we deduce that
\begin{equation*}
\begin{aligned}
id_{X(\Sigma^{(m)})}=&(\widehat{\Theta_2}{}^{-1}\circ \widetilde{\phi} \circ \widehat{\Theta_1} \circ \widetilde{\phi}^{-1})^{\ast} & \\
= & (\widetilde{\phi}^{-1})^\ast\circ \widehat{\Theta_1}{}^\ast \circ \widetilde{\phi}^\ast \circ (\widehat{\Theta_2}{}^{-1})^\ast  & \text{(\autoref{prop: composition})}\\
= & (\widetilde{\phi}^\ast)^{-1}\circ \widehat{\Theta_1}{}^\ast \circ \widetilde{\phi}^\ast \circ (\widehat{\Theta_2^{-1}})^\ast  & \text{(\autoref{cor: bi-semiauthentic induce isomorphism})}\\
= & (\widetilde{f}^\ast)^{-1} \circ \Theta_1^\ast \circ \widetilde{f} ^ \ast \circ (\Theta_2^{-1})^\ast & \text{(\autoref{example}~\ref{exchar1})}\\
= & (\Theta_2^{-1} \circ \widetilde{f} \circ \Theta_1 \circ \widetilde{f}^{-1}) ^\ast \\
= & \big [\theta(q_2)^{-1} \cdot \widetilde{f}_\ast (\theta(q_1))\big ]^\ast.
\end{aligned}
\end{equation*}
In particular, $ \theta(q_2) ^{-1} \cdot  \widetilde{f}_\ast (\theta(q_1)) \in \Out(\pi_1\Sigma^{(m)})  $ belongs to the kernel of the homomorphism $\Out(\pi_1\Sigma^{(m)})\to \Aut_{\mathrm{alg}}(X(\Sigma^{(m)}))$. According to \autoref{thm: MS theorem}, $ \theta(q_2)^{-1} \cdot  \widetilde{f}_\ast (\theta(q_1)) $ is the trivial outer automorphism class. In other words, $\theta(q_2)= \widetilde{f}_\ast (\theta(q_1)) $ for every $q_1\in Q_m$, so  diagram (\ref{equ: Qm surface}) commutes. 
\end{proof}

The commutative diagram (\ref{equ: Qm surface}) implies a homotopy equivariance  of the homeomorphism $\widetilde{h} : \Sigma^{(m)} \to \Sigma^{(m)}$ under deck transformations. We would now like to descend $\widetilde{h}$ to a homeomorphism of $\Sigma$,  which we approach in a group theoretic way.  
Note that $\pi_1\Sigma^{(m)}$ has trivial center. Thus,  applying \autoref{prop: standard abstract algebra} to the group extension $1\to \pi_1\Sigma^{(m)} \to\pi_1\Sigma\to Q_m \to 1$, we can obtain a unique isomorphism $f_1: \pi_1\Sigma \to \pi_1\Sigma$ such that the following diagram commutes. 
\begin{equation}\label{equ: characteristic surface 2}
\begin{tikzcd}
1 \arrow[r] & \pi_1\Sigma^{(m)} \arrow[r, hook,"p"] \arrow[d,"\widetilde{f}"',"\cong"] & \pi_1\Sigma \arrow[r, two heads] \arrow[d,"f_1"',"\cong"] & Q_m \arrow[r] \arrow[d,"\phi^\flat_{(m)}"',"\cong"] & 1 \\
1 \arrow[r] & \pi_1\Sigma^{(m)} \arrow[r, hook,"p"] & \pi_1\Sigma \arrow[r, two heads] & Q_m \arrow[r] & 1
\end{tikzcd}
\end{equation}
By  Dehn--Nielsen--Baer theorem, the isomorphism $f_1$ can be induced by a homeomorphism  $h_1: \Sigma\to \Sigma$ up to conjugacy. 

\begin{claim}\label{clm2}
$f_1$ is conjugacy equivalent with $f$ (see \autoref{def: conjugacy equivalent}). 
\end{claim}
\begin{proof}[Proof of \autoref{clm2}]
Invoking \autoref{prop: composition} and \autoref{example}~\ref{exchar1}, we have the following sequence of identical algebraic maps from $X(\Sigma)$ to $ X(\Sigma^{(m)})$  deduced  from   the commutative diagrams (\ref{equ: standard characteristic surface}) and (\ref{equ: characteristic surface 2}).
\begin{equation}\label{equ: pullback equal}
p^\ast \circ f_1^{ \ast} =\widetilde{f}^\ast \circ p^\ast =\widetilde{\phi}^\ast \circ \widehat{p}{}^\ast = \widehat{p}{}^\ast \circ \phi^\ast = p^\ast \circ f^\ast : X(\Sigma) \tto X(\Sigma^{(m)}) .
\end{equation}

Suppose by contrary that $f_1$ is not conjugacy equivalent with $f$. Then the homeomorphism $h_1$ is not homotopic to $h$, and \cite[Section 2]{Kor85} implies that there exists an unoriented simple closed curve $\alpha \subset \Sigma$ such that $h_1(\alpha)$ is not homotopic to $h(\alpha)$. We denote $\beta_1=h_1(\alpha)$ and $\beta=h(\alpha)$, and let $b_1,b\in \pi_1\Sigma$ be conjugacy representatives for $\beta_1,\beta$ equipped with arbitrary orientations.  According to \cite[Corollary 3.1]{Sun14}, there exists a hyperbolic structure $X$ on $\Sigma$ such that $l_X(\beta_1)\neq l_X(\beta)$, where $l_X$ denotes the hyperbolic length of the geodesic representative. 
The hyperbolic structure $X$ is given by a discrete faithful representation $\rho_0: \pi_1\Sigma \to \PSL_2(\R)$. It is well-known, see for example \cite{Pat75},  that $\rho_0$ lifts to an $\SL_2(\R)$-representation $\rho: \pi_1\Sigma \to \SL_2(\R)$, which we view as an $\SL_2(\C)$-representation via $\R\hookrightarrow \C$. 
According to (\ref{equ: pullback equal}), $\rho\circ f_1 \circ p: \pi_1\Sigma^{(m)}\to \SL_2(\R)$  and $\rho\circ f \circ p: \pi_1\Sigma^{(m)}\to \SL_2(\R)$ have the same character. 

Let $\widetilde{\alpha}$ be an elevation of $\alpha$ in $\Sigma^{(m)}$, and suppose that $\widetilde{\alpha}$ is an $n$-fold cover of $\alpha$. Let $g\in \pi_1\Sigma^{(m)}$ be a conjugacy representative of $\widetilde{\alpha}$ equipped with an arbitrary orientation. Then, 
$$
\abs{\tr(\rho\circ f_1 \circ p(g))}=\abs{\tr(\rho(b_1^{\pm n}))}=2\cosh\left(\frac{n}{2}l_X(\beta_1)\right),
$$
and 
$$
\abs{\tr(\rho\circ f \circ p(g))}=\abs{\tr(\rho(b ^{\pm n}))}=2\cosh\left(\frac{n}{2}l_X(\beta )\right).
$$
However, $l_X(\beta_1)\neq l_X(\beta)$, which implies that $ {\tr(\rho\circ f \circ p(g))}\neq  {\tr(\rho\circ f_1 \circ p(g))}$. Hence, we reach a contradiction with (\ref{equ: pullback equal}). 
\end{proof}

Following the notations in \autoref{sec: Standard characteristic quotients}, it is clear from diagram (\ref{equ: characteristic surface 2}) that  $(f_1)^\flat_{(m)}=\phi^\flat_{(m)}$. Since $f\simeq f_1$, we also have $f^\flat_{(m)}\simeq (f_1)^\flat_{(m)}$ by \autoref{lem: conjugacy equivalence criteria}. Consequently, $f^\flat_{(m)}\simeq \phi^\flat_{(m)}$ for every $m\in \N$, where $f\in \Aut(\pi_1\Sigma)$ is fixed at the beginning. According to \autoref{cor: genuine criteria 1}, we deduce that $\phi$ is genuine. 
\end{proof}

\begin{remark}
\begin{enumerate}[leftmargin=*, label=(\arabic*)]
\item In the setting of \autoref{lem: major step to realisation}, it suffices to assume that the restrictions of $\phi$ to a cofinal system of finite-index normal subgroups, including $\pi_1\Sigma$ itself,  are all bi-semiauthentic. We work with standard characteristic subgroups for the sake of convenience.
\item In the proof of  \autoref{clm2}, we utilize hyperbolic geometry to show that $\tr_{b^n}\neq \tr_{b_1^n}\in \C[X(\Sigma)]$, where $b,b_1\in \pi_1\Sigma$ represent non-homotopic simple closed curves on $\Sigma$. This can also be proven by a more theoretical approach, using the fact that the trace functions evaluated at  multicurves on $\Sigma$ (allowing homotopic components) form a $\C$-linear basis of $\C[X(\Sigma)]$, see \cite{Prz99,PS00}.
\end{enumerate}
\end{remark}

We  now finish the proof of  \autoref{thm: surface genuine}. 

\begin{proof}
We can find a $\phi$-corresponding pair of finite-index proper subgroups  $\pi_1\Sigma_1'\lneq \pi_1\Sigma_1$ and $\pi_1\Sigma_2'\lneq \pi_1\Sigma_2$, corresponding to non-trivial  finite covers $\Sigma_1'$ of $\Sigma_1$ and $\Sigma_2'$ of $\Sigma_2$, such that the restriction $\phi': \widehat{\pi_1\Sigma_1'}\to \widehat{\pi_1\Sigma_2'}$ is hereditarily bi-semiauthentic. In this case, $g(\Sigma'_1)=g(\Sigma_2')\ge 3$. Through a group isomorphism that identifies $\pi_1\Sigma_1'$ with $\pi_1\Sigma_2'$, \autoref{lem: major step to realisation} implies that  $\phi'$ is genuine.  Then, \autoref{cor: virtually genuine} implies that $\phi$ itself is also genuine. 
\end{proof}

\section{The uniform case}\label{sec: Uniform}
In this section, we prove \autoref{mainthm1} and \autoref{mainthm2} for uniform lattices in $\PSL_2(\C)$. We shall first prove these theorems for the fundamental groups of crossfibered manifolds, and then generalize them to all uniform lattices using the virtual crossfibering theorem (\autoref{thm: virtual crossfibering}).

\begin{lemma}\label{lem: crossfibered genuine}
Suppose $M_1$ and $M_2$ are crossfibered manifolds, and $\Phi: \widehat{\pi_1M_1}\to \widehat{\pi_1M_2}$  is an isomorphism. Then, $\pi_1M_1\cong\pi_1M_2$, and $\Phi$ is genuine. 
\end{lemma}
\begin{proof}
Let $(M_1,\psiAyi,\psiFyi,\gamma_1)$ and $(M_2,\psiAer,\psiFer,\gamma_2)$ be a $\Phi$-corresponding pair of crossfibering data, see \autoref{prop: detect crossfibering}, \autoref{conv: crossfibering 1}, and \autoref{conv: crossfibering 2}. For $i=1,2$, we abbreviate the fiber surface $S_{M_i,\psi_{\mathrm{F},i}}$ into $S_i$, and $\pi_1 S_i$ represents a distinguished normal subgroup in $\pi_1M_i$. 
The  periodic trajectory $\gamma_i$ freely homotopes onto $S_i$, so we can pick a conjugacy representative  $g_i\in \pi_1S_i\le \pi_1M_i$ for its free homotopy class. 
Since we have proven the regularity of $\Phi$ (\autoref{lem: crossfibered regular}), by possibly composing $\Phi$ with an inner automorphism of $\widehat{\pi_1M_2}$ -- which does not change its conjugacy equivalence class -- we can assume that $\Phi(g_1)=g_2$. Recall that in the   proof of \autoref{lem: crossfibered regular}, we have obtained a commutative diagram (\ref{equ: ses commutative crossfibering}) between short exact sequences: 
\begin{equation}\label{equ: closed finish ses}
\begin{tikzcd}
1 \arrow[r] & \widehat{\pi_1S_1} \arrow[r, hook] \arrow[d, "\phi"',"\cong"] & \widehat{\pi_1M_1} \arrow[r,"\widehat{\psiFyi}"] \arrow[d, "\Phi"',"\cong"] & \widehat{\Z} \arrow[d, "\text{id}" ] \arrow[r] & 1 \\
1 \arrow[r] & \widehat{\pi_1S_2} \arrow[r,hook]                    & \widehat{\pi_1M_2} \arrow[r,"\widehat{\psiFer}"]                    & \widehat{\Z} \arrow[r]                   & 1
\end{tikzcd}
\end{equation}
where $\phi$ is the restriction of $\Phi$, and the map $\widehat{\Z}\to \widehat{\Z}$ has been refined to be the identity map due to the regularity of $\Phi$ (see \autoref{conv: crossfibering 2}). In particular, we also have  $\phi(g_1)=g_2$. 

For $i=1,2$, let $[f_i]\in \Mod(S_i)$ be the monodromy   of the fibered class $\psi_{\mathrm{F}_i}$, which is a pseudo-Anosov mapping class, and let $[F_i]\in \Out(\pi_1S_i)$ be its corresponding outer automorphism class under the Dehn--Nielsen--Baer isomorphism (\ref{equ: DNB}). Then, the homomorphism $\widehat{\Z} \to \Out(\widehat{\pi_1S_i})$ induced by $\widehat{\pi_1M_i}$ acting via conjugation on $\widehat{\pi_1S_i}$ sends $1\in \widehat{\Z}$ to $[\widehat{F_i}]\in \Out(\widehat{\pi_1S_i})$. The commutative diagram (\ref{equ: closed finish ses}) implies the following commutative diagram. 
\begin{equation*}
\begin{tikzcd}
\widehat{\Z} \arrow[d,"\mathrm{id}"'] \arrow[r] & \Out(\widehat{\pi_1S_1}) \arrow[d,"\phi_\ast"] \\
\widehat{\Z} \arrow[r] &   \Out(\widehat{\pi_1S_2})
\end{tikzcd}
\end{equation*}
In particular, $\phi_\ast[\widehat{F_1}]=[\widehat{F_2}]$. 

Note that $S_1$ and $S_2$ have   genus   at least $2$. Consequently,  \autoref{lem: technical VHBS} implies that $\phi$ is virtually hereditarily bi-semiauthentic. Then, we apply \autoref{thm: surface genuine} to deduce that $\phi$ is genuine.  

Note that each $\pi_1S_i$ is finitely generated and has trivial center. In addition, $\widehat{\pi_1S_i} $ has trivial center according to \autoref{thm: lattice center free}, and the profinite completion map $\Out(\pi_1S_i)\to \Out(\widehat{\pi_1S_i})$ is injective according to \autoref{cor: completion map injective}. Furthermore, $id:\widehat{\Z}\to \widehat{\Z}$ is the profinite completion of $id: \Z\to \Z$.  Therefore, \autoref{lem: extension lemma} applies to the commutative diagram (\ref{equ: closed finish ses}), which implies that $\pi_1M_1\cong \pi_1M_2$ and $\Phi$ is genuine. 
\end{proof}

\begin{theorem}\label{thm: main for uniform}
Any uniform lattice $\Gamma \le \PSL_2(\C)$ is profinitely rigid among all lattices in $\PSL_2(\C)$. In addition, the profinite completion map $\Out(\Gamma)\to \Out(\widehat{\Gamma})$ is an isomorphism. 
\end{theorem}
\begin{proof}
Recall that the injectivity of $\Out(\Gamma)\to \Out(\widehat{\Gamma})$ has been proven by  \autoref{cor: completion map injective}, so it suffices to show the profinite rigidity of $\Gamma$ and the surjectivity of $\Out(\Gamma)\to \Out(\widehat{\Gamma})$. In other words, it suffices to show that if $\Delta\le \PSL_2(\C)$ is a lattice and $\Phi: \widehat{\Gamma}\to \widehat{\Delta}$ is an isomorphism, then $\Gamma\cong \Delta$ and  $\Phi$ is genuine.  

According to Selberg's lemma and \autoref{thm: virtual crossfibering}, there exists a finite-index subgroup $\Gamma'\le \Gamma$ such that $\Gamma'$ is the fundamental group of a crossfibered manifold. Let $\Delta'\le \Delta$ be the $\Phi$-corresponding finite-index subgroup, and let $\Phi' : \widehat{\Gamma' }\to \widehat{\Delta'}$ be the restriction of $\Phi$. \autoref{prop: torsion free} implies that $\Delta'$ is a torsion-free uniform lattice in $\PSL_2(\C)$, and \autoref{prop: detect crossfibering} then implies that $\Delta'$ is also the fundamental group of a crossfibered manifold. Consequently, \autoref{lem: crossfibered genuine} implies that $\Phi'$ is genuine. Then, applying \autoref{cor: virtually genuine}, we finally deduce that $\Phi$ is genuine. In particular, $\Gamma \cong \Delta$. 
\end{proof}

\section{The non-uniform case}\label{sec: non-uniform}
In this section, we prove \autoref{mainthm1} and \autoref{mainthm2} for non-uniform lattices in $\PSL_2(\C)$. The proof is based on the conclusions for uniform lattices proven in \autoref{sec: Uniform}, with Thurston's hyperbolic Dehn surgery theory serving once again as the key tool. We shall first deal with the case of torsion-free non-uniform lattices, and then generalize it to all non-uniform lattices.

\subsection{Peripheral regularity}
In this subsection, we recall a notion of peripheral regularity introduced by the author in \cite{Xu25}. 

Suppose $M$ is a compact orientable 3-manifold with  incompressible toral boundary, and we denote the boundary components of $M$ by $\partial_1 M,\cdots, \partial_kM$ respectively. Up to an identification of basepoints, we have an injective homomorphism  $ \text{incl}_\ast: \pi_1\partial_i M \to \pi_1M$ which is  well-defined up to conjugacy equivalence. Note that any  abelian subgroup  in a finitely generated 3-manifold group is closed in the profinite topology \cite{Ham01}. Thus, \autoref{prop: left exact} implies that $\widehat{\text{incl}_\ast}: \widehat{\pi_1\partial_iM}\to \widehat{\pi_1M}$ is injective. In particular, $\overline{\pi_1\partial _i M}\cong \widehat{\Z}^2$. 

\begin{definition}\label{def: peripheral regular}
Suppose $M $ and $N$ are compact orientable 3-manifolds with incompressible toral boundary. An isomorphism $\Phi: \widehat{\pi_1M }\to \widehat{\pi_1N}$ is {\em peripheral regular} if there exists a homeomorphism $\varphi: \partial M \to \partial N$, which we respectively denote as $\varphi_i:\partial_i M \to \partial_i N$ on the corresponding components, such that the following property holds. For each $i$, there exists $g_i \in \widehat{\pi_1N}$ such that the following diagram commutes. 
\begin{equation*}
\begin{tikzcd}[row sep=large]
\widehat{\pi_1\partial _i M} \arrow[rr, " \widehat{{\varphi_i}_\ast}" ,"\cong"'] \arrow[d,"\widehat{\text{incl}_\ast}"', hook ] & & \widehat{\pi_1\partial_i N} \arrow[d,"\widehat{\text{incl}_\ast}" , hook] \\
\widehat{\pi_1M} \arrow[r, "\Phi","\cong"'] & \widehat{\pi_1N} \arrow[r," \Inn_{g_i}","\cong"'] & \widehat{\pi_1N}
\end{tikzcd}
\end{equation*}

To specify this homeomorphism, we say that $\Phi$ is {\em peripheral regular with respect to $\varphi$}. 
\end{definition}


\begin{proposition}[{\cite[Theorem 1.4]{Xu25}}]\label{thm: peripheral regular}
Suppose $M$ and $N$ are orientable cusped finite-volume hyperbolic 3-manifolds, with incompressible toral boundaries obtained from truncating their cusps. 
Then, any isomorphism $\Phi: \widehat{\pi_1M}\to \widehat{\pi_1N}$ is peripheral regular. 
\end{proposition}

We remark that \autoref{thm: peripheral regular} can be alternatively proven by combining \autoref{thm: regular} with \cite[Theorem 7.2]{Xu25A}.

\subsection{Cusped hyperbolic manifolds}
Peripheral regularity is introduced primarily in order to   apply the construction of Dehn fillings as described in \autoref{subsec: hyperbolic Dehn surgery}. We have exemplified such an application in \autoref{prop: drilling lattice surjection}, and we will be utilizing this technique once again in the following proposition.

\begin{proposition}\label{prop: major step for non-uniform}
Let $\Gamma$ and $\Delta$ be torsion-free non-uniform lattices in $\PSL_2(\C)$. Suppose $\Phi:\widehat{\Gamma} \to \widehat{\Delta}$ is an isomorphism. Then $\Gamma\cong \Delta$, and $\Phi$ is genuine. 
\end{proposition}

\begin{proof}
Recall that $\Gamma$ and $\Delta$ are the fundamental groups of compact orientable 3-manifolds $M$ and $N$ with incompressible toral boundaries, whose interiors are homeomorphic to $\mathbb{H}^3/\Gamma$ and $\mathbb{H}^3/\Delta$ respectively. 
%
For the sake of convenience, we will identify $\Gamma$ and $\Delta$ with $\pi_1M$ and $\pi_1N$ henceforth. 

According to \autoref{thm: peripheral regular}, the isomorphism $\Phi:\widehat{\pi_1M}\to\widehat{\pi_1N}$ is peripheral regular with respect to a homeomorphism $\varphi: \partial M \to \partial N$, which we respectively denote as $\varphi_i: \partial _i M \to \partial_i N$ on the corresponding components. For $n\in \N$, let $n!$ denote the factorial of $n$.  According to \autoref{thm: hyperbolic dehn surgery}, we can find non-zero homology classes $d_i\in H_1(\partial_i M ; \Z ) $ such that for any $n\in \N$, the Dehn filled  orbifolds 
$$
X_n = M_{n!\cdot d_1,\cdots , n!\cdot d_k} \quad \text{and} \quad  Y_n= N_{n!\cdot {\varphi_1}_\ast(d_1),\cdots,n!\cdot {\varphi_k}_\ast(d_k)}
$$ 
admit complete hyperbolic structures. 

For  each $1\le i\le k$, let $\alpha_i \in \pi_1M$ be a conjugacy representative for the loop on $\partial _i M$ representing the homology class of $d_i$, and let $\beta_i \in \pi_1N$ be a conjugacy representative for the loop  on $\partial _i N$ representing the homology class of ${\varphi_i}_\ast(d_i)$. Denote 
$$
K_n= \langle \! \langle \alpha_1^{n!} ,\cdots, \alpha_k ^ { n!}\rangle \! \rangle \unlhd \pi_1M \quad \text{and} \quad L_n=\langle \! \langle \beta_1^{n!} ,\cdots, \beta_k^{n!} \rangle \! \rangle \unlhd \pi_1N .
$$
Let $r_n : \pi_1 M \to \pi_1M / K_n$ and  $s_n: \pi_1N \to \pi_1N / L_n$ be the quotient homomorphisms. 
According to \autoref{orbifold fundamental group}, we can identify $\pi_1M / K_n$ with $\piorb X_n$ such that the homomorphism $r_n$ is induced by the inclusion map $M\hookrightarrow X_n$,  where we identify their  basepoints.  Similarly, we can identify $\pi_1N / L_n$ with $\piorb Y_n$ such that $s_n$ is induced by the inclusion map $N\hookrightarrow Y_n$. 

 According to  \autoref{prop: right exact}, $\widehat{r_n}$ and $\widehat{s_n}$ are surjective, with $\ker(\widehat{r_n})= \overline{K_n}$ and $\ker(\widehat{s_n})=\overline{L_n}$. 
From \autoref{def: peripheral regular}, it is clear that  $\Phi(\overline{K_n})= \overline{L_n}$ for any $n\in \N$, since $\Phi(\alpha_i^{n!})$ conjugates with $\beta_i^{n!}$ in $\widehat{\pi_1N}$ for each $i$. 
Thus,  for each $n\in \N$, $\Phi$ descends to an isomorphism $\Psi_n: \widehat{\piorb X_n} \to \widehat{\piorb Y_n}$  that fits into the following commutative diagram. 
\begin{equation}\label{equ: dehn fill orbifold}
\begin{tikzcd}[column sep=large, row sep=large]
  \widehat{\pi_1M}   \arrow[r, two heads, "\widehat{r_n}"] \arrow[d, "\Phi"', "\cong"] & \widehat{\piorb X_n} \arrow[d,"\Psi_n","\cong"']  \\ 
  \widehat{\pi_1N}    \arrow[r, two heads, "\widehat{s_n}"]  & \widehat{\piorb Y_n}
\end{tikzcd}
\end{equation}

Note that $\piorb X_n$ and $\piorb Y_n$ are uniform lattices in $\PSL_2(\C)$  by  \autoref{prop: orbifold good}.
According to \autoref{thm: main for uniform}, $\piorb X_n\cong \piorb Y_n$, and $\Psi_n$ is conjugacy equivalent to the profinite completion of an isomorphism $h_n: \piorb X_n \to \piorb Y_n$. By Mostow's rigidity theorem \cite{Mos68}, $h_n$ is induced, up to conjugacy, by an isometry $H_n: X_n \to Y_n$ with respect to their hyperbolic structures. In particular, $H_n$ preserves the singular loci, which we denote as $\mathscr{S}(X_n)$ and $\mathscr{S}(Y_n)$. 
Let $ \mathscr{S}(X_n)^\epsilon $ and $ \mathscr{S}(Y_n)^\epsilon$ be sufficiently small $\epsilon$-neighbourhoods of the singular loci, which are disjoint unions of (open) singular solid tori. 
Note that when $n\ge 2$, $X_n-\mathscr{S}(X_n)^\epsilon$ reproduces $ M$, and $Y_n-\mathscr{S}(Y_n)^\epsilon$ reproduces  $N$. Hence, $H_n$ restricted to $X_n-\mathscr{S}(X_n)^\epsilon$ yields   a  homeomorphism $F_n: M\to N$ for each $n\ge 2$. In particular, $M$ and $N$ are homeomorphic, and thus $\Gamma=\pi_1M$ and $\Delta=\pi_1N$ are isomorphic. 

According to \cite{Joh79}, the mapping class group of $M$ is finite. Hence, there exists an infinite sequence of distinct integers $ (n_j)_{j\in \N}$, with $n_j\ge 2$, such that the homeomorphisms $F_{n_j}:M\to N$ are isotopic to each other. These homeomorphisms induce a group isomorphism $f: \pi_1M\to \pi_1N$ that is unique up to conjugacy equivalence. Recall that for each $j\in\N$, we have a commutative diagram of maps between manifolds and orbifolds. 
\begin{equation*}
\begin{tikzcd}[column sep=large, row sep=large]
 M \arrow[r, hook,"\mathrm{incl}" ] \arrow[d, " F_{n_j}" ']  & X_n \arrow[d,"H_{n_j}"] \\ 
 N \arrow[r, hook,"\mathrm{incl}" ] & Y_n
\end{tikzcd}
\end{equation*}
Via an identification of basepoints, we deduce that the following diagram commutes up to conjugacy equivalence. 
\begin{equation}\label{equ: up to conj equiv}
\begin{tikzcd}[column sep=tiny, row sep=tiny]
\pi_1M \arrow[rr, "r_{n_j}"] \arrow[dd, "f"'] &        & \piorb X_{n_j} \arrow[dd, "h_{n_j}"] \\
                                          & {\hspace{1mm}\boxed{\simeq} \hspace{-1mm}} &                                      \\
\pi_1N \arrow[rr, "s_{n_j}"]                  &        & \piorb Y_{n_j}                      
\end{tikzcd}
\end{equation}

Now we consider the standard characteristic quotients of $  \pi_1M$ and $ \pi_1N$. 
For any  $m\in \N$, we denote  by  
$\Phi^\flat_{(m)}:  \pi_1M/\pi_1M^{(m)}  \to   \pi_1N/\pi_1N^{(m)}$   the isomorphism between the $m$-th standard characteristic quotients  induced by $\Phi$, as described in \autoref{sec: Standard characteristic quotients}. With $m$ fixed, we can find $j\in \N$ such that the   quotients $\pi_1M\to \pi_1M/\pi_1M^{(m)}$ and $\pi_1N\to \pi_1N/\pi_1N^{(m)}$ factor through $r_{n_j}$ and $s_{n_j}$. In fact, it suffices to choose $$n_j\ge |\pi_1M/\pi_1M^{(m)}| = |\pi_1N/\pi_1N^{(m)}|.$$ In this case, each $\alpha_i^{n_j!}\in \pi_1M^{(m)}$ and  each $\beta_i^{n_j!}\in \pi_1N^{(m)}$, so $K_{n_j}\subseteq \pi_1M^{(m)}$ and $L_{n_j} \subseteq \pi_1N^{(m)}$. 
Let $u: \piorb X_{n_j}= \pi_1M/ K_{n_j}\to  \pi_1M/ \pi_1M^{(m)}$ and $v: \piorb Y_{n_j}= \pi_1N/ L_{n_j} \to \pi_1N/ \pi_1N^{(m)}$ denote the quotient homomorphisms. 
 From  (\ref{equ: dehn fill orbifold}), we derive the following commutative diagram. 
\begin{equation*} 
\begin{tikzcd}[column sep=small, row sep=large]
  \widehat{\pi_1M}   \arrow[rrr, two heads, "\widehat{r_{n_j}}"] \arrow[d, "\Phi"',"\cong" ] & & & \widehat{\piorb X_{n_j}} \arrow[d,"\Psi_{n_j}"',"\cong"]  \arrow[rr, two heads,"\widehat{u}"] & & \pi_1M/\pi_1M^{(m)}   \arrow[d,"\Phi^\flat_{(m)}"',"\cong"] \\ 
 \widehat{\pi_1N}    \arrow[rrr, two heads, "\widehat{s_{n_j}}"]  & & & \widehat{\piorb Y_{n_j}}  \arrow[rr, two heads,"\widehat{v}"] & & \pi_1N/\pi_1N^{(m)} 
\end{tikzcd}
\end{equation*}
Since $\Psi_{n_j}$ is conjugacy equivalent with $\widehat{h_{n_j}}$, combining with   (\ref{equ: up to conj equiv}),  we finally deduce that the following diagram commutes up to conjugacy equivalence. 
\begin{equation*} 
\begin{tikzcd}[column sep=small, row sep=small]
\pi_1M \arrow[rr, "r_{n_j}", two heads] \arrow[dd, "f"',"\cong"] &        & \piorb X_{n_j} \arrow[dd, "h_{n_j}"',"\cong"] \arrow[rr, two heads,"u"] & & \pi_1M/\pi_1M^{(m)} \arrow[dd,"\Phi^\flat_{(m)}"',"\cong"]\\
                                          & {\hspace{2mm}\boxed{\simeq} \hspace{0mm}} &                                     & {\boxed{\simeq}\hspace{-1mm}} & \\
\pi_1N \arrow[rr, "s_{n_j}"]                  &        & \piorb Y_{n_j}                      \arrow[rr, two heads,"v"] & & \pi_1N/\pi_1N^{(m)}
\end{tikzcd}
\end{equation*}
Consequently, $\Phi^\flat_{(m)}\simeq f^\flat_{(m)}$ for every $m\in \N$. As $f$ is fixed,  \autoref{cor: genuine criteria 1} implies that $\Phi$ is genuine. 
\end{proof}

\begin{remark}
Using orbifold Dehn fillings instead of manifold Dehn fillings is important in  our proof. Indeed, some finite quotients of $\pi_1M$ and $\pi_1N$ do not factor through any manifold Dehn fillings. However, all of their finite quotients factor through some orbifold Dehn fillings, so all these quotients can be witnessed from our construction. 
\end{remark}

\subsection{Finishing the proof}
\begin{theorem}\label{thm: main for non-uniform}
Any non-uniform lattice $\Gamma \le \PSL_2(\C)$ is profinitely rigid among all lattices in $\PSL_2(\C)$. In addition, the profinite completion map $\Out(\Gamma)\to \Out(\widehat{\Gamma})$ is an isomorphism. 
\end{theorem}
\begin{proof}
As in the uniform case, the injectivity of $\Out(\Gamma)\to \Out(\widehat{\Gamma})$ has been proven by  \autoref{cor: completion map injective}, so it suffices to show the profinite rigidity of $\Gamma$ and the surjectivity of $\Out(\Gamma)\to \Out(\widehat{\Gamma})$. 

Suppose $\Delta\le \PSL_2(\C)$ is a lattice and $\Phi: \widehat{\Gamma}\to \widehat{\Delta}$ is an isomorphism. It suffices to prove that  $\Gamma\cong \Delta$ and  $\Phi$ is genuine. Let $\Gamma'\le\Gamma$ be a finite-index torsion-free subgroup, and let $\Delta'\le \Delta$ be the $\Phi$-corresponding finite-index subgroup. We denote by $\Phi': \widehat{\Gamma'} \to \widehat{\Delta'}$ the restriction of $\Phi$. According to \autoref{prop: torsion free}, $\Delta'$ is also a torsion-free non-uniform lattice. 
Hence, \autoref{prop: major step for non-uniform} implies that $\Gamma' \cong \Delta'$ and $\Phi'$ is genuine. Applying \autoref{cor: virtually genuine}, we finally deduce that  $\Gamma \cong \Delta$ and that $\Phi$ is genuine.  
\end{proof}

\autoref{mainthm1} and \autoref{mainthm2} are combinations of \autoref{thm: main for uniform} and \autoref{thm: main for non-uniform}. We say a few more words about \autoref{maincor}. 

\begin{cor3mfd}
Suppose that a finitely generated group $\Gamma$ is virtually a lattice in $\PSL_2(\C)$. Then, $\Gamma$ is profinitely rigid among all finitely generated virtually 3-manifold groups. 
\end{cor3mfd}
\begin{proof}
Let $\Delta$ be a finitely generated virtually 3-manifold group, and let $\Phi:\widehat{\Gamma}\to \widehat{\Delta}$ be an isomorphism. According to Scott’s theorem \cite{Sco73}, any finitely generated 3-manifold group is the fundamental group of a compact 3-manifold. Thus, by passing to an orientable finite cover, $\Delta$ is virtually the fundamental group of a compact orientable 3-manifold. Choosing further finite-index subgroups, we can find a $\Phi$-corresponding pair of finite-index normal subgroups $\Gamma'\unlhd \Gamma$ and $\Delta'\unlhd \Delta$ such that $\Gamma'$ is a torsion-free lattice in $\PSL_2(\C)$, and $\Delta'$ is the fundamental group of a compact orientable 3-manifold $N$. 


Indeed, $\Gamma'$ is the fundamental group of a compact orientable 3-manifold $M$ with empty or incompressible toral boundary, whose interior is homeomorphic to the finite-volume hyperbolic 3-manifold $\mathbb{H}^3/ \Gamma'$.  As such, $\widehat{\pi_1M}\cong \widehat{\pi_1N}$. According to \cite[Lemma A.1]{Xu24},  up to capping off the boundary spheres of $N$ by 3-balls,  which does not affect its fundamental group, the manifold $N$ is irreducible  and  has empty or incompressible toral boundary. According to the results of Wilton--Zalesskii \cite{WZ17,WZ17b,WZ19} on the profinite detection of geometrization of 3-manifolds, $\Int(N)$ also admits a complete finite-volume hyperbolic structure. Consequently, $\Delta'$ is a  torsion-free lattice in $\PSL_2(\C)$. From  \autoref{mainthm1} and \autoref{mainthm2}, we deduce that $\Gamma'\cong \Delta'$, and the restriction $\Phi': \widehat{\Gamma'} \to \widehat{\Delta'}$ is  genuine. 

Since $\Gamma'$ and $\Delta'$ are finite-index normal subgroups in $\Gamma$ and $\Delta$, we obtain a commutative diagram between short exact sequences. 
\begin{equation*}
\begin{tikzcd}
1 \arrow[r] & \widehat{\Gamma'} \arrow[r] \arrow[d,"\Phi'"',"\cong"] & \widehat{\Gamma} \arrow[r] \arrow[d,"\Phi"',"\cong"] & \Gamma/\Gamma' \arrow[r] \arrow[d,"\phi"',"\cong"] & 1 \\
 1 \arrow[r] & \widehat{\Delta'} \arrow[r] & \widehat{\Delta} \arrow[r] & \Delta/ \Delta' \arrow[r] & 1
\end{tikzcd}
\end{equation*}
Note that $\Gamma'$ and $\Delta'$ are finitely generated groups with trivial center, and their profinite completions $\widehat{\Gamma'}$ and $\widehat{\Delta'}$ also have trivial center according to \autoref{thm: lattice center free}. In addition, the completion maps $\Out(\Gamma')\to \Out(\widehat{\Gamma'})$ and $\Out(\Delta')\to \Out(\widehat{\Delta'})$ are injective according to \autoref{mainthm2}. Since $\Gamma/\Gamma'$ and $\Delta/\Delta'$ are already finite groups, and $\Phi'$ is genuine, we deduce from \autoref{lem: extension lemma} that $\Phi$ itself is genuine. In particular, $\Gamma\cong \Delta$. 
\end{proof}

\bibliographystyle{abbrv}
\bibliography{main.bib}

\end{document}